\documentclass[11pt]{amsart}
\usepackage{amsmath}
\usepackage{amssymb}
\usepackage{amsthm}
\usepackage{latexsym}
\usepackage{graphicx}
\usepackage{hyperref}
\usepackage{enumerate}
\usepackage[all]{xy}
\usepackage{chngcntr}

\newtheorem{thm}{Theorem}[section]
\newtheorem{cor}[thm]{Corollary}
\newtheorem{lem}[thm]{Lemma}
\newtheorem{prop}[thm]{Proposition}
\newtheorem{defn}[thm]{Definition}

\newtheorem{cond}[thm]{Condition}
\newtheorem{rem}[thm]{Remark}
\newtheorem{thmintro}{Theorem}

\newtheorem{exintro}[thmintro]{Example}
\newtheorem{work}[thm]{Working hypothesis}

\renewcommand{\thethmintro}{\Alph{thmintro}}
\renewcommand{\theequation}{\arabic{section}.\arabic{equation}}
\counterwithin*{equation}{section}

\newcommand{\Z}{\mathbb Z}
\newcommand{\Q}{\mathbb Q}
\newcommand{\R}{\mathbb R}
\newcommand{\C}{\mathbb C}
\newcommand{\mf}{\mathfrak}
\newcommand{\mc}{\mathcal}

\newcommand{\mh}{\mathbb}
\def\Irr{{\rm Irr}}
\newcommand{\mr}{\mathrm}
\newcommand{\ind}{\mathrm{ind}}
\newcommand{\enuma}[1]{\begin{enumerate}[\textup{(}a\textup{)}] {#1} \end{enumerate}}

\newcommand{\cusp}{\mathrm{cusp}}
\newcommand{\nr}{\mathrm{nr}}
\newcommand{\unr}{\mathrm{unr}}

\newcommand{\Rep}{\mathrm{Rep}}
\newcommand{\Res}{\mathrm{Res}}

\newcommand{\der}{\mathrm{der}}
\newcommand{\red}{\mathrm{red}}

\newcommand{\Mod}{\mathrm{Mod}}
\newcommand{\fMod}{\mathrm{Mod}_{\mr f}}
\newcommand{\Hom}{\mathrm{Hom}}
\newcommand{\End}{\mathrm{End}}

\newcommand{\an}{\mathrm{an}}
\newcommand{\me}{\mathrm{me}}
\newcommand{\Modf}[1]{\mathrm{\Mod}_{\mr{f},#1}}

\begin{document}

\title[Endomorphism algebras for $p$-adic groups]{Endomorphism algebras and Hecke algebras\\ 
for reductive $p$-adic groups}
\date{\today}
\thanks{The author is supported by a NWO Vidi grant ``A Hecke algebra approach to the 
local Langlands correspondence" (nr. 639.032.528).}
\maketitle

\begin{center}
{\Large Maarten Solleveld} \\[1mm]
IMAPP, Radboud Universiteit Nijmegen\\
Heyendaalseweg 135, 6525AJ Nijmegen, the Netherlands \\
email: m.solleveld@science.ru.nl 
\end{center}
\vspace{6mm}

\begin{abstract}
Let $G$ be a reductive $p$-adic group and let Rep$(G)^{\mf s}$ be a Bernstein block in the 
category of smooth complex $G$-representations. We investigate the structure of Rep$(G)^{\mf s}$,
by analysing the algebra of $G$-endomorphisms of a progenerator $\Pi$ of that category.

We show that Rep$(G)^{\mf s}$ is ``almost" Morita equivalent with a (twisted) affine Hecke algebra.
This statement is made precise in several ways, most importantly with a family of (twisted)
graded algebras. It entails that, as far as finite length representations are concerned,
Rep$(G)^{\mf s}$ and $\End_G (\Pi)$-Mod can be treated as the module category of a twisted 
affine Hecke algebra.

We draw two major consequences. Firstly, we show that the equivalence of categories between 
Rep$(G)^{\mf s}$ and $\End_G (\Pi)$-Mod preserves temperedness of finite length representations.
Secondly, we provide a classification of the irreducible representations in Rep$(G)^{\mf s}$,
in terms of the complex torus and the finite group canonically associated to Rep$(G)^{\mf s}$.
This proves a version of the ABPS conjecture and enables us to express the set of irreducible 
$G$-representations in terms of the supercuspidal representations of the Levi subgroups of $G$.

Our methods are independent of the existence of types, and apply in complete generality.

In 2023 an appendix was added, to solve a problem with preservation of temperedness in 
Paragraph 9.1.
\end{abstract}
\vspace{5mm}

\tableofcontents

\section*{Introduction}

This paper investigates the structure of Bernstein blocks in the representation theory
of reductive $p$-adic groups. Let $G$ be such a group and let $M$ be a Levi subgroup.
Let $(\sigma,E)$ be a supercuspidal $M$-representation (over $\C$), and let $\mf s$ be its
inertial equivalence class (for $G$). To these data Bernstein associated a block Rep$(G)^{\mf s}$
in the category of smooth $G$-representations Rep$(G)$, see \cite{BeDe,Ren}.

Several questions about Rep$(G)^{\mf s}$ have been avidly studied, for instance:
\begin{itemize}\label{questions}
\item Can one describe Rep$(G)^{\mf s}$ as the module category of an algebra $H$ with an 
explicit presentation?
\item Is there an easy description of temperedness and unitarity of $G$-representations in
terms of $H$?
\item How to classify the set of irreducible representations $\Irr (G)^{\mf s}$?
\item How to classify the discrete series representations in Rep$(G)^{\mf s}$?
\end{itemize}
We note that all these issues have been solved already for $M = G$. In that case the real 
task is to obtain a supercuspidal representation, whereas in this paper we use a given 
$(\sigma,E)$ as starting point.

Most of the time, the above questions have been approached with types, following \cite{BuKu2}.
Given an $\mf s$-type $(K,\lambda)$, there is always a Hecke algebra $\mc H (G,K,\lambda)$
whose module category is equivalent with Rep$(G)^{\mf s}$.
This has been exploited very successfully in many cases, e.g. for $GL_n (F)$ \cite{BuKu1},
for depth zero representations \cite{Mor1,Mor2}, for the principal series of split groups \cite{Roc},
the results on unitarity from \cite{Ciu} and on temperedness from \cite{SolComp}.

However, it is often quite difficult to find a type $(K,\lambda)$, and even if one has it, 
it can be hard to find generators and relations for $\mc H (G,K,\lambda)$. For instance, types
have been constructed for all Bernstein components of classical groups \cite{Ste,MiSt}, but so far
the Hecke algebras of most of these types have not been worked out. Already for the principal
series of unitary $p$-adic groups, this is a difficult task \cite{Bad}. At the moment, it 
seems unfeasible to carry out the full Bushnell--Kutzko program for arbitary Bernstein
components.\\

We follow another approach, which builds more directly on the seminal work of Bernstein. 
We consider a progenerator $\Pi$ of Rep$(G)^{\mf s}$, and the algebra $\End_G (\Pi)$. There is
a natural equivalence from Rep$(G)^{\mf s}$ to the category $\End_G (\Pi)$-Mod of right
$\End_G (\Pi)$-modules, namely $V \mapsto \Hom_G (\Pi,V)$. 

Thus all the above questions can in principle be answered by studying the algebra $\End_G (\Pi)$.
To avoid superfluous complications, we should use a progenerator with an easy shape. 
Fortunately, such an object was already constructed in \cite{BeRu}. Namely, let $M^1$ be
subgroup of $M$ generated by all compact subgroups, write $B = \C [M / M^1]$ and
$E_B = E \otimes_\C B$. The latter is an algebraic version of the integral of the representations
$\sigma \otimes \chi$, where $\chi$ runs through the group $X_\nr (M)$ of unramified characters
of $M$. Then the (normalized) parabolic induction $I_P^G (E_B)$ is a progenerator of 
Rep$(G)^{\mf s}$. In particular we have the equivalence of categories
\[
\begin{array}{cccc}
\mc E : & \Rep (G)^{\mf s} & \longrightarrow & \End_G (I_P^G (E_B))\text{-Mod} \\
& V & \mapsto & \Hom_G (I_P^G (E_B),V)
\end{array}.
\]
For classical groups and inner forms of $GL_n$, the algebras $\End_G (I_P^G (E_B))$
were already analysed by Heiermann \cite{Hei1,Hei2,Hei4}. It turns out that they are isomorphic
to affine Hecke algebras (sometimes extended with a finite group). These results make use
of some special properties of representations of classical groups, which need not hold for
other groups.

We want to study $\End_G (I_P^G (E_B))$ in complete generality, for any Bernstein block of
any connected reductive group over any non-archimedean local field $F$. This entails that we
can only use the abstract properties of the supercuspidal representation $(\sigma,E)$, which
also go into the Bernstein decomposition. A couple of observations about $\End_G (I_P^G (E_B))$
can be made quickly, based on earlier work.
\begin{itemize}
\item The algebra $B$ acts on $E_B$ by $M$-intertwiners, and $I_P^G$ embeds $B$ as a commutative
subalgebra in $\End_G (I_P^G (E_B))$. As a $B$-module, $\End_G (I_P^G (E_B))$ has finite rank
\cite{BeRu,Ren}.
\item Write $\mc O = \{ \sigma \otimes \chi : \chi \in X_\nr (M) \} \subset \Irr (M)$. The group 
$N_G (M)/M$ acts naturally on $\Irr (M)$, and we denote the stabilizer of $\mc O$ in $N_G (M)/M$ 
by $W(M,\mc O)$. By \cite{BeDe}, the centre of $\End_G (I_P^G (E_B))$ is isomorphic to\\
$\C [\mc O / W(M,\mc O)] = \C [\mc O]^{W(M,\mc O)}$.
\item Consider the finite group 
\[
X_\nr (M,\sigma) = \{ \chi_c \in X_\nr (M) : \sigma \otimes \chi_c \cong \sigma \} .
\]
For every $\chi_c \in X_\nr (M,\sigma)$ there exists an $M$-intertwiner $\sigma \otimes \chi
\to \sigma \otimes \chi_c \chi$, which gives rise to an intertwiner $\phi_{\chi_c}$ in
$\End_M (E_B)$ and in $\End_G (I_P^G (E_B))$ \cite{Roc1}.
\item For every $w \in W(M,\mc O)$ there exists an intertwining operator 
\[
I_w (\chi) : I_P^G (\sigma \otimes \chi) \to I_P^G (w(\sigma \otimes \chi)) ,
\]
see \cite{Wal}. However, it is rational as a function of $\chi \in X_\nr (M)$ and in general has 
non-removable singularities, so it does not automatically yield an element of $\End_G (I_P^G (E_B))$. 
\end{itemize}
Based on this knowledge and on \cite{Hei2}, one can expect that $\End_G (I_P^G (E_B))$ has a 
$B$-basis indexed by $X_\nr (M,\sigma) \times W(M,\mc O)$, and that the elements of this 
basis behave somewhat like a group. However, in general things are more subtle than that.\\

\textbf{Main results.}\\
The action of any $w \in W(M,\mc O)$ on $\mc O \cong X_\nr (M) / X_\nr (M,\sigma)$ can be lifted 
to a transformation $\mf w$ of $X_\nr (M)$. Let $W(M,\sigma,X_\nr (M))$ be the group of 
permutations of $X_\nr (M)$ generated by $X_\nr (M,\sigma)$ and the $\mf w$. It satisfies
\[
W(M,\sigma,X_\nr(M)) / X_\nr (M,\sigma) = W(M,\mc O) .
\]
Let $K(B) = \C (X_\nr (M))$ be the quotient field of $B = \C[X_\nr (M)]$. In view of the rationality
of the intertwining operators $I_w$, it is easier to investigate the algebra
\[
\End_G (I_P^G (E_B)) \otimes_B K(B) = \Hom_G \big( I_P^G (E_B), I_P^G (E_B \otimes_B K(B)) \big) .
\]

\begin{thmintro}\label{thm:A} 
\textup{(see Corollary \ref{cor:5.6})} \\
There exist a 2-cocycle $\natural : W(M,\sigma,X_\nr (M))^2 \to \C^\times$ and an algebra isomorphism
\[
\End_G (I_P^G (E_B)) \otimes_B K(B) \cong K(B) \rtimes \C [W(M,\sigma,X_\nr (M)),\natural] .
\]
\end{thmintro}
Here $\C [W(M,\sigma,X_\nr (M)),\natural]$ is a twisted group algebra, it has basis elements 
$\mc T_w$ that multiply as $\mc T_w \mc T_{w'} = \natural (w,w') \mc T_{w w'}$.
The symbol $\rtimes$ denotes a crossed product: as vector space it just means the tensor product,
and the multiplication rules on that are determined by the action of $W(M,\sigma,X_\nr (M))$
on $K(B)$.

Theorem \ref{thm:A} suggests a lot about $\End_G (I_P^G (E_B))$, but the poles of some involved
operators make it impossible to already draw many conclusions about representations. In fact
the operators $\mc T_w$ with $w \in W(M,\mc O)$ involve certain parameters, powers
of the cardinality $q_F$ of the residue field of $F$. If we would manually replace $q_F$ by 1,
then $\End_G (I_P^G (E_B))$ would become isomorphic to $B \rtimes \C [W(M,\sigma,X_\nr (M)),\natural]$.
Of course that is an outrageous thing to do, we just mention it to indicate the relation 
between these two algebras.\\

To formulate our results about $\End_G (I_P^G (E_B))$, we introduce more objects.
The set of roots of $G$ with respect to $M$ contains a root system $\Sigma_{\mc O,\mu}$, namely
the set of roots for which the associated Harish-Chandra $\mu$-function has a zero on $\mc O$
\cite{Hei2}. This induces a semi-direct factorization
\[
W(M,\mc O) = W(\Sigma_{\mc O,\mu}) \rtimes R(\mc O) ,
\]
where $R(\mc O)$ is the stabilizer of the set of positive roots. We may and will assume 
throughout that $\sigma \in \Irr (M)$ is unitary and stabilized by $W(\Sigma_{\mc O,\mu})$.
The Harish-Chandra $\mu$-functions also determine parameter functions 
$\lambda,\lambda^* : \Sigma_{\mc O,\mu} \to \R_{\geq 0}$. The values $\lambda (\alpha)$ and 
$\lambda^* (\alpha)$ encode in a simple way for which $\chi \in X_\nr (M)$ the normalized parabolic 
induction $I_{M (P \cap M_\alpha}^{M_\alpha} (\sigma \otimes \chi)$ becomes reducible, see 
\eqref{eq:3.22} and \eqref{eq:9.9}. (Here $M_\alpha$ denotes the Levi subgroup of $G$ generated by 
$M$ and the root subgroups $U_\alpha, U_{-\alpha}$.)

To the data $\mc O, \Sigma_{\mc O,\mu}, \lambda, \lambda^*,q_F^{1/2}$ one can associate 
an affine Hecke algebra, which we denote in this introduction by
$\mc H \big( \mc O,\Sigma_{\mc O,\mu},\lambda,\lambda^*,q_F^{1/2} \big)$. There is a large subalgebra
\[
\End_G^\circ (I_P^G (E_B)) \; \subset \; \End_G (I_P^G (E_B))
\] 
such that the categories of finite length right modules of $\mc H \big(\mc O,\Sigma_{\mc O,\mu},
\lambda,\lambda^*,q_F^{1/2} \big)$ and of $\End_G^\circ (I_P^G (E_B))$ are equivalent. Suppose that 
$\natural$ descends to a 2-cocycle $\tilde \natural$ of $R(\mc O)$. Then the crossed product
\[
\tilde{\mc H}(\mc O) := \mc H \big( \mc O,\Sigma_{\mc O,\mu},\lambda,\lambda^* ,q_F^{1/2} \big) 
\rtimes \C [R(\mc O),\tilde{\natural}] 
\]
is a twisted affine Hecke algebra \cite[\S 2.1]{AMS3}. It is reasonable to expect that 
$\End_G (I_P^G (E_B))$ is Morita equivalent with $\tilde{\mc H}(\mc O)$. Indeed this 
is ``almost" true, and in important cases known already.
\begin{itemize} \label{importantCases}
\item The group $R(\mc O)$ is always trivial for $GL_n (F)$ \cite{BuKu1}, for inner forms of general
linear groups \cite{SeSt,Hei2} and for unipotent representations \cite{Lus-Uni,SolRamif}.
\item The 2-cocycle $\natural$ of $W(M,\sigma,X_\nr (M))$ is trivial for symplectic groups and special 
orthogonal groups \cite{Hei2} and for principal series representatons of split groups \cite{Roc}.
\end{itemize}
In these cases all the involved 2-cocycles are trivial, and there are equivalences of categories
\[
\Rep (G)^{\mf s} \;\cong\; \End_G (I_P^G (E_B))-\Mod \;\cong\; 
\mc H \big( \mc O,\Sigma_{\mc O,\mu},\lambda,\lambda^* ,q_F^{1/2} \big) \rtimes R(\mc O) -\Mod .
\]
However, examples with inner forms of $SL_n$ \cite{ABPS1} suggest that such a
Morita equivalence for $\End_G (I_P^G (E_B))$ might not hold for arbitrary groups.
It is conceivable that the 2-cocycles are always trivial for (quasi-)split reductive $F$-groups, 
but we would not know how to prove that.

In our completely general setting, we shall need to decompose $\End_G (I_P^G (E_B))$-modules according 
to their $B$-weights (which live in $X_\nr (M)$). The existence of such a decomposition cannot be 
guaranteed for representations of infinite length, and therefore we stick to finite length in most 
of the paper. All the algebras we consider have a large centre, so that every finite length module
actually has finite dimension. For Rep$(G)^{\mf s}$ ``finite length" is equivalent to ``admissible",
and we denote the corresponding subcategory by $\Rep_{\mr f}(G)^{\mf s}$.

It is known from \cite{Lus-Gr,AMS3} that the category of finite dimensional right modules
$\tilde{\mc H}(\mc O)-\fMod$ can be described with a family of (twisted) graded Hecke 
algebras. Write $X_\nr^+ (M) = \Hom (M/M^1,\R_{>0})$ and note that its 
Lie algebra is $a_M^* = \Hom (M / M^1 ,\R)$. For a unitary $u \in X_\nr (M)$, there is a 
graded Hecke algebra $\mh H_u$, built from the following data: the tangent space
$a_M^* \otimes_\R \C$ of $X_\nr (M)$ at $u$, a root subsystem $\Sigma_{\sigma \otimes u}
\subset \Sigma_{\mc O,\mu}$ and a parameter function $k^u_\alpha$ induced by 
$\lambda$ and $\lambda^*$. Further $W(M,\mc O)_{\sigma \otimes u}$ decomposes as 
$W(\Sigma_{\sigma \otimes u}) \rtimes R(\sigma \otimes u)$, and $\natural$ induces a
2-cocycle of the local R-group $R(\sigma \otimes u)$. This yields a twisted graded Hecke
algebra $\mh H_u \rtimes \C [R(\sigma \otimes u),\natural_u]$ \cite[\S 1]{AMS3}. 

We remark that these algebras depend mainly on the variety $\mc O$ and the group $W(M,\mc O)$.
Only the subsidiary data $k^u$ and $\natural_u$ take the internal structure of the
representations $\sigma \otimes \chi \in \mc O$ into account. The parameters $k^u_\alpha$,
depend only on the poles of the Harish-Chandra $\mu$-function (associated to $\alpha$)
on $\{ \sigma \otimes u \chi : \chi \in X_\nr^+ (M) \}$. It is not clear to us whether,
for a given $\sigma \otimes u$, they can be effectively computed in that way, further
investigations are required there. 

We do not know whether a 2-cocycle $\tilde \natural$ as used in $\tilde{\mc H}(\mc O)$ always 
exists. Fortunately, the description of $\End_G (I_P^G (E_B))-\fMod$ found via affine Hecke 
algebras turns out to be valid anyway.

\begin{thmintro}\label{thm:B} 
\textup{(see Corollaries \ref{cor:8.4} and \ref{cor:9.9})} \\
For any unitary $u \in X_\nr (M)$ there are equivalences between the following categories:
\begin{enumerate}[(i)]
\item representations in $\Rep_{\mr f} (G)^{\mf s}$ with cuspidal support in\\
$W(M,\mc O) \{ \sigma \otimes u \chi : \chi \in X_\nr^+ (M) \}$;
\item modules in $\End_G (I_P^G (E_B))-\fMod$ with all their $B$-weights in \\
$W(M,\sigma,X_\nr (M)) u X_\nr^+ (M)$;
\item modules in $\mh H (\tilde{\mc R}_u, W(M,\mc O)_{\sigma \otimes u}, k^u, 
\natural_u ) -\fMod$ with all their $\C[a_M^* \otimes_\R \C]$-weights in $a_M^*$.
\end{enumerate}
These equivalences commute with parabolic induction and Jacquet restriction (which for (ii)
and (iii) are just induction and restriction between the appropriate algebras).

Futhermore, suppose that there exists a 2-cocycle $\tilde \natural$ on
$R(\mc O) \cong W(M,\mc O ) / W(\Sigma_{\mc O,\mu})$ which on each subgroup 
$W(M,\mc O)_{\sigma \otimes u}$ is cohomologous to $\natural_u$. Then the above equivalences,
for all unitary $u \in X_\nr (M)$, combine to an equivalence of categories
\[
\End_G (I_P^G (E_B))-\fMod \; \longrightarrow \; \tilde{\mc H}(\mc O)-\fMod .
\]
Via $\mc E$, the left hand side is always equivalent with $\Rep_{\mr f}(G)^{\mf s}$.
\end{thmintro}

We stress that Theorem \ref{thm:B} holds for all Bernstein blocks of all reductive $p$-adic
groups. It provides a good substitute for types, when those are not available
or too complicated. The use of graded (instead of affine) Hecke algebras is only a small concession, 
since the standard approaches to the representation theory of affine Hecke algebras with unequal 
parameters run via graded Hecke algebras anyway. 

Let us point out that on the Galois side of the local Langlands correspondence,
analogous structures exist. Indeed, in \cite{AMS1,AMS2,AMS3} twisted graded Hecke algebras and
a twisted affine Hecke algebra were associated to every Bernstein component in the space of
enhanced L-parameters. By comparing twisted graded Hecke algebras on both sides of the local
Langlands correspondence, it might be possible to establish new cases of that correspondence.
 
For representations of $\End_G (I_P^G (E_B))$ and $\mh H_u \rtimes \C[R(\sigma \otimes u),
\natural_u]$ there are natural notions of temperedness and essentially discrete series,
which mimic those for affine Hecke algebras \cite{Opd-Sp}. The next result generalizes \cite{Hei3}.

\begin{thmintro}\label{thm:C} 
\textup{(see Theorem \ref{thm:9.2} and Proposition \ref{prop:9.3})} \\
Choose the parabolic subgroup $P$ with Levi factor $M$ as indicated by Lemma \ref{lem:9.1}.
Then all the equivalences of categories in Theorem \ref{thm:B} preserve temperedness.

Suppose that $\Sigma_{\mc O,\mu}$ has full rank in the set of roots of $(G,M)$. Then
these equivalences send essentially square-integrable representations in (i) to
essentially discrete series representations in (ii), and the other way round.

Suppose $\Sigma_{\sigma \otimes u}$ has full rank in the set of roots of $(G,M)$, for a
fixed unitary $u \in X_\nr (M)$. Then the equivalences in Theorem \ref{thm:B}, for that $u$,
send essentially square-integrable representations in (i) to essentially discrete series
representations in (iii), and conversely.
\end{thmintro}

Now that we have a good understanding of $\End_G (I_P^G (E_B))$, its finite dimensional
representations and their properties, we turn to the remaining pressing issue from page
\pageref{questions}: can one classify the involved irreducible representations?
This is indeed possible, because graded Hecke algebras have been studied extensively, see e.g.
\cite{BaMo1,BaMo2,CiOpTr,Eve,SolHomGHA,SolGHA,SolHomAHA}. The answer depends in a 
well-understood but involved and subtle way on the parameter functions $\lambda,\lambda^*,k^u$. 

With the methods in this paper, it is difficult to really compute the parameter functions
$\lambda$ and $\lambda^*$. Whenever a type $(K,\tau)$ and an associated Hecke algebra 
$\mc H (G,K,\tau)$ for $\Rep (G)^{\mf s}$ are known, $\mc H (G,K,\tau)$ is Morita equivalent
with $\End_G (I_P^G (E_B))$. In that case the values $q_F^{\lambda (\alpha)}$ and 
$q_F^{\lambda^* (\alpha)}$ can be read off from $\mc H (G,K,\tau)$, because they only depend
on the reducibility of certain parabolically induced representations and those properties
are preserved by a Morita equivalence. But, that does not cover all cases. 

We expect that the functorial properties of the progenerators $I_P^G (E_B)$ enable us to reduce 
the computation of $\lambda (\alpha), \lambda^* (\alpha)$ to cases where $G$ is simple and 
adjoint or simply connected. Thus it may be possible to prove that the parameter functions
$\lambda,\lambda^*$ are integers and of ``geometric type", as Lusztig conjectured in \cite{Lus2}.
We work that out in the sequel \cite{SolParam} to this paper.

The classification of $\Irr (G)^{\mf s}$ becomes more tractable if we just want to understand
$\Irr \big( \End_G (I_P^G (E_B)) \big)$ and $\Irr (\mh H_u \rtimes \C[R(\sigma \otimes u,\natural_u])$
as sets, and allow ourselves to slightly adjust the weights (with respect to respectively
$B$ and $\C[a_M^* \otimes_\R \C]$) in the bookkeeping. 
Then we can investigate $\Irr (\mh H_u \rtimes \C[R(\sigma \otimes u,\natural_u])$ via
the change of parameters $k^u \to 0$, like in \cite{SolAHA,SolHecke}. That replaces
\[
\mh H_u \rtimes \C[R(\sigma \otimes u),\natural_u] \qquad \text{by} \qquad
\C[a_M^* \otimes_\R \C] \rtimes \C[W(M,\mc O)_{\sigma \otimes u},\natural_u], 
\]
for which Clifford theory classifies the irreducible representations.

\begin{thmintro}\label{thm:D} 
\textup{(see Theorem \ref{thm:9.4})} \\
There exists a bijection
\[
\zeta \circ \mc E :\Irr (G)^{\mf s} \longrightarrow 
\Irr \big( \C [X_\nr (M)] \rtimes \C [W(M,\sigma,X_\nr(M)),\natural] \big)
\]
such that, for $\pi \in \Irr (G)^{\mf s}$ and a unitary $u \in X_\unr (M)$:
\begin{itemize}
\item the cuspidal support of $\zeta \circ \mc E (\pi)$ lies in $W(M,\mc O) u X_\nr^+ (M)$ if
and only if all the $\C[X_\nr (M)]$-weights of $(\zeta \circ \mc E (\pi))$ lie in 
$W(M,\sigma,X_\nr(M)) u X_\nr^+ (M)$,
\item $\pi$ is tempered if and only if all the $\C[X_\nr (M)]$-weights of 
$(\zeta \circ \mc E) (\pi)$ are unitary.
\end{itemize}\vspace{-2mm}
\end{thmintro}

Notice that on the right hand side the parameter functions $\lambda,\lambda^*$ and $k^u$
are no longer involved. Recall that in the important cases mentioned on page \pageref{importantCases},
$\natural$ is trivial. Then Theorem \ref{thm:D} and standard Morita equivalences provide bijections
\[
\Irr (G)^{\mf s} \longrightarrow \Irr \big( \C [X_\nr (M)] \rtimes W(M,\sigma,X_\nr(M)) \big)
\longrightarrow \Irr \big( \C [\mc O] \rtimes W(M,\mc O) \big) .
\]
Clifford theory identifies $\Irr ( \C [\mc O] \rtimes W(M,\mc O) )$ with the extended quotient 
\[
\mc O /\!/ W(M,\mc O) = \{ (\chi,\rho) : \chi \in \mc O, \rho \in \Irr (W(M,\mc O)_\chi) \} / W(M,\mc O) .
\] 
For $GL_n (F)$ such a bijection between $\Irr (G)^{\mf s}$ and $\mc O /\!/ W(M,\mc O)$ was already 
known from \cite{BrPl}, and for principal series representations of split groups from \cite{ABPS3,ABPS4}.
In general, in the language of \cite{ABPS}, 
\[
\Irr \big( \C [X_\nr (M)] \rtimes \C [W(M,\sigma,X_\nr(M)),\natural] \big)
\] 
is a twisted extended quotient $(\mc O /\!/ W(M,\mc O) )_\natural$. With that interpretation 
Theorem \ref{thm:D} proves a version of the ABPS conjecture \cite[\S 2.3]{ABPS} and:

\begin{thmintro}\label{thm:E}
\textup{(see Theorem \ref{thm:9.8})} \\
Theorem \ref{thm:D} (for all possible $\mf s = [M,\sigma]_G$ together) yields a bijection
\[
\Irr (G) \longrightarrow 
\bigsqcup\nolimits_M  \big( \Irr_\cusp (M) /\!/ (N_G(M)/M) \big)_\natural ,
\]
where $M$ runs over a set of representatives for the conjugacy classes of Levi subgroups of $G$
and $\Irr (M)_{\mr{cusp}}$ denotes the set of irreducible supercuspidal $M$-representations.
\end{thmintro}
It is quite surprising that such a simple relation between the space of irreducible representations
of an arbitrary reductive $p$-adic group and the supercuspidal representations of its Levi 
subgroups holds.\\

We note that Theorem \ref{thm:D} is about right modules of the involved algebra. If we insist on
left modules we must use the opposite algebra, which is isomorphic to
$\C [X_\nr (M)] \rtimes \C [W(M,\sigma,X_\nr(M)),\natural^{-1}]$. Then we would get the twisted
extended quotient $(\mc O /\!/ W(M,\mc O) )_{\natural^{-1}}$. 

The only noncanonical ingredient in Theorem \ref{thm:D} is the 
2-cocycle $\natural$. It is trivial on $W(\Sigma_{\mc O,\mu})$, but apart from that it depends
on some choices of $M$-isomorphisms $w(\sigma \otimes \chi) \to \sigma \otimes \chi'$
for $w \in R(\mc O)$ and $\chi, \chi' \in X_\nr (M)$. From Theorem \ref{thm:B} we see that the 
restrictions $\natural_u$ of $\natural$ have a definite effect on the involved module categories. 

Moreover, by \eqref{eq:8.11} $\natural_u^{-1}$ must be cohomologous to
a 2-cocycle obtained from the Hecke algebra of an $\mf s$-type (if such a type exists). 
This entails that in many cases $\natural_u$ must be trivial. At the same time, this argument
shows that in some instances, like \cite[Example 5.5]{ABPS1} and Example \ref{ex:2.E}, the 
2-cocycles $\natural_u$ and $\natural$ are cohomologically nontrivial. It would be interesting 
if $\natural$ could be related to the way $G$ is realized as an inner twist of a quasi-split 
$F$-group, like in \cite{HiSa}.\\

Besides $I_P^G (E_B)$, a smaller progenerator of $\Rep (G)^{\mf s}$ is available. Namely,
let $E_1$ be an irreducible subrepresentation of $\Res^M_{M^1} (E)$ and build
$I_P^G (\ind_{M^1}^M (E_1))$. We investigate the Morita equivalent subalgebra 
\[
\End_G \big( I_P^G (\ind_{M^1}^M (E_1)) \big) \; \subset \; \End_G \big( I_P^G (E_B) \big)
\] 
as well, because it should be even closer to an affine Hecke algebra. 

Unfortunately this turns out to be difficult, and we unable to make progress without further 
assumptions. In the large majority of cases, the restriction of $(\sigma,E)$ to $M^1$ decomposes 
without multiplicities bigger than one. (That does not always hold though, see Example \ref{ex:2.E}.)
With such multiplicity one as working hypothesis, we can slightly improve on Theorem \ref{thm:B}.

\begin{thmintro}\label{thm:F}
Suppose that the multiplicity of $E_1$ in $\Res^M_{M^1} (E)$ is one. There exists a 2-cocycle
$\natural_J : W(M,\mc O)^2 \to \C [\mc O]^\times$ and an algebra isomorphism
\[
\End_G \big( I_P^G (\ind_{M^1}^M (E_1)) \big) \cong \mc H \big( \mc O,\Sigma_{\mc O,\mu},
\lambda,\lambda^*,q_F^{1/2} \big) \rtimes \C [R (\mc O), \natural_J] .
\]
On the right hand side the first factor is a subalgebra but the second factor need not be.
The basis elements $J_r$ with $r \in R (\mc O)$ have products
\[
J_r J_{r'} = \natural_J (r,r') J_{r r'} \in \C [\mc O]^\times J_{r r'} .
\]
\end{thmintro}

Thus, the price we pay for the smaller progenerator $I_P^G (\ind_{M^1}^M (E_1))$ consists of
more complicated intertwining operators from the R-group $R(\mc O)$. In concrete cases this may be 
resolved by an explicit analysis of $R(\mc O)$. In general Theorem \ref{thm:10.8} could be useful
to say something about the relation between unitarity in $\Rep (G)^{\mf s}$ and unitarity in
$\End_G \big( I_P^G (\ind_{M^1}^M (E_1)) \big) -\Mod$.\\

\textbf{Structure of the paper.}\\
Most results about endomorphism algebras of progenerators in the cuspidal case ($M = G$)
are contained in Section \ref{sec:cuspidal}. A substantial part of this was already shown in
\cite{Roc1}, we push it further to describe $\End_M (E_B)$ better. 
Section \ref{sec:root} is elementary, its main purpose is to introduce some useful objects.

Harish-Chandra's intertwining operators $J_{P'|P}$ play the main role in Section \ref{sec:intertwining}.
We study their poles and devise several auxiliary operators to fit $J_{w(P)|P}$ into
$\Hom_G \big( I_P^G (E_B), I_P^G (E_B \otimes_B K(B)) \big)$. 
The actual analysis of that algebra is carried out in Section \ref{sec:rational}. First we express
it in terms of operators $A_w$ for $w \in W(M,\mc O)$, which are made by composing the
$J_{P'|P}$ with suitable auxiliary operators. Next we adjust the $A_w$ to $\mc T_w$ and we
prove Theorem \ref{thm:A}. Sections \ref{sec:cuspidal}--\ref{sec:rational}
are strongly influenced by \cite{Hei2}, where similar results were established in the
(simpler) case of classical groups.

At this point Lemma \ref{lem:6.1} forces us to admit that in general $\End_G (I_P^G (E_B))$
probably does not have a nice presentation. To pursue the analysis of this algebra, we localize it
on relatively small subsets $U$ of $X_\nr (M)$. In this way we get rid of $X_\nr (M,\sigma)$
from the intertwining group $W(M,\sigma,X_\nr (M))$, and several issues simplify.
For maximal effect, we localize with analytic rather than polynomial functions on $U$ -- after
checking (in Section \ref{sec:localization}) that it does not make a difference as far as
finite dimensional modules are concerned. We show that the localization of $\End_G (I_P^G (E_B))$
at $U$, extended with the algebra $C^\me (U)$ of meromorphic functions on $U$, is isomorphic
to a crossed product
$C^\me (U) \rtimes \C[W(M,\mc O)_{\sigma \otimes u}, \natural_u]$.

A presentation of the analytic localization of $\End_G (I_P^G (E_B))$ at $U$ is obtained in
Theorem \ref{thm:7.4}: it is almost Morita equivalent to an affine Hecke algebra. The only 
difference is that the standard large commutative subalgebra of that affine Hecke algebra must 
be replaced by the algebra of analytic functions on $U$. 

This presentation makes it possible to relate the localized version of $\End_G (I_P^G (E_B))$ to
the localized version of a suitable graded Hecke algebra. We do that in Section \ref{sec:GHA},
thus proving the first half of Theorem \ref{thm:B}. In Section \ref{sec:class} we translate that
to a classification of $\Irr (G)^{\mf s}$ in terms of graded Hecke algebras. Next we study the
change of parameters $k^u \to 0$ in graded Hecke algebras, and derive the larger part of Theorem 
\ref{thm:D}. All considerations about temperedness can be found in Section \ref{sec:temp}. 
There we finish the proofs of Theorems \ref{thm:B}, \ref{thm:C}, \ref{thm:D} and \ref{thm:E}.

Finally, in Section \ref{sec:smaller} we study the smaller progenerator $I_P^G (\ind_{M^1}^M (E_1))$.
Varying on earlier results, we establish Theorem \ref{thm:F}.\\

\textbf{Acknowledgement.}\\
We thank for George Lusztig for some helpful remarks on the first version of this paper.
We are also grateful to the referees for their work and their useful comments, especially for 
pointing out Example \ref{ex:2.E}.

\renewcommand{\thethmintro}{\arabic{section}.\Alph{thmintro}}

\section{Notations}
\label{sec:notations}

We introduce some of the notations that will be used throughout the paper.\\
$F$: a non-archimedean local field\\
$\mc G$: a connected reductive $F$-group\\
$\mc P$: a parabolic $F$-subgroup of $\mc G$\\
$\mc M$: a $F$-Levi factor of $\mc P$\\
$\mc U$: the unipotent radical of $\mc P$\\
$\overline{\mc P}$: the parabolic subgroup of $\mc G$ that is opposite to $\mc P$ with respect to $\mc M$\\
$G = \mc G (F)$ (and $M = \mc M (F)$ etc.): the group of $F$-rational points of $\mc G$\\
We often abbreviate the above situation to: $P = MU$ is a parabolic subgroup of $G$\\

\noindent
$\Rep (G)$: the category of smooth $G$-representations (always on $\C$-vector spaces)\\
$\Rep_{\mr f}(G)$: the subcategory of finite length representations\\
$\Irr (G)$: the set of irreducible smooth $G$-representations up to isomorphism\\
$I_P^G : \Rep (M) \to \Rep (G)$: the normalized parabolic induction functor\\
$X_\nr (M)$: the group of unramified characters of $M$, with its structure as a complex algebraic torus\\
$M^1 = \bigcap_{\chi \in X_\nr (M)} \ker \chi$\\

$\Irr (M)_\cusp$ subset of supercuspidal representations in $\Irr (M)$\\
$(\sigma,E)$: an element of $\Irr_\cusp (M)$\\
$\mc O = [M,\sigma]_M$: the inertial equivalence class of $\sigma$ for $M$, that is, the subset of
$\Irr (M)$ consisting of the $\sigma \otimes \chi$ with $\chi \in X_\nr (M)$\\
$\Rep (M)^{\mc O}$: the Bernstein block of $\Rep (M)$ associated to $\mc O$\\
$\mf s = [M,\sigma]_G$: the inertial equivalence class of $(M,\sigma)$ for $G$\\
$\Rep (G)^{\mf s}$: the Bernstein block of $\Rep (G)$ associated to $\mf s$\\
$\Irr (G)^{\mf s} = \Irr (G) \cap \Rep (G)^{\mf s}$\\

$W(G,M) = N_G (M) / M$\\
$N_G (M)$ acts on $\Rep (M)$ by $(g \cdot \pi)(m) = \pi (g^{-1} m g)$. 
This induces an action of $W(G,M)$ on $\Irr (M)$\\
$N_G (M,\mc O) = \{ g \in N_G (M) : g \cdot \sigma \cong \sigma \otimes \chi$ for some 
$\chi \in X_\nr (M) \}$\\
$W(M,\mc O) = N_G (M,\mc O) / M = \{ w \in W(G,M) : w \cdot \sigma \in \mc O \}$\\

$X_\nr (M,\sigma) = \{ \chi \in X_\nr (M) : \sigma \otimes \chi \cong \sigma \}$\\
$B = \C [X_\nr (M)]$: the ring of regular functions on the complex
algebraic torus $X_\nr (M)$\\
$K(B) = \C (X_\nr (M))$: the quotient field of $B$, the field of
rational functions on $X_\nr (M)$\\
The covering map 
\[
X_\nr (M) \to \mc O : \chi \mapsto \sigma \otimes \chi
\]
induces a bijection $X_\nr (M) / X_\nr (M,\sigma) \to \mc O$. In this way we regard $\mc O$ as a complex
algebraic variety.
We define $\C [X_\nr (M) / X_\nr (M,\sigma)], \C [\mc O]$ and $\C (X_\nr (M) / X_\nr (M,\sigma))$,
$\C (\mc O)$ like $B$ and $K(B)$.

\section{Endomorphism algebras for cuspidal representations}
\label{sec:cuspidal}

This section relies largely on \cite{Roc1}. Let 
\[
\ind_{M^1}^M : \Rep (M^1) \to \Rep (M) 
\]
be the functor of smooth, compactly supported induction. We realize it as
\begin{multline*}
\ind_{M^1}^M (\pi,V) = \{ f : M \to V \mid \pi (m_1) f(m) = f(m_1 m) \; \forall m \in M, m_1 \in M^1, \\
\mr{supp}(f) / M^1 \text{ is compact} \} ,
\end{multline*}
with the $M$-action by right translation. (Smoothness of $f$ is automatic because $M^1$ is open in $M$.)

Regard $(\sigma,E)$ as a representation of $M^1$, by restriction. Bernstein \cite[\S II.3.3]{BeRu} 
showed that $\ind_{M^1}^M (\sigma,E)$ is a progenerator of $\Rep (M)^{\mc O}$. This entails that
\[
V \mapsto \Hom_M \big( \ind_{M^1}^M (E),V \big)
\]
is an equivalence between $\Rep (M)^{\mc O}$ and the category 
$\End_M \big(\ind_{M^1}^M (\sigma,E)\big)-\Mod$ of right modules over the $M$-endomorphism algebra of 
$\ind_{M^1}^M (\sigma,E)$, see \cite[Theorem 1.5.3.1]{Roc1}.
We want to analyse the structure of $\End_M \big(\ind_{M^1}^M (\sigma,E)\big)$.

For $m \in M$, let $b_m \in \C [X_\nr (M)]$ be the element given by evaluating unramified characters at $m$.
We let $m$ act on $\C [X_\nr (M)]$ by
\[
m \cdot b = b_m b \qquad b \in \C [X_\nr (M)] .
\]
Then specialization/evaluation at $\chi \in X_\nr (M)$ is an $M$-homomorphism
\[
\mr{sp}_\chi : \C [X_\nr (M)] \to (\chi,\C) .
\]
Let $\delta_m \in \ind_{M^1}^M (\C)$ be the function which is 1 on $m M^1$ and zero on the rest of $M$.
Let $\C [M/M^1]$ be the group algebra of $M/M^1$, considered as the left regular representation of $M / M^1$. 
There are canonical isomorphisms of $M$-representations
\begin{equation}\label{eq:2.1}
\begin{array}{ccccc}
\C [X_\nr (M)] & \to & \C [M/ M^1] & \to & \ind_{M^1}^M (\C) \\
b_m & \mapsto & m M^1 & \mapsto & \delta_{m^{-1}} 
\end{array}.
\end{equation}
We endow $E \otimes_\C \ind_{M^1}^M (\C)$ with the tensor product of the $M$-representations
$\sigma$ and $\ind_{M^1}^M (\mr{triv})$. There is an isomorphism of $M$-representations
\begin{equation}\label{eq:2.2}
\begin{array}{ccc}
E \otimes_\C \ind_{M^1}^M (\C) & \cong & \ind_{M^1}^M (E) \\
e \otimes f & \mapsto & [v_{e \otimes f} : m \mapsto f(m) \sigma (m) e] \\
\sum_{m \in M / M^1} \sigma (m^{-1}) v(m) \otimes \delta_m & \text{\reflectbox{$\mapsto$}} & v 
\end{array}.
\end{equation}
Composing \eqref{eq:2.1} and \eqref{eq:2.2}, we obtain an isomorphism 
\begin{equation}\label{eq:2.3}
\begin{array}{ccc}
\ind_{M^1}^M (E) & \to & E \otimes_\C \C[X_\nr (M)] \\
v & \mapsto & \sum_{m \in M / M^1} \sigma (m) v(m^{-1}) \otimes b_m \\ 
v_{e \otimes \delta_m^{-1}} & \text{\reflectbox{$\mapsto$}} & e \otimes b_m
\end{array} .
\end{equation}
With \eqref{eq:2.3}, specialization at $\chi \in X_\nr (M)$ becomes a $M$-homomorphism
\begin{equation}\label{eq:2.spec}
\mr{sp}_\chi : \ind_{M^1}^M (\sigma,E) \to (\sigma \otimes \chi,E) .
\end{equation}
As $M / M^1$ is commutative, the $M$-action on $E \otimes_\C \C[X_\nr (M)]$ is $\C[X_\nr (M)]$-linear.
Via \eqref{eq:2.3} we obtain an embedding 
\begin{equation}\label{eq:2.4}
\C [X_\nr (M)] \to \End_M \big(\ind_{M^1}^M (\sigma,E)\big) .
\end{equation}
For a basis element $b_m \in \C [X_\nr (M)]$ and any $v \in \ind_{M^1}^M (E)$, it works out as
\begin{equation}\label{eq:2.5}
(b_m \cdot v)(m') = \sigma (m^{-1}) v (m m') . 
\end{equation}
For any $\chi_c \in X_\nr (M)$ we can define a linear bijection 
\begin{equation}\label{eq:2.25}
\begin{array}{cccc}
\rho_{\chi_c} : & \C [X_\nr (M)] & \to & \C [X_\nr (M)] \\
& b & \mapsto & [b_{\chi_c} : \chi \mapsto b (\chi \chi_c)] 
\end{array} .
\end{equation}
This provides an $M$-isomorphism 
\[
\mr{id}_E \otimes \rho_{\chi_c} : \ind_{M^1}^M (\sigma) \to \ind_{M^1}^M (\sigma \otimes \chi_c) .
\]
Let $(\sigma_1,E_1)$ be an irreducible subrepresentation of $\Res_{M^1}^M (\sigma,E)$, such that
the stabilizer of the subspace $E_1 \subset E$ in $M$ is maximal. We denote the multiplicity of
$\sigma_1$ in $\sigma$ by $\mu_{\sigma,1}$. Every other irreducible
$M^1$-subrepresentation of $\sigma$ is isomorphic to $m \cdot \sigma_1$ for some $m \in M$, and
$\sigma (m^{-1}) E_1$ is a space that affords $m \cdot \sigma_1$. Hence $\mu_{\sigma,1}$ depends
only on $\sigma$ and not on the choice of $(\sigma_1,E_1)$. (But note that, if $\mu_{\sigma,1} > 1$,
not every $M^1$-subrepresentation of $E$ isomorphic to $\sigma_1$ equals $\sigma (m^{-1}) E_1$
for an $m \in M$.)

Following \cite[\S 1.6]{Roc1} we consider the groups
\[
\begin{array}{lll}
M_\sigma^2 & = & \bigcap_{\chi \in X_\nr (M,\sigma)} \ker \chi , \\ 
M_\sigma^3 & = & \{ m \in M : \sigma (m) E_1 = E_1 \} , \\
M_\sigma^4 & = & \{ m \in M :m \cdot \sigma_1 \cong \sigma_1 \} .
\end{array}
\]
Notice that $X_\nr (M,\sigma) = \Irr (M / M_\sigma^2)$. There is a sequence of inclusions
\begin{equation}\label{eq:2.6}
M^1 \subset M_\sigma^2 \subset M_\sigma^3 \subset M_\sigma^4 \subset M . 
\end{equation}
Since $M^1$ is a normal subgroup of $M$ and $M / M^1$ is abelian, all these groups are normal in $M$.
By this normality, for any $m' \in M$:
\begin{equation}\label{eq:2.34}
\begin{array}{ccl}
M_\sigma^3 & = & \{ m \in M : \sigma (m) \sigma (m') E_1 = \sigma (m') E_1 \} , \\
M_\sigma^4 & = & \{ m \in M : m \cdot (m' \cdot \sigma_1) \cong m' \cdot \sigma_1 \} .
\end{array}
\end{equation}
In other words, $M_\sigma^4$ consists of the $m \in M$ that stabilize the isomorphism class of 
one (or equivalently any) irreducible $M^1$-subrepresentation of $\sigma$. In particular
$M_\sigma^2$ and $M_\sigma^4$ only depend on $\sigma$. On the other hand, it seems possible that 
$M_\sigma^3$ does depend on the choice of $E_1$.

Furthermore $[M : M_\sigma^4]$ equals the number of inequivalent irreducible constituents of
$\Res_{M^1}^M (\sigma)$ and, like \eqref{eq:2.2},
\[
\ind_{M^1}^{M_\sigma^2} (\C) \cong \C [X_\nr (M) / X_\nr (M,\sigma)] . 
\]
By \cite[Lemma 1.6.3.1]{Roc1}
\begin{equation}\label{eq:2.9}
[M_\sigma^4 : M_\sigma^3] = [M_\sigma^3 : M_\sigma^2] = \mu_{\sigma,1} . 
\end{equation}
\begin{exintro}\label{ex:2.E}
Consider the group called $G$ in \cite[\S 4.4]{HeVi}. This is a connected reductive $p$-adic
group, an extension of a four-dimensional torus (of split rank two) by the norm one elements
in the multiplicative group of a division algebra. Moreover $G$ is compact modulo centre,
it has a unique maximal compact subgroup $K = G^1$ and 
\[
K \setminus \! G / K = G / G^1 \cong \Z^2 .
\]
In \cite[Proposition 4.4]{HeVi} a particular character $\chi$ of $K$ is exhibited, and it is 
proven that the Hecke algebra $\mc H (G,K,\chi)$ is not commutative. 

Let $\sigma$ be an irreducible quotient of $\mr{ind}_K^G (\chi)$. Then $\sigma$ is supercuspidal
(by the compactness of $G / Z(G)$), and $\mu_{\sigma,1} = 2$. The last claim can be checked 
as follows. 

The arguments in \cite{HeVi} entail that $\mc H (G,K,\chi)$ has a vector space basis
$\{ T_g : g \in G / G^1 \}$ and that the centre of this Hecke algebra is the span of the basis
vectors indexed by a certain subgroup $G'_\chi / G^1$. By the calculations in 
\cite[Proposition 4.4]{HeVi} $G'_\chi$ 
is a proper subgroup of $G$, which guarantees the noncommutativity of $\mc H (G,K,\chi)$. 
Computations in the same spirit show that $G'_\chi / G^1$ corresponds to 
\[
(2 \Z)^2 \subset \Z^2 \cong G / G^1 .
\]
Moreover the basis of $\mc H (G,K,\chi)$ can be chosen so that 
\[
T_g \cdot T_{g'} = \natural (g,g') T_{g g'} \qquad g,g' \in G ,
\]
where $\natural : (G / G'_\chi)^2 \to \{\pm 1\}$ is a nontrivial 2-cocycle. Since 
$(G / G'_\chi) \cong (\Z / 2 \Z)^2$, there is only one such 2-cocycle (up to coboundaries). 
One checks that, for any one-dimensional representation $\tau$ of $\mc H (G'_\chi,K,\chi)$,
$\mc H (G,K,\chi) / \ker \tau$ is a group algebra of $(\Z / 2 \Z)^2$ twisted by a 
nontrivial 2-cocycle, so isomorphic with $M_2 (\C)$. It follows that every irreducible 
representation $\sigma'$ of $\mc H (G,K,\chi)$ has dimension two, and restricts with
multiplicity two to $\mc H (G^1,K,\chi) \cong \C$. Hence the irreducible $G$-representation
$\sigma$ corresponding to $\sigma'$ restricts with multiplicity $\mu_{\sigma,1} = 2$ to $G^1$.
In this example
\[
G = G_\sigma^4 ,\quad G'_\chi = G_\sigma^2 \quad \text{and} \quad 
[G_\sigma^4 : G_\sigma^3] = [G_\sigma^3 : G_\sigma^2] = 2 .
\]
\end{exintro}

When $\mu_{\sigma,1} = 1$, the groups $M_\sigma^2, M_\sigma^3$ and $M_\sigma^4$ coincide with
the group called $M^\sigma$ in \cite[\S 1.16]{Hei2}. Otherwise all the different 
$m \in M_\sigma^4 / M_\sigma^3$ give rise to different subspaces $\sigma (m) E_1$ of $E$. 
We denote the representation of $M_\sigma^3$ (resp. $M_\sigma^2$) on $E_1$ by $\sigma_3$ 
(resp. $\sigma_2$). The $\sigma_1$-isotypical component of $E$ is an irreducible representation 
$(\sigma_4,E_4)$ of $M_\sigma^4$. More explicitly
\begin{equation}\label{eq:2.7}
E_4 = \bigoplus\nolimits_{m \in M_\sigma^4 / M_\sigma^3} \sigma (m) E_1 \cong 
\ind_{M_\sigma^3}^{M_\sigma^4}(\sigma_3,E_1) .
\end{equation}
From \eqref{eq:2.7} we see that
\begin{equation}\label{eq:2.11}
(\sigma,E) \cong \ind_{M_\sigma^4}^M (\sigma_4,E_4) \cong \ind_{M_\sigma^3}^M (\sigma_3,E_1) . 
\end{equation}
The structure of $(\sigma_4,E_4)$ can be analysed as in \cite[\S 2]{GeKn}:

\begin{lem}\label{lem:2.1}
\enuma{ 
\item In the above setting
\[
\Res_{M_\sigma^3}^{M_\sigma^4} (\sigma_4) = 
\bigoplus\nolimits_{\chi \in \Irr (M_\sigma^3 / M_\sigma^2)} \sigma_3 \otimes \chi .
\]
\item All the $\sigma_3 \otimes \chi$ are inequivalent irreducible $M_\sigma^3$-representations.
\item There is a group isomorphism 
\[
\begin{array}{ccc}
M_\sigma^4 / M_\sigma^3 & \longrightarrow & \Irr (M_\sigma^3 / M_\sigma^2) \\
n M_\sigma^3 & \mapsto & \chi_{3,n}
\end{array}
\]
defined by $n \cdot \sigma_3 \cong \sigma_3 \otimes \chi_{3,n}$.
}
\end{lem}
\begin{proof}
(a) For any $\chi \in X_\nr (M,\sigma)$ we have $\sigma \otimes \chi \cong \sigma$, so 
$\sigma_3 \otimes \Res_{M_\sigma^3}^M \chi$ is isomorphic to an $M_\sigma^3$-subrepresentation
of $E$. As $M^1$-representation it is just $\sigma_1$, so $\sigma_3 \otimes \Res_{M_\sigma^3}^M \chi$
is even isomorphic to a subrepresentation of $E_4$. As every character of $M_\sigma^3 / M_\sigma^2$ can
be extended to a character of $M / M_\sigma^2$ (that is, to an element of $X_\nr (M,\sigma)$), all
the $\sigma_3 \otimes \chi$ with $\chi \in \Irr (M_\sigma^3 / M_\sigma^2)$ appear in $E_4$.

Further, all the $M_\sigma^3$-subrepresentations $(n^{-1} \cdot \sigma_3, \sigma (n) E_1)$ of 
$(\sigma_4,E_4)$ are extensions of the irreducible $M_\sigma^2$-representation $(\sigma_2,E_1)$. 
Hence they differ from each other only by characters of $M_\sigma^3 / M_\sigma^2$ 
\cite[Lemma 2.14]{GoHe}. This shows that $\Res_{M_\sigma^3}^{M_\sigma^4} (\sigma_4,E_4)$ is a direct 
sum of $M_\sigma^3$-representations of the form $\sigma_3 \otimes \chi$ with 
$\chi \in \Irr (M_\sigma^3 / M_\sigma^2)$. 

By Frobenius reciprocity, for any such $\chi$:
\begin{equation}\label{eq:2.8}
\Hom_{M_\sigma^4} \big( \ind_{M_\sigma^3}^{M_\sigma^4} (\sigma_3 \otimes \chi), \sigma_4 \big) \cong
\Hom_{M_\sigma^3} (\sigma_3 \otimes \chi , \sigma_4) \neq 0 .
\end{equation}
Thus there exists a nonzero $M_\sigma^4$-homomorphism 
$\ind_{M_\sigma^3}^{M_\sigma^4} (\sigma_3 \otimes \chi)
\to \sigma_4$. As these two representations have the same dimension and $\sigma_4$ is irreducible,
they are isomorphic. Knowing that, \eqref{eq:2.8} also shows that 
$\dim \Hom_{M_\sigma^3} (\sigma_3 \otimes \chi , \sigma_4) = 1$.\\
(b) The previous line is equivalent to: every $\sigma_3 \otimes \chi$ appears exactly once as a 
$M_\sigma^3$-subrepresentation of $\sigma_4$. As $\Res_{M_\sigma^3}^{M_\sigma^4} (\sigma_4)$ has length 
$[M_\sigma^4 : M_\sigma^3] = [M_\sigma^3 : M_\sigma^2]$, that means that they are mutually inequivalent.\\
(c) This is a consequence of parts (a), (b) and the Mackey decomposition of \\
$\Res_{M_\sigma^3}^{M_\sigma^4} (\sigma_4, E_4)$.
\end{proof}

For $\chi \in \Irr (M / M_\sigma^3)$ we define an $M$-isomorphism
\begin{equation}\label{eq:2.28}
\begin{array}{ccccr}
\phi_{\sigma,\chi} : & (\sigma,E) & \to & (\sigma \otimes \chi,E) & \\
& \sigma (m) e_1 & \mapsto & \chi (m) \sigma (m) e_1 &
\qquad e_1 \in E_1, m \in M .
\end{array}
\end{equation}
This says that $\phi_{\sigma,\chi}$ acts as $\chi (m) \mr{id}$ on the $M_\sigma^3$-subrepresentation
$\sigma (m) E_1$ of $E$. By Lemma \ref{lem:2.1} these $\phi_{\sigma,\chi}$ form a basis of
$\End_{M_\sigma^3}(E)$. We can extend $\phi_{\sigma,\chi}$ to an $M$-isomorphism
\begin{equation}\label{eq:2.10}
\begin{array}{cccc}
\phi_\chi = \phi_{\sigma,\chi} \otimes \rho_\chi^{-1} : & \ind_{M^1}^M (\sigma,E) & \to & 
\ind_{M^1}^M (\sigma,E) \\
& e \otimes \delta_m & \mapsto & \phi_{\sigma,\chi}(e) \otimes \chi (m) \delta_m 
\end{array} , 
\end{equation}
where $e \in E, m \in M$ and the elements are presented in $E \otimes_\C \ind_{M^1}^M (\C)$
using \eqref{eq:2.2}. Via \eqref{eq:2.3}, this becomes
\begin{equation}\label{eq:2.12}
\phi_\chi \in \mr{Aut}_M (E \otimes_\C \C [X_\nr (M)] ) : \qquad 
e \otimes b \mapsto \phi_{\sigma,\chi}(e) \otimes \rho_{\chi}^{-1} (b),
\end{equation}
where $e \in E, b \in \C [X_\nr (M)]$. Given $E_1$, $\phi_\chi$ is canonical.

For an arbitrary $\chi \in \Irr (M / M_\sigma^2) = X_\nr (M,\sigma)$ we can also construct such
$M$-homomorphisms, albeit not canonically. Pick $n \in M_\sigma^4$ (unique up to $M_\sigma^3$) as in 
Lemma \ref{lem:2.1}.c, such that $\chi_{3,n} = \chi |_{M_\sigma^3}$. Choose an $M_\sigma^3$-isomorphism
\[
\phi_{\sigma_3,\chi} : (\sigma_3,E_1) \to ((n^{-1} \cdot \sigma_3) \otimes \chi, \sigma (n) E_1).  
\]
We note that, when $\chi \notin \Irr ( M / M_\sigma^3 )$, $\psi_{\sigma_3,\chi}$ cannot commute
with all the $\phi_{\sigma,\chi'}$ for $\chi' \in \Irr (M / M_\sigma^3)$ because it does not
stabilize $E_1$.

For compatibility with \eqref{eq:2.28} we may assume that 
\begin{equation}\label{eq:2.30}
\phi_{\sigma_3,\chi \chi'} = \phi_{\sigma_3,\chi} \quad \text{for all } \chi' \in \Irr (M / M_\sigma^3).
\end{equation}
By Schur's lemma $\phi_{\sigma_3,\chi}$ is unique up to scalars, 
but we do not know a canonical choice when $M_\sigma^3 \not\subset \ker \chi$. By \eqref{eq:2.11}
\[
\Hom_M (\sigma, \sigma \otimes \chi) = \Hom_M (\ind_{M_\sigma^3}^M (\sigma_3), \sigma \otimes \chi) 
\cong \Hom_{M_\sigma^3} (\sigma_3, \sigma \otimes \chi) ,
\]
while $((n^{-1} \cdot \sigma_3) \otimes \chi, \sigma (n) E_1)$ is contained in $(\sigma \otimes \chi,E)$ 
as $M_\sigma^3$-representation. Thus $\phi_{\sigma_3,\chi}$ determines a 
$\phi_{\sigma,\chi} \in \Hom_M (\sigma, \sigma \otimes \chi)$, which is nonzero and hence bijective.
Then $\rho_\chi$ from \eqref{eq:2.25} and the formulas \eqref{eq:2.10} and \eqref{eq:2.12} provide 
\begin{equation}\label{eq:2.24}
\phi_\chi = \phi_{\sigma,\chi} \otimes \rho_\chi^{-1}  \in \mr{Aut}_M (\ind_{M^1}^M (E)) 
\cong \mr{Aut}_M (E \otimes_\C \C [X_\nr (M)]) . 
\end{equation}
For $\chi_c \in X_\nr (M)$ and $\chi \in X_\nr (M,\sigma)$, we see from \eqref{eq:2.24} that
\begin{equation}\label{eq:3.31}
\mr{sp}_{\chi_c} \circ \phi_{\chi} = \phi_{\chi} \circ \mr{sp}_{\chi_c \chi^{-1}} . 
\end{equation}
We also note that, regarding $b \in \C [X_\nr (M)]$ as multiplication operator:
\begin{equation}\label{eq:2.26}
b \circ \phi_\chi = \phi_\chi \circ b_\chi \in \End_M \big( E \otimes_\C \C [X_\nr (M)] \big) .
\end{equation}
For all $\chi, \chi' \in \Irr (M / M_\sigma^2)$, the uniqueness of $\phi_{\sigma_3,\chi}$ 
up to scalars implies that there exists a $\natural (\chi,\chi') \in \C^\times$ such that
\begin{equation}\label{eq:2.29}
\phi_\chi \phi_{\chi'} = \natural (\chi,\chi') \phi_{\chi \chi'} . 
\end{equation}
In other words, the $\phi_\chi$ span a twisted group algebra $\C[X_\nr (M,\sigma),\natural]$. 
By \eqref{eq:2.30} we have 
\begin{equation}\label{eq:2.35}
\natural (\chi,\chi') = 1 \text{ if } \chi \in \Irr (M / M_\sigma^3) \text{ or }
\chi' \in \Irr (M / M_\sigma^3). 
\end{equation}
If desired, we can scale the $\phi_{\sigma,\chi}$ so that $\phi_{\sigma_3,\chi}^{-1} = 
\phi_{\sigma_3,\chi^{-1}}$. In that case $\phi_{\chi}^{-1} = \phi_{\chi^{-1}}$ and 
$\natural (\chi,\chi^{-1}) = 1$ for all $\chi \in \Irr (M / M_\sigma^2)$. When 
$\mu_{\sigma,1} > 1$, not all $\phi_\chi$ commute and $\natural$ is nontrivial. Then it is unlikely
that all the $\phi_\chi$ with $\chi \in X_\nr (M,\sigma)$ can be normalized simultaneously 
in a canonical way, because they can always be rescaled by a character of $X_\nr (M,\sigma)$.

The next result is a variation on \cite[Proposition 3.6]{Hei2}.

\begin{prop}\label{prop:2.2}
\enuma{
\item The set $\{ \phi_{\sigma,\chi} : \chi \in X_\nr (M,\sigma) \}$ is a $\C$-basis
of $\End_{M^1}(E)$.
\item With respect to the embedding \eqref{eq:2.4}:
\[
\End_M (\ind_{M^1}^M (\sigma,E)) = 
\bigoplus_{\chi \in X_\nr (M,\sigma)} \C [X_\nr (M)] \phi_\chi =
\bigoplus_{\chi \in X_\nr (M,\sigma)} \phi_\chi \C [X_\nr (M)] .
\]
}
\end{prop}
\begin{proof}
(a) By \eqref{eq:2.7} and Lemma \ref{lem:2.1}
\[
\Res_{M_\sigma^3}^M (\sigma,E) = 
\bigoplus\nolimits_{m \in M / M_\sigma^3} (m^{-1} \cdot \sigma,\sigma (m) E_1 ) , 
\]
and all these summands are mutually inequivalent. Hence
\begin{equation}\label{eq:2.13}
\End_{M_\sigma^3} (E) = \bigoplus_{m \in M / M_\sigma^3} \End_{M_\sigma^3} (\sigma (m) E_1) =
\bigoplus_{m \in M / M_\sigma^3} \C \, \mr{id}_{\sigma (m) E_1} .
\end{equation}
The operators $\phi_{\sigma,\chi}$ with $\chi \in \Irr (M / M_\sigma^3)$ provide a basis of 
\eqref{eq:2.13}, because they are linearly independent.

For every $\chi_3 \in \Irr (M_\sigma^3 / M_\sigma^2)$ we choose an extension $\tilde \chi_3 \in
\Irr (M/M_\sigma^2)$. Then
\[
\{ \phi_{\sigma,\chi} : \chi \in \Irr (M / M_\sigma^2) \} = \{ \phi_{\sigma,\tilde \chi_3} 
\phi_{\sigma,\chi} : \chi \in \Irr (M / M_\sigma^3), \chi_3 \in \Irr (M_\sigma^3 / M_\sigma^2) \} .
\]
It follows from \eqref{eq:2.7} that
\begin{align*}
& \Res_{M^1}^M (\sigma,E) = \bigoplus\nolimits_{m \in M / M_\sigma^4} 
\big( (m \cdot \sigma_1)^{\mu_{\sigma,1}}, \sigma (m) E_4 \big) , \\
& \End_{M^1} (E) \cong \bigoplus_{m \in M / M_\sigma^4} \End_{M^1}(\sigma (m) E_4) \cong
\bigoplus_{m \in M / M_\sigma^4} \sigma (m) \End_{M^1}(E_4) \sigma (m^{-1}) .
\end{align*}
In view of the already exhibited basis of \eqref{eq:2.13}, it only remains to show that
\begin{equation}\label{eq:2.14}
\big\{ \mr{id}_{\sigma (m) E_1} \phi_{\tilde \chi_3} \big|_{E_4} \big\} 
\end{equation}
is a $\C$-basis of $\End_{M^1} (E_4)$. Every $\phi_{\tilde \chi_3}$ permutes the irreducible
$M_\sigma^3$-subrepresentations $\sigma (m) E_1$ of $E_4$ according to a unique 
$n \in M_\sigma^4 / M_\sigma^3$, so the set \eqref{eq:2.14} is linearly independent. As
\[
\dim \End_{M^1}(E_4) = \dim \End_{M^1} \big( \sigma_1^{\mu_{\sigma,1}} \big) = \mu_{\sigma,1}^2 =
[M_\sigma^4 : M_\sigma^3] [M_\sigma^3 : M_\sigma^2] ,
\]
equals the cardinality of \eqref{eq:2.14}, that set also spans $\End_{M^1} (E_4)$.\\
(b) As $M^1 \subset M$ is open, Frobenius reciprocity for 
compact smooth induction holds. It gives a natural bijection
\[
\End_M (\ind_{M^1}^M (E)) \to \Hom_{M^1}(E, \ind_{M^1}^M (E)) . 
\]
By \eqref{eq:2.3} the right hand side is isomorphic to 
\[
\Hom_{M^1} (E, E \otimes_\C \C [X_\nr (M)]) = \End_{M^1} (E) \otimes_\C \C [X_\nr (M)] ,
\]
where the action of $X_\nr (M)$ becomes multiplication on the second tensor factor
on the right hand side. Under these bijections $\phi_\chi \in \mr{Aut}_M (\ind_{M^1}^M (E))$ 
corresponds to 
\[
\phi_{\sigma,\chi} \otimes 1 \in \End_{M^1} (E) \otimes_\C \C [X_\nr (M)] .
\]
We conclude by applying part (a).
\end{proof}

We remark that \eqref{eq:2.29}, \eqref{eq:2.26} and Proposition \ref{prop:2.2}.b mean that
\begin{equation}\label{eq:2.36}
\End_M (E \otimes_\C \C[X_\nr (M)]) = \C[X_\nr (M)] \rtimes \C [X_\nr (M,\sigma),\natural] ,
\end{equation}
the crossed product with respect to the multiplication action of $X_\nr (M,\sigma)$ on $X_\nr (M)$.
This description confirms that
\begin{equation}
Z \big( \End_M (E \otimes_\C \C[X_\nr (M)]) \big) = \C[X_\nr (M) / X_\nr (M,\sigma)] \cong \C [\mc O] .
\end{equation}
Let us record what happens when we replace regular functions on the involved complex
algebraic tori by rational functions. More generally, consider a group $\Gamma$ and an integral 
domain $R$ with quotient field $Q$. Suppose that $V$ is a $\C \Gamma \times R$-module, which is 
free over $R$. Then $R \subset \End_\Gamma (V)$ and there is a natural isomorphism of $R$-modules
\begin{equation}\label{eq:2.21}
\Hom_\Gamma (V, V \otimes_R Q) \cong \End_\Gamma (V) \otimes_R Q .
\end{equation}
Applying this to \eqref{eq:2.3} and Proposition \ref{prop:2.2} we find
\begin{multline}\label{eq:2.22}
\Hom_M \big( \ind_{M^1}^M (E), \ind_{M^1}^M (E) \otimes_{\C [X_\nr (M)]} \C (X_\nr (M) \big) \cong \\
\bigoplus\nolimits_{\chi \in X_\nr (M,\sigma)} \phi_\chi \C (X_\nr (M)) =
\C (X_\nr (M)) \rtimes \C [X_\nr (M,\sigma),\natural] ,
\end{multline}
which generalizes \cite[Proposition 3.6]{Hei2}.

\section{Some root systems and associated groups}
\label{sec:root}

Let $\mc A_M$ be the maximal $F$-split torus in $Z(\mc M)$, put $A_M = \mc A_M (F)$ and let $
X_* (\mc A_M) = X_* (A_M)$ be the cocharacter lattice. We write 
\[
\mf a_M = X_* (A_M) \otimes_\Z \R \quad \text{and} \quad \mf a_M^* = X^* (A_M) \otimes_\Z \R .
\]
Let $\Sigma (\mc G,\mc M) \subset X^* (A_M)$ be the set of nonzero weights occurring in the adjoint 
representation of $A_M$ on the Lie algebra of $G$, and let $\Sigma_\red (A_M)$ be the set
of indivisible elements therein. 

For every $\alpha \in \Sigma_\red (A_M)$ there is a Levi subgroup $M_\alpha$ of $G$ which contains 
$M$ and the root subgroup $U_\alpha$, and whose semisimple rank is one higher than that of $M$. 
Let $\alpha^\vee \in \mf a_M$ be the unique element which is orthogonal to $X^* (A_{M_\alpha})$
and satisfies $\langle \alpha^\vee, \alpha \rangle = 2$.

Recall the Harish-Chandra $\mu$-functions from \cite[\S 1]{Sil2} and \cite[\S V.2]{Wal}. 
The restriction of $\mu^G$ to $\mc O$ is a rational, $W(M,\mc O)$-invariant function on $\mc O$ 
\cite[Lemma V.2.1]{Wal}. It determines a reduced root system \cite[Proposition 1.3]{Hei2}
\[
\Sigma_{\mc O,\mu} = \{ \alpha \in \Sigma_\red (A_M) : \mu^{M_\alpha}(\sigma \otimes \chi)
\text{ has a zero on } \mc O \} .
\]
For $\alpha \in \Sigma_\red (A_M)$ the function $\mu^{M_\alpha}$ factors through the quotient
map $A_M \to A_M / A_{M_\alpha}$. The associated system of coroots is
\[
\Sigma_{\mc O,\mu}^\vee = \{ \alpha^\vee \in \mf a_M :  \mu^{M_\alpha}(\sigma \otimes \chi)
\text{ has a zero on } \mc O \} .
\]
By the aforementioned $W(M,\mc O)$-invariance of $\mu^G$, $W(M,\mc O)$ acts naturally on
$\Sigma_{\mc O,\mu}$ and $\Sigma_{\mc O,\mu}^\vee$.
Let $s_\alpha$ be the unique nontrivial element of $W(M_\alpha,M)$. By \cite[Proposition 1.3]{Hei2}
the Weyl group $W(\Sigma_{\mc O,\mu})$ can be identified with the subgroup of $W(G,M)$ generated
by the reflections $s_\alpha$ with $\alpha \in \Sigma_{\mc O,\mu}$, and as such it is a normal
subgroup of $W(M,\mc O)$.

The parabolic subgroup $P = MU$ of $G$ determines a set of positive roots $\Sigma_{\mc O,\mu}(P)$ 
and a basis $\Delta_{\mc O,\mu}$ of $\Sigma_{\mc O,\mu}$. Let $\ell_{\mc O}$ be the length function on
$W(\Sigma_{\mc O,\mu})$ specified by $\Delta_{\mc O,\mu}$. Since $W(M,\mc O)$ acts on 
$\Sigma_{\mc O,\mu}$, $\ell_{\mc O}$ extends naturally to $W(M,\mc O)$, by
\[
\ell_{\mc O}(w) = | w (\Sigma_{\mc O,\mu}(P)) \cap \Sigma_{\mc O,\mu}(\bar P) |. 
\]
The set of positive roots also determines a subgroup of $W(M,\mc O)$:
\begin{equation}\label{eq:3.19}
\begin{aligned}
R(\mc O) & = \{ w \in W(M,\mc O) : w (\Sigma_{\mc O,\mu}(P)) = \Sigma_{\mc O,\mu}(P) \} \\
& = \{ w \in W(M,\mc O) : \ell_{\mc O}(w) = 0 \} .
\end{aligned}
\end{equation}
As $W(\Sigma_{\mc O,\mu}) \subset W(M,\mc O)$, 
a well-known result from the theory of root systems says:
\begin{equation}\label{eq:3.8}
W(M,\mc O) = R (\mc O) \ltimes W(\Sigma_{\mc O,\mu})  . 
\end{equation}
Recall that $X_\nr (M) / X_\nr (M,\sigma)$ is isomorphic to the character group of the lattice 
$M_\sigma^2 / M^1$. Since $M_\sigma^2$ depends only on $\mc O$, it is normalized by $N_G (M,\mc O)$. 
In particular the conjugation action of $N_G (M,\mc O)$ on $M_\sigma^2 / M^1$ induces an action of 
$W(M,\mc O)$ on $M_\sigma^2 / M^1$.

Let $\nu_F : F \to \Z \cup \{\infty\}$ be the valuation of $F$. Let $h^\vee_\alpha$ be the unique 
generator of $(M_\sigma^2 \cap M_\alpha^1) / M^1 \cong \Z$ such that 
$\nu_F (\alpha (h^\vee_\alpha)) > 0$. 
Recall the injective homomorphism $H_M : M/M^1 \to \mf a_M$ defined by 
\[
\langle H_M (m), \gamma \rangle = \nu_F (\gamma (m)) \qquad \text{for } m \in M, \gamma \in X^* (M).
\]
In these terms $H_M (h^\vee_\alpha) \in \R_{>0} \alpha^\vee$. Since $M_\sigma^2$ has finite index in
$M$, $H_M (M_\sigma^2 / M^1)$ is a lattice of full rank in $\mf a_M$. We write 
\[
(M_\sigma^2 / M^1)^\vee = \Hom_\Z (M_\sigma^2 / M^1, \Z) .
\]
Composition with $H_M$ and $\R$-linear extension of maps $H_M (M_\sigma^2 / M^1) \to \Z$ determines 
an embedding 
\[
H_M^\vee : (M_\sigma^2 / M^1)^\vee \to \mf a_M^*.
\]
Then $H_M^\vee (M_\sigma^2 / M^1)^\vee$ is a lattice of full rank in $\mf a_M^*$.

\begin{prop}\label{prop:3.5}
Let $\alpha \in \Sigma_{\mc O,\mu}$.
\enuma{
\item For $w \in W(M,\mc O)$: $w(h^\vee_\alpha) = h^\vee_{w(\alpha)}$.
\item There exists a unique $\alpha^\sharp \in (M_\sigma^2 / M^1)^\vee$ such that 
$H_M^\vee (\alpha^\sharp) \in \R \alpha$ and $\langle h^\vee_\alpha, \alpha^\sharp \rangle = 2$.
\item Write 
\[
\begin{array}{lll}
\Sigma_{\mc O} & = & \{ \alpha^\sharp : \alpha \in \Sigma_{\mc O,\mu} \} , \\
\Sigma^\vee_{\mc O} & = & \{ h^\vee_\alpha : \alpha \in \Sigma_{\mc O,\mu} \} .
\end{array}
\]
Then $(\Sigma_{\mc O}^\vee, M_\sigma^2 / M^1, \Sigma_{\mc O}, (M_\sigma^2 / M^1)^\vee)$
is a root datum with Weyl group $W(\Sigma_{\mc O,\mu})$.
\item The group $W(M,\mc O)$ acts naturally on this root datum, and $R(\mc O)$ is the
stabilizer of the basis determined by $P$.
}
\end{prop}
\begin{proof}
(a) Since $M_\sigma^2 / M^1$ and $\Sigma_{\mc O,\mu}$ are $W(M,\mc O)$-stable, we have
\[
\tilde w (M_\sigma^2 \cap M^{\alpha,1}) \tilde w^{-1} = M_\sigma^2 \cap M^{w(\alpha),1}.
\]
Hence $w (h^\vee_\alpha)$ is a generator of $(M_\sigma^2 \cap M^{w(\alpha),1}) / M^1$.
As $w(\alpha)(w (h^\vee_\alpha)) = \alpha (h^\vee_\alpha)$, it equals $h^\vee_{w (\alpha)}$.\\
(b) Let $\alpha^* \in \R \alpha \subset \mf a_M^*$ be the unique element 
which satisfies
\[
\langle H_M (h^\vee_\alpha), \alpha^* \rangle = 2.
\]
The group $W(\Sigma_{\mc O,\mu})$ acts naturally on $\mf a_M$ by
\begin{equation}\label{eq:3.18}
s_\alpha (x) = x - \langle x, \alpha \rangle \alpha^\vee =  
x - \langle x, \alpha^* \rangle H_M (h^\vee_\alpha) .
\end{equation}
This action stabilizes the lattice $H_M (M_\sigma / M^1)$. By construction  
$h^\vee_\alpha$ is indivisible in $M_\sigma^2 / M^1$. It follows that for all
$x \in H_M (M_\sigma^2 / M^1)$ we must have $\langle x, \alpha^* \rangle \in \Z$.
This means that $\alpha^*$ lies in $H_M^\vee (M_\sigma^2 / M^1)^\vee$, say
$\alpha^* = H_M^\vee (\alpha^\sharp)$. \\
(c) By construction the lattices $M_\sigma^2 / M^1$ and $(M_\sigma^2 / M^1)^\vee$ are
dual and $W(M,\mc O)$ acts naturally on them. In view of \eqref{eq:3.18}, the map
\[
M_\sigma^2 / M^1 \to M_\sigma^2 / M^1 : 
\bar m \mapsto \bar m - \langle \bar m, \alpha^\sharp \rangle h^\vee_\alpha
\]
coincides with the action of $s_\alpha$. Hence it stabilizes $\Sigma^\vee_{\mc O}$. 
Similarly, for $y \in \mf a_M^*$:
\[
y - \langle H_M (h^\vee_\alpha), y \rangle H_M (\alpha^\sharp) = 
y - \langle \alpha^\vee, y \rangle \alpha  = s_\alpha (y) . 
\]
This implies that the map 
\[
(M_\sigma^2 / M^1)^\vee \to (M_\sigma^2 / M^1)^\vee :
y \mapsto y - \langle h^\vee_\alpha, y \rangle \alpha^\sharp
\]
coincides with the action of $s_\alpha$ and stabilizes $\Sigma_{\mc O}$. Thus
$(\Sigma_{\mc O}^\vee, M_\sigma^2 / M^1, \Sigma_{\mc O}, (M_\sigma^2 / M^1)^\vee)$ is
a root datum and the Weyl groups of $\Sigma_{\mc O}$ and $\Sigma_{\mc O}^\vee$ can be
identified with $W(\Sigma_{\mc O,\mu})$.\\
(d) By part (a) $W(M,\mc O)$ acts naturally on the root datum, extending the action of
$W(\Sigma_{\mc O,\mu})$. The characterization of $R(\mc O)$ is obvious from \eqref{eq:3.19}
and the definition of $\Sigma_{\mc O}$ and $\Sigma_{\mc O}^\vee$.
\end{proof}

We note that $\Sigma_{\mc O}$ and $\Sigma_{\mc O}^\vee$ have almost the same type as
$\Sigma_{\mc O,\mu}$. Indeed, the roots $H_M^\vee (\alpha^\sharp)$ are scalar multiples
of the $\alpha \in \Sigma_{\mc O,\mu}$, the angles between the elements of $\Sigma_{\mc O}$
are the same as the angles between the corresponding elements of $\Sigma_{\mc O,\mu}$. 
It follows that every irreducible component of $\Sigma_{\mc O,\mu}$ has the same type
as the corresponding components of $\Sigma_{\mc O}$ and $\Sigma_{\mc O}^\vee$, except
that type $B_n / C_n$ might be replaced by type $C_n / B_n$.

For $\alpha \in \Sigma_\red (M) \setminus \Sigma_{\mc O,\mu}$, the function $
\mu^{M_\alpha}$ is constant on $\mc O$. In contrast, for $\alpha \in \Sigma_{\mc O,\mu}$ 
it has both zeros and poles on $\mc O$. By \cite[\S 5.4.2]{Sil2} 
\begin{equation}\label{eq:3.4}
\tilde{s_\alpha} \cdot \sigma' \cong \sigma' \quad \text{whenever } \mu^{M_\alpha}(\sigma') = 0 .
\end{equation}
As $\Delta_{\mc O,\mu}$ is linearly independent in $X^* (A_M)$ and $\mu^{M_\alpha}$ factors
through $A_M / A_{M_\alpha}$, there exists a $\tilde \sigma \in \mc O$ such that 
$\mu^{M_\alpha}(\tilde \sigma) = 0$ for all $\alpha \in \Delta_{\mc O,\mu}$. In view of 
\cite[Theorem 1.6]{Sil3} this can even be achieved with a unitary $\tilde \sigma$. 
We replace $\sigma$ by $\tilde \sigma$, which means that from now on we adhere to:

\begin{cond}\label{cond:3.1}
$(\sigma,E) \in \Irr (M)$ is unitary supercuspidal and $\mu^{M_\alpha}(\sigma) = 0$ 
for all $\alpha \in \Delta_{\mc O,\mu}$.
\end{cond}

By \eqref{eq:3.4} the entire Weyl group $W(\Sigma_{\mc O,\mu})$ stabilizes the isomorphism 
class of this $\sigma$. However, in general $R(\mc O)$ need not stabilize $\sigma$.
We identify $X_\nr (M) / X_\nr (M,\sigma)$ with 
$\mc O$ via $\chi \mapsto \sigma \otimes \chi$ and we define
\begin{equation}\label{eq:3.20}
X_\alpha = b_{h^\vee_\alpha} \in \C [X_\nr (M) / X_\nr (M,\sigma)] .
\end{equation}
For any $w \in W(M,\mc O)$ which stabilizes $\sigma$ in $\Irr (M)$, Proposition 
\ref{prop:3.5}.a implies
\begin{equation}\label{eq:3.21}
w (X_\alpha) = X_{w (\alpha)} \quad \text{for all } \alpha \in \Sigma_{\mc O,\mu} . 
\end{equation}
Let $q_F$ be the cardinality of the residue field of $F$.
According to \cite[\S 1]{Sil3} there exist $q_\alpha, q_{\alpha*} \in \R_{\geq 1}$,
$c'_{s_\alpha} \in \R_{>0}$ for $\alpha \in \Sigma_{\mc O,\mu}$, such that
\begin{equation}\label{eq:3.22}
\mu^{M_\alpha}(\sigma \otimes \cdot) = 
\frac{c'_{s_\alpha}  (1 - X_\alpha) (1 - X_\alpha^{-1})}{(1 - q_\alpha^{-1} X_\alpha)
(1 - q_\alpha^{-1} X_\alpha^{-1})} \frac{(1 + X_\alpha) (1 + X_\alpha^{-1})
}{(1 + q_{\alpha*}^{-1} X_\alpha)(1 + q_{\alpha*}^{-1} X_\alpha^{-1})} 
\end{equation}
as rational functions on $X_\nr (M) / X_\nr (M,\sigma) \cong \mc O$.

We have only little explicit information about the $q_\alpha$ and the
$q_{\alpha*}$ in general ($c'_{s_\alpha}$ is not important). Obviously, knowing them is
equivalent to knowing the poles of $\mu^{M_\alpha}$. These are precisely the reducibility
points of the normalized parabolic induction $I_{P \cap M_\alpha}^{M_\alpha}(\sigma \otimes \chi)$
\cite[\S 5.4]{Sil2}. When these reducibility points are known somehow, one can recover 
$q_\alpha$ and $q_{\alpha*}$ from them. In all cases that we are aware of, this method
shows that $q_{\alpha}$ and $q_{\alpha*}$ are integers. We refer to \cite{SolParam} for
further investigations in this direction.

\hspace{-2mm} We may modify the choice of $\sigma$ in Condition \ref{cond:3.1}, so that, as in 
\cite[Remark 1.7]{Hei2}:
\begin{equation}\label{eq:3.25}
q_\alpha \geq q_{\alpha*} \text{ for all } \alpha \in \Delta_{\mc O,\mu} .
\end{equation}
Comparing \eqref{eq:3.22}, Condition \ref{cond:3.1} and \eqref{eq:3.25}, we see that 
$q_{\alpha} > 1$ for all $\alpha \in \Sigma_{\mc O,\mu}$.
In particular the zeros of $\mu^{M_\alpha}$ occur at
\[
\{ X_\alpha = 1 \} = \{ \sigma' \in \mc O : X_\alpha (\sigma') = 1 \} 
\]
and sometimes at
\[
\{ X_\alpha = -1 \} = \{ \sigma' \in \mc O : X_\alpha (\sigma') = -1 \} . 
\]

\begin{lem}\label{lem:3.6}
Let $\alpha \in \Sigma_{\mc O,\mu}$ and suppose that $\mu^{M_\alpha}$ has a zero at both 
$\{ X_\alpha = 1 \}$ and $\{ X_\alpha = -1 \}$. Then the irreducible component of 
$\Sigma_{\mc O}^\vee$ containing $h^\vee_\alpha$ has type
$B_n \; (n \geq 1)$ and $h^\vee_\alpha$ is a short root.
\end{lem}
\begin{proof}
Consider any $h^\vee_\alpha \in \Sigma_{\mc O}^\vee$ which is not a short root in a type $B_n$ 
irreducible component. Then $\alpha^\sharp$ is not a long root in a type $C_n$ irreducible
component of $\Sigma_{\mc O}$, so there exists a $x \in \Sigma_{\mc O}^\vee \subset  M_\sigma^2 / M^1$ 
with $\langle x , \alpha^\sharp \rangle = -1$. Then
\[
s_\alpha (x) = x - \langle x , \alpha^\sharp \rangle h^\vee_\alpha
= x + h^\vee_\alpha \in M_\sigma^2 / M^1 .
\]
In $X^* (X_\nr (M) / X_\nr (M,\sigma))$ this becomes $s_\alpha (x) = x X_\alpha$. Assume that there 
exists a $\sigma' \in \mc O$ with $\mu^{M_\alpha}(\sigma') = 0$ and $X_\alpha (\sigma') = -1$. We compute
\[
x (\tilde{s_\alpha} \cdot \sigma') = (s_\alpha x)(\sigma') = (x X_\alpha)(\sigma')
= x (\sigma') X_\alpha (\sigma') = - x (\sigma') .
\]
As $x \in X_* (X_\nr (M) / X_\nr (M,\sigma))$, this implies that $\tilde{s_\alpha} \cdot \sigma'$
is not isomorphic to $\sigma'$. But that contradicts \eqref{eq:3.4}, so the assumption cannot hold.
\end{proof}

Consider $r \in R(\mc O)$. By the definition of $W(M,\mc O)$ there exists a $\chi_r \in X_\nr (M)$
such that 
\begin{equation}\label{eq:3.27}
\tilde r \cdot \sigma \cong \sigma \otimes \chi_r .
\end{equation}

\begin{lem}\label{lem:3.8}
\enuma{
\item The maps $\alpha \mapsto q_{\alpha}, \alpha \mapsto q_{\alpha*}$ and 
$\alpha \mapsto c'_{s_\alpha}$ are constant on $W(M,O)$-orbits in $\Sigma_{\mc O,\mu}$.
\item For $\alpha \in \Delta_{\mc O,\mu}$ and $r \in R(\mc O)$, either $X_\alpha (\chi_r) = 1$
or $X_\alpha (\chi_r) = -1$ and $q_\alpha = q_{\alpha*}$.
}
\end{lem}
\begin{proof}
It follows directly from the definitions in \cite[\S V.2]{Wal} that
\begin{equation}\label{eq:3.23}
\mu^{M_{w (\alpha)}}(\tilde w \cdot \sigma') = \mu^{M_\alpha}(\sigma') \quad
\text{for all } w \in W(M,\mc O).
\end{equation}
Since every $W(\Sigma_{\mc O,\mu})$-orbit in $\Sigma_{\mc O,\mu}$ meets $\Delta_{\mc O,\mu}$,
\eqref{eq:3.25} generalizes to
\begin{equation}\label{eq:3.24}
q_{\alpha} \geq q_{\alpha*} \qquad \forall \alpha \in \Sigma_{\mc O,\mu} .
\end{equation}
As $W(\Sigma_{\mc O,\mu})$ stabilizes $\sigma$, \eqref{eq:3.23}, \eqref{eq:3.24} and \eqref{eq:3.20} 
imply that part (a) holds at least on $W(\Sigma_{\mc O,\mu}$-orbits in $\Sigma_{\mc O,\mu}$. 

For $r \in R(\mc O)$ we work out \eqref{eq:3.23}:
\begin{multline*}
\mu^{M_\alpha} (\sigma \otimes \chi) = \mu^{M_{r(\alpha)}}( \tilde r \cdot (\sigma \otimes \chi)) \;
= \; \mu^{M_{r(\alpha)}} (\sigma \otimes \chi_r r(\chi)) = \\
\mr{sp}_{\chi_r r (\chi)} \Big( \frac{c'_{s_{r \alpha}}  (1 - X_{r \alpha}) (1 - X_{r \alpha}^{-1})
}{(1 - q_{r \alpha}^{-1} X_{r \alpha}) (1 - q_{r \alpha}^{-1} X_{r \alpha}^{-1})} 
\frac{(1 + X_{r \alpha}) (1 + X_{r \alpha}^{-1})}{(1 + q_{r \alpha*}^{-1} X_{r \alpha})
(1 + q_{r \alpha*}^{-1} X_{r \alpha}^{-1})} \Big) = \\
\mr{sp}_\chi \Big( \frac{c'_{s_\alpha} (1 - X_{r \alpha} (\chi_r) X_\alpha) (1 - X_{r \alpha}^{-1} 
(\chi_r) X_\alpha^{-1})}{(1 - q_\alpha^{-1} X_{r \alpha} (\chi_r) X_\alpha)
(1 - q_{\alpha}^{-1} X_{r \alpha}^{-1} (\chi_r) X_\alpha^{-1})} \; \times \\
\hspace{14mm} \frac{(1 + X_{r \alpha} (\chi_r) 
X_\alpha) (1 + X_{r \alpha}^{-1} (\chi_r) X_\alpha^{-1})}{(1 + q_{\alpha*}^{-1} X_{r \alpha} 
(\chi_r) X_\alpha)(1 + q_{\alpha*}^{-1} X_{r \alpha}^{-1} (\chi_r) X_\alpha^{-1})} \Big) 
\end{multline*}
Comparing the zero orders along the subvarieties $\{X_\alpha = $ constant\}, we see that
$X_{r \alpha} (\chi_r) \in \{1,-1\}$. Then we look at the pole orders.

When $X_{r \alpha} (\chi_r) = 1$, we obtain $q_{r \alpha} = q_{\alpha}$ and
$q_{r \alpha*} = q_{\alpha*}$.

When $X_{r \alpha} (\chi_r) = -1$, we find $q_{r \alpha} = q_{\alpha*}$ and
$q_{r \alpha*} = q_{\alpha}$. Together with \eqref{eq:3.24} that implies 
$q_{r \alpha} = q_{\alpha*} = q_{r \alpha*} = q_{\alpha}$.

Knowing all this, another glance at \eqref{eq:3.23} reveals that $c'_{s_{r \alpha}} = c'_{s_\alpha}$.
\end{proof}

Of course, $\chi_r$ is in general not unique, only up to $X_\nr (M,\sigma)$. If 
$\tilde r \cdot \sigma \cong \sigma$, then we take $\chi_r = 1$, otherwise we just pick
one of eligible $\chi_r$. We note that then
\[
\tilde{r}^{-1} \cdot \sigma \cong \sigma \otimes r^{-1} (\chi_r^{-1}) ,
\]
which implies 
\begin{equation}\label{eq:3.28}
r^{-1} (\chi_r) \chi_{r^{-1}} \in X_\nr (M,\sigma) .
\end{equation}
For $r \in R (\mc O)$ of order larger than two,  we may take $\chi_{r^{-1}} = r^{-1} (\chi_r^{-1})$.

\begin{lem}\label{lem:3.7}
For all $w \in W(\Sigma_{\mc O,\mu}), r \in R(\mc O)$: 
$w(\chi_r) \chi_r^{-1} \in X_\nr (M,\sigma)$.
\end{lem}
\begin{proof}
We abbreviate $w' = r^{-1} w^{-1} r$. Since $w r w' r^{-1} = 1 \in W(M,\mc O)$, 
\[
\tilde w \cdot \tilde r \cdot \tilde{w'} \cdot \tilde{r}^{-1} \cdot \sigma \cong \sigma \in \Irr (M). 
\]
We can also work out the left hand side stepwise. Recall from Condition \ref{cond:3.1} that 
$W(\Sigma_{\mc O,\mu})$ stabilizes $\sigma \in \Irr (M)$. With \eqref{eq:3.28} we compute
\begin{align*}
\tilde w \cdot \tilde r \cdot \tilde{w'} \cdot \tilde{r}^{-1} \cdot \sigma & \cong
\tilde w \cdot \tilde r \cdot \tilde{w'} \cdot (\sigma \otimes \chi_{r^{-1}}) \\
& \cong \tilde w \cdot \tilde r \cdot \tilde{w'} \cdot (\sigma \otimes r^{-1}(\chi_r^{-1})) \\ 
& \cong \tilde w \cdot \tilde r \cdot (\sigma \otimes w' r^{-1} (\chi_r^{-1})) \\
& \cong \tilde w \cdot (\sigma \otimes \chi_r \otimes r w' r^{-1} (\chi_r^{-1})) \\
& \cong \sigma \otimes w(\chi_r) \otimes w r w' r^{-1} (\chi_r^{-1}) 
= \sigma \otimes w(\chi_r) \chi_r^{-1} . \qedhere
\end{align*}
\end{proof}

Now we have three collections of transformations of $\mc O$:
\[
\begin{array}{ccc@{\qquad}l}
\sigma \otimes \chi & \mapsto & w (\sigma \otimes \chi) \cong \sigma \otimes w (\chi) & 
w \in W(\Sigma_{\mc O,\mu}) , \\
\sigma \otimes \chi & \mapsto & 
r (\sigma \otimes \chi) \cong \sigma \otimes r (\chi) \chi_r & r \in R(\mc O) ,\\
\sigma \otimes \chi & \mapsto & \sigma \otimes \chi \chi_c & \chi_c \in X_\nr (M,\sigma) . 
\end{array} 
\]
These give rise to the following transformations of $X_\nr (M)$:
\begin{equation}\label{eq:3.29}
\begin{array}{rccc@{\qquad}l}
\mf{w} : & \chi & \mapsto & w (\chi) & w \in W(\Sigma_{\mc O,\mu}) , \\
\mf{r} : & \chi & \mapsto & r (\chi) \chi_r & r \in R(\mc O) ,\\
\chi_c : & \chi & \mapsto & \chi \chi_c & \chi_c \in X_\nr (M,\sigma) . 
\end{array}  
\end{equation}
Let $W(M,\sigma,X_\nr (M))$ be the group of transformations of $X_\nr (M)$ generated by the
$\mf{w}, \mf{r}$ and $\phi_{\chi_c}$ from \eqref{eq:3.29}. Since $X_\nr (M,\sigma)$ is $W(M,\mc O)$-stable,
it constitutes a normal subgroup of $W(M,\sigma,X_\nr (M))$. Further $W(\Sigma_{\mc O,\mu})$ embeds 
as a subgroup in $W(M,\sigma,X_\nr (M))$, and $R(\mc O)$ as the subset $\{ \mf{r} : r \in R(\mc O) \}$. 

By \eqref{eq:3.8} and Lemma \ref{lem:3.7}, the multiplication map
\begin{equation}\label{eq:3.30}
X_\nr (M,\sigma) \times R(\mc O) \times W(\Sigma_{\mc O,\mu}) \to W(M,\sigma,X_\nr (M)) 
\end{equation}
is a bijection (but usually not a group homomorphism). We note that $R(\mc O)$ does not necessarily 
normalize $W(\Sigma_{\mc O,\mu})$ in $W(M,\sigma,X_\nr (M))$:
\begin{multline*}
\mf{r} \mf{w} \mf{r}^{-1} (\chi) = \mf{r} \mf{w} (r^{-1}(\chi) r^{-1}(\chi_r^{-1})) \\
= \mf{r} (wr^{-1} (\chi) w r^{-1} (\chi_r^{-1})) = (rwr^{-1}) (\chi) (r w r^{-1})(\chi_r^{-1}) \chi_r .
\end{multline*}
Rather, $W(M,\sigma,X_\nr (M))$ is a nontrivial extension of $W(M,\mc O)$ by $X_\nr (M,\sigma)$.

Via the quotient maps 
\[
W(M,\sigma,X_\nr (M)) \to W(M,\mc O) \to W(\Sigma_{\mc O,\mu})
\]
we lift $\ell_{\mc O}$ to $W(M,\sigma,X_\nr (M))$.

\section{Intertwining operators}
\label{sec:intertwining}

We abbreviate 
\[
E_B = E \otimes_\C B = E \otimes_\C \C [X_\nr (M)] \cong \ind_{M^1}^M (E) . 
\]
By \cite[\S III.4.1]{BeRu} or \cite{Ren}, the parabolically induced representation
$I_P^G (E_B)$ is a progenerator of $\Rep (G)^{\mf s}$, Hence, as in \cite[Theorem 1.8.2.1]{Roc1},
\begin{equation}\label{eq:4.equiv}
\begin{array}{cccc}
\mc E : &\Rep (G)^{\mf s} & \to & \End_G ( I_P^G (E_B) )\text{-Mod} \\
& V & \mapsto & \Hom_G ( I_P^G (E_B), V)
\end{array}
\end{equation}
is an equivalence of categories. This equivalence commute with parabolic induction, in the
following sense. Let $L$ be a Levi subgroup of $G$ containing $L$. Then $PL$ and $\overline{P}L$
are opposite parabolic subgroups of $G$ with common Levi factor $L$. The normalized parabolic
induction functor $I_{PL}^G$ provides a natural injection
\begin{equation}\label{eq:4.33}
\End_L ( I_{P \cap L}^L (E_B) ) \longrightarrow \End_G ( I_P^G (E_B) ) ,
\end{equation}
which allows us to consider $\End_L ( I_{P \cap L}^L (E_B) )$ as a subalgebra of 
$\End_G ( I_P^G (E_B) )$. We write $\mf s_L = [M,\sigma ]_L$ and we let $\mc E_L$ be the
analogue of $\mc E$ of $L$.

\begin{prop}\label{prop:4.5}
\textup{\cite[Proposition 1.8.5.1]{Roc1}} 
\enuma{
\item The following diagram commutes:
\[
\begin{array}{ccc}
\Rep (G)^{\mf s} & \xrightarrow{\quad \mc E \quad} & \End_G ( I_P^G (E_B) ) -\Mod \\
\uparrow I_{PL}^G & & \uparrow \ind_{\End_L ( I_{P \cap L}^L (E_B) )}^{\End_G ( I_P^G (E_B) )} \\
\Rep (L)^{\mf s_L} & \xrightarrow{\quad \mc E_L \quad} &  \End_L ( I_{P \cap L}^L (E_B) ) -\Mod 
\end{array}
\]
\item Let $J^G_{\overline{P}L} : \Rep (G) \to \Rep (L)$ be the normalized Jacquet restriction 
functor and let $\mr{pr}_{\mf s_L} : \Rep (L) \to \Rep (L)^{\mf s_L}$ be the projection coming
from the Bernstein decomposition. Then the following diagram commutes:
\[
\begin{array}{ccc}
\Rep (G)^{\mf s} & \xrightarrow{\quad \mc E \quad} & \End_G ( I_P^G (E_B) ) -\Mod \\
\downarrow \mr{pr}_{\mf s_L} \circ J_{\overline{P}L}^G & & 
\downarrow \mr{Res}_{\End_L ( I_{P \cap L}^L (E_B) )}^{\End_G ( I_P^G (E_B) )} \\
\Rep (L)^{\mf s_L} & \xrightarrow{\quad \mc E_L \quad} &  \End_L ( I_{P \cap L}^L (E_B) ) -\Mod 
\end{array}
\]
}
\end{prop}

We want to find elements of $\End_G \big( I_P^G (E_B) \big)$ that do not come from
$\End_M (E_B)$. Harish-Chandra devised by now standard intertwining operators for 
$I_P^G (E)$. However, they arise as a rational functions of $\sigma \in \mc O$, so their images 
lie in $I_P^G (E \otimes_\C \C (X_\nr (M)))$ and they may have poles. We will exhibit variations 
which have fewer singularities.

We denote the $M$-representation \eqref{eq:2.3} on $E_B$ by $\sigma_B$.
Similarly we have the $M$-representation $\sigma_{K(B)}$ on
\[
E_{K(B)} = E \otimes_\C K(B) = E_B \otimes_B K(B) = E \otimes_\C \C (X_\nr (M)) . 
\]
The specialization at $\chi \in X_\nr (M)$ from \eqref{eq:2.spec} is a $M$-homomorphism
\[
\mr{sp}_\chi : (\sigma_B,E_B) \to (\sigma \otimes \chi,E) . 
\]
It extends to the subspace of $E_{K(B)}$ consisting of functions that are regular at $\chi$.

Let $\delta_P : P \to \R_{>0}$ be the modular function. We realize $I_P^G (E)$ as
\[
\{ f : G \to E \mid f \text{ is smooth, } f(u m g) = \sigma (m) \delta_P^{1/2}(m) f(g)
\; \forall g \in G, m \in M, u \in U \} .
\]
As usual $I_P^G (\sigma)(g)$ is right translation by $g$. With $I_P^G$, we can regard
$\mr{sp}_\chi$ also as a $G$-homomorphism
\[
I_P^G (\sigma_B, E_B) \to I_P^G (\sigma \otimes \chi, E) .
\]
Fix a maximal $F$-split torus $\mc A_0$ in $\mc G$, contained in $\mc M$. Let $x_0$ be a special 
vertex in the apartment of the extended Bruhat--Tits building of $(\mc G,F)$ associated to $\mc A_0$. 
Its isotropy group $K = G_{x_0}$ is a good maximal compact subgroup of $G$, so it 
contains representatives for all elements of the Weyl group $W(G,A_0)$ and $G = P K$
by the Iwasawa decomposition.  

The vector space $I_P^G (E)$ is naturally in bijection with 
\begin{multline*}
I_{P \cap K}^K (E) = \{ f : K \to E \mid f \text{ is smooth, } f(u m k) 
= \sigma (m) f(k) \\ \forall k \in K, m \in M \cap K, u \in U \cap K\} .
\end{multline*}
Notice that this space is the same for $(\sigma,E)$ and $(\sigma \otimes \chi,E)$, 
for any $\chi \in X_\nr (M)$. 

\subsection{Harish-Chandra's operators $J_{P'|P}$} \

Let $P' = M U'$ be another parabolic subgroup of $G$ with Levi factor $M$. 
Following \cite[\S IV.1]{Wal} we consider the $G$-map
\begin{equation}\label{eq:3.13}
\begin{array}{cccc}
J_{P' | P}(\sigma) : & I_P^G (E) & \to & I_{P'}^G (E) \\
& f & \mapsto & \big[ g \mapsto \int_{(U \cap U') \backslash U'} f (u' g) \textup{d} u' \big]
\end{array} .
\end{equation}
The integral does not always converge. Rather, $J_{P'|P}$ should be considered as a map
\[
\begin{array}{ccc}
X_\nr (M) \times I_{P \cap K}^K (E) & \to & I_{P' \cap K}^K (E) \\
(\chi, f) & \mapsto & J_{P' | P}(\sigma \otimes \chi) f 
\end{array} ,
\]
where $I_{P \cap K}^K (E)$ is identified with $I_P^G (\sigma \otimes \chi,E)$ as above. With this 
interpretation $J_{P'|P}$ is rational in the variable $\chi$ \cite[Th\'eor\`eme IV.1.1]{Wal}.
In yet other words, it defines a map
\begin{equation}\label{eq:3.1}
\begin{array}{ccc}
I_{P \cap K}^K (E) & \to & I_{P' \cap K}^K (E) \otimes_\C \C (X_\nr (M)) \\
f & \mapsto & \big[ \chi \mapsto J_{P' | P}(\sigma \otimes \chi) f \big] 
\end{array} . 
\end{equation}
For $h \in G$, let $\lambda (h)$ be the left translation operator on functions on $G$:
\[
\lambda (h) f : g \mapsto f (h^{-1} g) . 
\]
For every $w \in W(G,M)$ we choose a representative $\tilde w \in N_K (M)$ (that is is
possible because the maximal compact subgroup $K$ is in good position with respect to 
$A_0 \subset M$). Then $w(P) := \tilde w P \tilde w^{-1}$ is a parabolic subgroup of $G$ 
with Levi factor $M$ and unipotent radical $\tilde w U \tilde w^{-1}$. 
For any $\pi \in \Rep (M)$, $\lambda (\tilde w)$ gives a $G$-isomorphism
\[
\lambda (\tilde w) : I_{P'}^G (\pi) \to I_{w (P')}^G (\tilde w \cdot \pi) .
\]
We let $w \cdot E_B$ (resp. $w \cdot E_{K(B)}$) be the vector space $E_B$ (resp. $E_{K(B)}$)
endowed with the representation $\tilde w \cdot \sigma$ (resp. $\tilde w \cdot \sigma_{K(B)}$).

Using \cite[Th\'eor\`eme IV.1.1]{Wal} we define
\begin{equation}\label{eq:3.2}
\begin{array}{cccc}
J_{K(B),w} : & I_P^G (E_B) & \to & I_P^G (w \cdot E_{K(B)}) \\
& f & \mapsto & \big[ \chi \mapsto \lambda (\tilde w) J_{w^{-1}(P) | P}(\sigma \otimes \chi) 
\mr{sp}_\chi (f) \big]
\end{array}.
\end{equation}
It follows from \cite[Proposition IV.2.2]{Wal} that $J_{P'|P}$ and $J_{K(B),w}$ extend to 
$G$-isomorphisms
\begin{equation}\label{eq:3.3}
\begin{array}{lcccl}
J_{P'|P} & : & I_P^G (E_{K(B)}) & \to & I_{P'}^G (E_{K(B)}) , \\
J_{K(B),w} & : & I_P^G (E_{K(B)}) & \to & I_P^G (w \cdot E_{K(B)}) .
\end{array}
\end{equation}
The algebra $B$ embeds in $\End_G (I_P^G (E_B))$ and in $\End_G (I_{P'}^G (E_{K(B)}))$ via 
\eqref{eq:2.4} and parabolic induction. That makes $J_{P'|P}$ and $J_{K(B),w}$ $B$-linear.

The group $W(G,M)$ acts on $B = \C[X_\nr (M)]$ and $K(B) = \C (X_\nr (M))$ by
\[
w \cdot b_m = b_{w(m)} = b_{\tilde w m \tilde w^{-1}}, \qquad (w \cdot b)(\chi) = b (w^{-1} \chi) ,
\]
for $w \in W(G,M), m \in M, b \in K(B), \chi \in X_\nr (M)$. This determines $M$-isomorphisms 
\[
\begin{array}{cccc}
\tau_w : & (\tilde w \cdot \sigma_B,w \cdot E_B) & \to & ((\tilde w \cdot \sigma)_B, (w \cdot E)_B) \\
& (\tilde w \cdot \sigma_{K(B)},w \cdot E_{K(B)}) & \to & 
((\tilde w \cdot \sigma)_{K(B)}, (w \cdot E)_{K(B)}) \\
& e \otimes b & \mapsto & e \otimes w \cdot b 
\end{array}.
\]
With functoriality we obtain $G$-isomorphisms
\[
I_{P'}^G (w (E_B)) \to I_{P'}^G ((w \cdot E)_B) \quad \text{and} \quad 
I_{P'}^G (w (E_{K(B)})) \to I_{P'}^G ((w \cdot E)_{K(B)}),
\]
which we also denote by $\tau_w$. Composition with \eqref{eq:3.2} gives
\[
\tau_w \circ J_{K(B),w} : I_P^G (E_B) \to I_P^G ((w \cdot E)_{K(B)}) . 
\]
In order to associate to $w$ an element of $\Hom_G \big( I_P^G (E_B), I_P^G (E_{K(B)}) \big)$, 
it remains to construct a suitable $G$-intertwiner from $I_P^G ((w \cdot E)_{K(B)})$ to 
$I_P^G (E_{K(B)})$. For this we do not want to use $\tau_{w^{-1}} \circ J_{K(B),w^{-1}}$, 
then we would end up with a simple-minded $G$-automorphism of $I_P^G (E_{K(B)})$ (essentially 
multiplication with an element of $K(B)$). We rather employ an idea from \cite{Hei2}: 
construct a $G$-intertwiner $I_P^G (w \cdot E) \to I_P^G (E)$ and extend it to 
$I_P^G ((w \cdot E)_B) \to I_P^G (E_B)$ by making it constant on $X_\nr (M)$.

With this motivation we analyse the poles of the operators $J_{P'|P}$ and $J_{K(B),w}$. 
They are closely related to zeros of the Harish-Chandra $\mu$-functions. Namely, for
$\alpha \in \Sigma_\red (M)$:
\begin{equation}\label{eq:3.33}
J_{P \cap M_\alpha | s_\alpha (P \cap M_\alpha)} (\sigma \otimes \chi) J_{s_\alpha (P \cap M_\alpha) 
| P \cap M_\alpha}(\sigma \otimes \chi) = \frac{\text{constant}}{\mu^{M_\alpha}(\sigma \otimes \chi)}
\end{equation}
as rational functions of $\chi \in X_\nr (M)$ \cite[\S IV.3 and V.2]{Wal}.

\begin{prop}\label{prop:3.2}
Let $P' = M U'$ be a parabolic subgroup of $G$ with Levi factor $M$, and consider $J_{P'|P}$ in the form 
\eqref{eq:3.1}.
\enuma{
\item All the poles of $J_{P'|P}$ occur at 
\[
\bigcup\nolimits_{\alpha \in \Sigma_{\mc O,\mu}(P) \cap \Sigma_{\mc O,\mu}(\overline{P'})}
\{ \chi \in X_\nr (M) : \mu^{M_\alpha}(\sigma \otimes \chi) = 0 \} .
\]
\item Suppose that $\chi_2 \in X_\nr (M)$ satisfies $\mu^{M_\alpha}(\sigma \otimes \chi_2) = 0$ for 
precisely one $\alpha \in \Sigma_{\mc O,\mu}(P) \cap \Sigma_{\mc O,\mu}(\overline{P'})$. Then
$J_{P'|P}$ has a pole of order one at $\chi_2$ and 
\[
(X_\alpha (\chi) - X_\alpha (\chi_2)) J_{P'|P}(\sigma \otimes \chi) : I_P^G (\sigma \otimes \chi)
\to I_{P'}^G (\sigma \otimes \chi)
\]
is bijective for all $\chi$ in a certain neighborhood of $\chi_2$ in $X_\nr (M)$.
\item There exists a neighborhood $V_1$ of 1 in $X_\nr (M)$ on which
\[
\prod_{\alpha \in \Sigma_{\mc O,\mu}(P) \cap \Sigma_{\mc O,\mu}(\overline{P'})}
(X_\alpha - 1) J_{P'|P} : I_{P \cap K}^K (E) \to I_{P \cap K}^K (E) \otimes_\C \C (X_\nr (M))
\]
has no poles. The specialization of this operator to $\chi \in V_1$
is a $G$-isomorphism $I_P^G (\sigma \otimes \chi) \to I_{P'}^G (\sigma \otimes \chi)$.
}
\end{prop}
\begin{proof}
As in \cite[p. 279]{Wal} we define
\[
d (P,P') = | \{ \alpha \in \Sigma_\red (A_M) : \alpha \text{ is positive with respect to
both } P \text{ and } \overline{P'} \} | .
\]
Choose a sequence of parabolic subgroups $P_i = M U_i$ such that $d(P_i ,P_{i+1}) = 1$,
$P_0 = P$ and $P_{d(P,P')} = P'$. From \cite[p. 283]{Wal} we know that 
\[
J_{P'|P} = J_{P'|P_{d(P,P')-1}} \circ \cdots \circ J_{P_2 |P_1} \circ J_{P_1|P} . 
\]
In this way we reduce the whole proposition to the case $d(P,P') = 1$. 
Assume that, and let $\alpha \in \Sigma_\red (A_M)$ be the unique element which 
is positive respect to both $P$ and $\overline{P'}$. 

When $\alpha \not\in \Sigma_{\mc O,\mu}$, \cite[Proposition 1.10]{Hei2} says that the
specialization of $J_{P'|P}$ at any $\chi \in X_\nr (M)$ is regular and bijective. That
proves parts (a) and (c) for such a $P'$, while (b) is vacuous because $\mu^{M_\alpha}$
is constant on $\mc O$ \cite[Theorem 1.6]{Sil3}.

Suppose now that $\alpha \in \Sigma_{\mc O,\mu}$. We have
\[
(U \cap U') \backslash U' \cong U_{-\alpha} \subset M_\alpha .  
\]
Hence $J_{P'|P}$ arises by induction from $J_{P' \cap M_\alpha | P \cap M_\alpha}$,
and it suffices to consider the latter operator. We apply \cite[Lemme 1.8]{Hei2} with
$M_\alpha$ in the role of $G$, that yields parts (a) and (b) of our proposition.
Part (c) follows because $X_\alpha - 1$ has a zero of order one at $\{X_\alpha = 1\}$.
\end{proof}

\subsection{The auxiliary operators $\rho_w$} \

With Proposition \ref{prop:3.2} we define, for $w \in W(G,M)$, a $G$-homomorphism
\[
\rho'_{\sigma \otimes \chi,w} = \lambda (\tilde w) \mr{sp}_\chi 
\prod_{\alpha \in \Sigma_{\mc O,\mu} (P) \cap \Sigma_{\mc O,\mu}(w^{-1}(\overline P))\hspace{-12mm}}
\hspace{-9mm} (X_\alpha - 1) J_{w^{-1}(P)|P} (\sigma \otimes \chi) : 
I_P^G (\sigma \otimes \chi) \to I_P^G (\tilde{w}(\sigma \otimes \chi))
\]
We note that $\rho'_{\sigma \otimes \chi,w}$ is not canonical, because it depends on the 
choice of a representative $\tilde w \in N_K (M)$ for $w$.

\begin{lem}\label{lem:3.3}
For $w \in W(\Sigma_{\mc O,\mu})$, $\rho'_{\sigma,w}$ arises by parabolic induction 
from an $M$-isomorphism $\rho^{-1}_{\sigma,w} : (\sigma,E) \to (\tilde{w} \cdot \sigma ,E)$.
\end{lem}
\begin{proof}
We compare $\rho'_{\sigma \otimes \chi}$ with Harish-Chandra's operator \cite[\S V.3]{Wal} 
\[
{}^\circ c_{P|P}(w,\sigma \otimes \chi) \in \Hom_{G \times G} 
\big( \End_\C (\sigma \otimes \chi,E), \End_\C (\tilde w (\sigma \otimes \chi)) \big) .
\]
By Proposition \ref{prop:3.2} and \cite[Lemme V.3.1]{Wal} both are rational as functions
of $\chi \in X_\nr (M)$, and regular on a neighborhood of 1. For generic $\chi$ the
$G$-representations $I_P^G (\sigma \otimes \chi)$ and $I_P^G (\tilde{w}(\sigma \otimes \chi))$ 
are irreducible, so there ${}^\circ c_{P|P}(w,\sigma \otimes \chi)$ specializes to a
scalar times conjugation by $\rho'_{\sigma \otimes \chi,w}$. It follows that 
$\rho'_{\sigma \otimes \cdot,w}$ equals a rational function times the
intertwining operator associated by Harish-Chandra to $w$ and $\sigma$.

Let us make this more precise. By Condition \ref{cond:3.1} there exists an $M$-isomorphism
\[
\phi_{\tilde w} : (\tilde w \cdot \sigma ,E) \to (\sigma,E) . 
\]
For any $\chi \in X_\nr (M)$, it gives an $M$-isomorphism $\tilde w \cdot \sigma \otimes 
w \chi \to \sigma \otimes w \chi$. Consider the $G$-homomorphism 
\begin{equation}\label{eq:3.5}
I_P^G (\phi_{\tilde w}) \circ \rho'_{\sigma \otimes \chi,w} :
I_P^G (\sigma \otimes \chi,E) \to I_P^G (\sigma \otimes w \chi,E) .
\end{equation}
By the above, this is equal to a rational function times the operator
\begin{equation}\label{eq:3.6}
{}^\circ c_{P|P}(w,\sigma \otimes \chi) \in 
\Hom_G (I_P^G (\sigma \otimes \chi), I_P^G (\sigma \otimes w \chi))
\end{equation}
considered in \cite{Sil1,Sil2}. By the easier part of the Knapp--Stein theorem for $p$-adic 
groups \cite[p.244]{Sil1}, \cite[\S 5.2.4]{Sil2}, the operator \eqref{eq:3.6} specializes at 
$\chi = 1$ to the identity, while by Proposition 
\ref{prop:3.2} the operator \eqref{eq:3.5} specializes at $\chi = 1$ to an isomorphism.
Hence \eqref{eq:3.5} for $\chi = 1$ is a nonzero scalar multiple of the identity operator and 
\[
\rho'_{\sigma,w} = z I_P^G (\phi_{\tilde w})^{-1} = I_P^G (z \phi_{\tilde w}^{-1})
\]
for some $z \in \C^\times$.
\end{proof}

From $\rho'_{\sigma,w}$ and Lemma \ref{lem:3.3} we obtain an isomorphism of
$M$-$B$-representations
\[
\rho^{-1}_{\sigma,w} \otimes \mr{id}_B : (\sigma_B, E \otimes_\C B) \to 
( (\tilde w \cdot \sigma)_B, E \otimes_\C B) .
\]
Applying $I_{P'}^G$ with $P' = M U'$, this yields an isomorphism of $G$-$B$-representations
\begin{equation}\label{eq:3.7}
I_{P'}^G (\rho^{-1}_{\sigma,w} \otimes \mr{id}_B) : 
I_{P'}^G (E_B) \to I_{P'}^G ((\tilde w \cdot E)_B) 
\end{equation}
whose specialization at $P' = P, \chi = 1$ is $\rho'_{\sigma,w}$. (However, its specialization
at other $\chi \in X_\nr (M)$ need not be equal to $\rho'_{\sigma \otimes \chi,w}$.)
To comply with the notation from \cite{Hei2} we define
\begin{equation}\label{eq:3.11}
\rho_{P',w} = I_{P'}^G (\rho_{\sigma,w} \otimes \mr{id}_B) : 
I_{P'}^G ((\tilde w \cdot E)_B) \to I_{P'}^G (E_B) .
\end{equation}
Following the same procedure with $K(B)$ instead of $B$, we can also regard $\rho_{P',w}$
as an isomorphism of $G$-$B$-representations
\[
I_{P'}^G (\rho_{\sigma,w} \otimes \mr{id}_{K(B)}) : 
I_{P'}^G ((\tilde w \cdot E)_{K(B)}) \to I_{P'}^G (E_{K(B)}) .
\]
When $P' = P$, we often suppress it from the notation.
We need a few calculation rules for the operators $\rho_{P',w}$.

\begin{lem}\label{lem:3.4}
Let $w,w_1,w_2 \in W(\Sigma_{\mc O,\mu})$.
\enuma{
\item $J_{P'|P}(\sigma \otimes \cdot) \circ \rho_{P,w} = \rho_{P',w} \circ 
J_{P'|P}(\tilde w \sigma \otimes \cdot) : 
I_P^G ((\tilde w \cdot E)_{K(B)}) \to I_{P'}^G (E_{K(B)})$.
\item As operators $I_{w_2^{-1} w_1^{-1} (P)}^G (E_{K(B)}) \to I_P^G (E_{K(B)})$:
\begin{multline*}
\rho_{w_1} \tau_{w_1} \lambda (\tilde{w_1}) \rho_{w_1^{-1}(P),w_2} \tau_{w_2} \lambda (\tilde{w_2}) = \\
\prod\nolimits_\alpha \Big( \mr{sp}_{\chi = 1} \frac{\mu^{M_\alpha} (\sigma \otimes \cdot)}{(X_\alpha - 1)
(X_\alpha^{-1} - 1)} \Big) \rho_{w_1 w_2} \tau_{w_1 w_2} \lambda (\widetilde{w_1 w_2}) 
\end{multline*}
where the product runs over $\Sigma_{\mc O,\mu}(P) \cap \Sigma_{\mc O,\mu}(w_2^{-1}(\overline P))
\cap \Sigma_{\mc O,\mu}(w_2^{-1} w_1^{-1} (P))$.
\item For $r \in R(\mc O)$:
\[
\lambda (\tilde r) \rho_{r^{-1}(P),w} \lambda (\tilde r)^{-1} = 
\rho_{P,\tilde{r} \cdot \sigma, r w r^{-1}}   
\lambda \big( \widetilde{r w r^{-1}} \tilde{r} \tilde{w}^{-1} \tilde{r}^{-1} \big).
\]
} 
\end{lem}
\begin{proof}
(a) In this setting $J_{P'|P}$ is invertible \eqref{eq:3.3}, so we can reformulate the claim as
\[
J_{P'|P}(\tilde w \cdot \sigma \otimes \cdot )^{-1} \circ \rho^{-1}_{P',w} \circ 
J_{P'|P}(\sigma \otimes \cdot ) = \rho^{-1}_{P,w} .
\]
The left hand side one first transfers everything from $I_P^G$ to $I_{P'}^G$ by means of 
$\int_{(U \cap U') \backslash U'}$, then we apply $\rho_{P',w}^{-1} = I_{P'}^G (\rho''_{\sigma,w}
\otimes \mr{id}_B)$ and finally we transfer back from $I_{P'}^G$ to $I_P^G$ (in the opposite fashion). 
In view of \eqref{eq:3.7}, this is just a complicated way to express 
$\rho_{P,w}^{-1} = I_P^G (\rho''_{\sigma,w} \otimes \mr{id}_B)$.\\
(b) The map
\[
\tau_{w_1} \lambda (\tilde{w_1}) \rho_{P',w_2} \lambda (\tilde{w_1})^{-1} \tau_{w_1}^{-1} :
I_{w_1 (P')}^G ((\tilde{w_1} \tilde{w_2} \cdot E)_{K(B)}) \to 
I_{w_1 (P')}^G ((\tilde{w_1} \cdot E)_{K(B)})
\]
is denoted simply $\rho_{w_2}$ in \cite{Hei2}. We note that part (a) proves the first formula
of \cite[Proposition 2.4]{Hei2} in larger generality, without a condition on $w$. Knowing this, the 
claim can shown in the same way as the second part of \cite[Propsition 2.4]{Hei2} (on page 729).\\
(c) By definition
\[
\begin{array}{lll}
\rho_{r^{-1}(P),w}^{-1} & = & \lambda (\tilde{s_\alpha}) \mr{sp}_{\chi = 1} \prod_\alpha
(X_\alpha (\chi) - 1) J_{w^{-1} r^{-1} (P)|r^{-1}(P)}(\sigma \otimes \chi) , \\
\rho_{P,\tilde r \cdot \sigma,r w r^{-1}}^{-1} & = 
& \lambda \big( \widetilde{r w r^{-1}} \big) \mr{sp}_{\chi = 1} \prod_\beta 
(X_\beta (\chi) - 1) J_{r w^{-1} r^{-1} (P)|P}(\tilde r \cdot \sigma \otimes \chi) .
\end{array}
\]
Here $\alpha$ runs over $\Sigma_{\mc O,\mu}(r^{-1}P) \cap \Sigma_{\mc O,\mu}(w^{-1} r^{-1} \overline{P})$
and $\beta$ over
\[
\Sigma_{\mc O,\mu}(P) \cap \Sigma_{\mc O,\mu}(r w^{-1} r^{-1}\overline{P}) =
r \big( \Sigma_{\mc O,\mu}(r^{-1}P) \cap \Sigma_{\mc O,\mu}(w^{-1} r^{-1} \overline{P}) \big) .
\]
It follows that
\begin{align*}
\lambda (\tilde r) \rho_{r^{-1}(P),w}^{-1} \lambda (\tilde r)^{-1} & 
= \lambda \big( \tilde{r} \tilde{w} \tilde{r}^{-1} \big) \mr{sp}_{\chi = 1} \prod\nolimits_\beta 
(X_\beta (\chi) - 1) J_{r w^{-1} r^{-1} (P)|P}(\tilde r \cdot \sigma \otimes \chi) \\
& = \lambda \big( \tilde{r} \tilde{w} \tilde{r}^{-1} \widetilde{r w r^{-1}}^{-1} \big)
\rho_{P,\tilde r \cdot \sigma,r w r^{-1}}^{-1} .
\end{align*}
Taking inverses yields the claim.
\end{proof}

Now we associate similar operators to elements of the group $R(\mc O)$ from \eqref{eq:3.19} and 
\eqref{eq:3.8}. We may assume that the representatives $\tilde w \in N_K (M)$ are chosen so that
\begin{equation}\label{eq:3.10}
\widetilde{r w_{\mc O}} = \tilde r \tilde{w_{\mc O}} 
\text{ for all } r \in R(\mc O), w_{\mc O} \in W (\Sigma_{\mc O,\mu}) .
\end{equation}
For $r \in R(\mc O)$, Proposition \ref{prop:2.2} and \cite[Proposition 1.10]{Hei2} say that
$J_{r(P)|P}$ is rational and regular on $X_\nr (M)$, and that its specialization at any 
$\chi$ is a $G$-isomorphism $I_P^G (\sigma \otimes \chi) \to I_{r(P)}^G (\sigma \otimes \chi)$.
For such $r$ we construct an analogue of $\rho_w$ in a simpler way. 
Let $\chi_r$ be as in \eqref{eq:3.27} and pick an $M$-isomorphism
\begin{equation}\label{eq:3.12}
\rho_{\sigma,r} : \tilde r \cdot \sigma \to \sigma \otimes \chi_r .
\end{equation}
Recall $\rho_{\chi_r}$ from \eqref{eq:2.25}. It combines
with $\rho_{\sigma,r}$ to an $M$-isomorphism
\[
\rho_{\sigma,r} \otimes \rho_{\chi_r}^{-1} : ((\tilde r \cdot \sigma)_B ,E_B) \to (\sigma_B,E_B),
\]
which is not $B$-linear when $\chi_r \neq 1$.
With parabolic induction we obtain a $G$-isomorphism
\[
\rho_{P',r} = I_{P'}^G ({\rho_{\sigma,r}} \otimes \rho_{\chi_r}^{-1}) : 
I_{P'}^G ((\tilde r \cdot \sigma)_B ,E_B) \to I_{P'}^G (\sigma_B ,E_B) .
\]
The same works with $K(B)$ instead of $B$. 

We note that Lemma \ref{lem:3.4}.a also applies to $\rho_r$, with the same proof:
\begin{equation}\label{eq:3.9}
J_{P'|P}(\sigma \otimes \cdot) \circ \rho_{P,r} = \rho_{P',r} \circ 
J_{P'|P}(\tilde r \sigma \otimes \chi_r^{-1} \otimes \cdot) : 
I_P^G ((\tilde r \cdot E)_{K(B)}) \to I_{P'}^G (E_{K(B)}). 
\end{equation}
For an arbitrary $w \in W(M,\mc O)$, we use \eqref{eq:3.8} and \eqref{eq:3.10} to write 
$\tilde w = \tilde r \tilde{w_{\mc O}}$ with $r \in R(\mc O)$ and $w_{\mc O} \in W(\Sigma_{\mc O,\mu})$. 
Then we put $\chi_w = \chi_r$ and
\[
\begin{array}{ccccccc}
\rho_{\sigma,w} & = & \rho_{\sigma,r} \rho_{\sigma,w_{\mc O}} & : & 
\tilde{w} \cdot \sigma \otimes \chi & \to & \sigma \otimes \chi_r \chi , \\
\rho_{P',w} & = & \rho_{P',r} \tau_r \lambda (\tilde r) \rho_{r^{-1}(P'),w_{\mc O}} 
\lambda (\tilde r)^{-1} \tau_r^{-1} & : & 
I_{P'}^G ((w \cdot E)_B) & \to & I_{P'}^G (E_B) .
\end{array} 
\]
Let us discuss the multiplication relations between all the operators constructed in this section
and the $\phi_\chi$ with $\chi \in X_\nr (M,\sigma)$ from \eqref{eq:2.24}. 
Via $I_{P'}^G$, we regard $\phi_\chi$ also as an element of $\Hom_G (I_{P'}^G (E_B))$. We note that
\begin{equation}\label{eq:3.26}
\mr{sp}_{\chi'} \circ \phi_\chi = \mr{sp}_{\chi'} \circ (\phi_{\sigma,\chi} \otimes \rho_\chi^{-1}) =
\phi_{\sigma,\chi} \otimes \mr{sp}_{\chi' \chi^{-1}} \qquad \chi' \in X_\nr (M) .
\end{equation}
From the very definition of $J_{P'|P}$ in \eqref{eq:3.13} we see that
\[
\phi_{\sigma,\chi}^{-1} \circ J_{P'|P}(\sigma \otimes \chi' \chi) \circ \phi_{\sigma,\chi} =
J_{P'|P}(\sigma \otimes \chi') ,
\]
which quickly implies
\begin{equation}\label{eq:3.14}
J_{P|P'} \circ \phi_\chi = \phi_\chi \circ J_{P|P'} 
\in \Hom_G \big( I_P^G (E_{K(B)}), I_{P'}^G (E_{K(B)}) \big) .
\end{equation}
For any $w \in W(M,\mc O)$, we take 
\begin{equation}\label{eq:3.32}
\phi_{\tilde w \cdot \sigma, w (\chi)} \in 
\Hom_M (\tilde w \cdot \sigma, \tilde w \cdot \sigma \otimes w (\chi))
\end{equation}
equal to $\phi_{\sigma,\chi}$ as $\C$-linear map. Next we define 
\[
\phi_{w (\chi)}^w = \phi_{\tilde w \cdot \sigma, w (\chi)} \otimes \rho_{w (\chi)}^{-1} \in
\End_M ( (w \cdot E)_B) ,
\]
and we tacitly extend to an element of $\End_G ( I_P^G (w \cdot E)_B )$ by functoriality. Then
\begin{equation}\label{eq:3.15}
\tau_w \lambda (\tilde w) \phi_\chi = \tau_w \lambda (\tilde w) (\phi_{\sigma,\chi} \otimes \rho_\chi^{-1}) 
= (\phi_{\tilde w \cdot \sigma, w (\chi)} \otimes \rho_{w (\chi)}^{-1}) \tau_w \lambda (\tilde w) =
\phi^w_{w (\chi)} \tau_w \lambda (\tilde w) .
\end{equation}
By the irreducibility of $\sigma$ there exists a $z \in \C^\times$ such that
\begin{equation}\label{eq:3.16}
{\rho_{\sigma,w}} \phi_{\tilde w \cdot \sigma,w (\chi)} = 
z \phi_{\sigma,w(\chi)} {\rho_{\sigma,w}} : \tilde w \cdot \sigma \to \sigma \otimes w(\chi) .
\end{equation}
With that we compute
\begin{equation}\label{eq:3.17}
\begin{aligned}
\rho_w \circ \phi_{w (\chi)}^w & = I_P^G \big( ({\rho_{\sigma,w}} \otimes \mr{id}_B) 
(\phi_{\tilde w \cdot \sigma,w (\chi)} \otimes \rho_{w(\chi)}^{-1}) \big) \\
& = I_P^G \big( z \phi_{\sigma,w(\chi)} {\rho_{\sigma,w}} \otimes \rho_{w(\chi)}^{-1} \big) \\
& = z I_P^G \big( \phi_{\sigma,w(\chi)} \otimes \rho_{w(\chi)}^{-1} \big) 
I_P^G \big( {\rho_{\sigma,w}} \otimes \mr{id}_B \big) \; = \; z \phi_{w(\chi)} \rho_w .
\end{aligned}
\end{equation}
From \eqref{eq:3.14}--\eqref{eq:3.17} we deduce that
\[
\rho_w \tau_w \lambda (\tilde w) J_{w^{-1}(P)|P'} \phi_\chi = z \phi_{w (\chi)} \tau_w 
\lambda (\tilde w) J_{w^{-1}(P)|P'} \in \Hom_G \big( I_{P'}^G (E_B), I_P^G (E_{K(B)}) \big) .
\]

\section{Endomorphism algebras with rational functions}
\label{sec:rational}

\subsection{The operators $A_w$} \

Let $B$ and the $\phi_\chi$ with $\chi \in X_\nr (M,\sigma)$ from \eqref{eq:2.24} act on
$I_P^G (E_B)$ and $I_P^G (E_{K(B)})$ by parabolically inducing their actions on
$E_B$ and $E_{K(B)}$.
For $w \in W(M,\mc O)$ we combine the operators $J_{K(B),w}, \tau_w$ and $\rho_w$
from Section \ref{sec:intertwining} to a $G$-homomorphism
\[
A_w = \rho_w \circ \tau_w \circ J_{K(B),w} : I_P^G (E_B) \to I_P^G (E_{K(B)}) . 
\]
With \eqref{eq:3.3} we can also regard $A_w$ as an invertible element of 
$\End_G \big( I_P^G (E_{K(B)}) \big)$. According to \cite[Proposition 3.1]{Hei2}, 
$A_w$ does not depend on the choice of the representative $\tilde w \in N_K (M)$
of $w$. Hence $A_w$ is canonical for $w \in W(\Sigma_{\mc O,\mu})$, while for
$w \in R(\mc O)$ it depends on the choices of $\chi_r$ and $\rho''_r \in 
\Hom_M (\sigma, \tilde r \cdot \sigma \otimes \chi_r^{-1})$. Further 
\cite[Lemme 3.2]{Hei2} says that, for every $\chi \in X_\nr (M)$ such that
$J_{w^{-1}(P)|P}(\sigma \otimes \chi)$ is regular:
\[
\mr{sp}_{(w \chi) \chi_w} A_w (v) = \rho_w \lambda (\tilde w) 
J_{w^{-1}(P)|P}(\sigma \otimes \chi) \mr{sp}_\chi (v)  \qquad v \in I_P^G (E_B) .
\]
Consequently, for any $b \in B = \C [X_\nr (M)]$:
\begin{multline}\label{eq:4.1}
\mr{sp}_{(w \chi) \chi_w} A_w (b v) = \rho_w \circ \lambda (\tilde w) \circ 
J_{w^{-1}(P)|P}(b(\chi) \mr{sp}_\chi (v)) \\
= b (\chi) \mr{sp}_{(w \chi) \chi_w} A_w (v)
= \mr{sp}_{(w \chi) \chi_w} \big( (w \cdot b)_{\chi_w^{-1}} A_w (v) \big) .
\end{multline}
In view of Proposition \ref{prop:3.2}.a, this holds for $\chi$ in a nonempty Zariski-open
subset of $X_\nr (M)$. Thus
\begin{equation}\label{eq:4.2}
A_w \circ b = (w \cdot b)_{\chi_w^{-1}} \circ A_w \qquad \in
\Hom_G \big( I_P^G (E_B), I_P^G (E_{K(B)}) \big) .
\end{equation}
From \eqref{eq:3.14}--\eqref{eq:3.17} we see that for all $w \in W(M,\mc O), \chi \in X_\nr (M,\sigma)$
there exists a $z (w,\chi) \in \C^\times$ such that
\begin{equation}\label{eq:4.7}
A_w \circ \phi_\chi = z (w,\chi) \phi_{w(\chi)} \circ A_w 
\qquad \in \Hom_G \big( I_P^G (E_B), I_P^G (E_{K(B)}) \big).
\end{equation}
Compositions of the operators $A_w$ are not as straightforward as one could expect.

\begin{prop}\label{prop:4.1}
Let $w_1, w_2 \in W(\Sigma_{\mc O,\mu})$.
\enuma{
\item As $G$-endomorphisms of $I_P^G (E_{K(B)})$:
\begin{align*}
A_{w_1} \circ A_{w_2} & = \prod\nolimits_\alpha \Big(  \mr{sp}_{\chi = 1} \frac{\mu^{M_\alpha}
(\sigma \otimes \cdot)}{ (X_\alpha - 1)(X_\alpha^{-1} - 1)} \Big) \mu^{M_\alpha}(\sigma 
\otimes w_2^{-1} w_1^{-1} \cdot)^{-1} A_{w_1 w_2} \\
& = A_{w_1 w_2} \prod\nolimits_\alpha \Big( \mr{sp}_{\chi = 1} \frac{\mu^{M_\alpha}
(\sigma \otimes \cdot)}{ (X_\alpha - 1)(X_\alpha^{-1} - 1)} \Big) \mu^{M_\alpha}(\sigma 
\otimes \cdot)^{-1} 
\end{align*}
where the products run over $\Sigma_{\mc O,\mu}(P) \cap \Sigma_{\mc O,\mu}(w_2^{-1}(\overline P))
\cap \Sigma_{\mc O,\mu}(w_2^{-1} w_1^{-1} (P))$.
\item If $\ell_{\mc O}(w_1 w_2) = \ell_{\mc O}(w_1) + \ell_{\mc O}(w_2)$, then 
$A_{w_1 w_2} = A_{w_1} \circ A_{w_2}$.
\item For $\alpha \in \Delta_{\mc O,\mu}$: 
\[
A_{s_\alpha}^2 = \frac{4 c'_{s_\alpha}}{(1 - q_{\alpha}^{-1})^2 
(1 + q_{\alpha*}^{-1})^2 \mu^{M_\alpha}(\sigma \otimes \cdot)} .
\]
}
\end{prop}
\begin{proof}
The second equality in part (a) is an instance of \eqref{eq:4.2}.

Lemma \ref{lem:3.4} is equivalent to two formulas established in \cite[Proposition 2.4]{Hei2} for
classical groups. With those at hand, the parts (a) and (b) can be shown in the same way as
\cite[Proposition 3.3 and Corollaire 3.4]{Hei2}. Part (c) is a special case of part (a), made
explicit with \eqref{eq:3.22}.
\end{proof}

For $r \in R (\mc O)$, Proposition \ref{prop:3.2}.a implies that $J_{r^{-1}(P)|P}$ does not have
any poles on $\mc O$. Hence it maps $I_P^G (E_B)$ to itself, and
\begin{equation}\label{eq:4.6}
A_r = \rho_{P,r} \tau_r \lambda (\tilde r) J_{r^{-1}(P)|P} \; \in \; \End_G (I_P^G (E_B)) .
\end{equation}
The maps $A_r$ with $r \in R(\mc O)$ behave more multiplicatively than in Proposition \ref{prop:4.1}, 
but still they do not form a group homomorphism in general.

\begin{prop}\label{prop:4.2}
Let $r,r_1,r_2 \in R(\mc O)$ and $w, w' \in W(\Sigma_{\mc O,\mu})$.
\enuma{
\item Write $\chi (r_1,r_2) = \chi_{r_1} r_1 (\chi_{r_2}) \chi_{r_1 r_2}^{-1} \in X_\nr (M,\sigma)$ 
and recall $\phi_{\chi (r_1,r_2)}$ from \eqref{eq:2.25}. There exists a 
$\natural (r_1,r_2) \in \C^\times$ such that
\[
A_{r_1} \circ A_{r_2} = \natural (r_1,r_2) \phi_{\chi (r_1,r_2)} \circ A_{r_1 r_2} . 
\]
\item $A_r \circ A_w = A_{r w}$.
\item There exists a $\natural (w',r) \in \C^\times$ such that
\[
A_{w'} \circ A_r = \natural (w',r) \phi_{w' (\chi_r^{-1}) \chi_r} \circ A_{w' r} .
\]
If $w' (\chi_r) = \chi_r$, then $\natural (w',r) = 1$ and $A_{w'} \circ A_r = A_{w' r}$.
}
\end{prop}
\begin{proof}
(a) By \eqref{eq:3.12}
\[
\sigma \otimes \chi_{r_1 r_2} \cong \widetilde{r_1 r_2} \cdot \sigma \cong \tilde{r_1}
\tilde{r_2} \cdot \sigma \cong \tilde{r_1} \cdot (\sigma \otimes \chi_{r_2}) \cong 
\tilde{r_1} \cdot \sigma \otimes r_1 (\chi_{r_2}) \cong \sigma \otimes \chi_{r_1} r_1 (\chi_{r_2}) .
\]
Hence the unramified characters $\chi_{r_1 r_2}$ and $\chi_{r_1} r_1 (\chi_{r_2})$ differ only
by an element $\chi_c \in X_\nr (M,\sigma)$ (as already used in the statement).
With \eqref{eq:3.9} we compute
\begin{align*}
A_{r_1} \circ A_{r_2} & = \rho_{r_1} \tau_{r_1} \lambda (\tilde{r_1}) J_{r_1^{-1}(P)|P}(\sigma \otimes \cdot)
\rho_{r_2} \tau_{r_2} \lambda (\tilde{r_2}) J_{r_2^{-1}(P)|P}(\sigma \otimes \cdot) \\
& = \rho_{P,r_1} \tau_{r_1} \lambda (\tilde{r_1}) \rho_{r^{-1}(P),r_2} J_{r_1^{-1}(P)|P}(\tilde{r_2} \cdot 
\sigma \otimes \cdot) \tau_{r_2} \lambda (\tilde{r_2}) J_{r_2^{-1}(P)|P}(\sigma \otimes \cdot) \\
& = \rho_{P,r_1} \tau_{r_1} \lambda (\tilde{r_1}) \rho_{r^{-1}(P),r_2} \tau_{r_2} \lambda (\tilde{r_2}) 
J_{r_2^{-1}r_1^{-1}(P)|r_2^{-1} P}(\sigma \otimes \cdot) J_{r_2^{-1}(P)|P}(\sigma \otimes \cdot) .
\end{align*}
Now we use that $r_1,r_2 \in R(\mc O)$, which by \cite[Proposition 1.9]{Hei2} or \cite[IV.3.(4)]{Wal} 
implies that the J-operators in the previous line compose in the expected way. Hence
\begin{equation}\label{eq:4.3}
A_{r_1} \circ A_{r_2} = \rho_{P,r_1} \tau_{r_1} \lambda (\tilde{r_1}) \rho_{r^{-1}(P),r_2} 
\tau_{r_2} \lambda (\tilde{r_2}) J_{r_2^{-1}r_1^{-1}(P)|P}(\sigma \otimes \cdot) .
\end{equation}
Comparing \eqref{eq:4.3} with the definition of $A_{r_1 r_2}$, we see that it remains to relate
\begin{equation}\label{eq:4.4}
\rho_{P,r_1} \tau_{r_1} \lambda (\tilde{r_1}) \rho_{r^{-1}(P),r_2} \tau_{r_2} \lambda (\tilde{r_2}) 
\end{equation}
to $\rho_{P,r_1 r_2} \tau_{r_1 r_2} \lambda (\widetilde{r_1 r_2})$. Both \eqref{eq:4.4} and 
\[
\phi_{\chi (r_1,r_2)} \rho_{P,r_1 r_2} \tau_{r_1 r_2} \lambda (\widetilde{r_1 r_2})
\]
give $G$-homomorphisms 
\[
I_{r_2^{-1} r_1^{-1}(P)}^G (\sigma \otimes \chi) \to I_P^G \big( (\sigma \otimes \chi_{r_1} r_1 
(\chi_{r_2})) \otimes \chi_{r_1}^{-1} r_1 (\chi_{r_2}^{-1}) \chi \big) 
\]
that are constant in $\chi \in X_\nr (M)$, because $\tilde{r_i} \in K$. 
For generic $\chi$ the involved $G$-representations are irreducible, so then
\begin{equation}\label{eq:4.5}
\mr{sp}_\chi \rho_{P,r_1} \tau_{r_1} \lambda (\tilde{r_1}) \rho_{r^{-1}(P),r_2} \tau_{r_2} 
\lambda (\tilde{r_2}) = \natural (\chi) \mr{sp}_\chi \phi_{\chi (r_1,r_2)} 
\rho_{P,r_1 r_2} \tau_{r_1 r_2} \lambda (\widetilde{r_1 r_2})
\end{equation}
for some $\natural (\chi) \in \C^\times$. But then $\natural (\chi)$ does not depend on $\chi$ 
(for generic $\chi$), so it is a constant $\natural (r_1,r_2)$ and in fact \eqref{eq:4.5} 
already holds without specializing at $\chi$. With 
\eqref{eq:4.3} we find the required expression for $A_{r_1} \circ A_{r_2}$.\\
(b) Pick any $\chi \in X_\nr (M,\sigma)$. With Lemma \ref{lem:3.4}.a one easily computes
\begin{multline}\label{eq:4.9}
\mr{sp}_\chi A_r \circ A_w = \mr{sp}_\chi A_{r w} = \\
\rho_{P,r} \lambda (\tilde r) \rho_{r^{-1}(P),w} \lambda (\tilde{w}) J_{w^{-1} r^{-1} 
(P)|P}(\sigma \otimes w^{-1} r^{-1} (\chi \chi_r^{-1})) \mr{sp}_{w^{-1} r^{-1} (\chi \chi_r^{-1})} .
\end{multline}
(c) We relate this to part (b) by setting $w = r^{-1} w' r$. 
By Lemma \ref{lem:3.4}, \eqref{eq:4.9} becomes 
\begin{equation}\label{eq:4.12}
\rho_{P,r} \rho_{P, \tilde r \cdot \sigma,w} \lambda (\tilde{w'}) \lambda (\tilde r) 
J_{w^{-1} r^{-1} (P)|P}(\sigma \otimes w^{-1} r^{-1} 
(\chi \chi_r^{-1})) \mr{sp}_{w^{-1} r^{-1} (\chi \chi_r^{-1})} .
\end{equation}
A similar computation yields
\begin{multline}\label{eq:4.11}
\mr{sp}_{\chi'} A_{w'} \circ A_r = \\ 
\rho_{P,w'} \lambda (\tilde{w'}) \rho_{w'^{-1}(P),r} 
\lambda (\tilde r) J_{w^{-1} r^{-1} (P)|P}(\sigma \otimes w^{-1} r^{-1} 
(\chi') r^{-1}(\chi_r^{-1})) \mr{sp}_{w^{-1} r^{-1} (\chi') r^{-1} (\chi_r^{-1})} 
%& = \rho_{P,w'} \rho_{P,\tilde{w'} \sigma,r} 
%\lambda (\tilde{w'} \tilde r) J_{w^{-1} r^{-1} (P)|P}(\sigma \otimes w^{-1} r^{-1} 
%(\chi') r^{-1}(\chi_r^{-1})) \mr{sp}_{w^{-1} r^{-1} (\chi') r^{-1} (\chi_r^{-1})}
\end{multline}
Thus it remains to compare
\begin{equation}\label{eq:4.10}
\rho_{P,r} \rho_{P, \tilde r \cdot \sigma,w} \mr{sp}_{\chi \chi_r^{-1}} \quad \text{and} \quad
\rho_{P,w'} \lambda (\tilde{w'}) \rho_{w'^{-1}(P),r} \lambda (\tilde{w'})^{-1} 
\mr{sp}_{\chi' w' (\chi_r^{-1})} .
\end{equation}
Lemma \ref{lem:3.7} guarantees that $w'(\chi_r^{-1}) \chi_r \in X_\nr (M,\sigma)$. Recall the 
convention \eqref{eq:3.32}. By Schur's lemma there exists a $\natural (w',r) \in \C^\times$ such that
\begin{equation}\label{eq:4.13}
\rho_{\sigma,r} \rho_{\tilde r \cdot \sigma,w}  = \natural (w',r) \phi_{\sigma,w'(\chi_r^{-1}) 
\chi_r} \rho_{\sigma,w'} \rho_{\tilde{w'} \sigma,r} : 
\tilde{w'} \tilde{r} \sigma \otimes \chi \chi_r^{-1} \to \sigma \otimes \chi .
\end{equation}
Instead of $A_{r w r^{-1}} \circ A_r$ we consider $\phi_{w' (\chi_r^{-1}) \chi_r} \circ 
A_{w'} \circ A_r$. Set $\chi' = \chi w'(\chi_r) \chi_r^{-1}$, compose \eqref{eq:4.11} on 
the left with $\natural (w',r) \phi_{\sigma,w' (\chi_r^{-1}) \chi_r}$ and recall \eqref{eq:3.31}. 
With \eqref{eq:4.13} we find
\begin{multline*}
\mr{sp}_\chi \natural (w',r) \phi_{w' (\chi_r^{-1}) \chi_r} A_{w'} A_r \; = \;
\natural (w',r) \phi_{\sigma,w' (\chi_r^{-1}) \chi_r} \mr{sp}_{\chi'} A_{w'} A_r = \\ 
\rho_{\sigma,r} \rho_{\tilde r \cdot \sigma,w} \lambda (\tilde{w'}) 
\lambda (\tilde r) J_{w^{-1} r^{-1} (P)|P}(\sigma \otimes w^{-1} r^{-1} 
(\chi \chi_r^{-1})) \mr{sp}_{w^{-1} r^{-1} (\chi \chi_r^{-1})} .
\end{multline*}
The last line equals \eqref{eq:4.12} and \eqref{eq:4.11}. This holds for every 
$\chi \in X_\nr (M)$, so we obtain the desired expression for $A_{w'} \circ A_r$.

If in addition $w' (\chi_r) = \chi_r$, then
\[
\rho_{\sigma,r}^{-1} \circ \rho_{\sigma,w'} \circ \rho_{\tilde{w'} \sigma,r} = 
\rho_{\tilde r \cdot \sigma,w} .
\]
In that case the two sides of \eqref{eq:4.10} are equal (with $\chi' = \chi$).
\end{proof}

With Bernstein's geometric lemma we can determine the rank of $\End_G (I_P^G (E_B))$
as $B$-module:

\begin{lem}\label{lem:5.1}
The $B$-module $\End_G (I_P^G (E_B))$ admits a filtration with successive subquotients 
isomorphic to $\Hom_M (w \cdot E_B,E_B)$, where $w \in W(M,\mc O)$.
This same holds for $\Hom_G \big( I_P^G (E_B),I_P^G (E_{K(B)})\big)$, with subquotients
$\Hom_M (w \cdot E_B,E_{K(B)})$.
\end{lem}
\begin{proof}
This is similar to \cite[Proposition 1.8.4.1]{Roc1}. Let $r_P^G : \Rep (G) \to \Rep (M)$ be 
the normalized Jacquet restriction functor associated to $P = M U$. By Frobenius reciprocity
\begin{equation}\label{eq:7.20}
\Hom_G (I_P^G (E_B), I_P^G (E_B) ) \cong \Hom_M (r_P^G I_P^G (E_B), E_B) .
\end{equation}
According to Bernstein's geometric lemma \cite[Th\'eor\`eme VI.5.1]{Ren}, 
$r_P^G I_P^G (E_B)$ has a filtration whose successive subquotients are
\[
I^M_{(M \cap w^{-1} M w)(M \cap P)} \circ w \circ r^M_{M \cap w M w^{-1})(M \cap P)} E_B
\]
with $w \in W(M,A_0) \backslash W(G,A_0) / W(M,A_0)$. That induces a filtration of
\eqref{eq:7.20} with subquotients isomorphic to 
\begin{equation}\label{eq:7.21}
\Hom_M \Big( I^M_{(M \cap w^{-1} M w)(M \cap P)} \circ w \circ r^M_{M \cap w M w^{-1})
(M \cap P)} E_B , E_B \Big) .
\end{equation}
By the Bernstein decomposition and the definition of $W(M,\mc O)$, \eqref{eq:7.21} is zero
unless $w \in W(M,\mc O)$. For $w \in W(M,\mc O)$, \eqref{eq:7.21} simplifies to 
$\Hom_M (w \cdot E_B , E_B)$, which we can analyse further with \eqref{eq:2.22}.
Thus \eqref{eq:7.20} has a filtration with subquotients
\begin{equation}\label{eq:7.22}
\Hom_M (w \cdot E_B , E_B) \cong \bigoplus_{\chi \in X_\nr (M,\sigma)} \phi_{\chi} B = 
\bigoplus_{\chi \in X_\nr (M,\sigma)} B \phi_{\chi} 
\end{equation}
where $w$ runs through $W(M,\mc O)$. The same considerations apply to\\
$\Hom_G \big( I_P^G (E_B),I_P^G (E_{K(B)})\big)$.
\end{proof}

Now we can generalize \cite[Theorem 3.8]{Hei2} and describe the space of\\ 
$G$-homomorphisms that we are after in this subsection:

\begin{thm}\label{thm:4.3}
As vector spaces over $K(B) = \C (X_\nr (M))$:
\[
\Hom_G \big( I_P^G (E_B), I_P^G (E_{K(B)}) \big) = \bigoplus_{w \in W(M,\mc O)} \bigoplus_{\chi \in
X_\nr (M,\sigma)} K(B) A_w \phi_\chi .
\]
\end{thm}
\begin{proof}
We need Proposition \ref{prop:2.2} and \eqref{eq:4.2}. With those, the proof (for classical groups)
in \cite[Proposition 3.7]{Hei2} applies and shows that the operators $\phi_\chi A_w$ with 
$w \in W(M,\mc O)$ and $\chi \in X_\nr (M,\sigma)$ are linearly independent over $K(B)$. 
Further by \eqref{eq:7.22}, with the second $E_B$ replaced by $E_{K(B)}$, the dimension of\\
$\Hom_G \big( I_P^G (E_B), I_P^G (E_{K(B)}) \big)$ over $K(B)$ is exactly
$|X_\nr (M,\sigma)| \, |W(M,\mc O)|$.
\end{proof}

Since all elements of $K(B) A_w \phi_\chi$ extend naturally to $G$-endomorphisms of\\
$I_P^G (E_{K(B)})$, Theorem \ref{thm:4.3} shows that 
$\Hom_G \big( I_P^G (E_B), I_P^G (E_{K(B)}) \big)$ is a subalgebra of
$\End_G (I_P^G (E_{K(B)}))$.
The multiplication relations from Proposition \ref{prop:4.2} become more transparant
if we work with the group $W(M,\sigma,X_\nr (M))$ from \eqref{eq:3.29}. For 
$\chi_c \in X_\nr (M,\sigma), r \in R(\mc O)$ and $w \in W(\Sigma_{\mc O,\mu})$ we define
\[
A_{\chi_c \mf{r} \mf{w}} = \phi_{\chi_c} A_r A_w \in \End_G (I_P^G (E_{K(B)})) . 
\]
By \eqref{eq:2.29} and Propositions \ref{prop:4.1} and \ref{prop:4.2}, all the 
$A_{\chi \mf{r} \mf{w}}$ are invertible in \\
$\End_G (I_P^G (E_{K(B)}))$. By \eqref{eq:2.26} and \eqref{eq:4.2}, for $b \in \C (X_\nr (M))$:
\begin{equation}\label{eq:4.16}
A_{\chi_c \mf{r} \mf{w}} \, b \, A_{\chi_c \mf{r} \mf{w}}^{-1} = 
(r w \cdot b)_{\chi_c^{-1} \chi_r^{-1}} = b \circ (\chi_c \mf{r} \mf{w})^{-1} \in \C (X_\nr (M)) .
\end{equation}
This implies that we may change the order of the factors in Theorem \ref{thm:4.3}, for any 
$w \in W(M,\mc O)$, $\chi_c \in X_\nr (M,\sigma)$:
\begin{equation}\label{eq:4.8}
B A_w \phi_\chi = A_w \phi_\chi B \qquad \text{and} \qquad
K(B) A_w \phi_\chi = A_w \phi_\chi K(B) .
\end{equation}

\subsection{The operators $\mc T_w$} \label{par:Tw} \

To simplify the multiplication relations between the $A_w$, we will introduce a variation.
For any $\alpha \in \Sigma_{\mc O,\mu}$ we write
\[
g_\alpha = \frac{(1 - X_\alpha)(1 + X_\alpha) (1 - q_\alpha^{-1}) (1 + q_{\alpha*}^{-1})}
{2 (1 - q_\alpha^{-1} X_\alpha) (1 + q_{\alpha*}^{-1} X_\alpha)} \in \C (X_\nr (M)) .
\]
By \eqref{eq:3.21} and Lemma \ref{lem:3.8}.a
\begin{equation}\label{eq:4.24}
w \cdot g_\alpha = g_{w (\alpha)} \qquad \alpha \in \Sigma_{\mc O,\mu}, w \in W(M,\mc O) .
\end{equation}
Our alternative version of $A_{s_\alpha} \; (\alpha \in \Delta_{\mc O,\mu})$ is
\begin{equation}\label{eq:4.25}
\mc{T}_{s_\alpha} = g_\alpha A_{s_\alpha} = 
\frac{(1 - X_\alpha)(1 + X_\alpha) (1 - q_\alpha^{-1}) (1 + q_{\alpha*}^{-1} )}
{2 (1 - q_\alpha^{-1} X_\alpha)(1 + q_{\alpha*}^{-1} X_\alpha)} A_{s_\alpha} .
\end{equation}
By Proposition \ref{prop:4.1} the only poles of $A_{s_\alpha}$ are those of $(\mu^{M_\alpha})^{-1}$,
and by Proposition \ref{prop:3.2} they are simple. A glance at \eqref{eq:4.25} then reveals that 
\begin{equation}\label{eq:4.26}
\text{the poles of } \mc{T}_{s_\alpha} \text{ are at }  \{X_\alpha = q_\alpha \}
\text{ and, if } b_{s_\alpha} > 0, \text{ at } \{ X_\alpha = -q_{\alpha*} \} .
\end{equation}

\begin{prop}\label{prop:4.6}
The map $s_\alpha \mapsto \mc{T}_{s_\alpha}$ extends to a group homomorphism $w \mapsto \mc T_w$
from $W(\Sigma_{\mc O,\mu})$ to the multiplicative group of $\End_G (I_P^G (E_{K(B)}))$.
\end{prop}
\begin{proof}
It suffices to check that the relations in the standard presentation of the Coxeter group
$W(\Sigma_{\mc O,\mu})$ are respected. For the quadratic relations, consider any $\alpha \in 
\Delta_{\mc O,\mu}$. With \eqref{eq:4.2} and Proposition \ref{prop:4.1}.c we compute
\begin{align*}
\mc{T}_{s_\alpha}^2 & = g_\alpha A_{s_\alpha} g_\alpha A_{s_\alpha} = 
g_\alpha g_{-\alpha} A_{s_\alpha}^2 \\
& = \frac{(1 - X_\alpha)(1 + X_\alpha)(1 - X_\alpha^{-1})(1 - X_\alpha^{-1}) \; c'_{s_\alpha}}{
(1 - q_\alpha^{-1} X_\alpha) (1 + q_{\alpha*}^{-1} X_\alpha) (1 - q_\alpha^{-1} 
X_\alpha^{-1})(1 + q_{\alpha*}^{-1} X_\alpha^{-1}) \mu^{M_\alpha}(\sigma \otimes \cdot) } = 1.
\end{align*}
For the braid relations, let $\alpha, \beta \in \Delta_{\mc O,\mu}$ with $s_\alpha s_\beta$ of 
order $m_{\alpha \beta} \geq 2$. Then
\[
s_\alpha s_\beta s_\alpha \cdots = s_\beta s_\alpha s_\beta \cdots \qquad 
\text{(with } m_{\alpha \beta} \text{ factors on both sides)},
\] 
and this is an element of $W(\Sigma_{\mc O,\mu})$ of length $m_{\alpha \beta}$. We know from
Proposition \ref{prop:4.1}.b that 
\begin{equation}\label{eq:4.27}
A_{s_\alpha} A_{s_\beta} A_{s_\alpha} \cdots = A_{s_\beta} A_{s_\alpha} A_{s_\beta} \cdots 
\text{(with } m_{\alpha \beta} \text{ factors on both sides)}.
\end{equation}
Applying \eqref{eq:4.24} repeatedly, we find
\begin{equation}\label{eq:4.28}
\mc T_{s_\alpha} \mc T_{s_\beta} \mc T_{s_\alpha} \cdots = g_\alpha (s_\alpha \cdot g_\beta) 
(s_\alpha s_\beta \cdot g_\alpha) \cdots A_{s_\alpha} A_{s_\beta} A_{s_\alpha} \cdots =
\big( \prod\nolimits_\gamma g_\gamma \big) A_{s_\alpha} A_{s_\beta} A_{s_\alpha} \cdots ,
\end{equation}
where the product runs over $\{ \alpha, s_\alpha (\beta), s_\alpha s_\beta (\alpha), \ldots \}$.
Similarly
\begin{equation}\label{eq:4.29}
\mc T_{s_\beta} \mc T_{s_\alpha} \mc T_{s_\beta} \cdots = 
\big( \prod\nolimits_{\gamma'} g_{\gamma'} \big) A_{s_\beta} A_{s_\alpha} A_{s_\beta} \cdots ,
\end{equation}
where $\gamma'$ runs through $\{ \beta, s_\beta (\alpha), s_\beta s_\alpha (\beta), \ldots \}$.

We claim that $\{ \alpha, s_\alpha (\beta), s_\alpha s_\beta (\alpha), \ldots \}$ is precisely
the set of positive roots in the root system spanned by $\{\alpha,\beta\}$. To see this, one has
to check it for each of the four reduced root systems of rank 2 ($A_1 \times A_1, A_2, B_2, G_2$). 
In every case, it is an easy calculation.

Of course this applies also to $\{ \beta, s_\beta (\alpha), s_\beta s_\alpha (\beta), \ldots \}$.
Hence the products in \eqref{eq:4.28} and \eqref{eq:4.29} run over the same set. In combination
with \eqref{eq:4.27} that implies 
\[
\mc T_{s_\alpha} \mc T_{s_\beta} \mc T_{s_\alpha} \cdots = 
\mc T_{s_\beta} \mc T_{s_\alpha} \mc T_{s_\beta} \cdots ,
\]
as required.
\end{proof}

Since $\mc T_w$ is the product of $A_w$ with an element of $K(B)$, the relation \eqref{eq:4.2}
remains valid:
\begin{equation}\label{eq:4.30}
\mc T_w \circ b = (w \cdot b) \circ \mc T_w \qquad b \in K(B), w \in W(\Sigma_{\mc O,\mu}) .
\end{equation}
The $\mc T_w$ also satisfy analogues of \eqref{eq:4.7} and Proposition \ref{prop:4.2}.c:

\begin{lem}\label{lem:4.7}
Let $w \in W(\Sigma_{\mc O,\mu}), r \in R(\mc O)$ and $\chi_c \in X_\nr (M,\sigma)$.
\enuma{
\item $A_r \mc T_w \phi_{\chi_c} = z(rw,\chi_c) \phi_{rw (\chi_c)} A_r \mc T_w$.
\item $\mc T_w A_r = \natural (w,r) \phi_{w (\chi_r^{-1}) \chi_r} A_r \mc T_{r^{-1} w r}$.

If $w (\chi_r) = \chi_r$, then $\natural (w,r) = 1$ and $A_r^{-1} \mc T_w A_r = \mc T_{r^{-1} w r}$.
}
\end{lem}
\begin{proof}
(a) In view of \eqref{eq:4.7}, it suffices to consider $r=1$ and $w=s_\alpha$ with
$\alpha \in \Delta_{\mc O,\mu}$. The element $X_\alpha \in \C [X_\nr (M)]$ is
$X_\nr (M,\sigma)$-invariant, so $g_\alpha \in \C (X_\nr (M))$ is also $X_\nr (M,\sigma)$-invariant.
Then \eqref{eq:4.7} implies
\[
\mc T_{s_\alpha} \phi_{\chi_c} = g_\alpha A_{s_\alpha} \phi_{\chi_c} =
g_\alpha z(s_\alpha,\chi_c) \phi_{s_\alpha (\chi_c)} A_{s_\alpha} = 
z(s_\alpha,\chi_c) \phi_{s_\alpha (\chi_c)} \mc T_{s_\alpha} .
\]
(b) First we consider the case $w = s_\alpha$ with $\Delta_{\mc O,\mu}$. By Proposition \ref{prop:4.2}.c
\[
\mc T_{s_\alpha} A_r = g_\alpha A_{s_\alpha} A_r = g_\alpha \natural (s_\alpha,r) 
\phi_{s_\alpha (\chi_r^{-1}) \chi_r} A_r A_{r^{-1} s_\alpha r} 
\]
With \eqref{eq:4.24} we obtain
\begin{equation}\label{eq:4.17}
\mc T_{s_\alpha} A_r = \natural (s_\alpha,r) \phi_{s_\alpha (\chi_r^{-1}) \chi_r} A_r 
g_{r^{-1}(\alpha)} A_{r^{-1} s_\alpha r}
= \natural (s_\alpha,r) \phi_{s_\alpha (\chi_r^{-1}) \chi_r} A_r \mc T_{r^{-1} s_\alpha r} .
\end{equation}
For a general $w \in W (\Sigma_{\mc O,\mu})$ we pick a reduced expression $w = s_{\alpha_1}
s_{\alpha_2} \cdots s_{\alpha_k}$. Then part (a) enables us to apply \eqref{eq:4.17} repeatedly.
Each time we move one $\mc T_{s_{\alpha_i}}$ over $A_r$, we pick up the same correction factors
as we would with $A$'s instead of $\mc T$'s. As the desired formula with just $A$'s is known
from Proposition \ref{prop:4.2}.c, this procedure yields the correct formula.

If $w(\chi_r) = \chi_r$, then the special case of Proposition \ref{prop:4.2}.c applies.
\end{proof}

Let $\chi_c \in X_\nr (M,\sigma), r \in R(\mc O), w \in W(\Sigma_{\mc O,\mu})$ we write,
like we did for $A_{\chi_c \mf{r} \mf{w}}$:
\begin{equation}\label{eq:4.18}
\mc T_{\chi_c \mf{r} \mf{w}} = \phi_{\chi_c} A_r \mc T_w \in \End_G (I_P^G (E_{K(B)})).
\end{equation}
Recall that the $\phi_{\chi_c}$ can be normalized so that $\phi_{\chi_c}^{-1} = \phi_{\chi_c^{-1}}$.
Similarly, we can normalize the $A_r$ so that $A_r^{-1} = A_{r^{-1}}$. Then
\[
\natural (\chi_c, \chi_c^{-1}) = 1 \quad \text{and} \quad \natural (r,r^{-1}) = 1.
\]

\begin{lem}\label{lem:4.4}
Let $\chi_c, \chi'_c \in X_\nr (M,\sigma), r,r' \in R(\mc O), w, w' \in W(\Sigma_{\mc O,\mu})$.
\enuma{
\item There exists a $\natural (\chi_c \mf{r} \mf{w},\chi'_c \mf{r'} \mf{w'}) \in \C^\times$ such that
\[
\mc T_{\chi_c \mf{r} \mf{w}} \circ \mc T_{\chi'_c \mf{r'} \mf{w'}} = \natural (\chi_c \mf{r} \mf{w},
\chi'_c \mf{r'} \mf{w'}) \mc T_{\chi_c \mf{r} \mf{w} \chi'_c \mf{r'} \mf{w'}} .
\]
\item If in addition $r = \chi'_c = 1$ and $w(\chi_{r'}) = \chi_{r'}$, then 
$\natural (\chi_c \mf{w},\mf{r'} \mf{w'}) = 1$.
\item The map $\natural : W(M,\sigma,X_\nr (M))^2 \to \C^\times$ is a 2-cocycle.
}
\end{lem}
\begin{proof}
(a) In the setting of Proposition \ref{prop:4.2} we write $r_3 = r r'$ and $w_3 = r^{-1} w r$.
That gives the following equalities in $W(M,\sigma,X_\nr (M))$:
\begin{equation}\label{eq:4.14}
\mf{r} \mf{r'} = \chi (r,r') \mf{r_3} \quad \text{ and } \quad
\mf{w} \mf{r'} = (w (\chi_{r'}^{-1}) \chi_{r'}) \mf{r'} \mf{w_3} .
\end{equation}
Thus the already established Propositions \ref{prop:4.2} and \ref{prop:4.6}, as well as
\eqref{eq:2.29} and Lemma \ref{lem:4.7} can be regarded as instances of the statement.

We denote equality up to nonzero scalar factors by $\dot{=}$. With aforementioned available
instances we compute
\begin{align}
\nonumber \mc T_{\chi_c \mf{r} \mf{w}} \circ \mc T_{\chi'_c \mf{r'} \mf{w'}} & = 
\phi_{\chi_c} A_r \mc T_w \phi_{\chi'_c} \mc A_{r'} \mc T_{w'} \\
\nonumber & \; \dot{=} \; \phi_{\chi_c} \phi_{r w (\chi'_c)} A_r \mc T_w A_{r'} \mc T_{w'} \\
\nonumber & \; \dot{=} \; \phi_{\chi_c} \phi_{r w (\chi'_c)} A_r \phi_{w (\chi_{r'}^{-1}) \chi_{r'}} 
A_{r'} \mc T_{r'^{-1} w r'} \mc T_{w'} \\
\label{eq:4.15} & \; \dot{=} \; \phi_{\chi_c} \phi_{r w (\chi'_c)} \phi_{r (w (\chi_{r'}^{-1}) \chi_{r'})} 
A_r A_{r'} \mc T_{r'^{-1} w r' w'} \\
\nonumber & \; \dot{=} \; \phi_{\chi_c} \phi_{r w (\chi'_c)} \phi_{r (w (\chi_{r'}^{-1}) \chi_{r'})} 
\phi_{\chi (r,r')} A_{r r'} \mc T_{r'^{-1} w r' w'} \\
\nonumber & \; \dot{=} \; \phi_{\chi_c r w (\chi'_c) r (w (\chi_{r'}^{-1}) \chi_{r'}) \chi (r,r')} 
A_{r r'} \mc T_{r'^{-1} w r' w'} 
\end{align}
In each of the above steps we preserved the underlying element of $W(M,\sigma,X_\nr (M))$, so
in the notation from \eqref{eq:4.14}
\[
\chi_c \mf{r} \mf{w} \chi'_c \mf{r'} \mf{w'} =  
\chi_c r w (\chi'_c) r (w (\chi_{r'}^{-1}) \chi_{r'}) \chi (r,r') \mf{r_3} \mf{w_3} \mf{w'} .
\]
(b) When $r = \chi'_c = 1$ and $w(\chi_{r'}) = \chi_{r'}$, the second, fifth and sixth steps 
of \eqref{eq:4.15} become trivial. Thanks to Propositions \ref{prop:4.1}.b and 
\ref{prop:4.2}.c the third and fourth steps become equalities, so the entire calculation
consists of equalities.\\
(c) This follows from the associativity of $\End_G (I_P^G (E_{K(B)}))$.
\end{proof}

By \eqref{eq:4.16} and \eqref{eq:4.30}
\begin{equation}\label{eq:4.19}
\mc T_{\chi_c \mf{r} \mf{w}} \, b \, \mc T_{\chi_c \mf{r} \mf{w}}^{-1} = 
(r w \cdot b)_{\chi_c^{-1} \chi_r^{-1}} = b \circ (\chi_c \mf{r} \mf{w})^{-1} \in \C (X_\nr (M)) .
\end{equation}
We embed the twisted group algebra $\C [W(M,\sigma,X_\nr (M)),\natural]$ in \\
$\Hom_G (I_P^G (E_B), I_P^G (E_{K(B)}))$ with the operators $\mc T_{\chi_c \mf{r} \mf{w}}$.
Then Theorem \ref{thm:4.3} and Lemma \ref{lem:4.4} show that the multiplication map
\[
K(B) \otimes_\C \C [W(M,\sigma,X_\nr (M)),\natural] \to 
\Hom_G \big( I_P^G (E_B), I_P^G (E_{K(B)}) \big)
\]
is bijective. That and \eqref{eq:4.19} can be formulated as:

\begin{cor}\label{cor:5.6}
The algebra $\Hom_G \big( I_P^G (E_B), I_P^G (E_{K(B)}) \big)$ is the crossed product 
\[
\C (X_\nr (M)) \rtimes \C [W(M,\sigma,X_\nr (M)),\natural]
\]
with respect to the canonical action of $W(M,\sigma,X_\nr (M))$ on $\C (X_\nr (M)) = K(B)$.
\end{cor}

We end this section with some investigations of the structure of $\End_G (I_P^G (E_B))$.
By the theory of the Bernstein centre \cite[Th\'eor\`eme 2.13]{BeDe}, its centre is
\begin{equation}\label{eq:6.19}
Z \big( \End_G (I_P^G (E_B)) \big) = \C [\mc O]^{W(M,\mc O)} = 
\C [X_\nr (M)]^{W(M,\sigma,X_\nr (M))} . 
\end{equation}
From Lemma \ref{lem:5.1} and \eqref{eq:7.22} we know that $\End_G (I_P^G (E_B))$ is a free 
$B$-module of rank $|W(M,\sigma,X_\nr (M))|$. The $\phi_{\chi_c}$ with $\chi_c \in X_\nr (M,\sigma)$
and the $A_r$ with $r \in R(\mc O)$ belong to $\End_G (I_P^G (E_B))$, but the $A_w$ with
$w \in W(\Sigma_{\mc O,\mu}) \setminus \{1\}$ do not, because they have poles. To see whether
these poles can be removed in a simple way, we analyse their residues.

\begin{lem}\label{lem:6.1}
Let $\alpha \in \Delta_{\mc O,\mu}$ and $\chi_\pm \in X_\nr (M)$ with 
$X_\alpha (\chi_\pm) = \pm 1$.
\enuma{
\item $\mr{sp}_{\chi_+} (1 - X_\alpha) A_{s_\alpha}$ is a scalar multiple of
$\mr{sp}_{\chi^+} \phi_{\chi_+ s_\alpha ( \chi_+^{-1})}$.

If $s_\alpha (\chi_+) = \chi_+$, then $\mr{sp}_{\chi_+} (1 - X_\alpha) A_{s_\alpha} = 
\pm \mr{sp}_{\chi_+}$. If $\chi_+ \in X_\nr (M_\alpha)$, then
$\mr{sp}_{\chi_+} (1 - X_\alpha) A_{s_\alpha} = \mr{sp}_{\chi_+}$.
\item Suppose that $b_{s_\alpha} > 0$. Then $\mr{sp}_{\chi_-} (1 + X_\alpha) A_{s_\alpha}$ 
is a scalar multiple of \\
$\mr{sp}_{\chi_-} \phi_{\chi_- s_\alpha ( \chi_-^{-1})}$.
In case $s_\alpha (\chi_-) = \chi_-$:
\[
\mr{sp}_{\chi_-} (1 + X_\alpha) A_{s_\alpha} = \pm \frac{(1 + q_\alpha^{-1}) 
(1 - q_{\alpha*}^{-1}) }{(1 - q_\alpha^{-1}) (1 + q_{\alpha*}^{-1})} \mr{sp}_{\chi_-} .
\]
}
\end{lem}
\emph{Remark.} With a closer analysis of the operators $J_{s_\alpha (P)|P}(\sigma \otimes \chi)$,
as in \cite[\S IV.1]{Wal}, it could be possible to prove that the signs $\pm$ in this lemma
are always $+1$.
\begin{proof}
(a) By Proposition \ref{prop:3.2}, $\mr{sp}_{\chi_+} (1 - X_\alpha) A_{s_\alpha}$ defines
a $G$-isomorphism\\ $I_P^G (\sigma \otimes s_\alpha (\chi_+)) \to I_P^G (\sigma \otimes \chi_+)$,
parabolically induced from an $M_\alpha$-isomorphism
\[
I_{P \cap M_\alpha}^{M_\alpha}(\sigma \otimes \chi_+) \to
I_{P \cap M_\alpha}^{M_\alpha}(\sigma \otimes s_\alpha (\chi_+)). 
\]
The same holds for $\mr{sp}_{\chi^+} \phi_{\chi_+ s_\alpha ( \chi_+^{-1})}$. Further, 
$I_{P \cap M_\alpha}^{M_\alpha}(\sigma \otimes \chi_+)$ is irreducible \cite[\S 4.2]{Sil2}. 
By Schur's lemma, these two operators are scalar multiples of each other.

Suppose now that $s_\alpha (\chi_+) = \chi_+$. By the above $\mr{sp}_{\chi_+} (1 - X_\alpha) 
A_{s_\alpha} \in \C \mr{sp}_{\chi_+}$. From Proposition \ref{prop:4.6} and
\eqref{eq:4.25} we see that in fact $\mr{sp}_{\chi_+} (1 - X_\alpha) A_{s_\alpha} = 
\pm \mr{sp}_{\chi_+}$.

The variety $X_\nr (M_\alpha) \subset X_\nr (M)$ is connected and fixed pointwise by
$s_\alpha$. The sign in $\pm \mr{sp}_{\chi_+}$ established above depends algebraically
on $\chi_+$, so it is constant on $X_\nr (M_\alpha)$. Therefore it suffices to consider 
$\chi_+ = 1$. Let us unravel the definitions:
\begin{align}\nonumber
\mr{sp}_{\chi = 1} (1 - X_\alpha) A_{s_\alpha} & =
\mr{sp}_{\chi = 1} (1 - X_\alpha) \rho_{s_\alpha} \tau_{s_\alpha} J_{K(B),s_\alpha} \\
\label{eq:4.31} & = I_P^G (\rho_{\sigma,s_\alpha} \otimes \mr{sp}_{\chi = 1}) (1 - X_\alpha)
\tau_{s_\alpha} \lambda (\tilde{s_\alpha}) J_{s_\alpha (P) | P}(\sigma \otimes \cdot) \\
\nonumber & = I_P^G (\rho_{\sigma,s_\alpha} \otimes \mr{sp}_{\chi = 1})\tau_{s_\alpha} 
\lambda (\tilde{s_\alpha}) (1- X_\alpha^{-1}) J_{s_\alpha (P) | P}(\sigma \otimes \cdot) \\
\nonumber & =  I_P^G (\rho_{\sigma,s_\alpha}) \mr{sp}_{\chi = 1} \lambda (\tilde{s_\alpha}) 
(X_\alpha - 1) J_{s_\alpha (P) | P}(\sigma \otimes \cdot) .
\end{align}
By construction $I_P^G (\rho_{\sigma,s_\alpha})$ is the inverse of
\[
\mr{sp}_{\chi = 1} \tau_{s_\alpha} \lambda (\tilde{s_\alpha}) (X_\alpha - 1) 
J_{s_\alpha (P) | P}(\sigma \otimes \cdot) = \mr{sp}_{\chi = 1} \lambda (\tilde{s_\alpha}) 
(X_\alpha - 1) J_{s_\alpha (P) | P}(\sigma \otimes \cdot) ,
\]
see \eqref{eq:3.7} and \eqref{eq:3.11}. We find 
$\mr{sp}_{\chi = 1} (X_\alpha - 1) A_{s_\alpha} = \mr{sp}_{\chi = 1}$.\\
(b) This is analogous to part (a). For the second claim we use that
$(g_\alpha A_{s_\alpha})^2 = 1$ and
\begin{equation}\label{eq:4.32}
\mr{sp}_{\chi_-} (g_\alpha A_{s_\alpha}) = \frac{(1 - q_\alpha^{-1}) (1 + q_{\alpha*}^{-1})}
{(1 + q_\alpha^{-1}) (1 - q_{\alpha*}^{-1})} \mr{sp}_{\chi_-} (1 + X_\alpha) A_{s_\alpha} . \qedhere
\end{equation}
\end{proof}

\begin{rem}
Lemma \ref{lem:6.1} shows that some of the poles of $A_{s_\alpha}$ may occur at
$\chi$'s that are not fixed by $s_\alpha$, so those poles cannot be removed by
adding an element of $\C (X_\nr (M))$ to $A_{s_\alpha}$. For instance, suppose that $X_\alpha / 3
\in X^* (X_\nr (M))$ and $(X_\alpha / 3)(\chi) = e^{2 \pi i/3}$. Then $X_\alpha (\chi) = 1$
but $(X_\alpha / 3) (s_\alpha \chi) = e^{-2\pi i / 3}$. 
Similar considerations apply to $\mc T_{s_\alpha}$. In particular the method from 
\cite[\S 5]{Hei2} does not apply in our generality. 

This problem is only made worse by the possible nontriviality of $\natural$ on $X_\nr (M,\sigma)^2$.
Although we expect that there exist $|W(M,\sigma,X_\nr (M))|$ elements that generate \\
$\End_G (I_P^G (E_B))$ as $B$-module, we do not have good candidates.
That renders it hard to find a nice presentation of $\End_G (I_P^G (E_B))$.
\end{rem}

\section{Analytic localization on subsets on $X_\nr (M)$}
\label{sec:localization}

In this and the upcoming sections, when we talk about modules for an algebra, we tacitly
mean right modules.  Each of the algebras $H$ we consider has a large commutative subalgebra $A$ 
such that $H$ has finite rank as $A$-module. 
For an $H$-module $V$, we denote the set of $A$-weights by Wt$(V)$. 

Every finite dimensional $H$-module $V$ decomposes canonically, as $A$-module, as the direct
sum of the subspaces
\[
V_\chi := \{ v \in V : (a - a(\chi))^{\dim V} (v) = 0 \}
\]
for $\chi \in \mr{Wt}(V)$. For this reason it is much easier to work with representations
of finite length. We denote the category of finite dimensional right $H$-modules by $H-\fMod$.
For a subset $U \subset \Irr (A)$, we let $H-\Modf{U}$ be the full subcategory of $H-\fMod$ formed
by the modules all whose $A$-weights lie in $U$.

For Rep$(G)^{\mf s}$, the role of weights is played by the cuspidal support. When $\pi \in 
\mr{Rep}(G)^{\mf s}$ has finite length, we define Sc$(\pi)$ as the set of $\sigma' \in \mc O$
which appear in the Jacquet restriction $J_P^G (\pi)$.

Lemma \ref{lem:6.1} does not provide enough control over the poles of $A_{s_\alpha}$ to deal
with all of them in one stroke. 
Therefore we approach $\End_G (I_P^G (E_B))$ via localization on suitable subsets of $X_\nr (M)$.
Let $U$ be a $W(M,\sigma,X_\nr (M))$-stable subset of $X_\nr (M)$, open with respect to the analytic
topology. Then $U$ is a complex submanifold of $X_\nr (M)$, so we can consider the algebra $C^{an}(U)$ 
of complex analytic functions on $U$. The natural map $\C [X_\nr (M)] \to C^{an}(U)$
is injective because $U$ is Zariski dense in $X_\nr (M)$. This and \eqref{eq:6.19} enable us
to construct the algebra
\[
\End_G (I_P^G (E_B))^\an_U := \End_G (I_P^G (E_B)) \underset{B^{W (M,\sigma,X_\nr (M))}}{\otimes} 
C^\an (U)^{W(M,\sigma,X_\nr (M))} . 
\]
Its centre is
\begin{equation}\label{eq:7.29}
Z \big( \End_G (I_P^G (E_B))^\an_U \big) = C^\an (U)^{W(M,\sigma,X_\nr (M))} .
\end{equation}
We note that by \cite[Lemma 4.4]{Opd-Sp}
\begin{equation}\label{eq:7.5}
\C [X_\nr (M)] \underset{\C [X_\nr (M)]^{W(M,\sigma,X_\nr (M))}}{\otimes} 
C^\an (U)^{W(M,\sigma,X_\nr (M))} \cong C^\an (U) .
\end{equation}
The subalgebra $C^\an (U)$ of $\End_G (I_P^G (E_B))^\an_U$ plays the same role as
$B = \C [X_\nr (M)]$ in $\End_G (I_P^G (E_B))$.
\begin{rem}\label{rem:6.7}
The set of $\C [X_\nr (M)]$-weights of any module for $\End_G (I_P^G (E_B))$ or 
$\End_G (I_P^G (E_B))^\an_U$ is stable under the subgroup $X_\nr (M,\sigma) R (\mc O) \subset
W(M,\sigma,X_\nr (M))$, because $\mc T_{\chi_c \mf r}$ belongs to $\End_G (I_P^G (E_B))$ 
for all $\chi_c \in X_\nr (M,\sigma), r \in R(\mc O)$ and this element satisfies \eqref{eq:4.30}.
\end{rem}

\begin{lem}\label{lem:7.1}
There are natural equivalences between the following categories:
\begin{enumerate}[(i)]
\item $\End_G (I_P^G (E_B))_U^\an -\fMod$;
\item $\End_G (I_P^G (E_B)) -\Modf{U}$, or equivalently those $V \in \End_G (I_P^G (E_B)) -\fMod$ 
with all weights of the centre in $U / W(M,\sigma,X_\nr (M))$;
\item finite length representations in $\Rep (G)^{\mf s}$, whose cuspidal support is contained in
$\{ \sigma \otimes \chi : \chi \in U \}$.
\end{enumerate}
These equivalences commute with parabolic induction and Jacquet restriction in the sense of
Proposition \ref{prop:4.5}.
\end{lem}
\begin{proof}
The equivalence between (ii) and (iii) follows from \eqref{eq:4.equiv} and the way the
$B$-action on $I_P^G (E_B)$ is constructed in \eqref{eq:2.4}. We saw in Proposition \ref{prop:4.5}
how it relates to induction and restriction.

The equivalence between (i) and (ii) is analogous to \cite[Proposition 4.3]{Opd-Sp}.
For a Levi subgroup $L$ of $G$ containing $M$ there are analogous algebras
\[
\End_L (I_{P \cap L}^L (E_B)) \subset \End_L (I_{P \cap L}^L (E_B))_U^\an .
\] 
The equivalence between (i) and (ii) for these two algebras works in the same way, basically
it only depends on the inclusion $\C [X_\nr (M)] \subset C^\an (U)$. Hence these equivalences
of categories commute with induction and restriction between the level of $L$ and the level of $G$.
\end{proof}

Now we specialize to very specific submanifolds of $X_\nr (M)$. Let
\begin{equation}\label{eq:7.2}
X_\nr (M) = X_\unr (M) \times X_\nr^+ (M) = 
\Hom (M / M^1 , S^1) \times \Hom (M / M^1, \R_{>0})
\end{equation}
be the polar decomposition of the complex torus $X_\nr (M)$. Fix a unitary unramified
character $u \in X_\unr (M)$. The following condition is a variation on 
\cite[Condition 4.9]{Opd-Sp} and \cite[Condition 2.1.1]{SolAHA}.
\begin{cond}\label{cond:7.U}
Let $U_u$ be a (small) connected open neighborhood of $u$ in $X_\nr (M)$, such that 
\begin{itemize}
\item $U_u$ is stable under the stabilizer of $u$ in $W(M,\sigma,X_\nr (M))$ and under $X_\nr^+ (M)$;
\item $W(M,\sigma,X_\nr (M)) u \cap U_u = \{u\}$;
\item $\Re (X_\alpha (\chi u^{-1})) > 0$ for all $\alpha \in \Sigma_{\sigma \otimes u},
\chi \in U_u$.
\end{itemize}
\end{cond}
We note that such a neighborhood $U_u$ always exists because $W(M,\sigma,X_\nr (M))$ is finite
and its action on $X_\nr (M)$ preserves the polar decomposition \eqref{eq:7.2}. 
The first two bullets of Condition \ref{cond:7.U} entail that $W(M,\sigma,X_\nr (M)) U_u$ is 
homeomorphic to $W(M,\sigma,X_\nr (M)) u \times U_u$. 
The last bullet implies that, if $\mu^{M_\alpha}(\sigma \otimes \chi) = 0$ for some $\chi \in U_u$ and 
$\alpha \in \Sigma_\red (A_M)$, then $\mu^{M_\alpha}(\sigma \otimes u) = 0$. This replaces the 
conditions on $U_u$ in relation to the functions $c_\alpha$ in \cite{Opd-Sp,SolAHA}.

In the remainder of this section we consider
\begin{equation}\label{eq:7.4}
U := W(M,\sigma,X_\nr (M)) U_u ,
\end{equation}
an open neighborhood of $W(M,\sigma,X_\nr (M)) u X_\nr^+ (M)$. By Lemma \ref{lem:7.1} the
family of algebras $\End_G (I_P^G (E_B))_U^\an$, for all possible $u \in X_\unr (M)$, 
suffices to study the entire category of finite length representations in $\Rep (G)^{\mf s}$.

We want to find a presentation of $\End_G (I_P^G (E_B))_U^\an$, as explicit as possible.
For $w \in W(M,\sigma,X_\nr (M))$ we write $U_{wu} = w (U_u)$. By \eqref{eq:7.5} 
$\End_G (I_P^G (E_B))_U^\an$ contains the element $1_{wu} \in C^\an (U)$ defined by
\[
1_{wu} (\chi) = \left\{
\begin{array}{l@{\quad}l}
1 & \chi \in U_{wu} \\
0 & \chi \in U \setminus U_{wu} 
\end{array}
\right. .
\]
The $1_{wu}$ with $wu \in W(M,\sigma,X_\nr (M)) u$ form a system of mutually orthogonal
idempotents in $C^\an (U)$ and
\begin{equation}\label{eq:7.6}
1_U = \sum\nolimits_{wu \in W(M,\sigma,X_\nr (M)) u} 1_{wu} .
\end{equation}
This yields a decomposition of $C^{an}(U)$-modules
\begin{align*}
\End_G (I_P^G (E_B))_U^\an & = \bigoplus_{wu, vu \in W(M,\sigma,X_\nr (M)) u} 1_{wu} 
\End_G (I_P^G (E_B))_U^\an 1_{vu} \\
& =\bigoplus_{wu, vu \in W(M,\sigma,X_\nr (M)) u} C^\an (U_{wu})
\End_G (I_P^G (E_B))_U^\an C^\an (U_{vu}) .
\end{align*}
Here the submodules with $wu = vu$ are algebras, while those with $wu \neq vu$ are not.

\begin{lem}\label{lem:7.2}
The inclusion $1_u \End_G (I_P^G (E_B))_U^\an 1_u \to \End_G (I_P^G (E_B))_U^\an$ 
is a Morita equivalence.  
\end{lem}
\begin{proof}
The Morita bimodules are $\End_G (I_P^G (E_B))_U^\an 1_u$ and $1_u \End_G (I_P^G (E_B))_U^\an$.
Most of the required properties are automatically fulfilled, it only remains to verify
that $1_u$ is a full idempotent in $Y := \End_G (I_P^G (E_B))_U^\an$:
\begin{equation}\label{eq:7.7}
Y 1_u Y \text{ should equal } Y .
\end{equation}
In view of \eqref{eq:7.6}, it suffices to show that $1_{wu} \in Y 1_u Y$ for all
$w \in W(M,\sigma,X_\nr (M))$. 

By \eqref{eq:3.30} there exist $\chi_c \in X_\nr (M,\sigma), r \in R(\mc O)$ and 
$v \in W(\Sigma_{\mc O,\mu})$ such that $w = \chi_c \mf{r} \mf{v}$. We may and will assume that 
$\ell_{\mc O}(\chi_c \mf{r} \mf{v})$ is minimal under the condition $\chi_c \mf{r} \mf{v} u = w u$. 
By \eqref{eq:4.30} 
\[
\mc T_v 1_{u} \mc T_{v^{-1}} = 1_{vu} \mc T_v \mc T_v^{-1} = 1_{vu} .
\]
We claim that $\mr{sp}_{v \chi} \mc T_v$ and $\mr{sp}_\chi \mc T_v^{-1}$ are regular for all 
$\chi \in U_u$, or equivalently
\begin{equation}\label{eq:7.9}
\mc T_v \text{ does not have poles on } U_u \text{ and } \mc T_v^{-1} 
\text{ does not have poles on } U_{vu} .
\end{equation}
We will prove this with induction to $\ell_{\mc O}(v)$. The case $\ell_{\mc O}(v) = 0$ is trivial.
For the induction step, write $v = s_\alpha v'$ with $\alpha \in \Delta_{\mc O,\mu}$ and 
$\ell_{\mc O}(v') = \ell_{\mc O}(v) - 1$. By the minimality of $\ell_{\mc O}(\chi_c \mf{r} \mf{v})$, 
$v' u \neq vu$. Then
\[
s_\alpha (U_{v' u}) \cap U_{v' u} = U_{v u} \cap U_{v' u} = \emptyset. 
\]
By \cite[Lemma 3.15]{Lus-Gr} $X_\alpha ( \chi ) \neq 1$ for all $\chi \in U_{v' u}$
and, if $X_\alpha \in 2 X^* (X_\nr (M))$, also $X_\alpha ( \chi ) \neq -1$ for all 
$\chi \in U_{v' u}$. If $\mc T_{s_\alpha}$ would have a pole on $U_{v'u}$, then
$U_{v' u} = X_\nr^+ (M) U_{v' u}$ and \eqref{eq:4.26} entail that
\begin{equation}\label{eq:7.8}
\exists \chi \in U_{v' u} : X_\alpha (\chi) = 1 \text{ or }
b_{s_\alpha} > 0 \text{ and } \exists \chi \in U_{v' u} : X_\alpha (\chi) = -1.
\end{equation}
That would contradict the already derived properties of $U_{v' u}$, so 
$\mc T_{s_\alpha}$ is regular on $U_{v' u}$. Notice that $v'$ has minimal length for sending
$u$ to $X_\nr (M,\sigma) R(\mc O)^{-1} v' u$, because $v$ has minimal length under the condition
$v u \in X_\nr (M,\sigma) R(\mc O)^{-1} w u$. Hence the induction hypothesis applies, and it tells
us that $\mc T_{v'}$ is regular on $U_u$. We conclude that $\mc T_v = \mc T_{s_\alpha} \mc T_{v'}$ 
is regular on $U_u$.

If we replay the argument up to \eqref{eq:7.8} with $U_{vu}$ and $U_{vu'}$ exchanged, we arrive
at the conclusion that $\mc T_{s_\alpha}$ is regular on $U_{vu}$. By the induction hypothesis
$\mc T_{v'}^{-1}$ does not have poles on $U_{v' u}$. Therefore $\mc T_v^{-1} = \mc T_{v'}^{-1}
\mc T_{s_\alpha}$ is regular on $U_{vu}$, affirming \eqref{eq:7.9}.

From \eqref{eq:7.9} and \eqref{eq:4.30} we obtain $\mc T_v 1_u = 1_{vu} \mc T_v \in Y$
and $\mc T_{v^{-1}} 1_{vu} = 1_u \mc T_v \in Y$. By \eqref{eq:4.6} and \eqref{eq:2.10}:
\[
\mc T_{\chi_c \mf{r} \mf{v}} 1_u = \phi_{\chi_c} A_r \mc T_v 1_u \in Y \quad \text{and}
\quad 1_u \mc T_{\mf{v}^{-1} \mf{r}^{-1} \chi_c^{-1}} = 
1_u \mc T_v^{-1} A_r^{-1} \phi_{\chi_c}^{-1} \in Y.
\]
Then \eqref{eq:4.19} says
\[
1_{wu} = 1_{\chi_c \mf{r} \mf {v} u} = 
\mc T_{\chi_c \mf{r} \mf{v}} 1_u \mc T_{\mf{v} \mf{r}^{-1} \chi_c^{-1}} = 
\phi_{\chi_c} A_r \mc T_v 1_u 1_u \mc T_v^{-1} A_r^{-1} \phi_{\chi_c}^{-1} \in Y.
\] 
This confirms \eqref{eq:7.7}.
\end{proof}

For later use we analyse the Morita equivalence from Lemma \ref{lem:7.2} more deeply.

\begin{lem}\label{lem:6.6}
\enuma{
\item There are equivalences of categories
\[
\begin{array}{ccc}
\!\! 1_u \End_G (I_P^G (E_B))_U^\an 1_u -\Mod \!\! & \longleftrightarrow & 
\End_G (I_P^G (E_B))_U^\an -\Mod \\
V 1_u & \text{\reflectbox{$\mapsto$}} & V \\
V_u & \mapsto & \!\! 
V_u \underset{1_u \End_G (I_P^G (E_B))_U^\an 1_u}{\otimes} \End_G (I_P^G (E_B))_U^\an \!\! 
\end{array} .
\]
\item Let $V \in \End_G (I_P^G (E_B))_U^\an -\Mod$. The $\C [X_\nr (M)]$-weights of $V 1_u$
are precisely the $\C [X_\nr (M)]$-weights of $V$ that lie in $U_u$.
\item Let $W^u$ be a set of shortest length representatives for\\
$W(M,\sigma,X_\nr (M)) / W (M,\sigma,X_\nr (M))_u$ and let
$V_u \in 1_u \End_G (I_P^G (E_B))_U^\an 1_u -\fMod$. The $\C [X_\nr (M)]$-weights
of $V_u \underset{1_u \End_G (I_P^G (E_B))_U^\an 1_u}{\otimes} \End_G (I_P^G (E_B))_U^\an$ are 
\[
\{ \chi_c \mf{r} \mf{w} (\chi) : \chi \text{ is a } \C [X_\nr (M)]\text{-weight of } V_u \}.
\]
}
\end{lem}
\begin{proof}
(a) follows from the explicit form of the bimodules in Lemma \ref{lem:7.2}.\\
(b) is clear from the definition of $1_u$.\\
(c) Since $V_u$ has finite dimension, we can decompose it according to its $\C[X_\nr (M)]$-weights
(or equivalently its $C^\an (U_u)$-weights). For every such weight $\chi \in U_u$ and every 
$n \in \Z_{\geq 1}$ we write
\[
V_u^{\chi,n} = \{ v \in V_u : v (f - f(\chi))^n = 0 \; \forall f \in C^\an (U_u) \}.
\]
Then $V_u = \sum_{\chi,n} V_u^{\chi,n}$. From Corollary \ref{cor:5.6} we see that
\[
\mr{sp}_\chi (\End_G (I_P^G (E_B))_U^\an) \quad \text{is spanned by} \quad
\{ \mr{sp}_\chi (f_{w,\chi} \mc T_w) : w \in W(M,\sigma,X_\nr (M)) \} ,
\]
where $f_{w,\chi} \in \C [X_\nr (M)]$ is such that $\mr{sp}_\chi (f_{w,\chi} \mc T_w)$ 
is regular and nonzero (or zero if that is not possible). Hence 
\[
V := V_u \underset{1_u \End_G (I_P^G (E_B))_U^\an 1_u}{\otimes} \End_G (I_P^G (E_B))_U^\an
\]
equals $\sum_{\chi,n,w} V_u^{\chi,n} f_{w,\chi} \mc T_w$. Since 
$V_u^{\chi,n} f_{w,\chi} \mc T_w \subset V_u$ for $W(M,\sigma,X_\nr (M))_u$, 
\[
V = \sum\nolimits_{\chi,n} \sum\nolimits_{w \in W^u} V_u^{\chi,n} f_{w^{-1},\chi} \mc T_w^{-1} .
\]
From \eqref{eq:7.9} we know that $\mr{sp}_\chi \mc T_w^{-1}$ is regular for all $\chi \in U_u$
and $w \in W^u$, so we can take $f_{w^{-1},\chi} = 1_u$. For $v \in V_u^{\chi,1}$ 
we see from \eqref{eq:4.30} that 
\[
v \mc T_w^{-1} f = v (w^{-1} f) \mc T_w^{-1} \mc T_{\chi_c \mf{r}} =
f (w \chi) v \mc T_w^{-1} \qquad \forall f \in C^\an (U) .
\]
Hence $v \mc T_w^{-1}$ is a $C^\an (U)$-weight vector for the weight $w \chi$. It remains to see
that $V$ has no other $C^\an (U)$-weights. Suppose 
\[
\sum\nolimits_{\chi,n \geq 1} \sum\nolimits_{w \in W^u} v_{\chi,n,w^{-1}} \mc T_w^{-1} \in V
\]
is a weight vector not in $\sum_{\chi} \sum_{w \in W^u} V_u^{\chi,1} \mc T_w^{-1}$,
where $v_{\chi,n,w^{-1}} \in V_u^{\chi,n}$. Then the same holds for 
\[
v := \sum\nolimits_{\chi,n > 1} \sum\nolimits_{w \in W^u} v_{n,\chi,w^{-1}} \mc T_w^{-1},
\] 
so $v_{\chi,n,w^{-1}} \in V_u^{\chi,n} \setminus \{0\}$ for some $\chi,n>1,w$. We may assume that 
$n$ is minimal for this property. Then direct computation of $v (f - f(w \chi))^{n-1}$ shows that for
some $f \in C^\an (U)$ it has a nonzero term in $\sum_{\chi} \sum_{w \in W^u} V_u^{\chi,1} 
\mc T_w^{-1}$. Therefore $v$ cannot be a weight vector, and we indeed found all 
$\C [X_\nr (M)]$-weights already.
\end{proof}

Lemma \ref{lem:6.6}.a is compatible with parabolic induction and restriction, but we have to be
careful with the formulation. Let $L$ be a Levi subgroup of $G$ containing $M$ and let
$W_L (M,\sigma,X_\nr (M))$ be the version of $W (M,\sigma,X_\nr (M))$ for $L$. Then
\[
U_L := W_L (M,\sigma,X_\nr (M)) U_u
\]
is a union of connected components of $U$ and $C^\an (U_L)$ is a subalgebra of $C^\an (U)$. 
From $I_{PL}^G $ and \eqref{eq:7.5} we obtain a natural injective algebra homomorphism
\[
\End_L (I_{P \cap L}^L (E_B))_{U_L}^\an \to \End_G (I_P^G (E_B))_U^\an .
\]
We warn that in general this homomorphism is not unital, so naive restriction may not send unital
modules to unital modules. Instead, we define a functor
\[
\begin{array}{cccc}
\Res_{U_L} : & \End_G (I_P^G (E_B))_U^\an -\Mod & \longrightarrow & 
\End_L (I_{P \cap L}^L (E_B))_{U_L}^\an -\Mod \\
& V & \mapsto & V 1_{U_L} 
\end{array}.
\]
On the other hand, the restricted homomorphism
\[
1_u \End_L (I_{P \cap L}^L (E_B))_{U_L}^\an 1_u \to 1_u \End_G (I_P^G (E_B))_U^\an 1_u
\quad \text{is unital.}
\]

\begin{lem}\label{lem:6.5}
The following diagrams, with horizontal maps from Lemma \ref{lem:6.6}.a, commute:
\[
\begin{array}{ccc}
1_u \End_G (I_P^G (E_B))_U^\an 1_u -\Mod & \longleftrightarrow & \End_G (I_P^G (E_B))_U^\an -\Mod \\
\uparrow \ind & & \uparrow \ind \\
1_u \End_L (I_{P \cap L}^L (E_B))_{U_L}^\an 1_u -\Mod & \longleftrightarrow &
\End_L (I_{P \cap L}^L (E_B))_{U_L}^\an -\Mod \\[3mm]
1_u \End_G (I_P^G (E_B))_U^\an 1_u -\Mod & \longleftrightarrow & \End_G (I_P^G (E_B))_U^\an -\Mod \\
\downarrow \Res & & \downarrow \Res_{U_L} \\
1_u \End_L (I_{P \cap L}^L (E_B))_{U_L}^\an 1_u -\Mod & \longleftrightarrow &
\End_L (I_{P \cap L}^L (E_B))_{U_L}^\an -\Mod
\end{array}
\]
\end{lem}
\begin{proof}
Consider the first diagram, with horizontal maps from left to right. It commutes because all arrows
are inductions from subalgebras. Next consider the second diagram with horizontal maps from right
to left. It commutes because $1_u 1_{U_L} V = 1_u V$.

As the horizontal maps are equivalences, the diagrams remain commutative if we reverse the directions
of one or two horizontal arrows.
\end{proof}

\noindent Lemma \ref{lem:7.2} tells us that we should understand the subalgebra 
$1_u \End_G (I_P^G (E_B))_U^\an 1_u$ of $\End_G (I_P^G (E_B))_U^\an$ better. Let $C^\me (U)$ 
be the ring of meromorphic functions on $U$. We proceed via 
\begin{equation}\label{eq:7.28}
\begin{aligned}
& \Hom_G \big( I_P^G (E_B),I_P^G (E_{K(B)}) \big) \subset \\
& \End_G (I_P^G (E_B)) \underset{B^{W (M,\sigma,X_\nr (M))}}{\otimes} 
C^\me (U)^{W(M,\sigma,X_\nr (M))} =: \End_G (I_P^G (E_B))_U^\me .
\end{aligned}
\end{equation}
For the same reasons as in \eqref{eq:7.5}, $C^\me (U)$ embeds in $\End_G (I_P^G (E_B))_U^\me$.

For $\chi_c \in X_\nr (M,\sigma)$ and $w \in W(M,\mc O)$, \eqref{eq:4.16} says that
\begin{equation}\label{eq:7.13}
1_u \phi_{\chi_c} A_w 1_u = \left\{
\begin{array}{cc}
1_u \phi_{\chi_c} A_w = \phi_{\chi_c} A_w 1_u & \chi_c \mf{w} (u) = u \\
0 & \text{otherwise}
\end{array}
\right. 
\end{equation}
Since $X_\nr (M,\sigma)$ acts freely on $X_\nr (M)$, for a given $w \in W(M,\mc O)$ there exists
at most one $\chi_c = \chi_c (w) \in X_\nr (M,\sigma)$ such that $\chi_c \mf{w}$ fixes $u$. Let
$W(M,\mc O)_{\sigma \otimes u}$ be the $W(M,\mc O)$-stabilizer of $\sigma \otimes u \in \Irr (M)$.
Then
\begin{equation}\label{eq:7.11}
\begin{array}{cccc}
\Omega_u : & W (M,\mc O)_{\sigma \otimes u} & \to & W(M,\sigma, X_\nr (M))_u \\
& w & \mapsto & \chi_c (w) \mf{w}
\end{array}
\end{equation}
is a group isomorphism. With Theorem \ref{thm:4.3}, \eqref{eq:4.25} and \eqref{eq:7.13} this yields
\begin{equation}\label{eq:7.14}
\begin{aligned}
1_u \End_G (I_P^G & (E_B))_U^\me 1_u =
\bigoplus_{w \in W(M,\mc O)_{\sigma \otimes u}} C^{me}(U_u) A_{\Omega_u (w)} = \\
& \bigoplus_{w \in W(M,\mc O)_{\sigma \otimes u}} C^{me}(U_u) A_{\Omega_u (w)} =
\bigoplus_{w \in W(M,\mc O)_{\sigma \otimes u}} C^{me}(U_u) \mc{T}_{\Omega_u (w)}.
\end{aligned}
\end{equation}

\subsection{Localized endomorphism algebras with meromorphic functions} \

Consider the set of roots 
\[
\Sigma_{\sigma \otimes u} := \{ \alpha \in \Sigma_\red (A_M) : \mu^{M_\alpha}(\sigma \otimes u) = 0 \} .
\]
This is a root system \cite[\S 1]{Sil1}, that can be shown with the same argument as for Proposition
\ref{prop:3.5}.c. The parabolic subgroup $P = MU$ of $G$ determines 
a positive system $\Sigma_{\sigma \otimes u}(P)$ and a basis $\Delta_{\sigma \otimes u}$ of 
$\Sigma_{\sigma \otimes u}$. The relevant R-group (the Knapp--Stein R-group) is 
\[
R(\sigma \otimes u) = \{ w \in W(M,\mc O)_{\sigma \otimes u} : w ( \Sigma_{\sigma \otimes u}(P) )
= \Sigma_{\sigma \otimes u}(P) \}
\]
Like in \eqref{eq:3.8} 
\begin{equation}\label{eq:7.12}
W(M,\mc O)_{\sigma \otimes u} = W(\Sigma_{\sigma \otimes u}) \rtimes R(\sigma \otimes u) .
\end{equation}
We note that $R(\sigma \otimes u)$ need not be contained in $R(\mc O)$, even though 
$W(\Sigma_{\sigma \otimes u}) \subset W(\Sigma_{\mc O,\mu})$.

To obtain generators of $1_u \End_G (I_P^G (E_B))^\an_U 1_u$ with nice and simple relations, we vary 
on the previous constructions.
We follow the setup from Sections \ref{sec:root}--\ref{sec:rational}, but now with base point
$\sigma \otimes u$ of $\mc O$, root system $\Sigma_{\sigma \otimes u}$, Weyl group
$W(\Sigma_{\sigma \otimes u})$ and R-group $R(\sigma \otimes u)$. On $X_\nr (M)$ we have new
functions $X_\alpha^u (\chi) := X_\alpha (u)^{-1} X_\alpha (\chi)$. We recall that by \eqref{eq:3.22}
$X_\alpha (u) \in \{1,-1\}$ for all $\alpha \in \Sigma_{\sigma \otimes u}$.

Further, $E$ is by default 
endowed with the $M$-representation $\sigma \otimes u$, and we get a slightly different version
$J_{P'|P}^u$ of $J_{P'|P}$. Instead of $\rho'_{\sigma,w}$ we use
\[
\rho'_{\sigma \otimes u,w} := \lambda (\tilde w) \mr{sp}_{\chi = 1} 
\prod_{\alpha \in \Sigma_{\sigma \otimes u}(P) \cap \Sigma_{\sigma \otimes u}(\overline{w^{-1} P})} 
(X_\alpha^u - 1) J^u_{w^{-1} (P) | P} (\sigma \otimes u\chi) .
\]
Then Lemmas \ref{lem:3.3} and \ref{lem:3.4} remain true with obvious small modifications. 
In particular, in Lemma \ref{lem:3.4}.b we have to replace the product over
\[
\alpha \in \Sigma_{\mc O,\mu}(P) \cap \Sigma_{\mc O,\mu}(w_2^{-1}(\overline P))
\cap \Sigma_{\mc O,\mu}(w_2^{-1} w_1^{-1} (P))
\]
by the analogous product (with $\sigma \otimes u$ instead of $\sigma$) over
\[
\alpha \in \Sigma_{\sigma \otimes u}(P) \cap \Sigma_{\sigma \otimes u}(w_2^{-1}(\overline P))
\cap \Sigma_{\sigma \otimes u}(w_2^{-1} w_1^{-1} (P)) .
\]
For $r \in R(\sigma \otimes u)$ we can now take $\chi_r = 1$ and 
\begin{equation}\label{eq:5.5}
\rho_{\sigma \otimes u,r} : \tilde r (\sigma \otimes u) \to \sigma \otimes u .
\end{equation}
With these we define $\rho^u_{P',rw}$ and
\[
A^u_{rw} = \rho^u_{rw} \circ \tau_{rw} \circ J^u_{K(B),rw} 
\qquad r w \in W(M,\mc O)_{\sigma \otimes u}.
\]
as before. The superscripts $u$ are meant to distinguish these operators from their ancestors 
without $u$ (or rather with $u = 1$). Then \eqref{eq:4.2} becomes
\begin{equation}\label{eq:7.15}
A^u_{rw} \circ b = (rw \cdot b) \circ A^u_{rw} \qquad b \in \C (X_\nr (M)) .
\end{equation}
Let $w_1,w_2 \in W(\Sigma_{\sigma \otimes u})$ and $r_1,r_2 \in R(\sigma \otimes u)$. 
By Proposition \ref{prop:4.1}:
\begin{equation}\label{eq:7.16}
A^u_{w_1} \circ A^u_{w_2} = \prod\nolimits_\alpha \Big( \mr{sp}_{\chi = 1} \frac{\mu^{M_\alpha}
(\sigma \otimes u \otimes \cdot)}{ (X_\alpha^u - 1)((X_\alpha^u)^{-1} - 1)} \Big) \mu^{M_\alpha}
(\sigma \otimes u \otimes w_2^{-1} w_1^{-1} \cdot)^{-1} A^u_{w_1 w_2} ,
\end{equation}
where the product runs over
\[
\alpha \in \Sigma_{\sigma \otimes u}(P) \cap \Sigma_{\sigma \otimes u}(w_2^{-1}(\overline P))
\cap \Sigma_{\sigma \otimes u}(w_2^{-1} w_1^{-1} (P)) .
\]
In particular for $\alpha \in \Delta_{\sigma \otimes u}$:
\begin{equation}\label{eq:7.26}
(A^u_{s_\alpha})^2 = \frac{4 c'_{s_\alpha}}{(1 - X_\alpha (u) q_\alpha^{-1} )^2
(1 + X_\alpha (u) q_{\alpha*}^{-1} )^2 \mu^{M_\alpha}(\sigma \otimes u \otimes \cdot)} .
\end{equation}
Similarly Proposition \ref{prop:4.2} yields the following multiplication rules:
\begin{equation}\label{eq:7.17}
\begin{aligned}
& A^u_{r_1} \circ A^u_{w_1} = A^u_{r_1 w_1} , \\
& A^u_{w_2} \circ A^u_{r_2} = A^u_{w_2 r_2} ,\\
& A^u_{r_1} \circ A^u_{r_2} = \natural_u (r_1,r_2) A^u_{r_1 r_2} .
\end{aligned}
\end{equation}
Here $\natural_u$ is a two-cocycle $R(\sigma \otimes u)^2 \to \C^\times$. 
By appropriate normalizations of the $\rho_{\sigma \otimes u,r}$ we can achieve that
\begin{equation}\label{eq:7.18}
\natural_u (1,r) = \natural_u (r,1) = 1 \qquad \text{and} \qquad \natural_u (r,r^{-1}) = 1
\end{equation}
for all $r \in R(\sigma \otimes u)$. In other words, we may assume that $A_1^u = 1$ and
$(A^u_r)^{-1} = A^u_{r^{-1}}$.

The arguments for Theorem \ref{thm:4.3} apply only partially in the current situation,
because we may have fewer than $|W(M,\mc O)|$ operators $A^u_{rw}$. Rather, 
\cite[Proposition 3.7]{Hei2} shows that in $\End_G (I_P^G (E_B)) \otimes_B K(B)$
\begin{equation}\label{eq:7.19}
\{ A^u_{rw} : rw \in W(M,\mc O)_{\sigma \otimes u} \} \text{ is } K(B)\text{-linearly independent.}
\end{equation}
We note that \eqref{eq:7.16} and \eqref{eq:7.26} mean that, when $X_\alpha (u) = -1$, in 
effect the roles of $q_\alpha$ and $q_{\alpha*}$ are exchanged. With that in mind we define
\[
g_\alpha^u = \frac{(1 - X_\alpha^2) (1 - X_\alpha (u) q_\alpha^{-1})
(1 + X_\alpha (u) q_{\alpha*}^{-1}) }{2 (1 - X_\alpha (u) q_\alpha^{-1} X_\alpha)
(1 + X_\alpha (u) q_{\alpha*}^{-1} X_\alpha)} \in \C (X_\nr (M)) ,
\]
and $\mc{T}_{s_\alpha}^u := g_\alpha^u A_{s_\alpha}^u$. This gives rise to elements 
$\mc T_w^u$ for $w \in W(\Sigma_{\sigma \otimes u})$, which satisfy the analogues of 
Proposition \ref{prop:4.6}, Lemma \ref{lem:4.7} and Lemma \ref{lem:4.4} -- with 
$W(M,\mc O)_{\sigma \otimes u}$ instead of $W(M,\sigma,X_\nr (M))$. 

To show that the set \eqref{eq:7.19} spans \eqref{eq:7.14} as $C^{me}(U_u)$-module,
we vary on the proofs of \cite[Th\'eor\`eme 3.8]{Hei2} and of Theorem \ref{thm:4.3}.

\begin{lem}\label{lem:7.3}
Regard $C^{an}(U), \Hom_G \big( I_P^G (E_B), I_P^G (E_{K(B)}) \big)$ and $C^{me} (U)$ as
subsets of, respectively, $\End_G (I_P^G (E \otimes_\C C^{an}(U))), 
\Hom_G \big( I_P^G (E \otimes_\C C^{an}(U)), I_P^G (E \otimes_\C C^{me}(U)) \big)$ and
$\End_G (I_P^G (E \otimes_\C C^{me}(U)))$. 
\enuma{
\item Then $1_u \End_G (I_P^G (E_B))_U^\me 1_u$ equals
\[
\mr{span} \Big( C^{me}(U_u) \Hom_G \big( I_P^G (E_B), I_P^G (E_{K(B)}) \big) C^{an}(U_u) \Big) =
\bigoplus_{w \in W(M,\mc O)_{\sigma \otimes u}} C^{me}(U_u) A^u_w .
\]
\item The elements $A_r^u \mc T_w^u$ with $rw \in W(M,\mc O)_{\sigma \otimes u}$ span
a subalgebra isomorphic to $\C [W(M,\mc O)_{\sigma \otimes u}, \natural_u]$, where $\natural_u$
is the 2-cocycle from \eqref{eq:7.17}. This provides an algebra isomorphism 
\[
1_u \End_G (I_P^G (E_B))_U^\me 1_u \cong 
C^\me (U_u) \rtimes \C [W(M,\mc O)_{\sigma \otimes u}, \natural_u],
\]
where we take the crossed product with respect to the canonical action of \\
$W(M,\mc O)_{\sigma \otimes u}$ on $C^\me (U_u)$.
}
\end{lem}
\begin{proof}
(a) Recall from Lemma \ref{lem:5.1} and \eqref{eq:7.22} that $\Hom_G \big( I_P^G (E_B), 
I_P^G (E_{K(B)}) \big)$ has a filtration with successive subquotients isomorphic to
\[
\Hom_M (w \cdot E_B, E_{K(B)}) \cong 
\bigoplus\nolimits_{\chi_c \in X_\nr (M,\sigma)} \phi_{\chi_c} K(B) ,
\]
where $w$ runs through $W(M,\mc O)_{\sigma \otimes u}$. Considering the left hand side as 
a subset of $\Hom_M \big( w \cdot (E \otimes_\C C^{an}(U)), E \otimes_\C C^{me}(U) \big)$, 
we can compose it on the left with $C^{me}(U)$ and on the right with $C^{an}(U)$. 
Using \eqref{eq:2.26} we find
\begin{multline*}
\mr{span} \big( C^{me}(U_u) \Hom_M (w \cdot E_B, E_{K(B)}) C^{an}(U_u) \big) \cong \\
\left\{ \begin{array}{cl}
\bigoplus_{\chi \in X_\nr (M,\sigma)} C^{me}(U_u) \phi_{\chi} 
C^{an}(U_u) & w X_\nr (M,\sigma) U_u \cap U_u \neq \emptyset \\
0 & \text{otherwise}
\end{array} \right. .
\end{multline*}
In view of the construction of $U_u$, the above condition on $w$ is equivalent to $w \in
W(M,\mc O)_{\sigma \otimes u}$. Furthermore $C^{me}(U_u) \phi_{\chi} C^{an}(U_u) = 0$
for all $\chi \in X_\nr (M,\sigma) \setminus \{1\}$. We conclude that 
\begin{equation}\label{eq:7.23}
\mr{span} \big( C^{me}(U_u) \Hom_G \big( I_P^G (E_B), I_P^G (E_{K(B)}) \big) C^{an}(U_u) \big)
\end{equation}
has a filtration with subquotients isomorphic to 
\[
\mr{span} \big( C^{me}(U_u) \Hom_M (w \cdot E_B, E_{K(B)}) C^{an}(U_u) \big) \cong C^{me}(U_u) 
\qquad w \in W(M,\mc O)_{\sigma \otimes u} .
\]
In particular \eqref{eq:7.23} has dimension $|W(M,\mc O)_{\sigma \otimes u}|$ over
$C^{me}(U_u)$ (notice that $C^{me}(U_u)$ is a field because $U_u$ is connected).
By \eqref{eq:7.19} the $A^u_w$ are $C^{me}(U_u)$-linearly independent, so
$\bigoplus_{w \in W(M,\mc O)_{\sigma \otimes u}} C^{me}(U_u) A^u_w$
is the whole of \eqref{eq:7.23}.\\
(b) The properties involving the elements $A_r^u \mc T_w^u$ can be shown in the same way as 
at the end of Section \ref{sec:rational}.
\end{proof}

\begin{rem}\label{rem:6.4}
The elements $A_r^u \mc T_w^u$ from Lemma \ref{lem:7.3}.b multiply like the elements 
of $W(M,\mc O)_{\sigma \otimes u}$ (up to a 2-cocycle trivial on $W(\Sigma_{\sigma \otimes u})$), 
they normalize $C^{me}(U_u)$ and act on it in a prescribed way. Furthermore, our construction 
consists entirely of steps needed to achieve those properties. In this sense, these elements
$A_r^u \mc T_w^u$ are canonical up to rescaling the $A_r^u$.
\end{rem}

\subsection{Localized endomorphism algebras with analytic functions} \

We set out to find a $C^\an (U_u)$-basis of $1_u \End_G (I_P^G (E_B))_U^\an 1_u$. An element of 
$1_u \End_G (I_P^G (E_B))_U^\me 1_u$ lies in $1_u \End_G (I_P^G (E_B))_U^\an 1_u$
precisely when it does not have any poles on $U_u$.
By \eqref{eq:7.16} the poles of $A_{s_\alpha} \; (\alpha \in \Delta_{\sigma \otimes u})$
are precisely the zeros of $\mu^{M_\alpha}(\sigma \otimes u \otimes \cdot)$. In view of
Condition \ref{cond:7.U}, the only poles on $U_u$ are those at $\{ X_\alpha^u = 1 \} =
\{ X_\alpha = X_\alpha (u) \}$. The intersection of this set with $U_u$ is connected and
equals 
\[
U_u^{s_\alpha} = u X_\nr (M_\alpha) \cap U_u.
\]
By Lemma \ref{lem:6.1}.a, for $\alpha \in \Delta_{\sigma \otimes u}, \chi \in U_u^{s_\alpha}$:
\begin{equation}\label{eq:6.2} 
\mr{sp}_\chi (1 - X_\alpha^u) A_{s_\alpha}^u = \mr{sp}_\chi .
\end{equation}
For $\alpha \in \Sigma_{\mc O,\mu}$ we define
\[
f_\alpha = \frac{X_\alpha^2 (q_\alpha q_{\alpha*} - 1) + 
X_\alpha (q_\alpha - q_{\alpha*}) }{X_\alpha^2 - 1} =
\frac{X_\alpha^u (q_\alpha q_{\alpha*} - 1) + 
X_\alpha (u) (q_\alpha - q_{\alpha*}) }{X_\alpha^u - (X_\alpha^u)^{-1}} . 
\]
When $q_{\alpha*} = 1$ (which happens for most roots), $f_\alpha$ reduces to
\[
\frac{(X_\alpha + 1) (q_\alpha - 1)}{X_\alpha - X_\alpha^{-1}} =
\frac{q_\alpha - 1}{1 - X_\alpha^{-1}} .
\]
By \eqref{eq:3.21} and Lemma \ref{lem:3.8}.a
\begin{equation}\label{eq:4.23}
w \cdot f_\alpha = f_{w (\alpha)} \qquad \alpha \in \Sigma_{\mc O,\mu}, w \in W(M,\mc O) .
\end{equation}
One checks that, for $\alpha \in \Delta_{\sigma \otimes u}, \chi \in U_u^{s_\alpha}$:
\begin{equation}\label{eq:5.3}
\begin{split}
\mr{sp}_\chi \big( (1 - X_\alpha^u) f_\alpha \big) =
\mr{sp}_\chi \Big( \frac{q_\alpha q_{\alpha*} - 1 + 
X_\alpha (u) (q_\alpha - q_{\alpha*}) }{-1 - (X_\alpha^u)^{-1}} \Big) \\
= -(q_\alpha - X_\alpha (u)) (q_{\alpha*} + X_\alpha (u)) / 2 . 
\end{split}
\end{equation}
By \eqref{eq:6.2} and \eqref{eq:5.3}, the element
\[
\frac{(q_\alpha - X_\alpha (u)) (q_{\alpha*} + X_\alpha (u))}{2}
A_{s_\alpha} + f_\alpha \in \Hom_G \big( I_P^G (E_B), I_P^G (E_{K(B)}) \big)
\]
does not have any poles on $U_u$. Therefore 
\[
T_{s_\alpha}^u := 1_u \frac{(q_\alpha - X_\alpha (u)) (q_{\alpha*} + X_\alpha (u))}{2}
A_{s_\alpha} + 1_u f_\alpha 
\]
belongs to $1_u \End_G (I_P^G (E_B))_U^\an 1_u$. We note that
\begin{align}\nonumber
1 + f_\alpha = \frac{X_\alpha^2 q_\alpha q_{\alpha*} - 1 + 
X_\alpha (q_\alpha - q_{\alpha*}) }{X_\alpha^2 - 1} =
\frac{( X_\alpha^u q_\alpha - X_\alpha (u)) (X_\alpha^u q_{\alpha*} + X_\alpha (u))
}{(X_\alpha^u)^2 - 1} ,\\
\label{eq:5.1} T_{s_\alpha}^u = 1_u (1 + f_{-\alpha}) \mc T_{s_\alpha}^u + 1_u f_\alpha =
\mc T_{s_\alpha}^u (1 + f_\alpha) 1_u + f_\alpha 1_u \in 1_u \End_G (I_P^G (E_B))_U^\an 1_u .
\end{align}
The quadratic relations for the operators $T_{s_\alpha}^u$ read:

\begin{lem}\label{lem:6.3}
$(T_{s_\alpha}^u + 1_u)( T_{s_\alpha}^u - q_\alpha q_{\alpha*} 1_u) = 0$ 
for $\alpha \in \Delta_{\sigma \otimes u}$.
\end{lem}
\begin{proof}
With \eqref{eq:5.1} and the multiplication rules for $\mc T_{s_\alpha}^u$ in \eqref{eq:4.30}
we compute
\begin{multline}\label{eq:5.4}
(T_{s_\alpha}^u + 1_u)( T_{s_\alpha}^u - q_\alpha q_{\alpha*} 1_u) = \\
1_u \big( (1+ f_{-\alpha}) \mc T_{s_\alpha}^u + f_\alpha 1_u + 1_u \big) \big( (1+ f_{-\alpha}) 
\mc T_{s_\alpha}^u + f_\alpha - q_\alpha q_{\alpha*}  \big) = \\
1_u \Big( (1 + f_{-\alpha})(1 + f_\alpha) \mc T_{s_\alpha}^u + (1 + f_{-\alpha}) \mc T_{s_\alpha}^u
(f_\alpha - q_\alpha q_{\alpha*} + f_{-\alpha} + 1) +
(f_\alpha + 1)(f_\alpha - q_\alpha q_{\alpha*} ) \Big) \\
= 1_u \Big( ( 1 + f_{-\alpha}) \mc T_{s_\alpha}^u (f_\alpha + f_{-\alpha} + 1 - q_\alpha q_{\alpha*} ) 
+ (1 + f_\alpha) (f_\alpha + f_{-\alpha} + 1 - q_\alpha q_{\alpha*} ) \Big).
\end{multline}
By direct calculation $f_\alpha + f_{-\alpha} + 1 - q_\alpha q_{\alpha*} = 0$,
so the last line of \eqref{eq:5.4} reduces to 0.
\end{proof}

With \eqref{eq:4.2} we compute, for $b \in C^\me (U_u)$:
\begin{align}\nonumber
b \, T_{s_\alpha}^u & = \frac{(q_\alpha - X_\alpha (u)) (q_{\alpha*} + 
X_\alpha (u))}{2} A_{s_\alpha}^u (s_\alpha \cdot b) + f_\alpha b \\
\label{eq:6.9} & = T_{s_\alpha}^u (s_\alpha \cdot b) + f_\alpha (b - s_\alpha \cdot b) \\
\nonumber & = T_{s_\alpha}^u (s_\alpha \cdot b) + \big( q_\alpha q_{\alpha*} - 1 + 
X_\alpha^{-1} (q_\alpha - q_{\alpha*}) \big) (1 - X_\alpha^{-2})^{-1} (b - s_\alpha \cdot b) .
\end{align}
Any $w \in W(\Sigma_{\sigma \otimes u})$ can be written as a reduced word 
$s_1 s_2 \cdots s_{\ell_{\mc O}(w)}$ in the generators $s_\alpha$ with $\alpha \in 
\Delta_{\sigma \otimes u}$. We pick such a reduced expression and we define
\begin{equation}\label{eq:6.6}
T_w^u = T_{s_1}^u T_{s_2}^u \cdots T^u_{s_{\ell_{\mc O}(w)}} \in 1_u \End_G (I_P^G (E_B))_U^\an 1_u. 
\end{equation}

\begin{lem}\label{lem:6.2}
Let $w \in W(\Sigma_{\sigma \otimes u})$ and $r \in R(\sigma \otimes u)$.
\enuma{
\item The operator \eqref{eq:6.6} does not depend on the choice of the reduced expression
$w = s_1 s_2 \cdots s_{\ell_{\mc O}(w)}$.
\item $(A_r^ u)^{-1} T_w^u A_r^u = T^u_{r^{-1} w r}$.
}
\end{lem}
\begin{proof}
(a) In view of the defining relations in the Coxeter group $W(\Sigma_{\sigma \otimes u})$, it 
suffices to show the following statement. Let $\alpha, \beta \in \Delta_{\sigma \otimes u}$ with 
$s_\alpha s_\beta$ of order $m_{\alpha \beta} > 1$. Then
\begin{equation}\label{eq:6.7}
T_{s_\alpha}^u T_{s_\beta}^u T_{s_\alpha}^u \cdots T_{s_{\alpha / \beta}}^u = 
T_{s_\beta}^u T_{s_\alpha}^u T_{s_\beta}^u \cdots T_{s_{\beta / \alpha}}^u 
\qquad ( m_{\alpha \beta} \text{ factors on both sides).}
\end{equation}
Consider the affine Hecke algebra $\mc H$ with root system $\{ h_\alpha^\vee : \alpha \in 
\Sigma_{\sigma \otimes u} \}$, torus $X_\nr (M)$, parameter $q_F^{1/2}$ and labels 
\[
\lambda (h_\alpha^\vee) = \log (q_\alpha q_{\alpha*} ) / \log (q_F) ,\; 
\lambda^* (h_\alpha^\vee) = \log (q_\alpha q_{\alpha*}^{-1}) / \log (q_F) .
\]
By definition $\mc H$ is generated by a subalgebra $\C [X_\nr (M)]$ and elements 
$T_{s_\alpha} \; (\alpha \in \Delta_{\sigma \otimes u})$ that satisfy:
\begin{itemize}
\item the braid relations \eqref{eq:6.7} from $W(\Sigma_{\sigma \otimes u})$;
\item $(T_{s_\alpha} + 1)( T_{s_\alpha} - q_\alpha q_{\alpha*} ) = 0$;
\item $b T_{s_\alpha} = T_{s_\alpha} (s_\alpha \cdot b) + f_\alpha (b - s_\alpha \cdot b)
\quad b \in \C [X_\nr (M)]$.
\end{itemize} 
Alternatively, $\mc H \otimes_{\C [X_\nr (M)]^{W(\Sigma_{\sigma \otimes u})}} 
\C (X_\nr (M))^{W(\Sigma_{\sigma \otimes u})}$ can be generated by\\
$\C (X_\nr (M))$ and the elements 
\begin{equation}\label{eq:5.6}
\tau_{s_\alpha} := (T_{s_\alpha} + 1) (1 + f_\alpha)^{-1} - 1 .
\end{equation}
These elements stem from \cite[\S 5.1]{Lus-Gr}, where $1 + f_\alpha$ is denoted $\mc G (\alpha)$.
By \cite[Proposition 5.2]{Lus-Gr} they satisfy the same relations as our $\mc T_{s_\alpha}^u$,
namely Proposition \ref{prop:4.6} and \eqref{eq:4.30}. 
That gives a new presentation of $\mc H \otimes_{\C 
[X_\nr (M)]^{W(\Sigma_{\sigma \otimes u})}} \C (X_\nr (M))^{W(\Sigma_{\sigma \otimes u})}$, 
with defining relations
\begin{itemize}
\item the braid relations from $W(\Sigma_{\sigma \otimes u})$ (but now for the $\tau_{s_\alpha}$);
\item $\tau_{s_\alpha}^2 = 1$; 
\item $b \tau_{s_\alpha} = \tau_{s_\alpha} (s_\alpha \cdot b) \quad b \in \C (X_\nr (M))$.
\end{itemize}
We map $U_u$ to $X_\nr (M)$ by considering $u \in U_u$ and $1 \in X_\nr (M)$ as basepoints,
so $\chi \mapsto u^{-1} \chi$. That gives an injection
\begin{equation}\label{eq:5.2}
\C (X_\nr (M)) \to C^\me (U_u) .
\end{equation}
We checked that the $\tau_{s_\alpha}$ and the $\mc T_{s_\alpha}^u$ satisfy the same relations. 
Hence there is a unique algebra homomorphism 
\[
\mc H \otimes_{\C [X_\nr (M)]^{W(\Sigma_{\sigma \otimes u})}} 
\C (X_\nr (M))^{W(\Sigma_{\sigma \otimes u})} \longrightarrow 1_u \End_G (I_P^G (E_B))_U^\me 1_u
\]
that extends \eqref{eq:5.2} and sends $\tau_w$ to $\mc T_w^u$ for $w \in W(\Sigma_{\sigma \otimes u})$.
From \eqref{eq:5.1} and \eqref{eq:5.6} we see that $T_{s_\alpha}$ is mapped $T_{s_\alpha}^u$ for
$\alpha \in \Delta_{\sigma \otimes u}$.
As the $T_{s_\alpha}$ satisfy the braid relations \eqref{eq:6.7}, so do the $T_{s_\alpha}^u$.\\
(b) From the definition of $T_{s_\alpha}^u$ and Proposition \ref{prop:4.2}.c we obtain 
\begin{equation}\label{eq:6.10}
%%\begin{aligned}
T_{s_\alpha}^u A_r^u = 
1_u \frac{(q_\alpha - X_\alpha (u)) (q_{\alpha*} + X_\alpha (u))}{2} 
A_r^u A_{s_{r^{-1} \alpha}} + A_r^u 1_u f_{r^{-1} \alpha} = A_r^u T^u_{s_{r^{-1} \alpha}} .
%%\end{aligned}
\end{equation}
Recall from \eqref{eq:7.18} that $(A_r^u )^{-1} = A^u_{r^{-1}}$.
Applying that and \eqref{eq:6.10} repeatedly, we find
\begin{equation}\label{eq:6.8}
(A_r^u)^{-1} T_w^u A_r^u = (A_r^u)^{-1} T^u_{s_1} T^u_{s_2} \cdots T^u_{s_{\ell_{\mc O}(w)}} A_r^u 
= T^u_{r^{-1} s_1 r} T^u_{r^{-1} s_2 r} \cdots T^u_{r^{-1} s_{\ell_{\mc O}(w)} r} . 
\end{equation}
Since conjugation by $r \in R(\sigma \otimes u)$ preserves the lengths of elements of 
$W(\Sigma_{\sigma \otimes u})$,
\[
r^{-1} w r = (r^{-1} s_1 r) (r^{-1} s_2 r) \cdots (r^{-1} s_{\ell_{\mc O}(w)} r)
\]
is a reduced expression. Now part (a) guarantees that the right hand side of \eqref{eq:6.8}
equals $T_{r^{-1} w r}^u$. 
\end{proof}

The arguments from \cite[\S 5]{Hei2} apply to the operators $A^u_r$ and $T_w^u$ in\\
$\Hom_G (I_P^G (E_B), I_P^G (E_{K(B)}))$, provided that we only look at $U_u \subset X_\nr (M)$.
In particular \cite[Proposition 5.9]{Hei2} proves that, for any $\chi \in U_u$:
\begin{equation}\label{eq:7.25}
\{ \mr{sp}_\chi A^u_r T^u_w : rw \in W(M,\mc O)_{\sigma \otimes u} \}
\text{ is } \C\text{-linearly independent} 
\end{equation}
in $\Hom_G \big( I_P^G (E_B),I_P^G (E, \sigma \otimes \chi)\big)$.

\begin{thm}\label{thm:7.4}
The algebra
\[
1_u \End_G (I_P^G (E_B))_U^\an 1_u = 
\mr{span} \big( C^\an (U_u) \End_G (I_P^G (E_B)) C^\an (U_u) \big)
\]
can be expressed as
\[
\bigoplus_{r \in R(\sigma \otimes u)} \bigoplus_{w \in W(\Sigma_{\sigma \otimes u})}
C^\an (U_u) \, A^u_r \, T^u_w \, =  \bigoplus_{w \in W(\Sigma_{\sigma \otimes u})}
\bigoplus_{r \in R(\sigma \otimes u)}  T^u_w \, A^u_r \, C^\an (U_u) .
\]
\end{thm}
\begin{proof}
By \eqref{eq:7.15}, for all $rw \in W(M,\mc O)_{\sigma \otimes u}$,
\[
1_u A^u_r T^u_w = 1_u A^u_r T^u_w 1_u = A^u_r T^u_w 1_u
\] 
and it is a nonzero element of $1_u \End_G (I_P^G (E_B))_U^\an 1_u$. 
Lemma \ref{lem:7.3} tells us that 
\[
\bigoplus_{rw \in W(M,\mc O)_{\sigma \otimes u}} \hspace{-5mm} C^\an (U_u) A^u_r T^u_w \;
\subset \; 1_u \End_G (I_P^G (E_B))_U^\an 1_u \; \subset \;
\bigoplus_{rw \in W(M,\mc O)_{\sigma \otimes u}} \hspace{-5mm} C^\me (U_u) A^u_r T^u_w .
\]
With \eqref{eq:7.25} that entails, for any $\chi \in U_u$, the $\C$-vector space
\[
\{ \mr{sp}_\chi (A) : A \in 1_u \End_G (I_P^G (E_B))_U^\an 1_u \}
\text{ has a basis } \{ \mr{sp}_\chi A^u_r T^u_w : rw \in W(M,\mc O)_{\sigma \otimes u} \} .
\]
Suppose that $f_{rw} \in C^\me (U_u)$ and
\[
A := \sum\nolimits_{rw \in W(M,\mc O)_{\sigma \otimes u}} f_{rw} A^u_r T^u_w \in
1_u \End_G (I_P^G (E_B))_U^\an 1_u .
\]
It remains to show that all the $f_{rw}$ belong to $C^\an (U_u)$. Consider any $\chi \in U_u$.
By the above there are unique $z_{rw} \in \C$ such that 
\[
\mr{sp}_\chi (A) = 
\sum\nolimits_{rw \in W(M,\mc O)_{\sigma \otimes u}} z_{rw} \mr{sp}_\chi A^u_r T^u_w.
\] 
Then $f_{rw} (\chi) = z_{rw}$. Hence none of the $f_{rw}$ has a pole at any $\chi \in U_u$.
In other words, they are analytic.
\end{proof}

Theorem \ref{thm:7.4} and the multiplication rules \eqref{eq:7.15}, \eqref{eq:7.17},
Lemma \ref{lem:6.3}, \eqref{eq:6.9}, \eqref{eq:6.6}, Lemma \ref{lem:6.2} provide a presentation
of $1_u \End_G (I_P^G (E_B))_U^\an 1_u$. We note that it is quite similar to an affine Hecke
algebra (when $R(\sigma \otimes u) = 1$) or to a twisted affine Hecke algebra 
\cite[Proposition 2.2]{AMS3}. The only difference is that the complex torus $T$ in the
definition of an affine Hecke algebra has been replaced by the complex manifold $U_u$,
and $\C [T]$ by $C^\an (U_u)$.

This observation enables us to compute the centre of $1_u \End_G (I_P^G (E_B))_U 1_u$
with the methods from \cite[Proposition 3.11]{Lus-Gr} and \cite[\S 1.2]{SolAHA}:
\begin{equation}\label{eq:7.30}
Z \big( 1_u \End_G (I_P^G (E_B))_U^\an 1_u \big) = C^{an}(U_u)^{W(M,\mc O)_{\sigma \otimes u}} .
\end{equation}

\section{Link with graded Hecke algebras}
\label{sec:GHA}

We will provide an easier presentation of $1_u \End_G (I_P^G (E_B))_U 1_u$, which 
comes from a graded Hecke algebra \cite{Lus-Gr}. 
Let us recall the construction of the graded Hecke algebras that we want to use,
starting from the affine Hecke algebra $\mc H$ in the proof of Lemma \ref{lem:6.2}.
We replace the complex torus $X_\nr (M) = \Hom (M/M^1,\C^\times)$ by its Lie algebra
\[
\mr{Lie}(X_\nr (M)) = \Hom (M/M^1 ,\C) = a_M^* \otimes_\R \C .
\]
The algebra of regular functions $\C [X_\nr (M)] = \C [M / M^1]$ is replaced by
the algebra of polynomial functions
\[
\C [\mr{Lie}(X_\nr (M))] = \C [a_M^* \otimes_\R \C ] = S(a_M \otimes_\R \C) ,
\]
where $S$ denotes the symmetric algebra of a vector space. The group $W(M,\mc O)$ 
acts naturally on Lie$(X_\nr (M))$ and on $\C [a_M^* \otimes_\R \C]$.

Recall that in Proposition \ref{prop:3.5} we associated to every $\alpha \in 
\Sigma_{\mc O,\mu}$ elements $h_\alpha^\vee \in M / M^1 \subset a_M, 
\alpha^\sharp \in a_M^*$. These elements form root systems $\Sigma^\vee_{\mc O}$ 
and $\Sigma_{\mc O}$, respectively. The quadruple
\[
\big( a_M, \{ h_\alpha^\vee ,\alpha \in \Sigma_{\sigma \otimes u} \}, a_M^* ,
\{ \alpha^\sharp : \alpha \in \Sigma_{\sigma \otimes u} \}, 
\{ h_\alpha^\vee ,\alpha \in \Delta_{\sigma \otimes u} \} \big)
\]
will be denoted $\tilde{\mc R}_u$, and is sometimes called a degenerate root datum.
Let $k^u : \Sigma_{\sigma \otimes u} \to \R_{\geq 0}$ be the 
$W(M,\mc O)_{\sigma \otimes u}$-invariant parameter function
\[
k^u_\alpha = \left\{ \begin{array}{ll}
\log (q_\alpha) & \text{if } X_\alpha (u) = 1  \\
\log (q_{\alpha*}) & \text{if } X_\alpha (u) = -1 
\end{array} \right. .
\]
We also need the 2-cocycle 
\[
\natural_u : R(\sigma \otimes u)^2 = 
(W (M,\mc O)_{\sigma \otimes u} / W(\Sigma_{\sigma \otimes u}) )^2 \to \C^\times
\]
from Lemma \ref{lem:7.3}. It gives rise to a twisted group algebra 
$\C [W(M,\mc O)_{\sigma \otimes u}, \natural_u]$
with basis $\{ N_w : w \in W(M,\mc O)_{\sigma \otimes u} \}$ and multiplication rules
\[
N_{w_1} N_{w_2} = \natural_u (w_1, w_2) N_{w_1 w_2} .
\]
\begin{lem}\label{lem:8.5}
Recall the bijection $\Omega_u : W(M,\mc O)_{\sigma \otimes u} \to W(M,\sigma, X_\nr (M))_u$
from \eqref{eq:7.11}. The 2-cocycles $\natural_u$ and $\natural \circ \Omega_u$ of
$W(M,\mc O)_{\sigma \otimes u}$ are cohomologous. Further, $\Omega_u$ induces a canonical
algebra isomorphism
\[
\widetilde{\Omega_u} : \C [W(M,\mc O)_{\sigma \otimes u}, \natural_u] \to 
\C [W(M,\sigma,X_\nr (M))_u , \natural ].
\]
\end{lem}
\begin{proof}
By definition
\begin{equation}\label{eq:8.1}
A_{r'}^u \mc T_{w'}^u \circ A_{r}^u \mc T_w^u = \natural_u (r' w', rw) A_{r'r}^u \mc T^u_{r^{-1} w' r w} 
\end{equation}
for $rw, r'w' \in W(M,\mc O)_{\sigma \otimes u}$. For a generic $\chi \in X_\nr (M)$, 
$I_P^G (\sigma \otimes \chi)$ and $I_P^G (\sigma \otimes \Omega_u (rw)^{-1} \chi)$ are
irreducible. Both $\mr{sp}_\chi A_r^u \mc T_w^u$ and $\mr{sp}_\chi \mc T_{\Omega_u (rw)}$ are
$G$-homomorphisms 
\[
I_P^G (\sigma \otimes \chi) \to I_P^G (\sigma \otimes \Omega_u (rw)^{-1} \chi) ,
\]
so they differ only by a scalar factor (at least away from their poles). Furthermore 
$\mr{sp}_\chi A_r^u \mc T_w^u$ and $\mr{sp}_\chi \mc T_{\Omega_u (rw)}$ are rational functions 
of $\chi \in X_\nr (M)$, so there exists a unique $f_{rw}^u \in \C (X_\nr (M))$ such that
\begin{equation}\label{eq:8.2}
A_r^u \mc T_w^u = f_{rw}^u \mc T_{\Omega_u (rw)} .
\end{equation}
By \eqref{eq:4.26} $\mc T_{\Omega_u (rw)}$ and $A_r^u \mc T_w^u$ are regular at every unitary
$\chi \in X_\nr (M)$, in particular at $u$. Specializing \eqref{eq:8.1} at $u$ and combining
it with \eqref{eq:8.2}, we obtain
\begin{multline*}
f_{r'w'}^u (u) \mr{sp}_u (\mc T_{\Omega_u (r'w')}) f_{rw}^u (u) \mr{sp}_u (\mc T_{\Omega_u (rw)})
= \mr{sp}_u (A_{r'}^u \mc T_{w'}^u) \mr{sp}_u (A_r^u \mc T_w^u) = \\
\natural_u (r'w',rw) \mr{sp}_u (A_{r'r}^u \mc T^u_{r^{-1} w' r w}) =
\natural_u (r'w',rw) f_{r'w'rw}^u (u) \mr{sp}_u (\mc T_{\Omega_u (r' w' r w)}) . 
\end{multline*}
On the other hand, by Lemma \ref{lem:4.7} 
\[
\mr{sp}_u (\mc T_{\Omega_u (r'w')}) \mr{sp}_u (\mc T_{\Omega_u (rw)}) =
\natural (\Omega_u (r'w'),\Omega_u (rw)) \mr{sp}_u (\mc T_{\Omega_u (r'w'rw)}) .
\]
Comparing the expressions with $\mc T_{\Omega_u (?)}$ we find that
\begin{equation}\label{eq:8.27}
\frac{\natural (\Omega_u (r'w'),\Omega_u (rw))}{\natural_u (r'w',rw)} =
\frac{f_{r'w'rw}^u (u)}{f_{r'w'}^u (u) f_{rw}^u (u)} .
\end{equation}
By definition, this says that $\natural_u$ and $\natural \circ \Omega_u$ are cohomologous 
2-cocycles. Moreover \eqref{eq:8.27} shows that 
\[
A_r^u T_w^u \mapsto f_{rw}^u (u) \mc T_{\Omega_u (rw)} \qquad r \in R(\sigma \otimes u),
w \in W(\Sigma_{\sigma \otimes u})
\]
defines the algebra isomorphism $\widetilde{\Omega_u}$ we were looking for. It is canonical
because every $f_{rw}^u$ is unique.
\end{proof}

The graded Hecke algebra associated to the above data is the vector space\\
$\C [a_M^* \otimes_\R \C] \otimes_\C \C [W(M,\mc O)_{\sigma \otimes u}, \natural_u]$ with 
multiplication rules
\begin{enumerate}[(i)]
\item $\C [a_M^* \otimes_\R \C]$ and $\C [W(M,\mc O)_{\sigma \otimes u}, \natural_u]$
are embedded as subalgebras;
\item for $f \in \C [a_M^* \otimes_\R \C]$ and $\alpha \in \Delta_{\sigma \otimes u}$:
\[
f N_{s_\alpha} - N_{s_\alpha} (s_\alpha \cdot f) = 
k^u_\alpha \frac{f - s_\alpha \cdot f}{h_\alpha^\vee} ;
\]
\item $N_r f = (r \cdot f) N_r$ for $f \in \C [a_M^* \otimes_\R \C]$ and 
$r \in R(\sigma \otimes u)$.
\end{enumerate}
This algebra is denoted
\begin{equation}\label{eq:8.10}
\mh H (\tilde{\mc R}_u, W(M,\mc O)_{\sigma \otimes u}, k^u, \natural_u) .
\end{equation}
Weights of representations of this algebra are by default with respect to the commutative
subalgebra $\C [a_M^* \otimes_\R \C]$.
An advantage of \eqref{eq:8.10} over $\End_G (I_P^G (E_B))$ is that a lot is known about its 
representation theory, even with arbitrary parameters $k^u_\alpha$. 
It is easy to see that the centre of \eqref{eq:8.10} is
\[
Z \big( \mh H (\tilde{\mc R}_u, W(M,\mc O)_{\sigma \otimes u}, k^u, \natural_u) \big) =
\C [ a_M^* \otimes_\R \C ]^{W(M,\mc O)_{\sigma \otimes u}} .
\]
To interpolate between $1_u \End_G (I_P^G (E_B))_U^\an 1_u$ and a graded Hecke algebra, we need
a version of the latter with analytic functions instead of polynomials. 
Let $\tilde U \subset a_M^* \otimes_\R \C$ be an open $W(M,\mc O)_{\sigma \otimes u}$-stable subset.
Like in \cite[\S 1.5]{SolAHA} we consider the algebra 
\begin{multline*}
\mh H (\tilde{\mc R}_u, W(M,\mc O)_{\sigma \otimes u}, k^u, \natural_u)_{\tilde U}^\an := \\
\mh H (\tilde{\mc R}_u, W(M,\mc O)_{\sigma \otimes u}, k^u, \natural_u) 
\underset{\C [a_M^* \otimes_\R \C]^{W(M,\mc O)_{\sigma \otimes u}}}{\otimes} 
C^\an (\tilde U)^{W(M,\mc O)_{\sigma \otimes u}} .
\end{multline*}
As vector space it is
\begin{equation}\label{eq:8.3}
\mh H (\tilde{\mc R}_u, W(M,\mc O)_{\sigma \otimes u}, k^u, \natural_u)_{\tilde U}^\an =
C^\an (\tilde U) \otimes_\C \C [W(M,\mc O)_{\sigma \otimes u}, \natural_u] ,
\end{equation}
compare with Theorem \ref{thm:7.4}. The multiplication relations (ii) and (iii) in the definition of 
$\mh H (\tilde{\mc R}_u , W(M,\mc O)_{\sigma \otimes u}, k^u, \natural_u)$ now hold for all 
$f \in C^{an}(\tilde U)$.

\begin{lem}\label{lem:8.2}
\enuma{
\item The following categories are naturally equivalent:
\begin{itemize}
\item $\mh H (\tilde{\mc R}_u, W(M,\mc O)_{\sigma \otimes u}, k^u, 
\natural_u)_{\tilde U}^\an -\fMod$;
\item $\mh H (\tilde{\mc R}_u, W(M,\mc O)_{\sigma \otimes u}, k^u, \natural_u) -\Modf{\tilde U}$.
\end{itemize}
\item For a Levi subgroup $L$ of $G$ containing $M$, there is a version of \\
$\mh H (\tilde{\mc R}_u, W(M,\mc O)_{\sigma \otimes u}, k^u, \natural_u)$ that uses only those
elements of $W(M,\mc O)_{\sigma \otimes u}$ and $\Sigma_{\sigma \otimes u}$ that come from $L$.
With that as parabolic subalgebra, the above equivalence of categories commutes with parabolic
induction and restriction in the same sense as Lemma \ref{lem:7.1}.
}
\end{lem}
\begin{proof}
This can be shown in the same way as \cite[Proposition 4.3]{Opd-Sp} and Lemmas \ref{lem:7.1}.
\end{proof}

We can also involve meromorphic functions on $\tilde U$, in an algebra 
\begin{multline}\label{eq:8.5}
\mh H (\tilde{\mc R}_u, W(M,\mc O)_{\sigma \otimes u}, k^u, \natural_u)_{\tilde U}^\me := \\
\mh H (\tilde{\mc R}_u, W(M,\mc O)_{\sigma \otimes u}, k^u, \natural_u) 
\underset{\C [a_M^* \otimes_\R \C]^{W(M,\mc O)_{\sigma \otimes u}}}{\otimes} 
C^{me}(\tilde U)^{W(M,\mc O)_{\sigma \otimes u}} ,
\end{multline}
which as vector space equals
\[
C^{me}(\tilde U) \otimes_\C \C [W(M,\mc O)_{\sigma \otimes u}, \natural_u] .
\]
As in \cite[\S 5.1]{Lus-Gr} we define, for $\alpha \in \Delta_{\sigma \otimes u}$, the 
following element of \eqref{eq:8.5}:
\begin{equation}\label{eq:8.9}
\tilde{\mc T}_{s_\alpha} = -1 + (N_{s_\alpha} + 1) \frac{h_\alpha^\vee}{k^u_\alpha + h_\alpha^\vee} . 
\end{equation}
According to \cite[Proposition 5.2]{Lus-Gr}, $s_\alpha \mapsto \tilde{\mc T}_{s_\alpha}$ extends
uniquely to a group homomorphism $w \mapsto \tilde{\mc T}_w$ from $W(\Sigma_{\sigma \otimes u})$
to the multiplicative group of \eqref{eq:8.5}, and
\[
\tilde{\mc T}_w f = (w \cdot f) \tilde{\mc T}_w 
\qquad f \in C^{me}(\tilde U), w \in W(\Sigma_{\sigma \otimes u}) .
\] 
An argument analogous to Lemma \ref{lem:6.2}.b shows that 
\[
N_r \tilde{\mc T}_w N_r^{-1} = \tilde{\mc T}_{r w r^{-1}} \qquad
r \in R(\sigma \otimes u), w \in W(\Sigma_{\sigma \otimes u}) .
\]
It is easy to see from \eqref{eq:8.9} that the $\C (a_M^* \otimes_\R \C)$-span of the
$\tilde{\mc T}_w$ coincides with the $\C (a_M^* \otimes_\R \C)$-span of the $N_w \; 
(w \in W(\Sigma_{\sigma \otimes u})$. With \eqref{eq:8.3} that yields 
\begin{multline}\label{eq:8.6}
\mh H (\tilde{\mc R}_u, W(M,\mc O)_{\sigma \otimes u}, k^u, \natural_u)_{\tilde U}^\me = \\
\mr{span} \big( C^\me (\tilde U) \mh H (\tilde{\mc R}_u, W(M,\mc O)_{\sigma \otimes u}, k^u, \natural_u) \big) 
= \bigoplus\nolimits_{rw \in W(M,\mc O)_{\sigma \otimes u}} N_r \tilde{\mc T}_w C^{me}(\tilde U) .
\end{multline}
In view of the above multiplication relations, \eqref{eq:8.6} means that the algebra 
\eqref{eq:8.5} is a crossed product 
\begin{equation}\label{eq:8.4}
C^{me}(\tilde U) \rtimes \C [W(M,\mc O)_{\sigma \otimes u}, \natural_u] , 
\end{equation}
where the latter factor is spanned by the $N_r \tilde{\mc T}_w$. We note that these elements 
$N_r \tilde{\mc T}_w$ are canonical in the same sense as Remark \ref{rem:6.4}.

Now we specialize to a particular $\tilde U$. The analytic map
\[
\begin{array}{cccc}
\exp_u : & a_M^* \otimes_\R \C = \Hom (M/M^1 ,\C) & \to & \Hom (M/M^1, \C^\times) = X_\nr (M) \\
& \lambda & \mapsto & u \exp (\lambda) 
\end{array}
\]
is a $W(M,\mc O)_{\sigma \otimes u}$-equivariant covering. Notice that
\[
\exp_u (a_M^*) = u X_\nr^+ (M) .
\]
Let $\log (U_u)$ be the connected component of $\exp_u^{-1}(U_u)$ that contains 0. 
By Condition \ref{cond:7.U} $\exp_u : \log (U_u) \to U_u$ is an isomorphism of analytic varieties.
In particular $f \mapsto f \circ \exp_u$ provides $W(M,\mc O)_{\sigma \otimes u}$-equivariant 
algebra isomorphisms
\[
C^{an}(U_u) \to C^{an}(\log (U_u)) \quad \text{and} \quad C^{me}(U_u) \to C^{me}(\log (U_u))
\]
From Lemma \ref{lem:7.3}, \eqref{eq:8.6} and the multiplication relations in these algebras,
we see that $\exp_u$ induces an algebra isomorphism
\[
\begin{array}{cccc}
\Phi_u : & 1_u \End_G (I_P^G (E_B))_U^\me 1_u & \to & 
\mh H (\tilde{\mc R}_u, W(M,\mc O)_{\sigma \otimes u}, k^u, \natural_u)_{\log (U_u)}^\me \\
& f A^u_r \mc T^u_w & \mapsto & (f \circ \exp_u) N_r \tilde{\mc T}_w 
\end{array}.
\]

\begin{prop}\label{prop:8.3}
The algebra homomorphism $\Phi_u$ is canonical. It restricts to an algebra isomorphism
\[
1_u \End_G (I_P^G (E_B))_U^\an 1_u \longrightarrow 
\mh H (\tilde{\mc R}_u, W(M,\mc O)_{\sigma \otimes u}, k^u, \natural_u )_{\log (U_u)}^\an .
\]
For a Levi subgroup $L$ of $G$ containing $M$, we looked at parabolic subalgebras in Lemmas
\ref{lem:6.5} and \ref{lem:8.2}. For any such $L$, $\Phi_u$ restricts to an isomorphism
between the respective parabolic subalgebras.
\end{prop}
\begin{proof}
We note that as linear map $\Phi_u$ can be expressed in terms of
Lemma \ref{lem:7.3}.b and \eqref{eq:8.4} as
\begin{equation}\label{eq:8.26}
\exp_u^* \otimes \mr{id} : C^{me}(U) \rtimes \C [W(M,\mc O)_{\sigma \otimes u}, \natural_u] 
\to C^{me}(\log (U_u)) \rtimes \C [W(M,\mc O)_{\sigma \otimes u}, \natural_u] . 
\end{equation}
As discussed in Remark \ref{rem:6.4}, the basis elements $A_r^u \mc T_u^w$ and $N_r \tilde{\mc T_w}$ 
are constructed in a canonical way, apart from possible renormalizations of the $A_r^u$ and the 
$N_r$. But the multiplication relations between the $N_r$ are defined in terms of the multiplication 
rules for the $A_r^u$, so that automatically works in the same way for the source and the target 
of $\Phi_u$. Hence \eqref{eq:8.26} shows that $\Phi_u$ is canonical.

By construction $\Phi_u (C^{an}(U_u)) = C^{an}(\log (U_u))$ and $\Phi_u (1_u A_r^u) = N_r$, where
$1_u A^u_r \in 1_u \End_G (I_P^G (E_B))_U^\an 1_u$ and $N_r \in \mh H (\tilde{\mc R}_u, 
W(M,\mc O)_{\sigma \otimes u}, k^u, \natural_u)$. Hence, by Theorem \ref{thm:7.4} it suffices 
to show that
\[
\Phi_u (T^u_w) \in \mh H (\tilde{\mc R}_u, W(M,\mc O)_{\sigma \otimes u}, k^u, \natural_u)_{\log (U_u)}^\an 
\quad \text{and} \quad \Phi_u^{-1}(N_w) \in 1_u \End_G (I_P^G (E_B))_U^\an 1_u 
\]
for all $w \in W(\Sigma_{\sigma \otimes u})$. The argument for that is a variation on
\cite[Theorem 9.3]{Lus-Gr} and \cite[Theorem 2.1.4]{SolAHA}. By \eqref{eq:6.6} for the $T^u_w$,
it suffices to consider $w = s_\alpha$ with $\alpha \in \Delta_{\sigma \otimes u}$. We compute
\begin{align*}
\Phi_u (1_u + T^u_{s_\alpha}) & = \Phi_u ( (1_u + 1_u \mc T^u_{s_\alpha}) (1 + f_\alpha ) ) \\
& = (1 + \tilde{\mc T}^u_{s_\alpha}) (1 + f_\alpha \circ \exp_u ) \\
& = (N_{s_\alpha} + 1) \Big( \frac{h_\alpha^\vee}{k^u_\alpha + h_\alpha^\vee} \Big)
\Big( \frac{X_\alpha^2 q_\alpha q_{\alpha*} + X_\alpha (q_\alpha - q_{\alpha*} ) - 1}{X_\alpha^2 - 1} 
\Big) \circ \exp_u \\
& = (N_{s_\alpha} + 1) \frac{h_\alpha^\vee}{k^u_\alpha + h_\alpha^\vee} \frac{ e^{2 h_\alpha^\vee} 
q_\alpha q_{\alpha*}  + 
X_\alpha(u) e^{h_\alpha^\vee} (q_\alpha - q_{\alpha*}) - 1}{e^{2 h_\alpha^\vee} - 1} \\
& = (N_{s_\alpha} + 1)  \Big( \frac{h_\alpha^\vee}{e^{2 h_\alpha^\vee} - 1}\Big)
\Big( \frac{ (e^{h_\alpha^\vee} q_\alpha - X_\alpha (u)) (e^{h_\alpha^\vee} q_{\alpha*} + 
X_\alpha (u))}{k^u_\alpha + h_\alpha^\vee} \Big)
\end{align*}
By Condition \ref{cond:7.U}, for all $v \in \log (U_u)$
\begin{equation}\label{eq:8.7}
h_\alpha^\vee (v) = \log (X_\alpha (u^{-1} \exp_u (v))) \text{ has imaginary part in } (-\pi /2, \pi / 2) .
\end{equation}
Hence $\frac{h_\alpha^\vee}{e^{2 h_\alpha^\vee} - 1}$ is an invertible analytic function on $\log (U_u)$. 
When $X_\alpha (u) = 1$, \eqref{eq:8.7} entails that $e^{h_\alpha^\vee} q_F^{b_{s_\alpha}} + X_\alpha (u)$
is an invertible and analytic on $\log (U_u)$, and by l'Hopital's rule, so is
\[
\frac{ e^{h_\alpha^\vee} q_\alpha - X_\alpha (u)}{k^u_\alpha + h_\alpha^\vee} =
\frac{ e^{h_\alpha^\vee} q_\alpha - 1}{\log (q_\alpha) + h_\alpha^\vee} .
\]
Similarly, when $X_\alpha (u) = -1$, $e^{h_\alpha^\vee} q_\alpha - X_\alpha (u)$ and 
\[
\frac{ e^{h_\alpha^\vee} q_{\alpha*} + X_\alpha (u)}{k^u_\alpha + h_\alpha^\vee} =
\frac{ e^{h_\alpha^\vee} q_{\alpha*} - 1}{\log (q_{\alpha*}) + h_\alpha^\vee}
\]
are invertible analytic functions on $\log (U_u)$.
The above computation and these considerations about invertibility allow us to conclude that
\[
\Phi_u (1_u + T^u_{s_\alpha}) = 1_u + \Phi_u (T^u_{s_\alpha}) \in 
\mh H (\tilde{\mc R}_u, W(M,\mc O)_{\sigma \otimes u}, k^u, \natural_u )_{\log (U_u)}^\an .
\]
Applying $\Phi_u^{-1}$ to the entire computation and rearranging, we obtain
\[
\Phi_u^{-1}(N_{s_\alpha} + 1) = (1_u + 1_u T^u_{s_\alpha}) 
\Big( \frac{h_\alpha^\vee (e^{h_\alpha^\vee} q_\alpha - X_\alpha (u)) (e^{h_\alpha^\vee} 
q_{\alpha*} + X_\alpha (u))}{(e^{2 h_\alpha^\vee} - 1)(k^u_\alpha + h_\alpha^\vee)} \Big)^{-1}
\circ \exp_u^{-1} .
\]
We just argued that the function between the large brackets is invertible and analytic on $\log (U_u)$.
So its composition with $\exp_u^{-1}$ is invertible and analytic on $U_u$. In particular
$\Phi_u^{-1}(N_{s_\alpha} + 1) = \Phi_u^{-1}(N_{s_\alpha}) + 1_u$ lies in
$1_u \End_G (I_P^G (E_B))_U^\an 1_u$.

On both sides of \eqref{eq:8.26}, the parabolic subalgebra (with meromorphic functions) associated to 
$L$ is obtained by using only the elements of $W(M,\mc O)_{\sigma \otimes u}$ that come from $L$. 
Clearly $\Phi_u$ restricts to an isomorphism between those subalgebras. The above calculations can
be restricted to those subalgebras, and then they show that $\Phi_u$ also provides an isomorphism
between the parabolic subalgebras with analytic functions.
\end{proof}

\nopagebreak

\section{Classification of irreducible representations}
\label{sec:class}

\subsection{Description in terms of graded Hecke algebras} \

In Sections \ref{sec:localization} and \ref{sec:GHA} we investigated the
following algebra homomorphisms:
\begin{equation}\label{eq:8.8}
\begin{split}
\End_G (I_P^G (E_B)) \; \hookrightarrow \; \End_G (I_P^G (E_B))_U^\an \hookleftarrow 
1_u \End_G (I_P^G (E_B))_U^\an 1_u \\
\xrightarrow{\sim} \; \mh H (\tilde{\mc R}_u, W(M,\mc O)_{\sigma \otimes u}, k^u, \natural_u 
)_{\log (U_u)}^\an \; \hookleftarrow \; \mh H (\tilde{\mc R}_u, W(M,\mc O)_{\sigma \otimes u}, 
k^u, \natural_u ) .
\end{split}
\end{equation}

\begin{cor}\label{cor:8.4}
There are equivalences between the following categories:
\begin{enumerate}[(i)]
\item $\mh H (\tilde{\mc R}_u, W(M,\mc O)_{\sigma \otimes u}, k^u, \natural_u )-\Modf{a_M^*}$;
\item $\End_G (I_P^G (E_B))-\Modf{W(M,\sigma,X_\nr (M)) u X_\nr^+ (M)}$;
\item $\{ \pi \in \Rep_{\mr f} (G)^{\mf s} : 
\mr{Sc}(\pi) \subset W(M,\mc O) \{ \sigma \otimes u \chi : \chi \in X_\nr^+ (M) \}$.
\end{enumerate}
Once the 2-cocycle $\natural_u$ has been fixed (by normalizing the elements $A_u^r$ and $N_r$),
these equivalences are canonical. The equivalences commute with parabolic induction and 
Jacquet restriction, in the sense of Proposition \ref{prop:4.5}.
\end{cor}
\begin{proof}
Since $\natural_u$ is given, we may apply Proposition \ref{prop:8.3}. 
By Lemmas \ref{lem:7.1}, \ref{lem:7.2}, \ref{lem:7.3}, \ref{lem:8.2} and Proposition \ref{prop:8.3},
the homomorphisms \eqref{eq:8.8} provide canonical equivalences between the categories of
finite dimensional modules of the respective algebras, with the restriction that we only 
consider modules all whose central weights lie in, respectively $U / W(M,\sigma,X_\nr (M))$ (twice),
$U_u / W(M,\mc O)_{\sigma \otimes u}$ and $\log (U_u) / W(M,\mc O)_{\sigma \otimes u}$ (twice).
Restricting from $\log (U_u)$ to $a_M^*$ and from $U_u$ to $u X_\nr^+ (M)$, we obtain the equivalence 
between (i) and (ii).

The equivalence between (ii) and (iii) can be shown in the same way as in Lemma \ref{lem:7.2}.
It is always canonical.

The compatibility with parabolic induction and (Jacquet) restriction was already checked in all
the results we referred to in this proof.
\end{proof}

Sometimes it is more convenient to use left modules instead of right modules. That could have
been achieved by considering the $G$-endomorphisms of $I_P^G (E_B)$ as acting from the right.
Then we would get the opposite algebra $\End_G (I_P^G (E_B))^{op}$, and item (ii) of Corollary
\ref{cor:8.4} would involve left modules of $\End_G (I_P^G (E_B))^{op}$. The constructions
summarised in \eqref{eq:8.8} relate those to left modules of 
$\mh H (\tilde{\mc R}_u, W(M,\mc O)_{\sigma \otimes u}, k^u, \natural_u )^{op}$.

Fortunately, with the multiplication rules (i)--(iii) before \eqref{eq:8.10} it is easy to
identify the opposite algebra of a (twisted) graded Hecke algebra. Namely, there is an
algebra isomorphism
\begin{equation}\label{eq:8.11}
\begin{array}{ccc}
\mh H (\tilde{\mc R}_u, W(M,\mc O)_{\sigma \otimes u}, k^u, \natural_u )^{op} & \to &
\mh H (\tilde{\mc R}_u, W(M,\mc O)_{\sigma \otimes u}, k^u, \natural_u^{-1} ) \\
N_{rw} f & \mapsto & f N_{w^{-1} r^{-1}}
\end{array}.
\end{equation}
The only subtlety to check in \eqref{eq:8.11} is that the 2-cocycles match up -- for that one
needs \eqref{eq:7.18}.

Thus the categories in Corollary \ref{cor:8.4} are also equivalent with the category of finite
dimensional left $\mh H (\tilde{\mc R}_u, W(M,\mc O)_{\sigma \otimes u}, k^u, \natural_u^{-1} 
)$-modules, all whose $\C [a_M^* \otimes_\R \C]$-weights lie in $a_M^*$. In other words, for the 
algebras that we consider it does not make too much difference whether we use left or right modules.

For $\chi_c \in X_\nr (M,\sigma)$, the $M$-representations $\sigma \otimes u$ and 
$\sigma \otimes u \chi_c$ are equivalent, and sometimes it is hard to distinguish them.
Fortunately, the equivalences of categories from Corollary \ref{cor:8.4} are essentially the 
same for $\sigma \otimes u$ and $\sigma \otimes u \chi_c$. To make that precise, we assume that
in \eqref{eq:5.5} the choices are made so that
\begin{equation}\label{eq:8.13}
\rho_{r, \sigma_{\chi_c u}} = \phi_{\sigma,\chi_c} \rho_{r,\sigma \otimes u} \phi_{\sigma,\chi_c}^{-1}.
\end{equation}

\begin{lem}\label{lem:8.1}
Let $\chi_c \in X_\nr (M,\sigma)$.
\enuma{
\item The algebras $\mh H (\tilde{\mc R}_u, W(M,\mc O)_{\sigma \otimes u}, k^u, \natural_u )$
and $\mh H (\tilde{\mc R}_{\chi_c u}, W(M,\mc O)_{\sigma \otimes \chi_c u}, k^{\chi_c u}, 
\natural_{\chi_c u} )$ are equal.
\item Let $V \in \mh H (\tilde{\mc R}_u, W(M,\mc O)_{\sigma \otimes u}, k^u, \natural_u ) 
-\Modf{a_M^*}$. Then its image in \\ $\End_G (I_P^G (E_B))-\fMod$ via Corollary 
\ref{cor:8.4} coincides with the image of $V$ as $\mh H (\tilde{\mc R}_{\chi_c u}, 
W(M,\mc O)_{\sigma \otimes \chi_c u}, k^{\chi_c u}, \natural_{\chi_c u} )$-module in 
$\End_G (I_P^G (E_B))-\fMod$, obtained from Corollary \ref{cor:8.4} for $\chi_c u$. 
}
\end{lem}
\begin{proof}
(a) The $M$-representations $\sigma \otimes u$ and $\sigma \otimes u \chi_c$ are the same in
$\Irr (M)$ and $W(M,\mc O)$ acts on that set, so $W(M,\mc O)_{\sigma \otimes u} = 
W(M,\mc O)_{\sigma \otimes \chi_c u}$. By the $X_\nr (M,\sigma)$-invariance of $\mu^{M_\alpha}$,
$\tilde{\mc R}_u = \tilde{\mc R}_{\chi_c u}$ and $k^u = k^{\chi_c u}$.

Conjugation with $\phi_{\chi_c}$ provides an algebra isomorphism
\[
\mr{Ad}(\phi_{\chi_c}) : 1_u \End_G (I_P^G (E_B))_U^\an 1_u 
\to 1_{\chi_c u} \End_G (I_P^G (E_B))_U^\an 1_{\chi_c u} ,
\]
and similarly with meromorphic functions on $U$. By \eqref{eq:8.13} this isomorphism sends 
$\mc T_w^u \; (w \in W(M,\mc O)_{\sigma \otimes u})$ to
$\mc T_w^{\chi_c u}$, so the 2-cocycles $\natural_u$ and $\natural_{\chi_c u}$ of 
$W(M,\mc O)_{\sigma \otimes u}$ coincide. Furthermore Ad$(\phi_{\chi_c})$ sends $f \in 
C^\an (U_u)$ to $f (\chi_c^{-1} \cdot)$, which equals\\
$(\exp_u \circ \exp_{\chi_c u}^{-1})^* f$.
Thus Proposition \ref{prop:8.3} gives a commutative diagram
\begin{equation}\label{eq:8.12}
\begin{array}{ccc}
\mh H (\tilde{\mc R}_u, W(M,\mc O)_{\sigma \otimes u}, k^u, \natural_u ) & \to &
1_u \End_G (I_P^G (E_B))_U^\an 1_u \\ 
|| & & \downarrow \mr{Ad} (\phi_{\chi_c}) \\
\mh H (\tilde{\mc R}_{\chi_c u}, W(M,\mc O)_{\sigma \otimes \chi_c u}, k^{\chi_c u}, 
\natural_{\chi_c u} ) & \to & 1_{\chi_c u} \End_G (I_P^G (E_B))_U^\an 1_{\chi_c u}
\end{array} .
\end{equation}
(b) Let us retrace what happens in \eqref{eq:8.8}. First we translate $V$ to a module for
$1_u \End_G (I_P^G (E_B))_U^\an 1_u$ (on the same vector space). Then we apply the Morita
equivalence in Lemma \ref{lem:7.2}. That yields a module
\[
V' := \End_G (I_P^G (E_B))_U^\an \underset{1_u \End_G (I_P^G (E_B))_U^\an 1_u}{\otimes} V
= \bigoplus_{u' \in W(M,\sigma,X_\nr (M)) u} V_{u'} ,
\]
where $V_{u'} \in 1_{u'} \End_G (I_P^G (E_B))_U^\an 1_{u'} -\fMod$ has the same dimension as $V$.
This $V'$ is also the $\End_G (I_P^G (E_B))$-module that results from \eqref{eq:8.8}.
From the proof of Lemma \ref{lem:7.2} we see that $V_{u'} = V \mc T_{w'}$ for a $w' \in 
W(M,\sigma,X_\nr (M))$ which satisfies $w'u' = u$ and whose length is minimal for that
property. In particular $V_{u \chi_c} = V \phi_{\chi_c}^{-1}$. 

Hence the module of $1_{\chi_c u} \End_G (I_P^G (E_B))_U^\an 1_{\chi_c u}$ obtained from $V'$
via Lemma \ref{lem:7.2} is Ad$(\phi_{\chi_c}^{-1})^*$. In view of the commutative diagram
\eqref{eq:8.12}, this procedure recovers $V$ as module of $\mh H (\tilde{\mc R}_{\chi_c u}, 
W(M,\mc O)_{\sigma \otimes \chi_c u}, k^{\chi_c u}, \natural_{\chi_c u} )$. Then Corollary
\ref{cor:8.4} for $\chi_c u$ implies that $V$ in the latter sense has the same image in
$\End_G (I_P^G (E_B)) -\fMod$ as $V$ in the former sense. 
\end{proof}

A weaker version of Lemma \ref{lem:8.1} holds for all points in $W(M,\sigma,X_\nr (M)) u$.

\begin{lem}\label{lem:8.6}
Let $w$ be an element of $W(M,\sigma,X_\nr (M))$ which is of minimal length in the coset
$w W(M,\sigma,X_\nr (M))_u$.
\enuma{
\item Conjugation by $\mc T_w$ gives rise to an algebra isomorphism
\[
\mr{Ad}(\mc T_w) : \mh H (\tilde{\mc R}_u, W(M,\mc O)_{\sigma \otimes u}, k^u, \natural_u ) \to 
\mh H (\tilde{\mc R}_{w(u)}, W(M,\mc O)_{\sigma \otimes w(u)}, k^{w(u)}, \natural_{w(u)} ) 
\]
with $\mr{Ad}(\mc T_w)(\C N_v) = \C N_{w v w^{-1}}$ for $v \in W(M,\mc O)_{\sigma \otimes u}$
and $\mr{Ad}(\mc T_w)(N_v) = N_{w v w^{-1}}$ for $v \in W(\Sigma_{\sigma \otimes u})$. 
\item Let $V \in \mh H (\tilde{\mc R}_u, W(M,\mc O)_{\sigma \otimes u}, k^u, \natural_u ) -\fMod$
with Wt$(V) \subset a_M^*$. Then $V$ and 
\[
\mr{Ad}(\mc T_w^{-1})^* V \in \mh H (\tilde{\mc R}_{w(u)}, W(M,\mc O)_{\sigma \otimes w(u)},
k^{w(u)}, \natural_{w(u)} ) -\fMod
\]
have the same image in $\End_G (I_P^G (E_B)) -\fMod$ (via Corollary \ref{cor:8.4} for,
respectively, $\sigma \otimes u$ and $\sigma \otimes w(u)$).
}
\end{lem}
\begin{proof}
(a) From the proof of Lemma \ref{lem:7.2} we see that $\mc T_w 1_u \in \End_G (I_P^G (E_B))_U^\an$
and $\mc T_w 1_u \mc T_w^{-1} = 1_{w(u)}$. Therefore conjugation by $\mc T_w$ gives an
algebra isomorphism
\[
\mr{Ad}(\mc T_w) : 1_u \End_G (I_P^G (E_B))_U^\an 1_u 
\to 1_{w(u)} \End_G (I_P^G (E_B))_U^\an 1_{w(u)} ,
\]
which on $C^\an (U_u)$ is just $f \mapsto f \circ w^{-1}$. The elements
$\mc T_w \mc T_v^u \mc T_w^{-1}$ with $v \in W(\Sigma_{\sigma \otimes u})$ satisfy
the same multiplication relations as the elements $\mc T_{w v w^{-1}}^u$ and they
have the same specialization at $w(u)$ (namely the identity on $I_P^G (E)$, by
Lemma \ref{lem:6.1}.a). Therefore $\mr{Ad}(\mc T_w)(\mc T_v^u) = \mc T_{w v w^{-1}}^{w(u)}$.
The same applies to the $A_r^u$ with $r \in R(\sigma \otimes u)$, but for those
we can only say that $\mr{sp}_{w(u)} \mc T_w A_r^u \mc T_w^{-1}$ and
$\mr{sp}_{w(u)} A_{w r w^{-1}}$ are equal up to a scalar factor. Hence
$\mr{Ad}(\mc T_w)(\mc A_r^u)$ is a scalar multiple of $\mc A_{w r w^{-1}}^{w(u)}$.

Via Proposition \ref{prop:8.3} $\mr{Ad}(\mc T_w)$ becomes an algebra isomorphism
\[
\mh H (\tilde{\mc R}_u, W(M,\mc O)_{\sigma \otimes u}, k^u, \natural_u )_{\log (U_u)}^\an \to 
\mh H (\tilde{\mc R}_{w(u)}, W(M,\mc O)_{\sigma \otimes w(u)}, k^{w(u)}, 
\natural_{w(u)} )_{w(\log (U_u))}^\an .
\]
It restricts to $f \mapsto f \circ w$ on $C^\an (\log (U_u))$ and sends $N_v$  to $N_{w v w^{-1}}$
for $v \in W(\Sigma_{\sigma \otimes u})$, and to a scalar multiple of that for $v \in 
W(M,\mc O)_{\sigma \otimes u}$. Now it is clear that $\mr{Ad}(\mc T_w)$ restricts to the 
required isomorphism between (twisted) graded Hecke algebras.\\
(b) This can be shown in the same way as Lemma \ref{lem:8.1}.b.
\end{proof}

Corollary \ref{cor:8.4} tells us that there is a surjection from the union of the sets
\[
\{ \pi \in \Irr \big( \mh H (\tilde{\mc R}_u, W(M,\mc O)_{\sigma \otimes u}, k^u, \natural_u ) \big)
: \mr{Wt}(\pi) \subset a_M^* \}
\]
with $u \in X_\unr (M)$ to $\Irr \big( \End_G (I_P^G (E_B)) \big)$. For $u$ and $u'$ in
different $W(M,\sigma,X_\nr (M))$-orbits, the images in  $\Irr \big( \End_G (I_P^G (E_B)) \big)$
are disjoint. For $u$ and $u'$ in the same $W(M,\sigma,X_\nr (M))$-orbit, Lemma \ref{lem:8.6}.b
tells us precisely which modules of 
\[
\mh H (\tilde{\mc R}_u, W(M,\mc O)_{\sigma \otimes u}, k^u, \natural_u ) \quad \text{and} 
\quad \mh H (\tilde{\mc R}_{u'}, W(M,\mc O)_{\sigma \otimes u'}, k^{u'}, \natural_{u'} )
\]
have the same image -- the relation between them comes from an element $w \in W(M,\sigma,X_\nr (M))$
with $w(u) = u'$. If we agree that $W(M,\sigma, X_\nr (M))_u$ acts trivially on
$\Irr \big( \mh H (\tilde{\mc R}_u, W(M,\mc O)_{\sigma \otimes u}, k^u, \natural_u ) \big)$,
it does not matter which $w$ with $w(u) = u'$ we pick. Thus we obtain a bijection
\begin{align}\nonumber
\bigsqcup\nolimits_{u \in X_\unr (M)} & \{ \pi \in \Irr \big( \mh H (\tilde{\mc R}_u, 
W(M,\mc O)_{\sigma \otimes u}, k^u, \natural_u ) \big) : \mr{Wt}(\pi) \subset a_M^* \} 
/ W(M,\sigma,X_\nr (M)) \\
\label{eq:8.14} & \longrightarrow \Irr \big( \End_G (I_P^G (E_B)) \big) .
\end{align}
Here the group action of $W(M,\sigma,X_\nr (M))$ on the disjoint union comes from the relations
described in Lemma \ref{lem:8.6}.

\subsection{Comparison by setting the $q$-parameters to 1} \

It is interesting to investigate what happens when in Corollary \ref{cor:8.4} we replace the
parameter function $k^u$ by 0. It is known that the analogous operation for affine Hecke algebras 
gives rise to a bijection on the level of irreducible representations \cite{SolAHA,SolHecke}.
Replacing all the $k^u$ by 0
corresponds to manually setting all the parameters $q_\alpha$ and $q_{\alpha*}$ 
to 1. In view of Corollary \ref{cor:5.6}, that transforms $\End_G (I_P^G (E_B))$ into
$\C [X_\nr (M)] \rtimes \C [W(M,\sigma,X_\nr (M)),\natural]$. Therefore we start by analysing
the irreducible representations of that simpler crossed product algebra.

\begin{lem}\label{lem:8.7}
There is a canonical bijection 
\begin{multline*}
\bigsqcup\nolimits_{u \in X_\unr (M)} \{ \pi \in \Irr \big( \C [X_\nr (M)] \rtimes 
\C [W(M,\sigma,X_\nr (M))_u,\natural] \big) : \mr{Wt}(\pi) \subset u X_\nr^+ (M) \} \\
\big/ \: W(M,\sigma,X_\nr (M)) \\ \longrightarrow \Irr \big( \C [X_\nr (M)] \rtimes 
\C [W(M,\sigma,X_\nr (M)),\natural] \big) .
\end{multline*}
Here $w \in W(M,\sigma,X_\nr (M))$ acts on the disjoint union by pullback along the algebra
isomorphism $\mr{Ad}(N_w^{-1})$ :
\[
\C [X_\nr (M)] \rtimes \C [W(M,\sigma,X_\nr (M))_{w(u)} ,\natural]
\to \C [X_\nr (M)] \rtimes \C [W(M,\sigma,X_\nr (M))_u ,\natural] .
\]
\end{lem}
\begin{proof}
Choose a central extension $\Gamma$ of $W(M,\sigma,X_\nr (M))$ such that $\natural$
becomes trivial in $H^2 (\Gamma,\C^\times)$. Then there exists a central idempotent\\
$p_\natural \in \C [ \ker (\Gamma \to W(M,\sigma,X_\nr (M))) ]$ such that
\begin{equation}\label{eq:8.16}
\C [W(M,\sigma,X_\nr (M)),\natural] \cong p_\natural \C [\Gamma] .
\end{equation}
The isomorphism sends $\C N_w$ to $\C p_\natural N_{\tilde w}$, for any lift
$\tilde w \in \Gamma$ of $w \in W(M,\sigma,X_\nr (M))$.

Lift the $W(M,\sigma,X_\nr (M))$-action on $X_\nr (M)$ to $\Gamma$ and note that
\eqref{eq:8.16} gives rise to a bijection
\begin{equation}\label{eq:8.15}
\Irr \big( \C [X_\nr (M)] \rtimes \C [W(M,\sigma,X_\nr (M)),\natural] \big) \longleftrightarrow
\{ V \in \big( \C [X_\nr (M)] \rtimes \C [\Gamma] \big) : p_\natural V \neq 0 \} . 
\end{equation}
By Clifford theory every irreducible representation $\pi$ of $\C [X_\nr (M)] \rtimes \C [\Gamma]$
is of the form
\begin{equation}\label{eq:8.20}
\mr{ind}_{\C [X_\nr (M)] \rtimes \C [\Gamma_\chi]}^{\C [X_\nr (M)] \rtimes \C [\Gamma]}
(\chi \otimes \rho) ,
\end{equation}
where $\chi \in X_\nr (M)$ and $\rho \in \Irr (\Gamma_\chi)$. Moreover the pair $(\chi,\rho)$ is
determined by $\pi$, uniquely up to the $\Gamma$-action 
\[
\gamma (\chi,\rho) = (\gamma (\chi), \mr{Ad} (N_\gamma^{-1})^* \rho) .
\] 
When $u$ is the unitary part of $\chi$, $\Gamma_u \supset \Gamma_\chi$. Again by Clifford theory, 
every irreducible representation of $\C [X_\nr (M)] \rtimes \Gamma_u$ is of the form 
\[
\mr{ind}_{\C [X_\nr (M)] \rtimes \C [\Gamma_{u,\chi}]}^{\C [X_\nr (M)] \rtimes \C [\Gamma_u]}
(\chi \otimes \rho) ,
\]
where $(\chi,\rho)$ is unique up to the action of $\Gamma_u$. Hence there is a canonical bijection
\[
\bigsqcup_{\chi \in u X_\nr^+ (M)} \Irr (\Gamma_\chi) \; \big/ \Gamma_u \longrightarrow
\{ \pi \in \Irr (\C [X_\nr (M)] \rtimes \Gamma_u) : \mr{Wt}(\pi) \subset u X_\nr^+ (M) \} .
\]
Comparing this with Clifford theory for $\C [X_\nr (M)] \rtimes \C [\Gamma]$, we deduce a
canonical bijection
\[
\bigsqcup_{u \in X_\unr (M)} \hspace{-3mm} \{ \pi \in \Irr \big( \C [X_\nr (M)] \rtimes 
\C [\Gamma_u] \big) : \mr{Wt}(\pi) \subset u X_\nr^+ (M) \} \big/ \Gamma
\to \Irr (\C [X_\nr (M)] \rtimes \Gamma ) . 
\]
Now we restrict on both sides to the subsets that are not annihilated by $p_\natural$ and we
use \eqref{eq:8.15}. 
\end{proof}

It is possible to vary on Lemma \ref{lem:8.7} by taking on the left hand side a variety in
which $u X_\nr^+ (M)$ embeds, for instance $a_M^* \otimes_\R \C$ with as embedding
$\exp_u^{-1} : u X_\nr^+ (M) \to a_M^*$. Then Lemma \ref{lem:8.7} provides a canonical bijection
between\\ $\Irr \big( \C [X_\nr (M)] \rtimes \C [W(M,\sigma,X_\nr (M)),\natural] \big)$ and
\begin{align}\nonumber
\bigsqcup\nolimits_{u \in X_\unr (M)} & \{ \pi \in \Irr \big( \C [a_M^* \otimes_\R \C] \rtimes 
\C [W(M,\sigma,X_\nr (M))_u,\natural] \big) : \mr{Wt}(\pi) \subset a_M^* \} \\
\label{eq:8.17} & \big/ \: W(M,\sigma,X_\nr (M)) .
\end{align}
With Lemma \ref{lem:8.5} we can identify \eqref{eq:8.17} with
\begin{align}\nonumber
\bigsqcup\nolimits_{u \in X_\unr (M)} & \{ \pi \in \Irr \big( \C [a_M^* \otimes_\R \C] \rtimes 
\C [W(M,\mc O)_{\sigma \otimes u},\natural_u] \big) : \mr{Wt}(\pi) \subset a_M^* \} \\
\label{eq:8.18} & \big/ \: W(M,\sigma,X_\nr (M)) .
\end{align}
Notice that
\[
\C [a_M^* \otimes_\R \C] \rtimes \C [W(M,\mc O)_{\sigma \otimes u},\natural_u] =
\mh H (\tilde{\mc R}_u , W(M,\mc O)_{\sigma \otimes u}, 0 ,\natural_u ) ,
\]
a (twisted) graded Hecke algebra with all parameters $k_\alpha$ equal to 0.
In \cite{SolHomGHA,SolAHA} we studied how graded Hecke algebras behave under deformations
of the parameters $k_\alpha$. To formulate that properly, we recall some results from the
representation theory of graded Hecke algebras.

For $P \subset \Delta_{\sigma \otimes u}$, we denote the Weyl group generated by the 
reflections $s_\alpha$ with $\alpha \in P$ by $W_P$. The set $\C[a_M^* \otimes_\R \C] \C[W_P]$
constitutes a parabolic subalgebra $\mh H (P,k)$ of $\mh H (\tilde{\mc R}_u, 
W(\Sigma_{\sigma \otimes u}),k)$. As algebra, it decomposes as a tensor product
\[
\C[ \mr{span}_\C (P)] \C [W_P] \otimes \C [ (a_M^* \otimes_\R \C)^{\perp P} ],
\]
where the subscript $\perp P$ denotes the subspace orthogonal to the set of coroots $P^\vee$.

The Langlands classification, proven for graded Hecke algebras in \cite{Eve}, expresses 
irreducible representations in terms of parabolic subalgebras, tempered representations and 
parabolic induction. See Definition \ref{def:temp} for temperedness. We need an extension that 
includes R-groups like $R(\sigma \otimes u)$. Such a version was proven
for affine Hecke algebras in \cite[\S 2.2]{SolAHA}. In view of Lusztig's second reduction
theorem \cite[\S 9]{Lus-Gr}, generalized in \cite[Corollary 2.1.5]{SolAHA}, that extended
Langlands classification also applies to graded Hecke algebras.

\begin{prop} \cite[Corollary 2.2.5]{SolAHA} \label{prop:8.9} \\
Let $\Gamma$ be a finite group acting linearly in $a_M^*$, stabilizing $\Sigma_{\sigma \otimes u}$
and $\Delta_{\sigma \otimes u}$. 
\enuma{
\item Suppose that the following data are given: $P \subset \Delta_{\sigma \otimes u}$,
$t \in (a_M^*)^{\perp P}$ which is strictly positive with respect to $\Delta_{\sigma \otimes u}
\setminus P$, a tempered $\tau \in \Irr (\mh H (P,k))$, an irreducible representation
$\rho$ of $\C [\Gamma_{P,\tau,t},\kappa]$ (where the 2-cocycle $\kappa$ is determined by
the action of $\Gamma_{P,\tau,t}$ on $\tau$). Then the 
$\mh H (\tilde{\mc R}_u , W(\Sigma_{\sigma \otimes u}) \Gamma, k)$-representation
\[
\mr{ind}_{\mh H (P,k) \rtimes \C [\Gamma_{P,\tau,t}]}^{\mh H (\tilde{\mc R}_u , 
W(\Sigma_{\sigma \otimes u}) \Gamma, k)} ((\tau \otimes t) \otimes \rho)
\]
has a unique irreducible quotient. It is called the Langlands quotient and we denote it by an L.
\item For every $\pi \in \Irr (\mh H (\tilde{\mc R}_u , W(\Sigma_{\sigma \otimes u}) \Gamma, k))$
there exist data as in part (a), unique up to the action of $\Gamma$, such that 
\[
\pi \cong L \Big( \mr{ind}_{\mh H (P,k) \rtimes \C [\Gamma_{P,\tau,t}]}^{\mh H (\tilde{\mc R}_u , 
W(\Sigma_{\sigma \otimes u}) \Gamma, k)} ((\sigma \otimes t) \otimes \rho) \Big) .
\]
}
\end{prop}
In Proposition \ref{prop:8.9}.a we can combine $\tau$ and $\rho$ in
\begin{equation}\label{eq:8.19}
\tau' := \mr{ind}_{\mh H(P,k) \rtimes \Gamma_{P,\tau,t}}^{\mh H (P,k) \rtimes \Gamma_{P,t}}
(\tau \otimes \rho),
\end{equation}
an irreducible tempered representation such that
\[
\mr{ind}_{\mh H (P,k) \rtimes \C [\Gamma_{P,t}]}^{\mh H (\tilde{\mc R}_u 
W(\Sigma_{\sigma \otimes u}) \Gamma, k)} (\tau' \otimes t) \cong
\mr{ind}_{\mh H (P,k) \rtimes \C [\Gamma_{P,\tau,t}]}^{\mh H (\tilde{\mc R}_u 
W(\Sigma_{\sigma \otimes u}) \Gamma, k)} ((\tau \otimes t) \otimes \rho) .
\]
Then Proposition \ref{prop:8.9} holds also with the alternative data $P,\tau',t$.

\begin{thm}\label{thm:8.8}
There exists a bijection
\begin{multline*}
\zeta_u : \{ V \in \Irr \big( \mh H (\tilde{\mc R}_u , W(M,\mc O)_{\sigma \otimes u}, k^u ,\natural_u ) 
\big) : \mr{Wt}(V) \subset a_M^* \} \\
\longrightarrow \{ V \in \Irr \big( \C [a_M^* \otimes_\R \C] \rtimes 
\C [W(M,\mc O)_{\sigma \otimes u},\natural_u ] \big) : \mr{Wt}(V) \subset a_M^* \}
\end{multline*}
such that 
\begin{itemize}
\item $\pi$ is tempered if and only if $\zeta_u (\pi)$ is tempered,
\item $\zeta_u$ is compatible with the Langlands classification from Proposition \ref{prop:8.9}.
\end{itemize}
\end{thm}
\begin{proof}
First we get rid of the 2-cocycle $\natural_u$. Choose a central extension 
\[
1 \to Z(\sigma \otimes u) \to \Gamma \to R(\sigma \otimes u) \to 1
\]
such that $\natural_u$ becomes trivial in $H^2 (\Gamma,\C^\times)$. Let $p_{\natural_u} \in 
\C[Z(\sigma \otimes u)]$ be a central idempotent such that
\[
p_{\natural_u} \C [Z(\sigma \otimes u)] \cong \C [R(\sigma \otimes u),\natural_u].
\]
For both $k = k^u$ and $k = 0$ that gives a bijection
\begin{multline*}
\Irr \big( \mh H (\tilde{\mc R}_u , W(M,\mc O)_{\sigma \otimes u}, k ,\natural_u ) \big) \to 
\{ V \in \Irr \big( \mh H (\tilde{\mc R}_u , W(\Sigma_{\sigma \otimes u}) \Gamma, k ) \big) :
p_{\natural_u} V \neq 0 \}.
\end{multline*}
Hence it suffices to construct the required bijection with $\Gamma$ instead of $R(\sigma \otimes u)$,
provided that it does not change the $Z(\sigma \otimes u)$-characters of representations.

Consider $\pi \in \Irr \big( \mh H (\tilde{\mc R}_u , W(\Sigma_{\sigma \otimes u})\Gamma, k ) \big)$
with Wt$(\pi) \subset a_M^*$. By Proposition \ref{prop:8.9}, with the modified data from 
\eqref{eq:8.19}, we have
\begin{equation}\label{eq:8.22}
\pi \cong L \Big( \mr{ind}_{\mh H (P,k) \rtimes \C [\Gamma_{P,t}]}^{\mh H (\tilde{\mc R}_u ,
W(\Sigma_{\sigma \otimes u}) \Gamma, k)} (\tau' \otimes t) \Big) ,
\end{equation}
for data $(P,\tau',t)$ that are unique up to the $\Gamma$-action. Since both Wt$(\pi)$ and
$t$ lie in $a_M^*$ and Wt$(\tau') + t$ consists of weights of $\pi$ \cite{Eve}, we must 
have Wt$(\tau') \subset a_M^*$. By \cite[Theorem 6.5.c]{SolHomGHA} the restrictions to
$\C [W_P \Gamma_{P,t}]$ of the set
\begin{equation}\label{eq:8.21}
\{ V \in \Irr \big( \mh H (P,k)\rtimes \C[\Gamma_{P,t}] \big) \text{ is tempered and Wt}(V) 
\subset a_M^* \}
\end{equation}
form a $\Q$-basis of the representation ring of $W_P \Gamma_{P,t}$. As $Z(\sigma \otimes u)
\subset \Gamma_{P,t}$, we can find a bijection $\zeta_{P,t}$ from \eqref{eq:8.21} to
$\Irr (W_P \Gamma_{P,t})$, such that $\zeta_{P,t}(V)$ occurs in $V |_{\C [W_P \Gamma_{P,t}]}$.
We regard $\zeta_{P,t}(V)$ as a $\C[a_M^* \otimes_\R \C] \rtimes \C [W_P \Gamma_{P,t}]$-representation
on which $\C[a_M^* \otimes_\R \C]$ acts via evaluation at $0 \in a_M^*$.

Now we define
\begin{equation}\label{eq:8.23}
\zeta_u (\pi) := L \Big( \mr{ind}_{\mh H (P,0) \rtimes \C [\Gamma_{P,t}]}^{\mh H 
(\tilde{\mc R}_u , W(\Sigma_{\sigma \otimes u}) \Gamma, 0)} (\zeta_{P,t}(\tau') \otimes t) \Big) .
\end{equation}
By Proposition \ref{prop:8.9}, this is a well-defined irreducible representation of \\
$\mh H (\tilde{\mc R}_u, W(\Sigma_{\sigma \otimes u}) \Gamma, 0)$. The only weight of
$\zeta_{P,t}(\tau')$ is 0, so by \cite[Theorem 6.4]{BaMo2}
\[
\mr{Wt}(\zeta_u (\pi)) \subset W(\Sigma_{\sigma_\otimes u}) \Gamma t \subset a_M^* .
\] 
The analogy between \eqref{eq:8.22} and \eqref{eq:8.23} is our compatibility with the 
extended Langlands classification.
The construction of $\zeta_u$ also works in the other direction (with $\zeta_{P,t}^{-1}$),
so it is bijective. Since $\zeta_u$ is built from operations that do not change anything
in $Z(\sigma \otimes u)$, it preserves the $Z(\sigma \otimes u)$-characters of
representations.
\end{proof}

A canonical choice for the bijection $\zeta_{P,t}$ in the above proof is provided by 
\cite[Theorem 6.2]{SolHecke}. By that and \cite[Proposition 6.10]{SolHecke}, the map $\zeta_u$ is 
canonical once the 2-cocycle $\natural_u$ has been fixed. But, there is one caveat.
Namely, \cite{SolHecke} deals with all parameter functions $k : R \to \R_{\geq 0}$, except for root 
systems of type $F_4$, for those only certain positive parameter functions are analysed. 
In the sequel \cite{SolParam} we show that all the parameter functions $k_u$ occurring in this paper
are among those investigated in \cite{SolHecke}.

For any $w \in W(M,\sigma,X_\nr (M))$, conjugation with $\mc T_w$ in $\End_G (I_P^G (E_B)) \otimes_B K(B)$
defines an isomorphism
\[
\mr{Ad}(\mc T_w) : \C [W(M,\sigma,X_\nr (M))_u,\natural] \to \C[W(M,\sigma,X_\nr (M))_{w(u)},\natural] .
\]
Recall from Lemma \ref{lem:8.5} that $\C [W(M,\mc O)_{\sigma \otimes u},\natural_u]$ is embedded in\\
$\C [W(M,\sigma,X_\nr (M))_u,\natural]$ as the span of $W(M,\sigma,X_\nr (M))_u$. Thus
Ad$(\mc T_w)$ can be transferred to an algebra isomorphism
\[
\mr{Ad}(N_w) : \C [W(M,\mc O)_{\sigma \otimes u},\natural_u] \to 
\C [W(M,\mc O)_{\sigma \otimes w(u)},\natural_{w(u)}] ,
\]
which sends $\C N_{\Omega_u (v)}$ to $\C N_{\Omega_{w(u)}(w v w^{-1})}$. 
We denote the differential of $w : U_u \to U_{w(u)}$ also by $w$, but now from $a_M^* \otimes_\R \C$
to itself. For $f \in \C[a_M^* \otimes_\R \C]$ we define Ad$(N_w)f = f \circ w^{-1}$.
These instances of Ad$(N_w)$ combine to an algebra isomorphism
\[
\mr{Ad}(N_w) : \mh H (\tilde{\mc R}_u , W(M,\mc O)_{\sigma \otimes u}, k^u ,\natural_u ) \to
\mh H (\tilde{\mc R}_{w(u)} , W(M,\mc O)_{\sigma \otimes w(u)}, k^{w(u)} ,\natural_{w(u)} ) .
\]
For $w \in W(M,\sigma,X_\nr (M))_u$, this is just the inner automorphism Ad$(N_{\Omega_u (w)})$
of $\mh H (\tilde{\mc R}_u , W(M,\mc O)_{\sigma \otimes u}, k ,\natural_u )$. For other $w \in 
W(M,\sigma,X_\nr (M))$ the notation Ad$(N_w)$ is only suggestive, because we have not defined
an element $N_w$.

With all that set, we define a bijection
\begin{equation}\label{eq:8.24}
\mr{Ad}(N_w^{-1})^* : \Irr \big( \mh H (\tilde{\mc R}_u , W(M,\mc O)_{\sigma \otimes u}, k^u ,
\natural_u ) \big) \to 
\mh H (\tilde{\mc R}_{u'} , W(M,\mc O)_{\sigma \otimes u'}, k^{u'} ,\natural_{u'} ) ,
\end{equation}
for any $w \in W(M,\sigma,X_\nr (M))$ such that $w(u) = u'$. Since inner automorphisms act trivially
on the set of irreducible representation of an algebra, \eqref{eq:8.24} does not depend on the
choice of $w$ with $w(u) = u'$. 
Clearly, the construction of Ad$(N_w)$ also works with $k=0$ instead of $k^u$ and $k^{w(u)}$. 

However, because of the lack of canonicity of $\zeta_u$ it is not clear whether 
\[
\zeta_{w(u)} \circ \mr{Ad}(N_w^{-1})^* = \mr{Ad}(N_w^{-1})^* \circ \zeta_u .
\]
To achieve that desirable equality we can enforce it in the following way. For every 
$W(M,\sigma,X_\nr (M))$-orbit in $X_\unr (M)$ we fix one representative $u$. Then we define
\begin{multline}\label{eq:8.25}
\zeta_{w(u)} := \mr{Ad}(N_w^{-1})^* \circ \zeta_u \circ \mr{Ad}(N_w)^* : \\
\{ V \in \Irr \big( \mh H (\tilde{\mc R}_{w(u)} , W(M,\mc O)_{\sigma \otimes w(u)}, k^{w(u)} ,
\natural_{w(u)} ) \big) : \mr{Wt}(V) \subset a_M^* \} \\
\longrightarrow \{ V \in \Irr \big( \C [a_M^* \otimes_\R \C] \rtimes 
\C [W(M,\mc O)_{\sigma \otimes w(u)},\natural_{w(u)} ] \big) : \mr{Wt}(V) \subset a_M^* \} .
\end{multline}
When $w$ is of minimal length in $w W(M,\sigma,X_\nr (M))_u$, it sends $\Delta_{\sigma \otimes u}$
to $\Delta_{\sigma \otimes w(u)}$. Then $w (a_M^*)_u^- = (a_M^*)_{w(u)}^-$, so Ad$(N_w^{-1})^*$
preserves temperedness. That particular Ad$(N_w^{-1})^*$ also maps a Langlands datum $(P,\tau,t')$ 
(as in Proposition \ref{prop:8.9}) to another Langlands datum, so it respects the compatibility 
with the Langlands classification from \eqref{eq:8.22} and \eqref{eq:8.23}.

As $\zeta_{u'}$, as defined in \eqref{eq:8.25}, does not depend on the 
choice of $w$ with $w(u) = u'$, this means that $\zeta_{u'}$ always satisfies the requirements
of Theorem \ref{thm:8.8}.

\begin{cor}\label{cor:8.10}
There exists a bijection
\[
\zeta : \Irr (\End_G (I_P^G (E_B)) \to 
\Irr \big( \C[X_\nr (M)] \rtimes \C[W(M,\sigma,X_\nr (M))],\natural] \big)
\]
such that Wt$(\pi) \subset W(M,\sigma,X_\nr (M)) uX_\nr^+ (M)$ if and only if\\
Wt$(\zeta (\pi)) \subset W(M,\sigma,X_\nr (M)) uX_\nr^+ (M)$.
\end{cor}
\begin{proof}
With \eqref{eq:8.14} we decompose $\Irr \big( \End_G (I_P^G (E_B))\big)$ as a disjoint union over\\
$X_\unr (M)$, modulo an action of $W(M,\sigma,X_\nr (M))$. Notice that the 
$W(M,\sigma,X_\nr (M))$-actions in \eqref{eq:8.14} and \eqref{eq:8.24} agree, because both are
induced by Ad$(\mc T_w)$. By Theorem \ref{thm:8.8} the terms in the disjoint union in \eqref{eq:8.5}
are in bijection with
\[
\big\{ V \in \Irr \big( \C [a_M^* \otimes_\R \C] \rtimes 
\C [W(M,\mc O)_{\sigma \otimes u},\natural_u ] \big) : \mr{Wt}(V) \subset a_M^* \big\} .
\]
By \eqref{eq:8.25} the bijections from Theorem \ref{thm:8.8} are $W(M,\sigma,X_\nr (M))$-equivariant.
That brings us to the left hand side of Lemma \ref{lem:8.7}. Applying that lemma, we finally obtain
the required bijection. 
\end{proof}

\section{Temperedness}
\label{sec:temp}

Like in \cite{Hei3,SolComp}, we want to show that the equivalence of categories
\[
\mc E : \Rep (G)^{\mf s} \to \End_G (I_P^G (E_B))\text{-Mod}
\]
preserves temperedness. At the moment we have not even defined temperedness for representations
of $\End_G (I_P^G (E_B))$, so we address that first. We also consider (essentially) discrete series representations of Hecke algebras, which correspond, under some extra conditions, to (essentially) 
square-integrable representations in Rep$(G)^{\mf s}$.

Our definition will mimick that for affine Hecke algebras \cite[\S 2]{Opd-Sp}. It depends on the
choice of the parabolic subgroup $P$ with Levi factor $M$. Before we just picked one, in this
section we have to be more careful. 

Recall that $\mc A_0$ is a maximal $F$-split torus of $\mc G$, contained in $\mc M$. 
By the standard theory of reductive groups \cite{Spr} there are (non-reduced) root systems
$\Sigma (\mc M,\mc A_0)$ and $\Sigma (\mc G,\mc A_0)$ in $X^* (\mc A_0)$. Further
$\Sigma (\mc G,\mc M) \cup \{0\}$ is the image of $\Sigma (\mc G,\mc A_0) \cup \{0\}$ in
the quotient $X^* (\mc A_0) \otimes_\Z \R / \R \Sigma (\mc M,\mc A_0)$. 

The root system $\Sigma_{\mc O,\mu}$ is contained in $\Sigma_\red (A_M) \subset \Sigma (\mc G,\mc M)$.
We write
\[
\Sigma_{\mc O,\mu} \tilde{+} \Sigma (\mc M,\mc A_0) = 
\{ \alpha \in \Sigma (\mc G,\mc A_0) : \alpha + \R \Sigma (\mc M,\mc A_0) \in \R \Sigma_{\mc O,\mu} \} ,
\]
a parabolic root subsystem of $\Sigma (\mc G,\mc A_0)$.

\begin{lem}\label{lem:9.1}
There exists a basis $\Delta$ of $\Sigma (\mc G,\mc A_0)$ which contains a basis $\Delta_{\mc M}$
of $\Sigma (\mc M,\mc A_0)$ and a basis $\Delta_{\mc M}^{\mc O}$ of 
$\Sigma_{\mc O,\mu} \tilde{+} \Sigma (\mc M,\mc A_0)$.
\end{lem}
\begin{proof}
Choose a linear function $t$ on $X^* (\mc A_0) \otimes_\Z \R$ such that, for all
$\alpha \in \Sigma (\mc M,\mc A_0), \beta \in \Sigma_{\mc O,\mu} \tilde{+} \Sigma (\mc M,\mc A_0) 
\setminus \Sigma (\mc M,\mc A_0)$ and $\gamma \in \Sigma (\mc G,\mc A_0) \setminus
\Sigma_{\mc O,\mu} \tilde{+} \Sigma (\mc M,\mc A_0)$:
\[
0 < |t (\alpha)| < |t (\beta)| < |t (\gamma)| .
\]
Now take the system of positive roots 
\[
\Sigma (\mc G,\mc A_0)^+ := \{ \alpha \in \Sigma (\mc G,\mc A_0) : t (\alpha) > 0 \} 
\]
and let $\Delta$ be the unique basis of $ \Sigma (\mc G,\mc A_0)$ contained therein. Then
$\Delta$ consists of the positive roots that cannot be written as sums of positive roots
with smaller $t$-values. Hence $\Delta$ consists of a basis of $\Sigma (\mc M,\mc A_0)$,
added to that some roots to create a basis of $\Sigma_{\mc O,\mu} \tilde{+} \Sigma (\mc M,\mc A_0)$
and completed with other roots (all with $t$-values as small as possible) to a basis
of $\Sigma (\mc G,\mc A_0)$.
\end{proof}

Let $\mc P_0$ be the "standard" minimal parabolic $F$-subgroup of $\mc G$ determined by $\mc A_0$ 
and $\Delta$ and put $\mc P = \mc P_0 \mc M$. Then $\Sigma (\mc G,\mc M)$ is spanned by 
$\Delta \setminus \Delta_{\mc M}$ and $\Delta_{\mc M}^{\mc O} \setminus \Delta_{\mc M}$ spans 
$\R \Sigma_{\mc O,\mu}$. We note that
\[
(a_M^*)^{\perp \Delta_{\mc O,\mu}} := 
\{ x \in a_M^* : \langle \alpha^\vee, x \rangle = 0 \; \forall \alpha \in \Delta_{\mc O,\mu} \}
= \{ x \in a_M^* : \langle \alpha^\vee, x \rangle = 0 \; \forall \alpha \in \Delta_{\mc M}^{\mc O} \}
\]
always contains $a_G^* = X^* (A_G) \otimes_\Z \R$, but can be larger (if $\Sigma_{\mc O,\mu}$
has smaller rank than $\Sigma (\mc G,\mc M)$). Consider the obtuse negative cones with
respect to $\Delta_{\mc O,\mu}$:
\begin{align*}
& a_M^{*-} = \big\{ \sum\nolimits_{\alpha \in \Delta_{\mc O,\mu}} x_\alpha \alpha : 
x_\alpha \in \R_{\leq 0} \big\} , \\
& a_M^{*--} = \big\{ \sum\nolimits_{\alpha \in \Delta_{\mc O,\mu}} x_\alpha \alpha : 
x_\alpha \in \R_{< 0} \big\} . 
\end{align*}

\begin{defn}\label{def:temp}
Let $\pi$ be a finite dimensional $\End_G (I_P^G (E_B))$-representation. Then $\pi$ tempered
if Wt$(\pi) \subset X_\unr (M) \exp (a_M^{*-})$.
We say that $\pi$ is discrete series if $(a_M^*)^{\perp \Delta_{\mc O,\mu}} = 0$ and
Wt$(\pi) \subset X_\unr (M) \exp (a_M^{*--})$.
We call $\pi$ essentially discrete series if Wt$(\pi) \subset 
\exp ((a_M^*)^{\perp \Delta_{\mc O,\mu}} \otimes_\R \C) X_\unr (M) \exp (a_M^{*--})$.

These definitions also apply to $1_u \End_G (I_P^G (E_B) )_U^\an 1_u$, provided that we replace
$\Delta_{\mc O,\mu}$ by $\Delta_{\sigma \otimes u}$. This means replacing $a_M^{*-}$ by
\[
(a_M^* )^-_u := \big\{ \sum\nolimits_{\alpha \in \Delta_{\sigma \otimes u}} x_\alpha \alpha : 
x_\alpha \in \R_{\leq 0} \big\}
\]
and similarly adjusting $a_M^{*--}$ to $(a_M^* )^{--}_u$ and 
$(a_M^*)^{\perp \Delta_{\sigma \otimes u}}$ to $(a_M^* )^{\perp \Delta_{\sigma \otimes u}}$.

Let $k : \Sigma_{\sigma \otimes u} \to \R$ be a $W(M,\mc O)_{\sigma \otimes u}$-invariant function.
We say that $V \in \mh H (\tilde{\mc R}_u , W(M,\mc O)_{\sigma \otimes u}, k, 
\natural_u ) -\fMod$ is tempered if Wt$(V) \subset i a_M^* + (a_M^*)^-_u$, essentially
discrete series if Wt$(V) \subset (a_M^* )^{\perp \Delta_{\sigma \otimes u}} +
i a_M^* + (a_M^* )^{--}_u$ and discrete series if it is essentially discrete series and
$(a_M^* )^{\perp \Delta_{\sigma \otimes u}} = 0$.
\end{defn}
We note that Definition \ref{def:temp} also makes sense for localized or completed versions
of Hecke algebras, because those still have a root system and a large commutative subalgebra with
respect to which one can consider weights.

\subsection{Preservation of temperedness and discrete series} \
\label{par:preserv}

We will investigate these aspects of the relation between $\Rep (G)^{\mf s}$ and\\
$\End_G (I_P^G (E_B))-\Mod$ via graded Hecke algebras. From Corollary \ref{cor:8.4} we recall
the equivalence of categories between
\begin{itemize}
\item $\End_G (I_P^G (E_B))-\Modf{W(M,\sigma,X_\nr (M)) u X_\nr^+ (M)}$,
\item $\mh H (\tilde{\mc R}_u , W(M,\mc O)_{\sigma \otimes u}, k^u, \natural_u )-\Modf{a_M^*}$.
\end{itemize}

\begin{prop}\label{prop:9.6}
\enuma{
\item The above equivalence of categories preserves temperedness.
\item The above equivalence of categories preserves discrete series.
\item Any $V \in \End_G (I_P^G (E_B))-\Modf{W(M,\sigma,X_\nr (M)) u X_\nr^+ (M)}$
is essentially discrete series if and only if the corresponding module for
$\mh H (\tilde{\mc R}_u ,W(M,\mc O)_{\sigma \otimes u}, k^u, \natural_u )$ is essentially 
discrete series and $\Sigma_{\sigma \otimes u}$ has full rank in $\Sigma_{\mc O,\mu}$.
}
\end{prop}
\begin{proof}
We have to consider all the steps in \eqref{eq:8.8}, for those give rise to the equivalence of
categories in Corollary \ref{cor:8.4}. 
Pullback along
\[
\End_G (I_P^G (E_B)) \to \End_G (I_P^G (E_B))_U^\an
\]
does not change the $\C [X_\nr (M)]$-weights, nor the root system, so that step certainly preserves
everything under consideration. Similarly pullback along
\[
\mh H (\tilde{\mc R}_u , W(M,\mc O)_{\sigma \otimes u}, k^u, \natural_u ) \longrightarrow
\mh H (\tilde{\mc R}_u , W(M,\mc O)_{\sigma \otimes u}, k^u, \natural_u )^\an_{\log (U_u)} 
\]
is innocent for our purposes. 

The algebra isomorphism
\[
1_u \End_G (I_P^G (E_B))_U^\an 1_u \longrightarrow 
\mh H (\tilde{\mc R}_u , W(M,\mc O)_{\sigma \otimes u}, k^u, \natural_u )^\an_{\log (U_u)} 
\]
from Proposition \ref{prop:8.3} has the effect $\exp_u$ on weights. Since the root systems on both
sides are the same,
\begin{align*}
& \exp_u (i a_M^* + (a_M^* )^-_u ) = X_\unr (M) \exp ( (a_M^* )^-_u ),\\
& \exp_u (i a_M^* + (a_M^* )^{--}_u ) = X_\unr (M) \exp ( (a_M^* )^{--}_u ),\\
& \exp_u \big( (a_M^* )^{\perp \Delta_{\sigma \otimes u}} + i a_M^* + (a_M^* )^{--}_u \big) = 
\exp \big( (a_M^* )^{\perp \Delta_{\sigma \otimes u}} \big) X_\unr (M) \exp ( (a_M^* )^{--}_u ).
\end{align*}
From Definition \ref{def:temp} we see that the equivalence of categories coming from this 
algebra isomorphism preserves temperedness and (essentially) discrete series.

It remains to investigate the Morita equivalent inclusion
\begin{equation}\label{eq:9.11}
1_u \End_G (I_P^G (E_B ))_U^\an 1_u \longrightarrow \End_G (I_P^G (E_B ))_U^\an 
\end{equation}
from Lemma \ref{lem:7.2}. Notice that here the root system changes from $\Sigma_{\sigma \otimes u}$
to $\Sigma_{\mc O,\mu}$. Let $V \in \End_G (I_P^G (E_B ))_U^\an -\fMod$ be a module corresponding
to $V_u \in \\ 1_u \End_G (I_P^G (E_B ))_U^\an 1_u -\fMod$. The relation between the 
$\C[X_\nr (M)]$-weights of $V_u$ and $V$ was described in Lemma \ref{lem:6.6}, in terms of a set
of representatives $W^u$ for $W(M,\sigma,X_\nr (M)) / W(M,\sigma,X_\nr (M))_u$.

In Definition \ref{def:temp} the unitary parts of $\C [X_\nr (M)]$-weights are irrelevant, the 
conditions depend only on the absolute values of $\C [X_\nr (M)]$-weights. Therefore it suffices to 
consider these weights as elements of $X_\nr (M) / X_\nr (M,\sigma) \cong \mc O$, or equivalently 
as characters of $\C [X_\nr (M) / X_\nr (M,\sigma)]$. We indicate this by Wt'$(V)$ and Wt'$(V_u)$. 
Then Lemma \ref{lem:6.6} becomes
\begin{align*}
& \mr{Wt'}(V_u) = \mr{Wt'} (V) \cap X_\nr (M,\sigma) U_u / X_\nr (M,\sigma) ,\\
& \mr{Wt'}(V) = \{ w (\chi) : w \in W^u , \chi \in \mr{Wt'}(V_u) \} .
\end{align*}
In view of Remark \ref{rem:6.7}, Wt'$(V)$ is stable under $R(\mc O)$. Therefore we may replace $W^u$
by $R (\mc O) W(\Sigma_{\mc O,\mu})^u$, where $W(\Sigma_{\mc O,\mu})^u$ is a set of shortest length
representatives for $W(\Sigma_{\mc O,\mu}) / W(\Sigma_{\mc O,\mu})_u$. Since Wt'$(V_u)$ is stable
under $R(\sigma \otimes u)$ and 
\[
W(\Sigma_{\mc O,\mu})_u \subset R(\sigma \otimes u) W(\Sigma_{\sigma \otimes u}),
\] 
we may also take for $W(\Sigma_{\mc O,\mu})^u$ a set of shortest length representatives of\\ 
$W(\Sigma_{\mc O,\mu}) / W(\Sigma_{\sigma \otimes u})$, then still
\[
\mr{Wt'}(V) = \{ r w (\chi) : r \in R (\mc O), w \in W(\Sigma_{\mc O,\mu})^u ,\chi \in \mr{Wt'}(V_u) \}.
\]
Now we are in a setting where, under the Morita equivalence from Lemma \ref{lem:6.6}, the root 
systems and $\C [X_\nr (M) / X_\nr (M,\sigma)]$-weights behave exactly as in \cite[Theorem 2.5.c]{AMS3}.
That enables us to apply the arguments for \cite[Proposition 2.7]{AMS3} (which can easily be rephrased
in terms of a root datum and an extended Weyl group acting on it). The conclusions from \cite[\S 2]{AMS3}
about the behaviour of temperedness and (essentially) discrete series with respect to \eqref{eq:9.11}
are exactly as stated in the proposition. 
\end{proof}

Now we will translate $\End_G (I_P^G (E_B))$ to the module category of an affine Hecke
algebra, as far as possible. Every finite dimensional $\End_G (I_P^G (E_B))$-module
decomposes canonically as a direct sum of submodules, each of which has weights in just one
set $W(M,\sigma,X_\nr (M)) u X_\nr^+ (M)$. Combining that with Corollary \ref{cor:8.4} and
Lemma \ref{lem:8.6}, we obtain equivalences of categories between $\Rep_{\mr f}(G)^{\mf s}$, 
$\End_G (I_P^G (E_B)) -\fMod$ and
\begin{equation}\label{eq:9.1}
\bigoplus\nolimits_{u \in X_\unr (M)} \mh H (\tilde{\mc R}_u , W(M,\mc O)_{\sigma \otimes u}, k^u, 
\natural_u ) -\Modf{a_M^*} \; \big/ W(M,\sigma,X_\nr (M)) .
\end{equation}
To make sense of this as category, the action of $w \in W(M,\sigma,X_\nr (M))$ on the summands 
indexed by $u$ with $w(u) = u$ is supposed to be trivial. Hence the quotient operation only takes 
place in the index set $X_\unr (M)$, and the result can be considered as a direct sum of module 
categories, indexed by $X_\unr (M) / W(M,\sigma,X_\nr (M))$. Unfortunately this is not canonical, it
depends on the choice of a set of representatives for the action of $X_\nr (M,\sigma)$ on $X_\unr (M)$.

With Lemma \ref{lem:8.1} we can rewrite \eqref{eq:9.1} as
\begin{equation}\label{eq:9.2}
\bigoplus\nolimits_{u \in X_\unr (M) / X_\nr (M,\sigma)} \mh H (\tilde{\mc R}_u , 
W(M,\mc O)_{\sigma \otimes u}, k^u, \natural_u ) -\Modf{a_M^*} \; \big/ W(M,\mc O) .
\end{equation}
This is very similar to the module category of an affine Hecke algebra with torus
$X_\unr (M) / X_\nr (M,\sigma) = \Irr (M_\sigma^2 / M^1)$. 
To make that precise, consider the algebra
\[
\End_G^\circ (I_P^G (E_B)) :=
\End_G (I_P^G (E_B)) \bigcap \bigoplus\nolimits_{w \in W(\Sigma_{\mc O,\mu})} \C (X_\nr (M)) \mc T_w .
\]
Corollary \ref{cor:5.6} and $\mc T_r \in \End_G (I_P^G (E_B))^\times$ for $r \in R(\mc O)$ entail that
\[
\End_G (I_P^G (E_B)) = \bigoplus\nolimits_{r \in R (\mc O)} \End_G^\circ (I_P^G (E_B)) \mc T_r .
\]
All the calculations in Sections \ref{sec:localization}--\ref{sec:class} also work with
$\End_G^\circ (I_P^G (E_B))$, provided we replace $W(M,\mc O)$ by $W(\Sigma_{\mc O,\mu})$ and
$W(M,\sigma,X_\nr (M))$ by $X_\nr (M,\sigma) \rtimes W(\Sigma_{\mc O,\mu})$ everywhere. (These 
restrictions only make the computations easier.) 

For any $u \in X_\unr (M)$, Lemmas \ref{lem:4.4} and \ref{lem:8.5} imply that the 2-cocycle $\natural_u$ 
of $W(M,\mc O)_{\sigma \otimes u}$ is trivial on $W(\Sigma_{\mc O,\mu})_{\sigma \otimes u}$. The proof
of Lemma \ref{lem:8.5} provides a canonical normalization for the involved $\mc T_w^u$, so that
$\natural_u$ disappears in this setting. In the end we find that, like \eqref{eq:9.2}, 
$\End_G^\circ (I_P^G (E_B))$ is equivalent with
\begin{equation}\label{eq:9.15}
\bigoplus\nolimits_{u \in X_\unr (M) / X_\nr (M,\sigma)} \mh H (\tilde{\mc R}_u , 
W(\Sigma_{\mc O,\mu})_{\sigma \otimes u}, k^u ) -\Modf{a_M^*} \; \big/ W(\Sigma_{\mc O,\mu}) .
\end{equation}
Recall the root datum
\[
(\Sigma_{\mc O}^\vee, M_\sigma^2 / M^1, \Sigma_{\mc O}, (M_\sigma^2 / M^1)^\vee)
\]
from Proposition \ref{prop:3.5}. Endow it with the basis determined by $P$, parameter $q_F$
and the labels 
\begin{equation}\label{eq:9.9}
\lambda (h_\alpha^\vee) = \log ( q_\alpha q_{\alpha*} ) / \log (q_F),\; 
\lambda^* (h_\alpha^\vee) = \log ( q_\alpha q_{\alpha*}^{-1} ) / \log (q_F) .
\end{equation}
To these data we associate the affine Hecke algebra
\begin{equation}\label{eq:9.4}
\mc H (\mc O,G) := \mc H \big( \Sigma_{\mc O}^\vee, M_\sigma^2 / M^1, \Sigma_{\mc O}, 
(M_\sigma^2 / M^1)^\vee, \lambda, \lambda^*, q_F \big).
\end{equation}
From Lusztig's reduction theorems \cite{Lus-Gr}, in the form \cite[Theorems 2.5 and 2.9]{AMS3},
we see that $\mc H (\mc O,G) - \fMod$ is also equivalent with \eqref{eq:9.15}.
The group $R(\mc O)$ acts on the root and on the algebra \eqref{eq:9.4}, preserving all 
the structure. With a 2-cocycle 
\begin{equation}\label{eq:9.3}
\tilde \natural : ( W(M,\mc O) / W(\Sigma_{\mc O,\mu}) )^2 \to \C^\times
\end{equation}
we build a twisted affine Hecke algebra
\[
\mc H (\mc O,G) \rtimes \C [R(\mc O),\tilde \natural] 
\] 
as in \cite[Proposition 2.2]{AMS3}. Let $\tilde{\natural}_u$ be the restriction of $\tilde \natural$
to $W (M,\mc O)_{\sigma \otimes u}$. From \cite[Theorems 2.5 and 2.9]{AMS3} we see that
$\mc H (\mc O,G) \rtimes \C [R(\mc O),\tilde \natural] -\fMod$ is equivalent with 
\begin{equation}\label{eq:9.12}
\bigoplus_{u \in X_\unr (M) / X_\nr (M,\sigma)} \mh H (\tilde{\mc R}_u , 
W(\Sigma_{\sigma \otimes u}), k^u) \rtimes \C[R(\sigma \otimes u), \tilde{\natural}_u ] 
-\Modf{a_M^*} \; \big/ W(M,\mc O) .
\end{equation}
Notice that here we do not see the entire 2-cocycle $\tilde{\natural}$, only its restrictions to the 
subgroups $W(M,\mc O)_{\sigma \otimes u}$. Let us summarise the above observations:

\begin{cor}\label{cor:9.9}
\enuma{
\item There exists an equivalence of categories
\[
\End_G^\circ (I_P^G (E_B)) -\fMod \longleftrightarrow \mc H (\mc O,G) -\fMod .
\]
Thus $\End_G (I_P^G (E_B))$ has a subalgebra, over which it is of finite rank, and that subalgebra
is almost Morita equivalent with an affine Hecke algebra.
\item Suppose that $\tilde{\natural}_u$ is cohomologous to $\natural_u$, for each $u \in X_\unr (M)$.
Then the categories $\mc H (\mc O,G)  \rtimes \C [R(\mc O),\tilde \natural] -\fMod$ and 
$\End_G (I_P^G (E_B)) -\fMod$ are equivalent.
}
\end{cor}

Although the equivalences of categories in Corollary \ref{cor:9.9} look like Morita equivalences, 
they are not quite, because we do not know whether they extend to modules of infinite length. 
Let us describe part (b) more explicitly. Start with $V \in \End_G (I_P^G (E_B))-\fMod$. Decompose it as
\[
V = \bigoplus\nolimits_{u \in X_\unr (M)} V_u \quad \text{where } \mr{Wt}(V_u) \subset u X_\nr^+ (M) .
\]
Pick a fundamental domain $\tilde X$ for the action of $X_\nr (M,\sigma)$ on $X_\unr (M)$.
Then put $\tilde V = \bigoplus_{u \in \tilde X} V_u$, this is the associated 
$\mc H (\mc O,G) \rtimes \C [R(\mc O),\tilde \natural]$-module. The \\
$\C [X_\nr (M)/X_\nr (M,\sigma)]$-action can be read of directly, to reconstruct how the rest of 
$\mc H (\mc O,G) \rtimes \C [R(\mc O),\tilde \natural]$ acts one needs Lemmas \ref{lem:8.1} and 
\ref{lem:8.6}. The effect of this equivalence on weights is simple. Whenever a module $\tilde V$ of 
$\mc H (\mc O,G) \rtimes \C [R(\mc O),\tilde \natural]$ has a weight $\chi X_\nr (M,\sigma) \in 
X_\nr (M)/X_\nr (M,\sigma)$, all elements of $\chi X_\nr (M,\sigma) \subset X_\nr (M)$ are weights 
of $V \in \End_G (I_P^G (E_B)) -\fMod$, and conversely. 

The problem with Corollary \ref{cor:9.9}.b lies in the existence of a 2-cocycle $\tilde \natural$
with the mentioned properties. We do not know whether such a 2-cocycle exists in general. 
Of course it is easy to fulfill the condition for one given $u \in X_\unr (M)$, but then it could 
fail for different $u' \in X_\unr (M)$.
Nevertheless, even if we cannot find such a $\tilde \natural$, we can still work in similar spirit.

Recall from the proof of Theorem \ref{thm:8.8} that there is a central extension $\Gamma_{\mc O}$ of
$R(\mc O)$ such that $\mc H (\mc O,G) \rtimes \C [R(\mc O),\tilde \natural] -\Mod$ is equivalent
with the subcategory of $\mc H (\mc O,G) \rtimes \Gamma_{\mc O}-\Mod$ determined by the appropriate 
character of $\ker (\Gamma_{\mc O} \to R(\mc O))$. That allows us to apply results about extended
affine Hecke algebras like $\mc H (\mc O,G) \rtimes \Gamma_{\mc O}$ to twisted affine Hecke algebras
like $\mc H (\mc O,G) \rtimes \C [R(\mc O),\tilde \natural]$.

It is known from \cite[\S 2.1--2.2]{AMS3} that the equivalence between $\mc H (\mc O,G) \rtimes 
\C [R(\mc O),\tilde \natural] -\fMod$ and \eqref{eq:9.12} is compatible with parabolic induction
and restriction, temperedness, discrete series and the Langlands classication, and its effect on
weights is also well-understood. All this is analogous to the equivalence between 
$\End_G (I_P^G (E_B))-\fMod$ and \eqref{eq:9.1}, as we worked out in Corollary \ref{cor:8.4} and
Proposition \ref{prop:9.6}. As a consequence, most results about finite dimensional modules of
extended affine Hecke algebras can be interpreted in terms of the category \eqref{eq:9.12}. 

That applies in particular to the results of \cite{SolComp} that do not involve infinite dimensional
modules or topological completions of Hecke algebras. In that paper it is assumed that Rep$(G)^{\mf s}$ 
is equivalent with the module category of an extended affine Hecke algebra, and properties of such an
equivalence are derived. In fact, all the proofs and results of \cite{SolComp} outside Paragraphs 1.2,
2.1 and 3.3 can be reformulated with \eqref{eq:9.12} instead of the module category of a twisted affine 
Hecke algebra, because they only use properties that are respected by such equivalences of categories.
If we do that, we do not need all of $\tilde \natural$ any more, it suffices to know its 
restrictions $\tilde{\natural}_u$. 

Once we realize that, we can generalize \cite{SolComp}. Namely, we can replace $\tilde \natural$
by a family of 2-cocyles $\tilde{\natural}_u$ of $W(M,\mc O)_{\sigma \otimes u}$, parametrized by 
$X_\nr (M) / X_\nr (M,\sigma)$ and equivariant (up to coboundaries) for $W(M,\mc O)$, but not 
necessarily constructed from a single 2-cocycle $\tilde \natural$ on $W(M,\mc O)$. From Lemmas
\ref{lem:8.1}.a and \ref{lem:8.6}.a we know that $\natural_u$ is such a family of 2-cocycles.
Thus, we want to apply \cite{SolComp} with \eqref{eq:9.2} instead of an extended affine Hecke algebra.

To do so, it remains to verify  the precise assumptions in \cite[\S 4.1]{SolComp} for the equivalence 
between \eqref{eq:9.2} and $\Rep_{\mr f}(G)^{\mf s}$. The results about 
parabolic induction and restriction in Corollary \ref{cor:8.4} take care of 
\cite[Condition 4.1.(i)--(ii)]{SolComp}. Next, \cite[Conditions 4.1.(iii) and 4.2.(ii)]{SolComp} are
about inclusions of parabolic subalgebras associated to Levi subgroups $L$ of $G$ containing $M$.
These are fulfilled by the naturality of the inclusion \eqref{eq:4.33} and because for graded Hecke
algebras we are using standard parabolic subalgebras anyway.

In \cite[Condition 4.2.iii]{SolComp} it is required that $\Sigma_{\mc O,\mu}^+$ lies in the
cone $\Q_{\geq 0} \Sigma (\mc G,\mc M)^+$ and that $\Q \Sigma_{\mc O,\mu}$ has a $\Q$-basis
consisting of simple roots of $\Sigma (\mc G,\mc M)$. Both are guaranteed by Lemma \ref{lem:9.1}.

Let $\Sigma_{\mc O,L}$ be the parabolic root subsystem of $\Sigma_{\mc O,\mu}$ consisting of
roots that come from the action of $A_M$ on the Lie algebra of $L$. Then
\[
N_L (M,\mc O) / M = W(\Sigma_{\mc O,L}) \rtimes R(\mc O,L) ,\text{ where }
R(\mc O,L) = R(\mc O) \cap N_L (M,\mc O) / M .
\]
In \cite[Condition 4.2.iv]{SolComp} it is required firstly that $R(\mc O,L)$ stabilizes 
$\Sigma_{\mc O,L}$ -- which is clear. Secondly, when $\Sigma_{\mc O,L}$ has full rank in 
$\Sigma (\mc L,\mc M)$, \cite[Condition 2.1]{SolComp} has to be fulfilled. That says 
\begin{itemize}
\item $R(\mc O,L) \subset R(\mc O,L')$ if $L \subset L'$;
\item the action of $R(\mc O,L)$ on $X_\nr (M)$ stabilizes the subsets $\exp (\C \Sigma_{\mc O,L})$ 
and $X_\nr (M)^L := \exp ((a_M^*)^{\perp L} \otimes_\R \C)$, where 
\[
(a_M^*)^{\perp L} = \{ x \in a_M^* : \langle \alpha^\vee , x\rangle = 0 \; \forall \alpha
\in \Sigma_{\mc O,L} \} ;
\]
\item $R(\mc O,L)$ acts on $X_\nr (M)^L$ by multiplication with elements of 
$X_\nr (M)^L \cap \exp (\C \Sigma_ {\mc O,L})$.
\end{itemize}
The first of these bullets is obvious. By the full rank assumption
\begin{equation}\label{eq:9.5}
(a_M^*)^{\perp L} = \{ x \in a_M^* : \langle \alpha^\vee , x\rangle = 0 \; \forall \alpha
\in \Sigma (\mc L,\mc M) \} = a_L^* = X^* (A_L) \otimes_\Z \R .
\end{equation}
Recall that the action of $r \in R(\mc O,L)$ on $X_\nr (M)$ consists of a part which is linear 
on the Lie algebra and a translation by $\chi_r$. By the $R(\mc O,L)$-stability of 
$\Sigma_{\mc O,L}$, the linear part stabilizes $\exp (\C \Sigma_{\mc O,L})$. Further the
linear part of the action of $r$ fixes $a_L^*$ pointwise, so by \eqref{eq:9.5} it fixes
$X_\nr (M)^L$ pointwise. The definition of $\chi_r$ in \eqref{eq:3.27} shows that it is
an unramified character of $L$ which is trivial on $Z(L)$. This means that 
$\chi_r \in X_\nr (M)^L \cap \exp (\C \Sigma_ {\mc O,L})$. Hence the second and third bullets hold.

We have verified everything needed to make the arguments in \cite[\S 4.2]{SolComp} about finite
length representations work with \eqref{eq:9.2}. Recall that a $G$-representation (of finite length) 
is called essentially square-integrable if its restriction to the derived group of $G$ is 
square-integrable.

\begin{prop}\label{prop:9.3}
Consider the equivalence between $\Rep_{\mr f}(G)^{\mf s}$ and 
\[
\bigoplus\nolimits_{u \in X_\unr (M) / X_\nr (M,\sigma)} \mh H (\tilde{\mc R}_u , 
W(M,\mc O)_{\sigma \otimes u}, k^u, \natural_u ) -\Modf{a_M^*} \; \big/ W(M,\mc O)
\] 
coming from Corollary \ref{cor:8.4} and Lemma \ref{lem:8.1}. 
\enuma{
\item This equivalence preserves temperedness.
\item Suppose that $\Sigma_{\sigma \otimes u}$ has smaller rank than $\Sigma (\mc G,\mc M)$. Then
$\Rep_{\mr f}(G)^{\mf s}$ contains no essentially square-integrable representations with cuspidal 
support in\\ $W(M,\mc O) \{ \sigma \otimes u \chi : \chi \in X_\nr^+ (M) \}$.
\item Suppose that $\Sigma_{\sigma \otimes u}$ has full rank in $\Sigma (\mc G,\mc M)$. 
The equivalence provides a bijection between the following sets:
\begin{itemize}
\item essentially square-integrable objects of $\Rep_{\mr f}(G)^{\mf s}$ 
with cuspidal support in $W(M,\mc O) \{ \sigma \otimes u \chi : \chi \in X_\nr^+ (M) \}$,
\item essentially discrete series objects of
$\mh H (\tilde{\mc R}_u , W(M,\mc O)_{\sigma \otimes u}, k ,\natural_u ) -\Modf{a_M^*}$.
\end{itemize}
This remains valid if we add "tempered" and/or "irreducible" on both sides.
\item When $Z(G)$ is compact, part (c) holds without "essentially".
}
\end{prop}
\begin{proof}
With \cite[Proposition 2.7 and Theorem 2.11.d]{AMS3} we translate the notions of temperedness and
(essentially) discrete series for extended affine Hecke algebras to notions for \eqref{eq:9.3}.
Then we apply \cite[Theorem 4.9 and Proposition 4.10]{SolComp}, as discussed above.
\end{proof}

With Proposition \ref{prop:9.6} we can translate Proposition \ref{prop:9.3} into a statement about
$\End_G (I_P^G (E_B))$.

\begin{thm}\label{thm:9.2}
\enuma{
\item The equivalence $\mc E : \Rep (G)^{\mf s} \to \End_G (I_P^G (E_B))\mr{-Mod}$ restricts to 
an equivalence between the subcategories of finite length tempered representations on both sides.
\item If $\Sigma_{\mc O,\mu}$ has smaller rank than $\Sigma (\mc G,\mc M)$, then 
$\Rep_{\mr f} (G)^{\mf s}$ contains no essentially square-integrable representations.
\item Suppose that $\Sigma_{\mc O,\mu}$ has full rank in $\Sigma (\mc G,\mc M)$. Then
$\mc E$ provides a bijection between the following sets:
\begin{itemize}
\item essentially square-integrable representations in $\Rep_{\mr f} (G)^{\mf s}$,
\item essentially discrete series representations in $\End_G (I_P^G (E_B)) -\fMod$.
\end{itemize}
This remains valid if we add "tempered" and/or "irreducible" on both sides.
\item When $Z(G)$ is compact, part (c) also holds without "essentially".
}
\end{thm}

\subsection{The structure of $\Irr (G)^{\mf s}$} \

It is interesting to combine the previous results on temperedness with Corollaries \ref{cor:8.4}
and \ref{cor:8.10}. 

\begin{thm}\label{thm:9.4}
There exist bijections
\[
\Irr (G)^{\mf s} \xrightarrow{\mc E} \Irr \big( \End_G (I_P^G (E_B)) \big)
\xrightarrow{\zeta} \Irr \big( \C [X_\nr (M)] \rtimes \C [W(M,\sigma,X_\nr(M)),\natural] \big)
\]
such that, for $\pi \in \Irr (G)^{\mf s}$ and $u \in X_\unr (M)$:
\begin{itemize}
\item the cuspidal support Sc$(\pi)$ lies in $W(M,\mc O) u X_\nr^+ (M) \Longleftrightarrow$\\
Wt$(\mc E (\pi)) \subset W(M,\sigma,X_\nr(M)) u X_\nr^+ (M) \Longleftrightarrow$\\
Wt$(\zeta \circ \mc E (\pi)) \subset W(M,\sigma,X_\nr(M)) u X_\nr^+ (M)$
\item $\pi$ is tempered $\Longleftrightarrow \mc E (\pi)$ is tempered
$\Longleftrightarrow \mr{Wt}(\zeta \circ \mc E (\pi)) \subset X_\unr (M)$
\end{itemize}
\end{thm}
\begin{proof}
The bijections and the first bullet come from Corollaries \ref{cor:8.4} and \ref{cor:8.10}. Let 
\[
V_0 \in \Irr \big( \C [X_\nr (M)] \rtimes \C [W(M,\sigma,X_\nr(M)),\natural] \big)
\] 
with Wt$(V) \subset W(M,\sigma,X_\nr(M)) u X_\nr^+ (M)$. Then Wt$(V_0) \subset X_\unr (M)$ if and
only if the irreducible representation $V_1$ of $\C [X_\nr (M)] \rtimes 
\C [W(M,\sigma,X_\nr(M))_u,\natural]$ associated to it by Lemma \ref{lem:8.7} has $u$ as its
only weight. This is the case if and only if the irreducible representation $V_2$ of
$\C [a_M^* \otimes_\R \C] \rtimes \C [W(M,\mc O)_{\sigma \otimes u},\natural_u]$ obtained 
from $V_1$ as in \eqref{eq:8.18} has $0 \in a_M^*$ as its only weight. 

Now $V_2$, a representation of a twisted graded Hecke algebra \\
$\mh H (\tilde{\mc R}_u, W(M,\mc O)_{\sigma \otimes u}, 0, \natural_u )$ with 
Wt$(V_2) \subset a_M^*$, is tempered if and only if Wt$(V_2) = \{0\}$.
To see that, notice that the weights of $V_2$ form full $W(\Sigma_{\sigma \otimes u})$-orbits.
Every $W(\Sigma_{\sigma \otimes u})$-orbit in $a_M^*$, except $\{0\}$, contains elements outside
the cone $(a_M^*)_u^-$.

By Theorem \ref{thm:8.8} $V_2$ is tempered if and only if
\[
\zeta_u^{-1} (V_2) \in \Irr \big( \mh H (\tilde{\mc R}_u, W(M,\mc O)_{\sigma \otimes u}, k^u, 
\natural_u ) \big)
\] 
is tempered. Next Proposition \ref{prop:9.3} says that $\zeta_u^{-1}(V_2)$ is tempered if and 
only if its image $V_3$ in $\Irr (\End_G (I_P^G (E_B)))$ is so. Comparing the above with the 
proof of Corollary \ref{cor:8.10}, we see that $V_3$ equals $\zeta^{-1}(V_0)$. 
Finally, in Theorem \ref{thm:9.2}.a we showed that $\zeta^{-1}(V_0)$ is tempered if and only
if $\mc E^{-1} (\zeta^{-1} (V_0))$ is tempered.
\end{proof}

The space $\Irr \big( \C [X_\nr (M)] \rtimes \C [W(M,\sigma,X_\nr(M)),\natural] \big)$
admits an alternative description, which clarifies the geometric structure in Theorem \ref{thm:9.4}.

\begin{lem}\label{lem:9.5}
There is a canonical bijection
\begin{multline*}
\bigsqcup\nolimits_{\chi \in X_\nr (M) / X_\nr (M,\sigma)}
\Irr (\C [W(M,\mc O)_{\sigma \otimes \chi},\natural_\chi]) \; \big/ W(M,\mc O) \longrightarrow \\
\Irr \big( \C [X_\nr (M)] \rtimes \C [W(M,\sigma,X_\nr(M)),\natural] \big) .
\end{multline*}
Here $\natural_\chi$ is defined as the restriction of $\natural_u$ to 
$(W(M,\mc O)_{\sigma \otimes \chi} )^2$, where $u$ is the unitary part of $\chi$.
\end{lem}
\begin{proof}
Let the central extension $\Gamma$ of $W(M,\sigma,X_\nr (M))$ and the central idempotent
$p_\natural$ be as in the proof of Lemma \ref{lem:8.7}, so that
\[
p_\natural \C [\Gamma] \cong  \C[W(M,\sigma,X_\nr (M)),\natural] .
\]
By \eqref{eq:8.15} and \eqref{eq:8.20}, every irreducible representation $\pi$ of \\
$\C [X_\nr (M)] \rtimes \C [W(M,\sigma,X_\nr(M)),\natural]$ is of the form
\[
\mr{ind}_{\C [X_\nr (M)] \rtimes \C [W(M,\sigma,X_\nr(M))_\chi,\natural]}^{\C [X_\nr (M)] \rtimes 
\C [W(M,\sigma,X_\nr(M)),\natural]} (\chi \otimes \rho), 
\]
where $\chi \in X_\nr (M)$ and $\rho \in \Irr (\C [W(M,\sigma,X_\nr(M))_\chi,\natural])$.
The pair $(\chi,\rho)$ is determined by $\pi$, uniquely up to the action of $\Gamma$ given by
\[
\gamma (\chi,\rho) = (\gamma (\chi), \mr{Ad}(\gamma^{-1})^* \rho) .
\]
Since $\Gamma$ is a central extension of $W(M,\sigma,X_\nr (M))$, this action descends to an
action of $W(M,\sigma,X_\nr (M))$ on the collection of such pairs. This yields a bijection
\begin{multline}\label{eq:9.6}
\bigsqcup\nolimits_{\chi \in X_\nr (M)} \Irr (\C [W(M,\sigma,X_\nr(M))_\chi,\natural]) \; \big/ 
W(M,\sigma,X_\nr (M)) \longrightarrow \\
\Irr \big( \C [X_\nr (M)] \rtimes \C [W(M,\sigma,X_\nr(M)),\natural] \big) .
\end{multline}
Recall from Lemma \ref{lem:8.5} that $\C [W(M,\sigma,X_\nr(M))_\chi,\natural]$ is canonically
isomorphic with $\C [W(M,\mc O)_{\sigma \otimes \chi}, \natural_\chi]$.
For $\chi_c \in X_\nr (M,\sigma)$ there are group isomorphisms
\begin{multline}\label{eq:9.7}
W(M,\mc O)_{\sigma \otimes \chi} \xrightarrow{\Omega_\chi} W(M,\sigma,X_\nr (M))_\chi
\xrightarrow{\mr{Ad}(\chi_c)} \\
W(M,\sigma,X_\nr (M))_{\chi \chi_c} \xleftarrow{\Omega_{\chi \chi_c}} 
W(M,\mc O)_{\sigma \otimes \chi \chi_c} .
\end{multline}
It follows from \eqref{eq:8.12} that conjugation with $\phi_{\chi_c}$ induces an algebra
isomorphism
\[
\C [W(M,\sigma,X_\nr(M))_\chi,\natural] \to \C [W(M,\sigma,X_\nr(M))_{\chi_c \chi},\natural]
\]
which via Lemma \ref{lem:8.5} translates to the identity map
\[
\C [W(M,\mc O)_{\sigma \otimes \chi}, \natural_\chi] \to
\C [W(M,\mc O)_{\sigma \otimes \chi_c \chi}, \natural_{\chi_c \chi}] .
\]
Hence, in \eqref{eq:9.6} we can canonically identify all the terms associated to $\chi$'s in
one $X_\nr (M,\sigma)$-orbit. If we do that, the action of $W(M,\sigma,X_\nr (M))$ descends
to an action of
\[
W(M,\mc O) = W(M,\sigma,X_\nr (M)) / X_\nr (M,\sigma) 
\]
and the left hand side of \eqref{eq:9.6} becomes
\[
\bigsqcup\nolimits_{\chi \in X_\nr (M) / X_\nr (M,\sigma)} 
\Irr (\C [W(M,\mc O)_{\sigma \otimes \chi},\natural_\chi]) \; \big/ W(M,\mc O) . \qedhere 
\]
\end{proof}

We note that in Lemma \ref{lem:9.5} $\natural_\chi$ is not necessarily equal to
$\natural \circ \Omega_u |_{W(M,\mc O)_{\sigma \otimes \chi}}$. These 2-cocycles are 
merely cohomologous (by Lemma \ref{lem:8.5} with $u$ the unitary part of $\chi$). 
An advantage of $\natural_\chi$ is that it factors via
\[
\big( W(M,\mc O)_{\sigma \otimes u} / W(\Sigma_{\sigma \otimes u}) \big)^2
\cong R(\sigma \otimes u)^2 . 
\]
The action of $w \in W(M,\mc O)$ on the left hand side of Lemma \ref{lem:9.5} comes from isomorphisms
\begin{multline}\label{eq:9.10}
\C [W(M,\mc O)_{\sigma \otimes \chi},\natural_\chi] \cong \C [W(M,\sigma,X_\nr (M))_\chi, \natural]
\xrightarrow{\mr{Ad}(N_w)} \\ \C [W(M,\sigma,X_\nr (M))_{w(\chi)}, \natural] \cong
\C [W(M,\mc O)_{\sigma \otimes w(\chi)},\natural_{w(\chi)}] . 
\end{multline}
Here the outer automorphisms are described in Lemma \ref{lem:8.5}, while the middle
isomorphism is computed in $\C [W(M,\sigma,X_\nr (M)),\natural]$. In particular we still make use
of the entire 2-cocycle $\natural$, not just of the $\natural_\chi$. 

Define the root system $\Sigma_{\sigma \otimes \chi}$ like $\Sigma_{\sigma \otimes u}$.
By Lemma \ref{lem:8.6}, the composed isomorphism \eqref{eq:9.10} sends $N_v$ to $N_{w v w^{-1}}$ 
for $v \in W(\Sigma_{\sigma \otimes \chi})$, and to a scalar multiple of that for
$v \in W(M,\mc O)_{\sigma \otimes \chi}$. Since
\[
X_\nr (M) / X_\nr (M,\sigma) \to \mc O : \chi \mapsto \sigma \otimes \chi
\]
is bijective, we can rewrite the left hand side of Lemma \ref{lem:9.5} as
\begin{equation}\label{eq:9.8}
\bigsqcup\nolimits_{\sigma' \in \mc O} \Irr ( \C[W(M,\mc O)_{\sigma'}, \natural_{\sigma'}])
\; \big/ W(M,\mc O) .
\end{equation}
In the terminology of \cite[\S 2.1]{ABPS}, \eqref{eq:9.8} is the twisted extended quotient\\
$\big( \mc O /\!/ W(M,\mc O) \big)_\natural$. 

\begin{thm}\label{thm:9.8}
\enuma{
\item There exists a bijection
\[
\widetilde{\zeta \circ \mc E} : \Irr (G)^{\mf s} \to \big( \mc O /\!/ W(M,\mc O) \big)_\natural ,
\]
such that, for $\pi \in \Irr (G)^{\mf s}$ and $u \in X_\unr (M)$:
\begin{itemize}
\item the cuspidal support Sc$(\pi)$ lies in $W(M,\mc O)\{ \sigma \otimes u \chi :
\chi \in X_\nr^+ (M) \}$ if and only if
$\widetilde{\zeta \circ \mc E}(\pi)$ has $\mc O$-coordinate in 
$W(M,\mc O)\{ \sigma \otimes u \chi : \chi \in X_\nr^+ (M) \}$;
\item $\pi$ is tempered if and only if $\widetilde{\zeta \circ \mc E}(\pi)$ has a unitary
(or equivalently tempered) $\mc O$-coordinate.
\end{itemize}
\item This gives rise to bijection
\[
\Irr (G) \longrightarrow \bigsqcup\nolimits_M  \big( \Irr_{\mr{cusp}}(M) /\!/ W(G,M) \big)_\natural ,
\]
where $M$ runs over a set of representatives for the conjugacy classes of Levi subgroups of $G$.}
\end{thm}
\begin{proof}
(a) This follows from Theorem \ref{thm:9.4}, Lemma \ref{lem:9.5} and \eqref{eq:9.8}.\\
(b) We write $\mc O = \mc O_\sigma$, so that the space of supercuspidal representations of $M$ becomes 
\[
\Irr_{\mr{cusp}}(M) = \bigsqcup\nolimits_{\mf s_M = [M,\sigma]_M} \Irr (M)^{\mf s_M} =
\bigsqcup\nolimits_{\mf s_M = [M,\sigma]_M} \mc O_\sigma .
\]
Let $\Irr (G,M)$ be the subset of $\Irr (G)$ with supercuspidal support in $\Irr (M)$. 
Part (a) gives rise to a bijection
\begin{equation}\label{eq:9.13}
\Irr (G,M) = \bigsqcup\nolimits_{\mf s = [M,\sigma]_G} \Irr (G)^{\mf s} \longrightarrow
\bigsqcup\nolimits_{\mf s = [M,\sigma]_G} \big( \mc O_\sigma /\!/ W(M,\mc O_\sigma) \big)_\natural .
\end{equation}
Now we must be careful because $\widetilde{\zeta \circ \mc E}$ is not entirely canonical. We choose
a set of $\sigma \in \Irr_\cusp (M)$ representing all possible inertial equivalence classes (for $G$) 
in \eqref{eq:9.13}, and (for each such $\sigma$) a set of representatives $w \in N_G (M) / M$ for 
$N_G (M) / N_G (M,\mc O_\sigma)$. For $\sigma' \in \mc O_\sigma$ and such a $w$, we define 
$\natural_{w \cdot \sigma'}$ as the push-forward of $\natural_{\sigma'}$ along Ad$(N_w)$
Then the set in \eqref{eq:9.8} can be enlarged to
\begin{equation}\label{eq:9.14}
\bigsqcup\nolimits_{\sigma,w} \bigsqcup\nolimits_{\sigma' \in \mc O} 
\Irr ( \C[W(M,\mc O_\sigma )_{w \cdot \sigma'}, \natural_{w \cdot \sigma'}])
\end{equation}
and the action of $W(M,\mc O)$ in \eqref{eq:9.8} extends to an action of $W(G,M) = N_G (M) / M$ 
on \eqref{eq:9.14}. As in \cite[(29)]{ABPS}, that puts \eqref{eq:9.13} in bijection with
\[
\Big( \bigsqcup\nolimits_{\mf s_M = [M,\sigma]_M} \mc O_\sigma /\!/ W(G,M) 
\Big)_\natural = \big( \Irr_{\mr{cusp}}(M) /\!/ W(G,M) \big)_\natural .
\]
When we let $M$ run over a set of representatives for the conjugacy classes of Levi subgroups of $G$,
$\Irr (G,M)$ exhausts $\Irr (G)$ and we find the claimed bijection.
\end{proof}

Theorem \ref{thm:9.8}.a reveals some geometry hidden in $\Irr (G)^{\mf s}$ and proves a version of 
the ABPS conjecture, namely \cite[Conjecture 2]{ABPS}.

\section{A smaller progenerator of $\Rep (G)^{\mf s}$} 
\label{sec:smaller}

In our main results we do not obtain an affine Hecke algebra, rather a category which is
almost equivalent to the category of finite length modules of an affine Hecke algebra. 
On the other hand, it is known that in many cases $\Rep (G)^{\mf s}$ is really equivalent to the 
module category of an affine Hecke algebra. To obtain more results in that direction, we 
study progenerators of $\Rep (G)^{\mf s}$ that are strictly contained in $I_P^G (E_B)$. 

\subsection{The cuspidal case} \

From \eqref{eq:2.35} we see that $\chi \mapsto \phi_\chi$ yields a group homomorphism
\begin{equation}\label{eq:2.31}
\Irr (M / M_\sigma^3) \to \mr{Aut}_M \big( \mr{ind}_{M^1}^M (E) \big) \cong 
\mr{Aut}_M \big( E \otimes_\C \C [X_\nr (M)] \big) .
\end{equation}
In other words, the group $\Irr (M / M_\sigma^3)$ acts on $\mr{ind}_{M^1}^M (E)$ and on
$E \otimes_\C \C [X_\nr (M)]$. With the isomorphism \eqref{eq:2.3} one can easily express 
the space of invariants under $\Irr (M / M_\sigma^3)$:
\begin{equation}\label{eq:2.32}
\big( \mr{ind}_{M^1}^M (E) \big)^{\Irr (M / M_\sigma^3)} = \mr{ind}_{M^1}^M (E_1) .
\end{equation}
Recall that every irreducible $M^1$-subrepresentation of $E$ is 
isomorphic to $(m^{-1} \cdot \sigma_1, \sigma (m) E_1)$ for some $m \in M$. More precisely,
it follows from \eqref{eq:2.6} and \eqref{eq:2.9} that
\[
\Res_{M^1}^M (\sigma,E) \cong \bigoplus\nolimits_{m \in M / M_\sigma^4} (m\cdot \sigma_1)^{\mu_{\sigma,1}} . 
\]
Since $\sigma_1$ and $m \cdot \sigma_1$ have isomorphic induction to $M$:
\begin{equation}\label{eq:2.15}
\ind_{M^1}^M (\sigma,E) \cong \bigoplus\nolimits_{m \in M / M_\sigma^4} \ind_{M^1}^M (m\cdot \sigma_1
)^{\mu_{\sigma,1}} \cong \ind_{M^1}^M (\sigma_1,E_1)^{[M : M_\sigma^4] \mu_{\sigma,1}} . 
\end{equation}
Notice that 
\[
[M : M_\sigma^4] \mu_{\sigma,1} = [M:M_\sigma^3]
\]
is the length of $\Res_{M^1}^M (E)$. From \eqref{eq:2.15} we see that $\ind_{M^1}^M (\sigma_1,E_1)$ is,
like $\ind_{M^1}^M (\sigma,E)$, a progenerator of $\Rep (M)^{\mc O}$ -- this was already shown by
Bernstein \cite{BeRu}. Further \eqref{eq:2.15} implies
\begin{equation}\label{eq:2.16}
\End_M ( \ind_{M^1}^M (\sigma,E) ) \cong 
\End_M (\ind_{M^1}^M (\sigma_1,E_1)) \otimes_\C M_{[M:M_\sigma^3]} (\C) ,
\end{equation}
where $M_d (\C)$ denotes the algebra of $d \times d$ complex matrices. For comparison with \cite{Hei2}
we analyse the Morita equivalent subalgebra $\End_M (\ind_{M^1}^M (\sigma_1,E_1))$ as well. 
Now \eqref{eq:2.5} cannot be used in general,
because $\sigma (m^{-1})$ need not preserve $E_1$. So we cannot easily embed 
$\C [X_\nr (M)] \cong \C [M / M^1]$ in $\End_M (\ind_{M^1}^M (\sigma_1,E_1))$. The formula \eqref{eq:2.5} 
does still apply when $m \in M_\sigma^3 / M^1$, which means that \eqref{eq:2.5} provides a homomorphism
\[
M_\sigma^3 / M^1 \to \mr{Aut}_M (\ind_{M^1}^M (\sigma_1,E_1)) . 
\]
That extends $\C$-linearly to an embedding
\begin{equation}\label{eq:2.17}
\begin{array}{ccccc}
\ind_{M^1}^{M_\sigma^3} (\C) & \cong & \C [M_\sigma^3 / M^1] & \to & \End_M (\ind_{M^1}^M (\sigma_1,E_1)) \\
\delta_m & \mapsto & m^{-1} M^1 & \mapsto & \sigma_3 (m) \lambda_m 
\end{array} ,
\end{equation}
where $\sigma_3 (m) \lambda_m (v) (m') = \sigma_3 (m) v(m^{-1} m')$ for $v \in \ind_{M^1}^M (\sigma_1,E_1)$.
We note that $\C [M_\sigma^3 / M^1]$ can be regarded as the ring of regular functions on the complex torus 
\[
\mc O_3 := \Irr (M_\sigma^3 / M^1), 
\]
a degree $\mu_{\sigma,1}$ cover of $X_\nr (M) / X_\nr (M,\sigma)$. Another way to construct the embedding
\eqref{eq:2.17} uses that $(\sigma_1,E_1)$ extends to the $M_\sigma^3$-representation $\sigma_3$.
The same reasoning as in \eqref{eq:2.3} gives an isomorphism of $M_\sigma^3$-representations
\[
\ind_{M^1}^{M_\sigma^3} (E_1) \to E_1 \otimes_\C \C [\mc O_3] . 
\]
Since $\C [\mc O_3]$ is commutative, the $M_\sigma^3$-action on $E_1 \otimes_\C \C [\mc O_3]$ is
$\C [\mc O_3]$-linear. We find 
\begin{equation}\label{eq:2.20}
\ind_{M^1}^M (E_1) \cong \ind_{M_\sigma^3}^M (E_1 \otimes_\C \C [\mc O_3]) 
\end{equation}
and $\C [\mc O_3]$ acts on it by $M$-intertwiners, induced from the action on $E_1 \otimes_\C \C [\mc O_3]$.

As worked out in \cite[Proposition 1.6.3.2]{Roc1}, the subalgebra
\begin{equation}\label{eq:2.27}
\C [M_\sigma^2 / M^1] \cong \C [ X_\nr (M) / X_\nr (M,\sigma)]  
\end{equation}
is the centre of $\End_M (\ind_{M^1}^M (\sigma_1,E_1))$. In view of \eqref{eq:2.16}, 
$\C [M_\sigma^2 / M^1]$ is also the centre of $\End_M (\ind_{M^1}^M (\sigma,E))$ -- 
which can be derived directly from Proposition \ref{prop:2.2}. Furthermore $\End_M (\ind_{M^1}^M 
(\sigma_1,E_1))$ is free of rank $\mu_{\sigma,1}^2$ as a module over its centre. 

The commutative subalgebra 
\[
\C[M_\sigma^3 / M^1] \cong \C [\mc O_3],
\]
embedded in $\End_M (\ind_{M^1}^M (\sigma_1,E_1))$ as in \eqref{eq:2.17} or \eqref{eq:2.20}, 
is free of rank $\mu_{\sigma,1}$ as a module over $\C [M_\sigma^2 / M^1]$.
To find generators for $\End_M (\ind_{M^1}^M (\sigma_1,E_1))$ as a module over $\C[M_\sigma^3 / M^1]$,
we consider any $m \in M_\sigma^4$. By definition there exists an $M_1$-isomorphism
\[
\phi_{m,\sigma_1} : (\sigma_1,E_1) \to (m^{-1} \cdot \sigma_1, E_1)  .
\]
Regarding the subspace of $\ind_{M^1}^M (E_1)$ supported on $m M^1$ as the $M^1$-representation
$m^{-1} \cdot \sigma_1$, $\phi_{m,\sigma_1}$ becomes an element of 
$\Hom_{M_1} (\sigma_1 , \ind_{M^1}^M (\sigma_1))$. Applying Frobenius reciprocity, we obtain 
\begin{equation}\label{eq:2.18}
\phi_m \in \End_M (\ind_{M^1}^M (\sigma_1,E_1)), \qquad \phi_m (v) = \phi_{m,\sigma_1} \lambda_m (v) . 
\end{equation}
For $m \in M_\sigma^3$ we can take $\phi_{m,\sigma_1} = \sigma_3 (m)$, and then \eqref{eq:2.18}
recovers \eqref{eq:2.17}.

\begin{lem}\label{lem:2.3}
For every $m \in M_\sigma^4 / M^1$ and for every $m \in M_\sigma^4 / M_\sigma^3$ 
we pick a representative $\tilde m \in M_\sigma^4$. 
\enuma{
\item The set $\{ \phi_{\tilde m} : m \in M_\sigma^4 / M^1 \}$ is a $\C$-basis of 
$\End_M (\ind_{M^1}^M (\sigma_1,E_1))$.
\item With respect to the embedding \eqref{eq:2.17}:
}
\[
\End_M \big( \ind_{M^1}^M (\sigma_1,E_1) \big) =
\bigoplus_{m \in M_\sigma^4 / M_\sigma^3} \phi_{\tilde m} \C [ M_\sigma^3 / M^1] =
\bigoplus_{m \in M_\sigma^4 / M_\sigma^3} \phi_{\tilde m} \C [ \mc O_3] .
\]
\end{lem}
\begin{proof}
(a) By Frobenius reciprocity
\begin{multline}\label{eq:2.19}
\End_M (\ind_{M^1}^M (\sigma_1,E_1)) \cong \Hom_{M_1} (\sigma_1 , \ind_{M^1}^M (\sigma_1)) \\
\cong \Hom_{M_1} \big( \sigma_1 , \bigoplus\nolimits_{m \in M / M^1} (m^{-1} \cdot \sigma_1) \big) .
\end{multline}
By the definition of $M_\sigma^4$, this reduces to 
\[
\bigoplus\nolimits_{m \in M_\sigma^4 / M^1}\Hom_{M_1} (\sigma_1 , m^{-1} \cdot \sigma_1) ,
\]
where each summand is one-dimensional. For every $m \in M_\sigma^4$, the element 
$\phi_m \in \End_M (\ind_{M^1}^M (\sigma_1,E_1))$ comes by construction from the nonzero element
$\phi_{m,\sigma_1} \in \Hom_{M^1}(\sigma_1, m^{-1} \cdot \sigma_1)$. It follows that, for any 
$m_1 \in M^1$, $\phi_m$ and $\phi_{m m_1}$ differ only by a scalar, and that the 
$\phi_{\tilde m}$ with $m \in M_\sigma^4 / M^1$ form a basis of \eqref{eq:2.19}.\\
(b) This follows directly from part (a) and \eqref{eq:2.17}.
\end{proof}

It is known from \cite[(1.6.1.1)]{Roc1} that the operators $\phi_m$ with $m \in M_\sigma^4$ 
commute up to scalars. However, by \cite[Proposition 1.6.1.2]{Roc1} the algebra 
$\End_M (\ind_{M^1}^M (\sigma_1,E_1))$ is commutative if and only if $\mu_{\sigma,1} = 1$. 

From \eqref{eq:2.21}, \eqref{eq:2.20} and Lemma \ref{lem:2.3}.b we obtain 
\begin{equation}\label{eq:2.23}
\Hom_M \big( \ind_{M^1}^M (E_1), \ind_{M^1}^M (E_1) \otimes_{\C [\mc O_3]} \C (\mc O_3) \big) \cong
\bigoplus\nolimits_{m \in M_\sigma^4 / M_\sigma^3} \phi_{\tilde m} \C (\mc O_3) .
\end{equation}
The action of $\Irr (M / M_\sigma^3)$ on $\ind_{M^1}^M (E)$ from \eqref{eq:2.31} extends naturally 
to an action on $\ind_{M^1}^M (E) \otimes_{\C [X_\nr (M)]} \C (X_\nr (M))$. 
From \cite[Proposition 4.3]{Hei2}, with the same proof, we obtain
\begin{equation}\label{eq:2.33}
\big( \ind_{M^1}^M (E) \otimes_{\C [X_\nr (M)]} \C (X_\nr (M) \big)^{\Irr (M / M_\sigma^3)}
= \ind_{M^1}^M (E_1) \otimes_{\C [\mc O_3]} \C (\mc O_3) .
\end{equation}

Unfortunately, we did not succeed in making the isomorphism \eqref{eq:2.16} explicit in 
terms of the endomorphisms of $\ind_{M^1}^M (E)$ and $\ind_{M^1}^M (E_1)$ exhibited above. Clearly 
\[
\C [\mc O_3] \subset \End_M (\ind_{M^1}^M (E_1))
\]
corresponds naturally to a subalgebra of
\[
\C [X_\nr (M)] \subset \End_M (\ind_{M^1}^M (E)) .
\]
From Lemma \ref{lem:2.3} and \eqref{eq:2.26} we see that the $\phi_{\tilde m}$ 
with $m \in M_\sigma^4 / M_\sigma^3$ should correspond to linear combinations of the $\phi_\chi$ 
with $\chi \in X_\nr (M,\sigma)$ and $\chi |_{M_\sigma^3} = \chi_{3,m}$, but we did not find a
canonical choice. Thus, although the progenerators $\ind_{M^1}^M (E)$ and $\ind_{M^1}^M (E_1)$ of 
the cuspidal Bernstein component $\Rep (M)^{\mc O}$ are equally good, they look somewhat differently. 
It seems technically difficult to analyse\\ 
$\End_G \big( I_P^G  (\ind_{M^1}^M (E_1)) \big)$ in general.

Many complications with the 2-cocycles $\natural$ stem from the multiplicity $\mu_{\sigma,1}$ of
the $M^1$-subrepresentation $E_1$ in $\Res^M_{M^1} E$. In Section \ref{sec:cuspidal} we saw that
$\natural \big|_{X_\nr (M,\sigma)}$ is trivial if and only if $\mu_{\sigma,1} = 1$. Recall that
$\mu_{\sigma,1}$ depends on $\sigma$ but not on the choice of $E_1$.
It is known from \cite[Remark 1.6.1.3]{Roc1} that $\mu_{\sigma,1} = 1$ in many cases:
\begin{itemize}
\item when the maximal $F$-split central torus of $M$ has dimension $\leq 1$,
\item when $M$ is quasi-split and $(\sigma,E)$ is generic,
\item for finite products $(\prod_i M_i, \boxtimes_i \sigma_i)$ with each $(M_i,\sigma_i)$ as
in the above cases.
\end{itemize}
We note that the first bullet includes semisimple groups, general linear groups, unitary groups
and inner forms of such groups, while the second bullet includes all tori.
On the other hand, in Example \ref{ex:2.E} we saw that $\mu_{\sigma,1} = 2$ also occurs.

\begin{work}\label{conj:10.1}
We are given $(\sigma,E) \in \Irr_\cusp (M)$ such that the restriction
of $\sigma$ to $M^1$ is multiplicity free, that is, $\mu_{\sigma,1} = 1$.
\end{work}

We stress that this hypothesis does not imply that the restriction of $(\sigma,E)$ to the
derived subgroup $M_\der$ is multiplicity-free. Indeed, a counterexample to that can be found in
\cite[\S 7]{AdPr}.

As Working hypothesis \ref{conj:10.1} constitutes a major difference between the assumptions 
from \cite{Hei2} and from this paper, it should not come as a surprise that it makes better results
possible. We will investigate how far we can get, following \cite{Hei2} rather closely. So, in the 
remainder of this section we assume Working hypothesis \ref{conj:10.1} (then it also holds for 
$\sigma \otimes \chi$ with $\chi \in X_\nr (M)$ because $M^1 \subset \ker \chi$).

From \eqref{eq:2.9} we know that $M_\sigma^2, M_\sigma^3$ and $M_\sigma^4$ now coincide to a single
group, which we call $M_\sigma$. Lemma \ref{lem:2.3} reduces to
\[
\End_M \big( \ind_{M^1}^M (\sigma_1,E_1) \big) = \C [ M_\sigma / M^1] = \C [ \mc O_3] .
\]
Further $X_\nr (M,\sigma) = \Irr (M / M_\sigma)$ and 
\eqref{eq:2.28}--\eqref{eq:2.12} provide a group homomorphism
\[
X_\nr (M,\sigma) \to \mr{Aut}_M (E_B) : \chi_c \mapsto \phi_{\chi_c} .
\]
In particular $\natural \big|_{X_\nr (M,\sigma)} = 1$ and \eqref{eq:2.36} simplifies to
\[
\End_M (E_B) = \C [X_\nr (M)] \rtimes X_\nr (M,\sigma) .
\]

\subsection{The non-cuspidal case} \

As before, we also write $\phi_{\chi_c}$ for $I_P^G (\phi_{\chi_c}) \in \mr{Aut}_G (I_P^G (E_B))$.
Then \eqref{eq:2.32} implies 
\begin{equation}\label{eq:10.1}
(I_P^G (E_B) )^{X_\nr (M,\sigma)} = I_P^G \big( (E_B)^{X_\nr (M,\sigma)} \big) = 
I_P^G (\ind_{M^1}^M (E_1)) . 
\end{equation}
By \eqref{eq:2.15} the progenerator $I_P^G (E_B)$ of $\Rep (G)^{\mf s}$ is isomorphic to
$I_P^G (\ind_{M^1}^M (E_1))^{[M : M_\sigma]}$. In particular $I_P^G (\ind_{M^1}^M (E_1))$ is also
a progenerator of $\Rep (G)^{\mf s}$ and 
\begin{equation}\label{eq:10.2}
\begin{aligned}
& \End_G (I_P^G (E_B)) \cong \End_G \big( I_P^G (\ind_{M^1}^M (E_1)) \big) \otimes_\C 
M_{[M : M_\sigma]}(\C) , \\
& \End_G (I_P^G (E_B))^{X_\nr (M,\sigma) \times X_\nr (M,\sigma)} = 
\End_G \big( I_P^G (\ind_{M^1}^M (E_1)) \big) .
\end{aligned}
\end{equation}
Here we use that $X_\nr (M,\sigma)$ embeds in $\mr{Aut}_G (I_P^G (E_B))$, which acts from the left
and from the right on $\End_G (I_P^G (E_B))$.

Under Working hypothesis \ref{conj:10.1}, \cite[\S 4.1--4.4]{Hei2} holds for $G,M,\sigma$.

\begin{lem}\label{lem:10.2}
Assume Working hypothesis \ref{conj:10.1} and let $w \in W(M,\mc O)$. 
\enuma{
\item There exists a $m_w \in M$, unique up to $M_\sigma$, such that:
\begin{itemize}
\item $\sigma (m_w) \rho_{\sigma,w} E_1 = E_1$,
\item $b_{m_w} A_w$ stabilizes $I_P^G (E_{K(B)})^{X_\nr (M,\sigma)}$.
\end{itemize}
\item For $s_\alpha$ with $\alpha \in \Sigma_{\mc O,\mu}$, we can take $m_{s_\alpha} \in M \cap 
M^1_\alpha$, and then there is a canonical choice up to $M^1$.
}
\end{lem}
\begin{proof}
(a) Working hypothesis \ref{conj:10.1} entails that there is a $m_w \in M$, unique up to $M_\sigma$,
which fulfills the first bullet. The proof of \cite[Lemme 4.5]{Hei2} then shows the second bullet.\\
(b) The element $m_{s_\alpha} \in M \cap M^1_\alpha$ comes from the proof of \cite[Lemme 4.5]{Hei2}.
As $M \cap M^1_\alpha / M^1 \cong \Z$, there is a unique choice $m_{s_\alpha} M^1$ such that
$\nu_F (\alpha (m_{s_\alpha}))$ is positive and minimal.
\end{proof}

As in \cite[\S 4.6]{Hei2} we define $J_{s_\alpha} = b_{m_{s_\alpha}} A_{s_\alpha}$ for 
$\alpha \in \Delta_{\mc O,\mu}$, but not for other $w \in W(\Sigma_{\mc O,\mu})$. 
By \eqref{eq:4.2} and Proposition \ref{prop:4.1}.c:
\begin{equation}\label{eq:10.3}
J^2_{s_\alpha} = b_{m_{s_\alpha}} (s_\alpha \cdot b_{m_{s_\alpha}}) A^2_{s_\alpha} =
A_{s_\alpha}^2 = \frac{4 c'_{s_\alpha}}{(1 - q_\alpha^{-1} )^2 
(1 + q_{\alpha*}^{-1} )^2 \mu^{M_\alpha}(\sigma \otimes \cdot)} .
\end{equation}
Let $w \in W(\Sigma_{\mc O,\mu})$, with a reduced expression $w = s_1 s_2 \cdots s_{\ell (w)}$.
Then \cite[Lemme 4.7.i]{Hei2} shows: the operator
\begin{equation}\label{eq:10.4}
J_w := J_{s_1} J_{s_2} \cdots J_{s_{\ell (w)}}
\end{equation}
depends only on $w$, not on the chosen reduced expression. For $r \in R(\mc O)$ we define
\[
J_{rw} := b_{m_r} A_r J_w = J_r J_w .
\]
By Lemma \ref{lem:10.2} and \eqref{eq:4.6}, $J_r$ belongs to $\End_G (I_P^G (E_B)^{X_\nr (M,\sigma)})$.
On the other hand, the $J_w$ with $w \in W(\Sigma_{\mc O,\mu}) \setminus \{1\}$ have poles, just like
the $A_w$. The multiplication rules for these elements are similar to Proposition \ref{prop:4.2}:

\begin{lem}\label{lem:10.3}
Assume Working hypothesis \ref{conj:10.1}, 
let $r,r_1,r_2 \in R(\mc O)$ and $w \in W(\Sigma_{\mc O,\mu})$.
\enuma{
\item Write $\chi (r_1,r_2) = \chi_{r_1} r_1 (\chi_{r_2}) \chi_{r_1 r_2}^{-1} \in X_\nr (M,\sigma)$.
There exists a $\natural_J (r_1,r_2) \in \C [X_\nr (M) / X_\nr (M,\sigma)]^\times \cong
\C^\times \times M_\sigma / M^1$ such that 
\[
J_{r_1} \circ J_{r_2} = \natural_J (r_1,r_2) \phi_{\chi (r_1,r,2)} \circ J_{r_1 r_2} .
\]
\item There exists a $\natural_J (w,r) \in \C [X_\nr (M) / X_\nr (M,\sigma)]^\times$ such that
\[
J_w \circ J_r = \natural (w,r) \phi_{w(\chi_r^{-1}) \chi_r} \circ J_{wr} .
\]
}
\end{lem}
\begin{proof}
(a) By \eqref{eq:4.2} and Proposition \ref{prop:4.2}.a 
\begin{equation}\label{eq:10.5}
\begin{aligned}
J_{r_1} J_{r_2} & = b_{m_{r_1}} A_{r_1} b_{m_{r_2}} A_{r_2} \; = \; b_{m_{r_1}} (r_1 \cdot b_{m_{r_2}}
)_{\chi_{r_1}^{-1}} A_{r_1} A_{r_2} \\
& = b_{m_{r_1}} (r_1 \cdot b_{m_{r_2}}) (r_1 \cdot b_{m_{r_2}})(\chi_{r_1}^{-1}) A_{r_1} A_{r_2} \\
& = b_{m_{r_1}} b_{\tilde r_1 m_{r_2} \tilde r_1^{-1}}  b_{m_{r_2}}(r_1^{-1} \chi_{r_1}^{-1})
\natural (r_1,r_2) \phi_{\chi (r_1,r_2)} A_{r_1 r_2} .
\end{aligned}
\end{equation}
On the other hand, by \eqref{eq:2.26} 
\begin{multline}\label{eq:10.6}
\phi_{\chi (r_1,r_2)} J_{r_1 r_2} = \phi_{\chi (r_1,r_2)} b_{m_{r_1 r_2}} A_{r_1 r_2} \\
= (b_{m_{r_1 r_2}} )_{\chi (r_1,r_2)^{-1}} \phi_{\chi (r_1,r_2)} A_{r_1 r_2} =
b_{m_{r_1 r_2}} \chi (r_1,r_2)^{-1} (m_{r_1 r_2}) \phi_{\chi (r_1,r_2)} A_{r_1 r_2} .
\end{multline}
Inserting \eqref{eq:10.6} into \eqref{eq:10.5}, we obtain
\begin{equation}\label{eq:10.7}
b_{m_{r_1}} b_{\tilde r_1 m_{r_2} \tilde r_1^{-1}} b_{m_{r_1 r_2}}^{-1} \chi (r_1,r_2) (m_{r_1 r_2}) 
(r_1^{-1} \chi_{r_1}^{-1})(m_{r_2}) \natural (r_1,r_2) \phi_{\chi (r_1,r_2)} J_{r_1 r_2} .
\end{equation}
This element stabilizes $I_P^G (E_{K(B)})^{X_\nr (M,\sigma)}$, so by the uniqueness of $m_w$
(up to $M_\sigma$):
\[
m := m_{r_1} \tilde r_1 m_{r_2} \tilde r_1^{-1} m_{r_1 r_2}^{-1} \text{ lies in } M_\sigma .
\]
The three middle terms in \eqref{eq:10.7} are nonzero scalars, so \eqref{eq:10.7} is of the form\\
$b_m z \phi_{\chi (r_1,r_2)} J_{r_1 r_2}$ for some $z \in \C^\times$. \\
(b) This can be derived from Proposition \ref{prop:4.2}.c, analogous to part (a).
\end{proof}

It is clear from Theorem \ref{thm:4.3} that $\{ J_w \phi_{\chi_c} : w \in W(M,\mc O), \chi_c \in 
X_\nr (M,\sigma) \}$ is a $K(B)$-basis of $\Hom_G \big( I_P^G (E_B), I_P^G (E_{K(B)}) \big)$.
Consider the idempotent
\[
p_{X_\nr (M,\sigma)} := |X_\nr (M,\sigma) |^{-1} \sum\nolimits_{\chi_c \in X_\nr (M,\sigma)} 
\phi_{\chi_c} \; \in \End_G (I_P^G (E_B)) .
\]
It satisfies $p_{X_\nr (M,\sigma)} I_P^G (E_B) = I_P^G (E_B)^{X_\nr (M,\sigma)}$ and
\begin{multline}\label{eq:10.8}
p_{X_\nr (M,\sigma)} \Hom_G \big( I_P^G (E_B), I_P^G (E_{K(B)}) \big) p_{X_\nr (M,\sigma)} = \\
\Hom_G \big( I_P^G (E_B)^{X_\nr (M,\sigma)}, I_P^G (E_{K(B)})^{X_\nr (M,\sigma)} \big) .
\end{multline}

\begin{prop}\label{prop:10.4}
Assume Working hypothesis \ref{conj:10.1}. As vector spaces over\\ 
$K(B)^{X_\nr (M,\sigma)} = \C (X_\nr (M) / X_\nr (M,\sigma))$:
\[
\Hom_G \big( I_P^G (E_B)^{X_\nr (M,\sigma)}, I_P^G (E_{K(B)})^{X_\nr (M,\sigma)} \big) =
\bigoplus_{w \in W(M,\mc O)} K(B)^{X_\nr (M,\sigma)} J_w p_{X_\nr (M,\sigma)} .
\]
\end{prop}
\begin{proof}
By Lemma \ref{lem:10.2}.a, each $J_w$ defines an element $J_w p_{X_\nr (M,\sigma)}$ of\\
$\Hom_G \big( I_P^G (E_B)^{X_\nr (M,\sigma)}, I_P^G (E_{K(B)})^{X_\nr (M,\sigma)} \big)$.
By Theorem \ref{thm:4.3} these elements are linearly independent over $K(B)$ and over
$K(B)^{X_\nr (M,\sigma)}$. This proves the inclusion $\supset$ of the proposition.

Further, Lemma \ref{lem:10.2}.a also shows that, for every $\chi_c \in X_\nr (M,\sigma)$,
\[
J_w p_{X_\nr (M,\sigma)} = p_{X_\nr (M,\sigma)} J_w p_{X_\nr (M,\sigma)} = 
\phi_{\chi_c} J_w p_{X_\nr (M,\sigma)}.
\] 
Since $J_w$ is invertible in $\End_G (I_P^G (E_{K(B)}))$, this also equals 
\begin{equation}\label{eq:10.9}
p_{X_\nr (M,\sigma)} J_w = p_{X_\nr (M,\sigma)} J_w \phi_{\chi_c} .
\end{equation}
With Theorem \ref{thm:4.3} we deduce
\begin{equation}\label{eq:10.10} 
\begin{aligned}
& \Hom_G \big( I_P^G (E_B)^{X_\nr (M,\sigma)}, I_P^G (E_{K(B)})^{X_\nr (M,\sigma)} \big) = \\
& \Big( \bigoplus\nolimits_{w \in W(M,\mc O)} \bigoplus\nolimits_{\chi_c \in X_\nr (M,\sigma)} 
K(B) \phi_{\chi_c} J_w \Big)^{X_\nr (M,\sigma) \times X_\nr (M,\sigma)} = \\
& \bigoplus\nolimits_{w \in W(M,\mc O)} \Big( p_{X_\nr (M,\sigma)} \bigoplus\nolimits_{\chi_c 
\in X_\nr (M,\sigma)} K(B) \phi_{\chi_c} J_w p_{X_\nr (M,\sigma)} \Big) = \\
& \bigoplus\nolimits_{w \in W(M,\mc O)} p_{X_\nr (M,\sigma)} K(B)  p_{X_\nr (M,\sigma)} J_w . 
\end{aligned}
\end{equation}
From \eqref{eq:2.26} we see that 
\[
p_{X_\nr (M,\sigma)} K(B) p_{X_\nr (M,\sigma)} = K(B)^{X_\nr (M,\sigma)} p_{X_\nr (M,\sigma)}.
\]
Combine that with \eqref{eq:10.10} and \eqref{eq:10.9}.
\end{proof}

As $\Hom_G ( I_P^G (E_B), I_P^G (E_{K(B)})$ is an algebra, so the vector space in 
Proposition \ref{prop:10.4}. The multiplication relations between the elements $J_w$ with 
$w \in W(\Sigma_{\mc O,\mu})$ are similar to those in Proposition \ref{prop:4.1}.a, but with some 
extra factors $b_m$ inserted. The relations between the other $J_w$ with $w \in W(M,\mc O)$ are
given in Lemma \ref{lem:10.3} (although in part (a) we do not need $\phi_{\chi (r_1,r_2)}$ because
it is the identity on $I_P^G (E_{K(B)})$). Like in Paragraph \ref{par:Tw}, we can simplify these
multiplication relations by replacing the $J_w$ with slightly different operators.

For $\alpha \in \Delta_{\mc O,\mu}$ we define 
\begin{equation}\label{eq:10.11}
\mc T'_{s_\alpha} = g_\alpha J_{s_\alpha} .
\end{equation}
Then Proposition \ref{prop:4.6} holds for these $\mc T'_{s_\alpha}$, with the same proof:
\eqref{eq:10.11} extends to a group homomorphism
\begin{equation}\label{eq:10.12}
W(\Sigma_{\mc O,\mu}) \to \mr{Aut}_G \big( I_P^G (E_{K(B)})^{X_\nr (M,\sigma)} \big) :
w \mapsto \mc T'_w .
\end{equation}
Now we can replay the proof of Lemma \ref{lem:4.7}.b for the $\mc T'_w J_r$, using Lemma 
\ref{lem:10.3}.b as a substitute for Proposition \ref{prop:4.2}.c. That leads to:
\begin{equation}\label{eq:10.13} 
\mc T'_w J_r = \natural_J (w,r) J_r \mc T'_{r^{-1} w r} \qquad 
w \in W (\Sigma_{\mc O,\mu}), r \in R(\mc O) ,
\end{equation}
where $\natural_J (w,r) \in \C [X_\nr (M) / X_\nr (M,\sigma)]^\times$ is as in Lemma \ref{lem:10.3}.
Recall that 
\[
X_\nr (M) / X_\nr (M,\sigma) \to \mc O : \chi \mapsto \sigma \otimes \chi
\]
is a bijection and that $W(M,\mc O)$ acts naturally on $\mc O$ (which induces an action on
$\C [\mc O]$).

\begin{thm}\label{thm:10.5}
Assume Working hypothesis \ref{conj:10.1}.
\enuma{
\item The set $\{ J_r \mc T'_w : r \in R(\mc O), w \in W(\Sigma_{\mc O,\mu}) \}$ is a
$\C (\mc O)$-basis of \\
$\Hom_G \big( I_P^G (E_B)^{X_\nr (M,\sigma)}, I_P^G (E_{K(B)})^{X_\nr (M,\sigma)} \big)$.
For any $b \in \C (\mc O)$:
\[
J_r \mc T'_w \circ b = (rw \cdot b) \circ J_r \mc T'_w .
\]
\item There exists a 2-cocycle $\natural_J : W(M,\mc O)^2 \to \C[\mc O]^\times$ (with respect
to the action of $W(M,\mc O)$ on $\C[\mc O]$) such that, for all 
$r_1, r_2 \in R (\mc O), w_1, w_2 \in W(\Sigma_{\mc O,\mu})$:
\[
J_{r_1} \mc T'_{w_1} J_{r_2} \mc T'_{w_2} =
\natural_J (r_1 w_1, r_2 w_2) J_{r_1 r_2} \mc T'_{r_2^{-1} w_1 r_2 w_2} .
\] 
When $r_2 = 1$, this simplifies to $\natural_J (r_1 w_1,w_2) = 1$.
}
\end{thm}
\begin{proof}
(a) The expression for $J_r \mc T'_w b$ follows from \eqref{eq:4.16} and \eqref{eq:4.30}. This
implies that
\[
J_r \mc T'_w (J_r J_w )^{-1} \in \C [\mc O]^\times \qquad
\forall r \in R(\mc O), w \in W(\Sigma_{\mc O,\mu}) .
\]
Then Proposition \ref{prop:10.4} says that the $J_r \mc T'_w$, just like the $J_r J_w$, form
a $\C (\mc O)$-basis.\\
(b) With \eqref{eq:10.12}, \eqref{eq:10.13} and Lemma \ref{lem:10.3}.a we compute
\begin{multline*}
J_{r_1} \mc T'_{w_1} J_{r_2} \mc T'_{w_2} = J_{r_1} \natural (w_1,r_2) J_{r_2} 
\mc T'_{r_2^{-1} w_1 r_2} \mc T'_{w_2} = (r_1 \cdot \natural_J (w_1,r_2) ) J_{r_1} J_{r_2}
\mc T'_{r_2^{-1} w_1 r_2 w_2} \\
= (r_1 \cdot \natural_J (w_1,r_2) ) \natural_J (r_1,r_2) J_{r_1 r_2} \mc T'_{r_2^{-1} w_1 r_2 w_2} .
\end{multline*}
This means that we must define 
\[
\natural_J (r_1 w_1,r_2 w_2) = (r_1 \cdot \natural_J (w_1,r_2) ) \natural_J (r_1,r_2) .
\]
The above computation with $r_2 = 1$ shows that $\natural_J (r_1 w_1,w_2) = 1$.
In view of the associativity of $\End_G \big( I_P^G (E_{K(B)})^{X_\nr (M,\sigma)} \big)$,
we can work out the product
\[
J_{r_1} \mc T'_{w_1} \circ J_{r_2} \mc T'_{w_2} \circ J_{r_3} \mc T'_{w_3}
\]
in two equivalent ways. Comparing the resulting expressions in
\[
\C [\mc O]^\times J_{r_1 r_2 r_3} \mc T'_{r_3^{-1} r_2^{-1} w_1 r_3^{-1} w_2 r_3 w_3} ,
\]
we deduce that $\natural_J$ is a 2-cocycle.
\end{proof}

Next we aim for a version of Theorem \ref{thm:10.5} with regular (instead of merely rational)
functions on $\mc O$. As $\mc T'_{s_\alpha} = b_{m_{s_\alpha}} \mc T_{s_\alpha}$ and
$b_{m_{s_\alpha}} \in \C [\mc O]^\times$, the poles of $\mc T'_{s_\alpha}$ are the same as the
poles of $\mc T_{s_\alpha}$. We know those from \eqref{eq:4.26}: they are simple and occur at 
\begin{equation}\label{eq:10.15}
\{ X_\alpha = q_\alpha \} \quad \text{ and \; (if } b_{s_\alpha} > 0 
\text{) at } \{ X_\alpha = -q_{\alpha*} \}.
\end{equation}
Since $X_\alpha = b_{h_\alpha^\vee}$ and $h_\alpha^\vee$ is indivisible in $M_\sigma / M^1$,
any set of the form 
\[
\{ \sigma' \in \mc O : X_\alpha (\sigma') = \text{constant} \}
\]
is connected. By Lemma \ref{lem:3.6} and the same indivisibility of $h_\alpha^\vee$, $s_\alpha$
pointwise fixes $\{ \sigma' \in \mc O : X_\alpha (\sigma') = 1 \}$ and (if $b_{s_\alpha} > 0$)
$\{ \sigma' \in \mc O : X_\alpha (\sigma') = -1 \}$.

For $\sigma' \in \mc O$ and $v \in I_P^G (E_B )^{X_\nr (M,\sigma)} = I_P^G (\ind_{M^1}^M (E_1))$
we define
\[
\mr{sp}_{\sigma'} (v) = |X_\nr (M,\sigma)|^{-1} \sum\nolimits_{\chi_c \in X_\nr (M,\sigma)}
\phi_{\chi_c} (\mr{sp}_\chi (v)) ,
\]
for any $\chi \in X_\nr (M,\sigma)$ with $\sigma' \cong \sigma \otimes \chi$.
By the invariance property of $v$, this does not depend on the choice of $\chi$. In view of
\eqref{eq:2.28}, it is a well-defined map
\[
\mr{sp}_{\sigma'} : I_P^G (E_B )^{X_\nr (M,\sigma)} \to E_1 .
\] 

\begin{lem}\label{lem:10.6}
Assume Working hypothesis \ref{conj:10.1}. Let $\alpha \in \Delta_{\mc O,\mu}$ and 
$\sigma_+, \sigma_- \in \mc O$ with $X_\alpha (\sigma_{\pm}) = \pm 1$.
\enuma{
\item $\mr{sp}_{\sigma_+} (1 - X_\alpha) J_{s_\alpha} = \mr{sp}_{\sigma_+} \mc T'_{s_\alpha} =
\mr{sp}_{\sigma_+}$.
\item Suppose that $b_{s_\alpha} > 0$. There exists $\epsilon_\alpha \in \{0,1\}$ such that
$\mr{sp}_{\sigma_-} \mc T'_{s_\alpha} = (-1)^{\epsilon_\alpha} \mr{sp}_{\sigma_-}$.
}
\end{lem}
\begin{proof}
(a) From $\mc T'_{s_\alpha} = g_\alpha J_{s_\alpha}$ and the definition of $g_\alpha$ we see that 
\[
\mr{sp}_{\sigma_+} \mc T'_{s_\alpha} = \mr{sp}_{\sigma_+} (1 - X_\alpha) J_{s_\alpha} .
\]
By Lemma \ref{lem:6.1}, $\mr{sp}_{\chi_+} (1 - X_\alpha) A_{s_\alpha} = \mr{sp}_{\chi_+}$ for all 
$\chi_+ \in X_\nr (M_\alpha)$. From Lemma \ref{lem:10.2}.b we see that
also $\mr{sp}_{\chi_+} (1 - X_\alpha) J_{s_\alpha} = \mr{sp}_{\chi_+}$. As
\[
\{ \chi \in X_\nr (M) : X_\alpha (\chi) \} \to \{ \sigma' \in \mc O : X_\alpha (\sigma') = 1 \}
\]
is a covering, there exists a $\chi_+ \in X_\nr (M_\alpha)$ with $\sigma \otimes \chi_+ \cong 
\sigma_+$. Consequently
\[
\mr{sp}_{\sigma_+} (1 - X_\alpha) J_{s_\alpha} = |X_\nr (M,\sigma) |^{-1} \sum\nolimits_{
\chi_c \in X_\nr (M,\sigma)} \phi_{\chi_c} \circ \mr{sp}_{\chi_+} = \mr{sp}_{\chi_+} .
\]
(b) Pick $\chi_- \in X_\nr (M)$ with $\sigma \otimes \chi_- \cong \sigma_-$. From Lemma \ref{lem:6.1}
we know that
\begin{equation}\label{eq:10.14}
\mr{sp}_{\chi_-} \mc T'_{s_\alpha} = z \mr{sp}_{\chi_-} \phi_{\chi_- s_\alpha (\chi_-^{-1})}
\end{equation}
for some $z \in \C^\times$. On $I_P^G (E_B)^{X_\nr (M,\sigma)}$ the operator 
$\phi_{\chi_- s_\alpha (\chi_-^{-1})}$ acts as the identity, so there \eqref{eq:10.14} reduces to
$\mr{sp}_{\chi_-} \mc T'_{s_\alpha} = z \mr{sp}_{\chi_-}$. Using $(\mc T'_{s_\alpha})^2 = 1$,
we deduce $z = \pm 1$. By definition, that means
\[
\mr{sp}_{\sigma_-} \mc T'_{s_\alpha} = \pm \mr{sp}_{\sigma_-} .
\]
Since $\{ \sigma' \in \mc O : X_\alpha (\sigma') = -1 \}$ is connected and this sign $\pm$
depends continuously on $\sigma'$, the sign is a constant $(-1)^{\epsilon_\alpha}$.
\end{proof}

Recall the elements $f_\alpha \in \C (\mc O)$ for $\alpha \in \Sigma_{\mc O,\mu}$ from \eqref{eq:4.23}.
One readily checks that
\[
(1 + f_{-\alpha}) \mc T'_{s_\alpha} = (q_\alpha - X_\alpha)(q_{\alpha*} + X_\alpha) J_{s_\alpha} / 2.
\]
For $\alpha \in \Delta_{\mc O,\mu}$ we define
\[
T'_{s_\alpha} = (1 + f_{-\alpha}) X_\alpha^{\epsilon_\alpha} \mc T'_{s_\alpha} + f_\alpha =
(X_\alpha^{\epsilon_\alpha} \mc T'_{s_\alpha} + 1)(1 + f_\alpha) - 1 .
\]

\begin{lem}\label{lem:10.7}
Assume Working hypothesis \ref{conj:10.1}. For all $\alpha \in \Delta_{\mc O,\mu}$, 
$T'_{s_\alpha}$ lies in $\End_G \big( I_P^G (E_B)^{X_\nr (M,\sigma)} \big)$.
\end{lem}
\begin{proof}
By \eqref{eq:10.15} the operators $(q_\alpha - X_\alpha)(q_{\alpha*} + X_\alpha) J_{s_\alpha}/2, 
f_\alpha$ and $T'_{s_\alpha}$ can only have poles at $\{ \sigma' \in \mc O : X_\alpha (\sigma') = 
\pm 1 \}$. Select $\sigma_\pm \in \mc O$ with $X_\alpha (\sigma_\pm ) = \pm 1$. 
By Lemma \ref{lem:10.6}, as operators on $I_P^G (E_B)^{X_\nr (M,\sigma)}$:
\[
\begin{array}{ccccc}
\! \mr{sp}_{\sigma_+} (1 - X_\alpha) (1 + f_{-\alpha}) X_\alpha^{\epsilon_\alpha} \mc T'_{s_\alpha} & 
\! = \! & - \mr{sp}_{\sigma_+} (1 - X_\alpha) f_\alpha & \! = \! & 
\mr{sp}_{\sigma_+} (q_\alpha - 1) (q_{\alpha *} + 1) /2 \\
\! \mr{sp}_{\sigma_-} (1 + X_\alpha) (1 + f_{-\alpha}) X_\alpha^{\epsilon_\alpha} \mc T'_{s_\alpha} & 
\! = \! & - \mr{sp}_{\sigma_-} (1 + X_\alpha) f_\alpha & \! = \! & 
\mr{sp}_{\sigma_-} (q_\alpha + 1) (q_{\alpha *} - 1) /2 
\end{array}
\]
Hence $\mr{sp}_{\sigma_+} (1 - X_\alpha) T'_{s_\alpha} = 0$ and 
$\mr{sp}_{\sigma_-} (1 + X_\alpha) T'_{s_\alpha} = 0$, which means that $T'_{s_\alpha}$ does not have
any poles and sends $I_P^G (E_B)^{X_\nr (M,\sigma)}$ to itself.
\end{proof}

In view of Theorem \ref{thm:10.5}.a, the proof of Lemma \ref{lem:6.3} remains valid in the current
setting, and shows that
\[
(T'_{s_\alpha} + 1)(T'_{s_\alpha} - q_\alpha q_{\alpha*}) = 0 .
\]
Similarly \eqref{eq:6.9} holds for $T'_{s_\alpha}$ and $b \in \C (\mc O)$.
Next, the proof of Lemma \ref{lem:6.2}.a also holds in our setting, if we take $u=1$ and replace 
$X_\nr (M)$ by $\mc O$. That provides a consistent definition of $T'_w$ for $w \in W(\Sigma_{\mc O,\mu})$.
The multiplication relations between these $T'_w$ and the $J_r$ with $r \in R(\mc O)$ are not as nice
as in Lemma \ref{lem:6.2}.b, because Lemma \ref{lem:10.3} is weaker than Proposition \ref{prop:4.2}.

Fortunately \cite[\S 5.7--5.12]{Hei2} still works for the elements $J_r T'_w$. 

\begin{thm}\label{thm:10.8}
Assume Working hypothesis \ref{conj:10.1}. The set $\{ J_r T'_w : r \in R(\mc O), w \in 
W(\Sigma_{\mc O,\mu}) \}$ is a $\C [\mc O]$-basis of $\End_G \big( I_P^G (E_B)^{X_\nr (M,\sigma)} \big)$.
The subalgebra 
\[
\bigoplus\nolimits_{w \in W(\Sigma_{\mc O,\mu})} \C [\mc O] T'_w
\]
is canonically isomorphic to the affine Hecke algebra $\mc H (\mc O,G)$ from \eqref{eq:9.4}.
\end{thm}
\begin{proof}
The first claim is \cite[Th\'eor\`eme 5.10]{Hei2} in our generality. The second claim follows
from the proof of Lemma \ref{lem:6.2}.a. It generalizes \cite[Proposition 7.4]{Hei2}.
\end{proof}

From Theorem \ref{thm:10.5} we see that conjugation by $J_r$ stabilizes the subalgebra 
$\mc H (G,\mc O)$. However, it seems that $J_r^{-1} T'_w J_r$ need not be contained in
$\C T'_{r^{-1} w r}$, or even in $\C [\mc O] T'_{r^{-1} w r}$. Therefore the main conclusions
(\S 7.6--7.8) of \cite{Hei2} do not necessarily hold in our setting.

Recall from \eqref{eq:4.equiv} and \eqref{eq:10.2} that 
\begin{equation}\label{eq:10.16}
\End_G \big( I_P^G (E_B)^{X_\nr (M,\sigma)} \big) - \Mod \; = \;
\End_G \big( I_P^G (\ind_{M^1}^M (E_1)) \big) - \Mod
\end{equation}
is naturally equivalent with $\Rep (G)^{\mf s}$ and with $\End_G (I_P^G (E_B)) -\Mod$. 
However, with only the above description available, deriving the representation theory of 
$\End_G \big( I_P^G (E_B)^{X_\nr (M,\sigma)} \big)$ from that of $\mc H (\mc O,G)$ is quite involved.
One problem is that Clifford theory (as in the proof of Lemma \ref{lem:8.7}) is not available if
the values of $\natural_J$ are not central in $\mc H (\mc O,G)$. A way around that is
by localization on subsets of $\mc O$, as in Section \ref{sec:localization}. In fact any $U_u 
\subset X_\nr (M)$ satisfying Condition \ref{cond:7.U} is diffeomorphic to its image in $\mc O$,
and the analytic localization of $\End_G \big( I_P^G (E_B)^{X_\nr (M,\sigma)} \big)$ on $U_u$ is
canonically isomorphic to $1_u \End_G (I_P^G (E_B))_U^\an 1_u$. In this way the results from
Sections \ref{sec:localization}--\ref{sec:class} provide an analysis of \eqref{eq:10.16} in terms
of a family of (twisted) graded Hecke algebras.

For aspects of \eqref{eq:10.16} that cannot be translated to graded Hecke algebras, like those
involving modules of infinite length, one must understand the 2-cocycle
\[
\natural_J : W(M,\mc O )^2 \to \C [\mc O]^\times
\] 
well. In practice that means one needs either irreducibility of $\Res_{M^1}^M (E)$ (for 
then we can take $m_w = 1$ for all $w$) or an explicit easy description of $R (\mc O)$, like for 
classical groups in \cite[Proposition 1.15 and \S 2.5]{Hei2}. Here we must warn the reader that 
exactly this aspect of \cite{Hei2} is incomplete, it was corrected in \cite[Appendix A]{Hei4}.

\appendix
\section{Correction (from 2023)}

\renewcommand{\theequation}{\Alph{section}.\arabic{equation}}

In the recent preprint \cite{Oha}, the algebras $\End_G (I_P^G (E_B))$ are compared with 
similar endomorphism algebras from \cite{Mor1}. The main results of \cite{Oha} made the author
realize that there could be an issue with the preservation of temperedness proven in
Theorem \ref{thm:9.2}. Further investigations revealed that there is indeed a problem, and that
it stems from \cite{Hei3}. In this appendix we explain the problem and we show how it can be
fixed.

By Definition \ref{def:temp}, a finite dimensional $\End_G (I_P^G (E_B))$-module $V$ is tempered
if all its $\C [X_\nr (M)]$-weights $t \in X_\nr (M)$ satisfy the following condition:
\begin{equation}\label{eq:A.1}
\log |t| \in \Hom (M_\sigma^2 / M^1,\R) \text{ lies in the negative cone } 
\big\{ \sum_{\alpha^\sharp \in \Delta_{\mc O}\!\!\!} \! c_\alpha \alpha^\sharp : 
c_\alpha \in \R_{\leq 0} \big\} .
\end{equation}
Equivalently, $|t(m)| \leq 1$ whenever $m \in M_\sigma^2 / M^1$ lies in the positive Weyl chamber
with respect to the simple roots $h_\alpha^\vee \in \Delta_{\mc O}^\vee$. Here the bases 
$\Delta_{\mc O}$ of $\Sigma_{\mc O}$ and $\Delta_{\mc O}^\vee$ of $\Sigma_{\mc O}^\vee$ are
determined by $P$. This is the same notion of temperedness as commonly used for affine Hecke 
algebras, e.g. in \cite{Opd-Sp,SolAHA}.

According to Theorem \ref{thm:9.2}, the equivalence of categories 
\[
\Rep_f (G)^{\mf s} \cong \End_G (I_P^G (E_B)) -\Mod_f
\]
preserves temperedness. The proof proceeds by reduction to results of \cite{SolComp}. One of
the conditions needed to apply \cite{SolComp} is that the definition of positivity for roots
in $\Sigma_{\mc O}^\vee$ corresponds (via $X_\nr (M)$ which is present on both sides) to the 
definition of positivity for roots in $\Sigma_\red (A_M)$. This is checked in Lemma \ref{lem:9.1}
and shortly before Proposition \ref{prop:9.3}. 

Unfortunately, the comparison between these two root systems is made in the wrong way. Namely,
for a positive root $\alpha \in \Sigma_\red (A_M)$, $\nu_F \circ \alpha \in \mf a_M^*$ is 
regarded as positive. That results in a definition of $h_\alpha^\vee$ as element of
$\R_{>0} \alpha^\vee (\varpi_F) \cap M_\sigma^2 / M^1$. This use of $\nu_F \circ \alpha$ as
positive element stems from \cite{Hei2,Hei3}, on which Section \ref{sec:temp} is partly based.
As a consequence Theorem \ref{thm:9.2} and \cite[Th\'eor\`eme 5]{Hei3} suffer from the same
problem: they do not send tempered $G$-representations to tempered $\End_G (I_P^G (E_B))
$-representations, but to anti-tempered $\End_G (I_P^G (E_B))$-representations. Here
anti-tempered means that the above condition \eqref{eq:A.1} for temperedness is replaced by
\[
\log |t| \in \Hom (M_\sigma^2 / M^1 ,\R) \quad \text{lies in} \quad \big\{ \sum\nolimits_{
\alpha^\sharp \in \Delta_{\mc O}} d_\alpha \alpha^\sharp : d_\alpha \in \R_{\geq 0} \big\} .
\]
Let us work out explicitly why this is the case, in contrast to the statements of
Theorem \ref{thm:9.2} and \cite[Th\'eor\`eme 5]{Hei3}. Consider a root $\alpha \in \Sigma_\red
(A_M)$ which is positive with respect to $P$. Then $|\alpha|_F \in X_\nr (M)$ is an unramified
character in positive position with respect to $P$, as used in Casselman's criteria for
temperedness of $G$-representations \cite[Proposition III.2.2]{Wal}. Since $|\alpha (\alpha^\vee
(\varpi_F))|_F < 1$, it follows that $|\alpha |_F$ becomes a negative multiple of $\alpha^\sharp$
in \cite{Hei2} and in Section \ref{sec:root}. That is off by a minus sign from what is needed
in \cite{SolComp}, and therefore the ``preservation of temperedness" in Theorem \ref{thm:9.2} 
and \cite[Th\'eor\`eme 5]{Hei3} relates to the preservation of temperedness in 
\cite[\S 4.2]{SolComp} by inserting an extra minus sign in the criterion for temperedness.

For example, consider the Steinberg representation St of $G = GL_2 (F)$, with $P$ the standard
Borel subgroup and $\alpha$ the unique positive root of the diagonal torus $T$. Then
$\End_G (I_P^G (E_B))$ is an affine Hecke algebra of type $GL_2$ \cite[\S 2.5]{SolHecke} and 
\[
J^G_{\overline P} (\mr{St}) = \delta_{\overline P}^{1/2} = |\alpha |_F^{-1/2} .
\]
Let $V$ be the $\End_G (I_P^G (E_B))$-module which corresponds to St under the equivalence
of categories \eqref{eq:4.equiv}. By Proposition \ref{prop:4.5}.b, $J^G_{\overline P} (\mr{St})$
corresponds to $\mr{Res}^{\End_G (I_P^G (E_B))}_{\C [X_\nr (T)]} V$ under the equivalence
of categories \eqref{eq:4.equiv} for $T$. In other words, $\mr{Res}^{\End_G (I_P^G (E_B))
}_{\C [X_\nr (T)]} V$ is the map $\mc O (X_\nr (T)) \to \C$ coming from evaluation at
$|\alpha |_F^{-1/2} \in X_\nr (T)$. The element $h_\alpha^\vee = \alpha^\vee (\varpi_F)$, from
Proposition \ref{prop:3.5} and \cite{Hei2}, satisfies
\[
|\alpha |_F^{-1/2} (h_\alpha^\vee) = | \alpha (\alpha^\vee (\varpi_F)) |_F^{-1/2} = 
|\varpi_F |_F^{-1} = q_F .
\]
Thus $\log |t| = \log (q_F) \alpha^\sharp / 2$ for the unique $\C [X_\nr (T)]$-weight 
$t = |\alpha |_F^{-1/2}$ of $V$.
As $\log (q_F) > 0$, we find that $V$ is an anti-tempered $\End_G (I_P^G (E_B))$-module.\\

Now that we have seen all this, it is clear how the problem can be fixed: we replace every
$h_\alpha^\vee \in M_\sigma^2 / M^1$ by $-h_\alpha^\vee$. Equivalently, the requirement
\begin{equation}\label{eq:A.2}
\nu_F (\alpha (h_\alpha^\vee)) > 0 \quad \text{is substituted by} \quad 
|\alpha (h_\alpha^\vee) |_F > 1 .
\end{equation}
It is best to replace $H_M$ by $-H_M$ at the same time, so that $H_M (h_\alpha^\vee)$ does not 
change. That leads to a definition of $H_M$ which is in any case more common: 
\begin{equation}\label{eq:A.3}
q_F^{\langle H_M (m), \gamma \rangle} = |\gamma (m)|_F \qquad m \in M, \gamma \in X^* (M).
\end{equation}
We point out that the specific formulas for $H_M$ and for $h_\alpha^\vee$ as element
of $G$ are never actually used in this paper. They only play an implicit role in the part of
Paragraph \ref{par:preserv} after Corollary \ref{cor:9.9}, because only that part makes use of
\cite{Hei3,SolComp}. Hence our entire paper remains valid with \eqref{eq:A.2} and \eqref{eq:A.3}
instead of the conventions just before Proposition \ref{prop:3.5}. With this improvement
Proposition \ref{prop:9.3} and Theorem \ref{thm:9.2} really become valid as stated. The same
goes for the results that use those two, namely Theorems \ref{thm:C}, \ref{thm:D}, \ref{thm:E}, 
\ref{thm:9.4} and \ref{thm:9.8}.

The same improvement could be used in \cite{Hei2,Hei3}, that would leave everything in \cite{Hei2}
valid and would repair the issue for \cite{Hei3}.

\end{document}